\documentclass[11pt,reqno]{amsart}
\usepackage{amsfonts,amssymb,amsmath}
\usepackage{enumitem, datetime}

\usepackage{xcolor}

\numberwithin{equation}{section} 
 
\newtheorem{theorem}{Theorem}[section] 
\newtheorem{proposition}[theorem]{Proposition} 
\newtheorem{corollary}[theorem]{Corollary} 
\newtheorem{lemma}[theorem]{Lemma} 

\theoremstyle{definition}
 
\newtheorem{remark}[theorem]{Remark} 

\usepackage[colorlinks=true,citecolor=cyan,urlcolor=blue,linkcolor=blue]{hyperref}

\newcommand*{\bigchi}{\mbox{\Large$\chi$}}

\def\D{\mathbb {D}}
\def\R{\mathbb {R}}
\def\C{\mathbb {C}}
\def\N{\mathbb {N}}
\def\H{\mathbb {H}}
\def\Z{\mathbb {Z}}
\def\O{\mathbb {O}}
\def\E{\mathbb {E}}

\def\sgn{\mathbf {sgn}}
\def\Db{\mathbf {D}}
\def\Ab{\mathbf {A}}
\def\ib{\mathbf {i}}
\def\jb{\mathbf {j}}
\def\kb{\mathbf {k}}

\def\d{\mathbf{ d}}

\def\ZC{\mathcal {Z}}
\def\CC{\mathcal {C}}
\def\PC{\mathcal {P}}
\def\NC{\mathcal {N}}
\def\OC{\mathcal {O}}
\def\BC{\mathcal {B}}
\def\YC{\mathcal {Y}}
\def\EC{\mathcal {E}}

\def\HC{\mathcal {H}}
\def\SC{\mathcal {S}}
\def\DC{\mathcal {D}}
\def\AC{\mathcal {A}}
\def\GC{\mathcal {G}}

\def\g{\mathfrak {g}}
\def\h{\mathfrak {h}}
\def\p{\mathfrak {p}}
\def\k{\mathfrak {k}}
\def\s{\mathfrak {s}}
\def\a{\mathfrak {a}}

\def\u{\mathfrak {u}}
\def\o{\mathfrak {o}}
\def\l{\mathfrak {l}}
\def\z{\mathfrak {z}}

\def\m{\mathfrak {m}}
\def\b{\mathfrak b}
\def\<>{\langle \cdot ,\, \cdot \rangle}
 
\begin{document} 
 
\title[Second cohomology groups of nilpotent orbits]{The second cohomology groups of nilpotent orbits in classical Lie algebras} 

\author[I. Biswas]{Indranil Biswas} 
\address{School of Mathematics, Tata Institute of 
Fundamental Research, Homi Bhabha Road, Mumbai 400005, India}
\email{indranil@math.tifr.res.in} 

\author[P. Chatterjee]{Pralay Chatterjee}
\address{The Institute of Mathematical Sciences, HBNI, CIT Campus, 
Tharamani, Chennai 600113, India}
\email{pralay@imsc.res.in} 

\author[C. Maity]{Chandan Maity}
\address{The Institute of Mathematical Sciences, HBNI, CIT Campus, 
Tharamani, Chennai 600113, India}
\email{cmaity@imsc.res.in} 


\begin{abstract} 
The second de Rham cohomology groups of nilpotent orbits in non-compact real forms of classical complex simple
Lie algebras are explicitly computed. Furthermore, the first de Rham cohomology groups of nilpotent orbits in
non-compact classical simple Lie algebras are computed; they are proven to be zero for nilpotent orbits 
in all the complex simple Lie algebras. A key component in these computations is a description of the
second and first cohomology groups of homogeneous spaces of general connected Lie groups which is obtained here.
This description, which generalizes a previous theorem of the first two authors, may be of independent interest.
\end{abstract}

\maketitle 

\begingroup
\def\addvspace#1{}
\tableofcontents
\endgroup

\section{Introduction}

\subsection{Context and motivations}
The nilpotent orbits in the semisimple Lie algebras, under the adjoint action of the associated
semisimple Lie groups, form a rich class of homogeneous spaces which are extensively studied; 
see the book \cite{CoM} and the relatively recent article \cite{M} for an account. Nevertheless, surprisingly 
for a very long period there seems to have been hardly any literature on the topological invariants of 
such orbits other than the description of fundamental groups in the case of simple 
Lie algebras. The computation of fundamental groups of such orbits 
is attributed to T. Springer, R. Steinberg \cite{SS} for the classical case and A. Alexeevski \cite{A} for the complex exceptional case; 
see \cite[Cor 6.1.6, p. 91, pp. 128--134]{CoM}, \cite[pp. 229--230]{M}.
We also refer to the works of D. King \cite{Ki} and E. Sommers \cite{So} in this regard.

It was only recently attention has been drawn to topological 
invariants of such orbits  other than the fundamental groups. In \cite{Ju}, D. Juteau 
describes the integral cohomology  of the minimal nontrivial nilpotent orbit in complex 
simple Lie algebras, and in \cite{Cr}, P. Crooks  computed the equivariant rational 
cohomology algebras of the regular and minimal nilpotent orbits in complex simple Lie 
algebras under the action of a maximal (compact) torus. Contribution towards this by 
some of us began in \cite{BC} where the second cohomology groups of the nilpotent 
orbits in all the complex simple Lie algebras, under the adjoint action of the 
corresponding complex group, are computed; see \cite[Theorems 5.4, 5.5, 5.6, 5.11, 5.12]{BC}. 
Recently in \cite{CM} the second cohomology groups of nilpotent orbits are computed for 
most of the nilpotent orbits in non-compact non-complex exceptional Lie algebras, and 
for the rest of cases of nilpotent orbits, which are not covered in the above 
computations, upper bounds for the dimensions of the second cohomology groups 
are obtained; see \cite[Theorems 3.2--3.13]{CM}.

Each orbit in a semisimple Lie algebra under the adjoint action is equipped with 
the Kostant-Kirillov two form. In \cite{BC}, besides finding invariants other 
than the fundamental groups, further motivation for initiating computation of 
the second cohomology groups naturally stem from the exactness criterion obtained 
in \cite[Proposition 1.2]{BC} for the Kostant-Kirillov two-forms on adjoint orbits of 
arbitrary elements in real semisimple Lie groups with 
semisimple maximal compact subgroups. It should also be mentioned that 
\cite[Proposition 1.2]{BC} generalises \cite[Theorem 1.2]{ABB}  where the 
exactness criterion are obtained for the Kostant-Kirillov form on an adjoint 
orbit of a semisimple element in a complex semisimple Lie algebra.

\subsection{Description of our results}
As it is always the case with non-compact real forms of the complex simple Lie
algebras, when the analogues questions as above are cast in this setting, 
because of their inherent complexities the computations become much more 
harder compared to the case of complex simple Lie algebras; see \S \ref{complexities}
for the precise description of the intricacies that occur in this case.
In the present paper we carry forward what is initiated in \cite{BC}, 
and using involved computations we describe the second and the first cohomology groups 
of all the nilpotent orbits in all the real classical Lie algebras which are 
non-complex and non-compact. 

We first obtain a Lie theoretic formulation of the first and second cohomologies of 
nilpotent orbits in Theorem \ref{thm-nilpotent-orbit} when the underlying group 
is the connected component of the $\R$ -points of a $\R$-simple algebraic group.
Applying this result, and developing the 
computational machineries in \S \ref{sec-Second-Cohomology-of-Nilpotent-Orbits}
 we next obtain the second cohomology groups 
of all the nilpotent orbits in Theorems \ref{sl-n-R}, \ref{sl-n-H}, \ref{su-pq}, \ref{so-pq}, 
\ref{so-pq-2}, \ref{so*}, \ref{sp-n-R}, \ref{sp-pq}. Our computations, in 
particular, yield that the second cohomology groups vanish for all the 
nilpotent orbits in $\s\l_n (\H)$, and for most of the nilpotent orbits 
in $\s\l_n (\R)$.  In the case of $\s\u (p,q)$ the second cohomology 
groups of the nilpotent orbits depend on their signed Young diagrams. 
It also follows that for the remaining cases of non-compact real forms 
of classical complex Lie algebras the second cohomology groups of the 
nilpotent orbits are related in a straightforward way to the partition 
component of the associated signed Young diagrams. 

The methods involved 
above also steered us to describing the first cohomology groups in all 
the simple Lie algebras except the non-compact non-complex exceptional 
ones. Using the generalities in Theorem \ref{thm-nilpotent-orbit} it 
follows that the first cohomology groups of nilpotent orbits in 
non-compact simple Lie algebras are at the most one dimensional.  
However, more precise descriptions of the first cohomology groups 
of nilpotent orbits are obtained in Theorems \ref{complex-simple-H1}, 
\ref{sl-n-H-H1}, \ref{sl-n-R-H1}, \ref{su-pq-H1},
\ref{so-pq-H1}, \ref{so-pq-H1-2}, \ref{so*-H1}, \ref{sp-n-R-H1}.
In particular, we obtain that the first cohomology groups of nilpotent 
orbits vanish in the cases of  all the complex simple Lie algebras and 
in $\s\l_n (\H)$, $\s\p (p,q)$, $\s\l_n (\R)$ when $n \geq 3$, $\s\o (p,q)$ 
when $p \neq 2, q \neq 2$ and $(p,q) \neq (1,1)$. We also refer to \cite{CM} 
for some results on the dimensions of the first cohomology groups
of nilpotent orbits in non-compact non-complex exceptional Lie algebras.

The computations described above necessitate a description of the second
 cohomology groups of homogeneous spaces of 
Lie groups in full generality, which we formulate in Theorem 
\ref{thm-i1}, generalizing \cite[Theorem 3.3]{BC}. We also
obtain a description of the first cohomology groups in Theorem \ref{H^1} in the same general
set-up as above. In \cite{BC} to facilitate the computations in complex simple Lie algebras a suitable 
description of the second cohomology group of any homogeneous space of a connected Lie
group was obtained in \cite[Theorem 3.3]{BC} under the assumption that all the maximal
compact subgroups of the Lie group are semisimple; see \cite{BC1} for a relatively simple
proof of \cite[Theorem 3.3]{BC}. As only the complex simple Lie groups were considered in 
\cite{BC}, this condition was not restrictive because the maximal compact subgroups in 
complex simple Lie groups are in fact simple Lie groups. However, in the present case of
non-compact real forms of complex simple Lie algebras, the maximal compact subgroups of 
the associated simple real Lie groups  are not necessarily semisimple, and hence 
\cite[Theorem 3.3]{BC} can not be applied anymore, in general, to do the computations. 
Hence we need to invoke Theorem \ref{thm-i1} and Theorem \ref{H^1} where no additional conditions 
are imposed on the maximal compact subgroups therein.  
Although general theories of cohomology groups of (compact) homogeneous spaces
are extensively studied in the past (see, for example \cite{Bo1}, \cite{CE}, \cite{GHV}) 
we are unable to locate Theorem \ref{thm-i1} and Theorem \ref{H^1} in the literature and 
hence proofs are included for the sake of completeness. On the other hand, in view of their
applicability the above results  may be of independent interest as they are general and hold
under a very mild restriction. We deduce stronger consequences of the above results 
in the special cases when the ambient Lie groups are complex semisimple or real simple; 
see Corollary \ref{H.1} and Theorem \ref{H.21}.

\subsection{Underlying complexities and our strategy}\label{complexities}
We now describe the underlying complexities that arise in the present setting of 
non-compact non-complex classical Lie algebras, and briefly sketch the strategy 
that is followed in this paper to overcome them. 

In our setting without loss of generality the simple real Lie groups $G$ may as well be assumed 
to be the connected component of $\R$-points of $\R$-simple algebraic groups defined 
over $\R$. We set $\g:= {\rm Lie}\, G$. In this setting the Lie theoretic formulation is obtained 
in Theorem \ref{thm-nilpotent-orbit}; the computations in \cite{CM} also use this result crucially.
Let $ X\in \g$ be a non-zero nilpotent element, and $\ZC_G(X)$ be the centralizer of $X$ in $G$. Let
$M$ be a maximal compact subgroup in $G$, and $K$ be a maximal compact subgroup in 
$\ZC_G(X)$ such that $K\subset M$. Let $\m,\, \z(\k)$ 
be the Lie algebras of $M$ and the center of $K$, respectively. Then Theorem 
\ref{thm-nilpotent-orbit} says that the computation of the second cohomology of the nilpotent orbits 
boils down to understanding the action of the component group $K/K^\circ$ on the subalgebra $\z(\k)\cap[\m,\m]$. 
Thus, when $M$ is semi-simple this amounts to describing the action of $K/K^\circ$ on 
$\z(\k)$, and hence knowing the isomorphism class of $K$ is enough to compute the second 
cohomology groups in this case. However, when $M$ is not semisimple it is not enough to know the 
isomorphism classes of $K$ and $M$, rather one also needs to understand the embedding 
of $K$ in $M$; the case of ${\s\u}(p,q)$ as in Section \ref{sec-su-pq} a typical example of such a situation.
Although in the setting of non-compact classical groups the isomorphism 
classes of $M$ are well-known, and the isomorphism classes of $K$ can be obtained 
immediately either using Lemma \ref{reductive-part-comp} or by the work of Springer 
and Steinberg \cite{SS}, hardly anything can be concluded, from these isomorphism 
classes, on how $K$ is embedded in $M$. 
Thus, one needs to go above and beyond the computations done in \cite{SS}
to find out this embedding explicitly for all the nilpotent orbits in $\g$, and 
major parts of the Sections \ref{sec-su-pq}, 
\ref{sec-so-pq}, \ref{sec-so*} and \ref{sec-sp-n-R} are devoted to do this.

\subsection{Organization of the paper}
The paper is organized as follows. In Section \ref{sec-notation} some notation and 
terminologies are fixed, and some backgrounds needed in the paper are recalled. 
Section \ref{sec-second-first-cohomologies-of-homogeneous-spaces} deals with the 
second and first cohomology groups of general homogeneous spaces of connected Lie groups; Theorem
\ref{thm-i1} and Theorem \ref{H^1} are proved, and some corollaries of them are
drawn which are of independent interest.
 In Section \ref{sec-Second-Cohomology-of-Nilpotent-Orbits} the 
second cohomology groups are computed in Theorems \ref{sl-n-R}, \ref{sl-n-H}, \ref{su-pq},
\ref{so-pq}, \ref{so-pq-2}, \ref{so*}, \ref{sp-n-R} and \ref{sp-pq}. In Section 
\ref{sec-First-Cohomology-of-Nilpotent-Orbits} the first cohomology groups are computed in 
Theorems \ref{complex-simple-H1}, \ref{sl-n-H-H1}, \ref{sl-n-R-H1}, \ref{su-pq-H1}, 
\ref{so-pq-H1}, \ref{so-pq-H1-2}, \ref{so*-H1} and \ref{sp-n-R-H1}.
In the Appendix we workout certain details on nilpotent orbits that are used in Section 
\ref{sec-Second-Cohomology-of-Nilpotent-Orbits}. 

The computations in the Appendix 
lead us to Remark \ref{CM-correction} where we rectify an error in 
\cite[p. 139, Lemma 9.3.1]{CoM} related to the parametrizations of nilpotent orbits.

\section{Notation and background}\label{sec-notation}

In this section we fix some notation, and recall some backgrounds which will be used throughout in this paper.
Subsequently, a few specialized notation are mentioned as and when they occur.

Once and for all fix a square root of $-1$ and call it $\sqrt{-1}$. The {\it center} 
of a group $G$ is denoted by $\ZC(G)$ while the {\it center} of a Lie algebra
$\g$ is denoted by $ \z (\g)$. The Lie groups will be denoted by the capital letters, 
while the Lie algebra of a Lie group will be denoted by the corresponding lower case 
German letter, unless a different notation is explicitly mentioned. Sometimes, for 
notational convenience, the Lie algebra of a Lie group $G$ is also denoted by 
${\rm Lie} (G)$. The connected component of $G$ containing the identity element 
is denoted by $G^{\circ}$. For a subgroup $H$ of $G$ and a subset $S$ of $\g$, the
subgroup of $H$ that fixes $S$ pointwise under the adjoint action is called the {\it centralizer}
of $S$ in $H$; the centralizer of $S$ in $H$ is denoted by $\ZC_{H} (S)$. Similarly, for a
Lie subalgebra $\h \,\subset\, \g$ and a subset $S \,\subset\, \g$, by $\z_\h (S)$ we will denote
the subalgebra of $\h$ consisting of all the elements that commute with every element of $S$. 

Let $\Gamma$ be a group acting linearly on a vector space $V$. The subspace of $V$ fixed pointwise
by the action of $\Gamma$ is denoted by $V^\Gamma$. If $G$ is a Lie group with Lie
algebra $\g$, then it is immediate that the adjoint (respectively, coadjoint) action
of $G^{\circ}$ on $\z(\g)$ (respectively, $\z(\g)^*$) is trivial; in particular, one
obtains a natural action of $G/G^{\circ}$ on $\z(\g)$ (respectively, $\z (\g)^*$). We
denote by $[\z (\g)]^{G/G^{\circ}}$ (respectively, $[\z (\g)^*]^{G/G^{\circ}}$) the space of points
of $\z (\g)$ (respectively, of $\z (\g)^*$) fixed pointwise under the action of $G/G^{\circ}$. 

For a real semisimple Lie group $G$, an element $X \in \g$ is called {\it nilpotent} if
the linear endomorphism $${\rm ad}(X)\,:\, \g \,\longrightarrow \,\g$$ is nilpotent.
The {\it set of nilpotent elements} in $\g$ is denoted by ${\NC}_{\g} $. 
A {\it nilpotent orbit} is an orbit of a nilpotent element in $\g$ under the adjoint
action of $G$. The orbit of a nilpotent element $X \,\in\, \g$ is denoted by $\OC_X$.
The {\it set of all nilpotent orbits in $\g$ under the adjoint action of $G$} is denoted by ${\NC}(G)$.

\subsection{\texorpdfstring{$\epsilon$-$\sigma$ Hermitian forms and associated 
groups}{Lg}}\label{sec-epsilon-sigma-forms}

The notation $\D$ will stand for either $\R$ or $ \C $ or $ \H$, unless mentioned otherwise. 
Take a right vector space $V$ defined over $\D$.
Let ${\rm End}_\D (V)$ be the right $\R$-algebra of {\it $\D$-linear endomorphisms} of $V$,
and let ${\rm GL(V)}$ be the {\it group of invertible elements} of ${\rm End}_\D (V)$.
For a $\D$-linear endomorphism $T \,\in\, {\rm End}_\D (V)$ and an ordered $\D$-basis $\BC$ of $V$,
the {\it matrix of $T$ with respect to $\BC$} is denoted by $[T]_{\BC}$.
Recall that if $\D\,=\, \C$
then ${\rm End}_\D(V)$ is also a (right) $\C$-algebra. 
When $\D$ is either $\R$ or $\C$, let
$$
{\rm tr} \,:\, {\rm End}_\D (V) \,\longrightarrow\, \D\, \ \ \text{ and }\ \
{\rm det} \,:\, {\rm End}_\D V \,\longrightarrow\, \D
$$
respectively be the usual {\it trace} and {\it determinant} maps. Let
$A$ be a central simple $\R$-algebra. Let
$$\text{Nrd}_A \,:\, A \,\longrightarrow\, \R$$ be the {\it reduced norm} on $A$, and
let $\text{Trd}_A \,:\, A \,\longrightarrow\, \R$ be the {\it reduced trace} on $A$.
When $\D \,=\, \R$ or $ \C$, define 
$${\rm SL}(V)\,:=\, \{ z \,\in \,{\rm GL}(V) \, \mid \, \text{det}(z)\,= \,1\} \ \ \text{ and } 
\ \ \s\l (V ) \,:=\, \{ y \,\in\, {\rm End}_\D (V) \, \mid \, \text{tr}(y)\,=\, 0 \}\, .$$ 
If $\D \,= \,\H$ then recall that ${\rm End}_\D(V)$ is a central simple $\R$-algebra. In that
case, define
$${\rm SL}(V)\,:=\, \{ z \,\in\, {\rm GL}(V) \, \mid \, \text{Nrd}_{ {\rm End}_\D V} (z)\,=\, 1\}$$
and
$$\s\l (V )\,:=\, \{ y\,\in\, {\rm End}_\D(V) \,\mid \,
\text{Trd}_{ {\rm End}_\D (V)} (y)\,=\, 0\}\, .$$

Let $\D$ be $\R$, $\C$ or $\H$, as above.
Let $\sigma$ be either the identity map $\text{Id}$ or an {\it involution} of $\D$,
meaning $\sigma$ is $\R$-linear with $\sigma^2 \,=\, \text{Id}$ and 
$\sigma (xy) \,=\, \sigma (y) \sigma (x)$ for all $x,\, y \,\in\, \D$. Let $\epsilon \,=\, \pm 1$.
Following \cite[\S~23.8, p. 264]{Bo} we call a map $$\langle \cdot,\,
\cdot \rangle \,\colon\, V \times V \,\longrightarrow\, \D$$ a
$\epsilon$-$\sigma$ {\it Hermitian form} if
\begin{itemize}
\item $ \langle \cdot,\, \cdot \rangle $ is additive in each argument, 

\item $ \langle v,\, u \rangle \,= \, \epsilon \sigma( \langle u, v \rangle)$, and

\item $ \langle v \alpha,\, u \rangle \,=\, \sigma (\alpha) \langle v,\, u \rangle$ for all
$v,\,u \,\in\, V$ and for all $\alpha \,\in\, \D$.
\end{itemize}

A $\epsilon$-$\sigma$ Hermitian form $ \langle \cdot, \, \cdot \rangle $ is called {\it
non-degenerate} if $ \langle v,\, u \rangle \,=\,0 $ for all $v$ if and only if $u \,=\, 0$.
All $\epsilon$-$\sigma$ Hermitian forms considered here will be assumed to be non-degenerate.

We define $${\rm U} (V,\, \langle \cdot,\, \cdot \rangle ) \,:=\,
\{T \,\in\, {\rm GL}(V) \, \mid \, \langle Tv ,\, Tu \rangle \,= \, \langle v ,
\, u \rangle ~~ \forall~~ v,u \,\in \,V \}$$ and
$$\u (V,\, \langle \cdot, \,\cdot \rangle ) \,:=\, \{T\,\in\, {\rm End}_\D(V) \, \mid \, 
\langle Tv ,\, u \rangle + \langle v ,\, Tu \rangle \,= \,0 ~~\forall~~ v,u \,\in \,V \}\, .$$ 
We next define $${\rm SU} (V, \langle \cdot, \cdot \rangle )\,: =\,
{\rm U} (V, \langle \cdot, \cdot \rangle ) \cap {\rm SL}(V)\ \ \text{ and } \ \
\s\u (V, \langle \cdot, \cdot \rangle ) \,:=\,
\u (V, \langle \cdot, \cdot \rangle ) \cap \s\l (V)\, .$$ Recall that 
$\s\u (V, \,\langle \cdot,\,\cdot \rangle )$ is
a simple Lie algebra (cf. \cite[Chapter I, Section 8]{K}).

If $\D \,=\, \C$, then multiplication by $\sqrt{-1}$ sends the non-degenerate $\epsilon$-$\sigma$
Hermitian forms on $V$ with $\epsilon\,=\, -1$,
$\sigma\,\neq\, \text{Id}$ to the non-degenerate $\epsilon$-$\sigma$ Hermitian forms with
$\epsilon \,=\, 1$, $\sigma \,\neq\, \text{Id}$, and this mapping is a bijection.
Hence, when $\D \,=\, \C$ and $\sigma \,\neq\, \text{Id}$, we consider only the case where
$\epsilon \,=\, 1$. 
If $\D \,=\, \H$ and $\sigma \,= \,\text{Id}$, then it can be easily seen
that $\langle \cdot,\, \cdot \rangle \,\equiv\, 0$. As $ \langle \cdot,\, \cdot \rangle$ 
is assumed to be non-degenerate, this, in particular, implies 
that there is no form $ \langle \cdot, \, \cdot \rangle$ on $V$ with $\D \,=\, \H,\, \sigma
\,=\, \text{Id}$.

We define the {\it usual conjugations} $\sigma_c$ on $\C$ by
$\sigma_c (x_1 + \sqrt{-1}x_2 ) \,=\, x_1-\sqrt{-1}x_2$, and on 
$\H$ by $\sigma_c (x_1 + \ib x_2 + \jb x_3 + \kb x_4 ) \,=\, x_1 - \ib x_2 - \jb x_3 - \kb x_4$, $x_i \in \R$ for $i= 1, \cdots , 4$.
Define $|z| \,:=\, (z \sigma_c (z))^{1/2}$, for $z \,\in\, \D$.
For $\alpha \,\in\, \H^*$ define
$$C_\alpha \,\colon\, \H \,\longrightarrow\, \H\, ,\ \ x\,\longmapsto\,
\alpha x \alpha^{-1}\, .$$
Clearly $C_\alpha $ is a $\R$-algebra automorphism, and $C_\alpha \,=\, C_{\alpha/ |\alpha|}$.
When $\D \,= \,\H$, the following facts justify that it is enough to consider the involution
$\sigma_c$ instead of arbitrary involutions.
The proof of the next lemma is a straightforward application of Skolem-Noether theorem which can
be found in \cite[\S~23.7, p. 262]{Bo}. 

\begin{lemma}[{cf.\,\cite[\S~23.7, p.~262]{Bo}}]\label{involution1}
Let $\sigma \,:\, \H \,\longrightarrow\, \H$ be $\R$-linear with 
$\sigma (xy) \,=\, \sigma (y) \sigma (x)$ for all $x,\,y \,\in\, \H$. 
Then $\sigma$ is an involution, meaning $\sigma^2 \,= \,{\rm Id}$, if and only if 
either $\sigma \,= \,\sigma_c$ or $\sigma \,= \,C_{\alpha} \circ \sigma_c$ for some
$\alpha $ with $\alpha^2 \,=\,-1$.
\end{lemma}

\begin{lemma}[{cf.\,\cite[\S~23.8, p.~264]{Bo}}]\label{involution2}
Let $\sigma \,:\, \H \,\longrightarrow\, \H$ be an involution, $\epsilon \,=\, \pm 1$ and
$$\langle \cdot,\, \cdot \rangle \,:\, V \times V \,\longrightarrow\, \D$$ a
$\epsilon$-$\sigma$ Hermitian form. Let $\delta \,=\, \pm 1$ and $\alpha\,\in\, \H$ such that
$|\alpha |\,=\,1$ and $\alpha \sigma (\alpha)^{-1} \,= \,\delta$. Then $\alpha \langle \cdot,\,
\cdot \rangle$ is a $\delta \epsilon$-$C_{\alpha} \circ \sigma$ Hermitian form.
\end{lemma} 

As a consequence of Lemma \ref{involution2} if $\sigma \,:\, \H \,\longrightarrow\, \H$ is an
involution, $\epsilon \,= \,\pm 1$ and $$\langle \cdot,\, \cdot \rangle \,: \,V \times V
\,\longrightarrow\, \D$$ a
$\epsilon$-$\sigma$ Hermitian form, then $\alpha \langle \cdot,\, \cdot \rangle$ is a
$\epsilon$-$ \sigma_c$ Hermitian form with $\alpha \,\in\, \H$ being such that 
$ \sigma \,=\, C_{\alpha} \circ \sigma_c$ and $\alpha^2 \,=\,-1$ (as in Lemma \ref{involution1}).
In particular, an immediate consequence is that if $\sigma$, $\epsilon$ and
$\langle \cdot, \,\cdot \rangle$ are as above, then there exists a 
$\epsilon$-$ \sigma_c$ Hermitian form, say, $\langle \cdot,\,
\cdot \rangle' \,:\, V \times V \,\longrightarrow\, \D$ such that 
${\rm SU} (V,\, \langle \cdot,\, \cdot \rangle ) \,=\, {\rm SU} (V,\, \langle \cdot,\,
\cdot \rangle' )$ and $\s\u (V,\, \langle \cdot,\, \cdot \rangle )\,= \,\s\u (V,\, \langle \cdot,\,
\cdot \rangle' )$.
In view of the above observations, without loss of generality, we may only consider
the involution $\sigma_c$. {\it From now on we will restrict to the
involution $\sigma_c$ instead of arbitrary involutions on $\D$}.

The case where $\D\,=\, \C$, $\sigma \,= \,\text{Id}$ and $\epsilon \,=\, \pm 1$ is already
investigated in \cite{BC}. Here the remaining three cases
\begin{enumerate}
\item $\D \,=\, \R$, $\sigma \,= \, \text{Id}$ and $\epsilon \,=\, \pm 1$,

\item $\D \,=\, \C$, $\sigma \,=\, \sigma_c$ and $\epsilon \,=\, 1$, and

\item $\D\,=\,\H$, $\sigma \,=\, \sigma_c$ and $\epsilon \,=\, \pm 1$
\end{enumerate}
will be investigated.

We next introduce certain standard nomenclature associated to the specific values of
$\epsilon$ and $\sigma$. If
$\sigma\,=\, \sigma_c$ and $\epsilon \,=\,1$, then $\langle \cdot ,\, \cdot \rangle$
is called a {\it Hermitian} form.
When $\sigma \,= \,\sigma_c$ and $\epsilon \,=\,-1$, then $\langle \cdot ,\, \cdot \rangle$
is called a {\it skew-Hermitian} form. 
The form $\langle \cdot, \, \cdot \rangle$ is called {\it symmetric} if
$\sigma \,=\, \text{Id}$ and $\epsilon \,=\,1 $. Lastly, if
$\sigma \,=\, \text{Id}$ and $\epsilon \,=\, -1 $, then $\langle \cdot, \, \cdot \rangle$ 
is called a {\it symplectic} form.
If $\langle \cdot, \, \cdot \rangle$ is a symmetric form on $V$, define
$${\rm SO} (V,\, \langle \cdot,\, \cdot \rangle)\,:=\, {\rm SU} (V, \,\langle \cdot,\, \cdot \rangle )
\ \ \text{ and }\ \ 
\s\o (V,\, \langle \cdot,\, \cdot \rangle )\,:=\, \s\u (V,\, \langle \cdot,\, \cdot \rangle)\, .$$
Similarly, if $\langle \cdot, \, \cdot \rangle$ is a symplectic form on $V$, then define
$${\rm Sp} (V,\, \langle \cdot,\, \cdot \rangle)\,:=\, {\rm SU} (V,\, \langle \cdot,\, \cdot \rangle )
\ \ \text{ and }\ \ \s\p (V, \,\langle \cdot, \,\cdot \rangle) \,:=\, \s\u (V,\,
 \langle \cdot,\, \cdot \rangle )\, .$$
When $\D\,=\,\H$ and $\<>$ is a skew-Hermitian form on $V$, define
$${\rm SO}^*(V,\, \langle \cdot,\, \cdot \rangle)\,:=\,{\rm SU}(V,\, \langle \cdot,\, \cdot \rangle )
\ \ \text{ and } \ \ \s\o^* (V,\, \langle \cdot,\, \cdot \rangle )\,:=\, \s\u (V,\, \langle \cdot,\,
\cdot \rangle )\, .$$

As before, $V$ is a right vector space over $\D$. We now introduce some terminologies associated to
certain types of $\D$-basis of $V$. When either $\D \,=\,\R,\, \sigma \,= \,\text{Id}$ or
$ \D\,= \,\C, \,\sigma \,=\, \sigma_c$ or $ \D \,=\, \H,\, \sigma \,=\, \sigma_c$, for a 
$1$-$\sigma$ Hermitian form $\<>$ on $V$, an orthogonal basis $\AC$ of $V$ is called
{\it standard orthogonal} if $\langle v,\, v \rangle \,=\, \pm 1$ for all
$v \,\in \,\AC$. For a standard orthogonal basis ${\AC}$ of $V$,
set $$p \,:=\, \# \{ v \,\in\, \AC \,\mid\, \langle v, \,v\rangle \,=\,1\} \ \ \text{ and }\ \
q \,:=\, \# \{ v \,\in \,\AC \,\mid\, \langle v, \,v\rangle \, =\, -1\}\, .$$ 
The pair $(p,\,q)$, which is independent of the choice
of the standard orthogonal basis ${\AC}$, is called the {\it signature} of $\<>$. When
$\D \,=\, \C$ and $\sigma \,=\, \sigma_c$, if $\<>$ is a skew-Hermitian form on $V$ then
$\sqrt{-1}\<>$ is a Hermitian form on $V$; in this case the {\it signature} of $\<>$ is defined to
be the signature of the Hermitian form $\sqrt{-1}\<>$.

In the case where $\D \,=\,\R$ or $\C$, $\sigma \,= \,\text{Id}$ and $\epsilon \,= \,-1$, the
dimension $\dim_\D V$ is an even number. 
Let $2n \,=\, \dim_\D V$. In this case an ordered basis $\BC \,:=\, (v_1,\, \cdots, \,v_n;\, v_{n+1},
\, \cdots ,\, v_{2n})$ of $V$ is said to be {\it symplectic} if $\langle v_i,\, v_{n+i} \rangle
\,=\, 1$ for all $1\,\leq\, i \,\leq\, n$ and
$\langle v_i,\, v_j \rangle \,=\, 0$ for all $j \,\neq\, n +i$. 
The ordered set $(v_1,\, \cdots ,\, v_n)$ is called the {\it positive part} of $\BC$ and it
is denoted by $\BC_+$.
Similarly, the ordered set $(v_{n+1},\, \cdots ,\, v_{2n})$ is called the {\it negative part} of
$\BC$, and it is denoted by $\BC_-$. The {\it complex structure on $V$ associated to the above
symplectic basis $\BC$} is defined to be the $\R$-linear map $$J_{\BC} \,:\, V
\,\longrightarrow\, V\, , \ \ v_i\,\longmapsto\, v_{n+i}\, ,\ \ v_{n+i}\,\longmapsto\, -v_{i}
\ \ \forall~ \ 1\,\leq\, i \,\leq\, n\, .$$
If $\D \,=\,\H$ and $\<>$ is a skew-Hermitian form on $V$,
an orthogonal $\H$-basis $$\BC\,:=\, (v_1,\, \cdots,\, v_m)$$
of $V$ ($m \,:=\, \dim_\H V$) is said to be {\it standard orthogonal} if
$\langle v_r,\, v_r \rangle \,=\, \jb$ for all $1 \,\leq\, r \,\leq\, m$ and
$\langle v_r, \, v_s \rangle \,=\, 0$ for all $r \,\neq \,s$.

Take $P\,=\, (p_{ij}) \,\in\, {\rm M}_{r\times s}(\D)$. Then $P^t$ denotes the {\it transpose} of $P$. If
$\D \,=\, \C$ or $\H$, then define $\overline{P} \,:=\, (\sigma_c(p_{ij}))$. Let
\begin{equation}\label{defn-I-pq-J-n}
{\rm I}_{p,q} \,:=\, \begin{pmatrix}
 {\rm I}_p  \\
 & -{\rm I}_q
 \end{pmatrix}\, , ~~~ \
 {\rm J}_n \,:=\, \begin{pmatrix}
 & -{\rm I}_n  \\
 {\rm I}_n & 
 \end{pmatrix}\,.
\end{equation}
 The classical groups and Lie algebras that we will be working with are:
 \begin{align*}
 {\rm SL}_n (\R)&:=\, \{g \,\in\, {\rm GL}_n (\R) \,\mid\, \det (g) \,=\,1  \}, \qquad \qquad
 {\s\l}_n (\R)\,:=\, \{z \,\in\, {\rm M}_n (\R) \,\mid\, \text{tr} (z) \,=\,0 \};
 \\
 {\rm SL}_n (\H)&:=\, \{g \,\in\, {\rm GL}_n (\H) \,\mid\, \text{Nrd}_{{\rm M}_n (\H)} (g) \,=\,1 \}, \quad
 {\s\l}_n (\H)\,:=\, \{z \,\in\, {\rm M}_n (\H) \,\mid\, \text{Trd}_{{\rm M}_n (\H)} (z) \,=\,0 \};
 \\
 {\rm SU} (p,q)&:=\, \{g \,\in\, {\rm SL}_{p+q}(\C) \,\mid\, \overline{g}^t{\rm I}_{p,q} g \,= \,{\rm I}_{p,q} \} , \quad
 {\s\u} (p,q)\,:=\, \{z \,\in\, \s\l_{p+q}(\C)\,\mid\, \overline{z}^t{\rm I}_{p,q} + {\rm I}_{p,q} z\,=\,0 \};
 \\
 {\rm SO} (p,q)&:=\, \{g \,\in\, {\rm SL}_{p+q}(\R)\,\mid\, g^t{\rm I}_{p,q} g \,=\, {\rm I}_{p,q} \},\quad
 {\s\o} (p,q)\,:=\, \{z \,\in\, \s\l_{p+q}(\R) \,\mid\, z^t {\rm I}_{p,q} + {\rm I}_{p,q} z \,=\,0 \};
 \\
 {\rm Sp} (p,q)&:=\, \{g \,\in\, {\rm SL}_{p+q}(\H)\, \mid\, \overline{g}^t{\rm I}_{p,q} g \,=\, {\rm I}_{p,q} \},\quad
 {\s\p} (p,q)\,:=\, \{z \,\in\, \s\l_{p+q}(\H) \,\mid\, \overline{z}^t{\rm I}_{p,q} + {\rm I}_{p,q} z \,=\,0 \};
 \\
 {\rm Sp} (n,\R)&:=\, \{g \,\in\, {\rm SL}_{2n}(\R) \,\mid\, g^t {\rm J}_{n} g \,=\,{\rm J}_{n} \}, \quad
 {\s\p} (n,\R)\,:=\, \{z \,\in\, \s\l_{2n}(\R) \,\mid\, z^t {\rm J}_{n} + {\rm J}_{n} z \,=\,0 \};
\\
 {\rm SO}^* (2n)&:=\, \{g \,\in\, {\rm SL}_{n}(\H) \,\mid\, \overline{g}^t \jb{\rm I}_n g\,= \, \jb{\rm I}_n \}, \quad
 {\s\o}^* (2n)\,:=\, \{z \,\in\, \s\l_{n}(\H) \,\mid\, \overline{z}^t \jb {\rm I}_n + \jb {\rm I}_n z \,=\,0 \}.
\end{align*}

For any group $H$, let $H^n_\Delta$ denote the diagonally embedded copy of $H$ 
in the $n$-fold direct product $H^n$. 
Let $V$ be a vector space over $\D$. Define $\mathfrak{d}_V \,: \,{\rm End}_\D (V) \,\longrightarrow\,
\D^*$ to be $\mathfrak{d}_V \,:=\, \det$ if $\D\,=\, \C$ or $\R$, and $\mathfrak{d}_V \,:=\,
{\rm Nrd}_{{\rm End}_\D V}$ if
$\D \,= \,\H$. Let now $V_i$, $1 \,\leq\, i \,\leq\, m$, be right vector spaces over $\D$.
As before, $\D$ is either $\R$ or $\C$ or $\H$. 
For every $1 \,\leq\, i \,\leq\, m$, let $H_i \,\subset \,{\rm GL} (V_i) $ be a 
matrix subgroup. Now define the subgroup
$$
S\big( \prod_i H_i \big) \, :=\, \Big\{ (h_1,\, \cdots,\, h_m) \,\in\, \prod_{i=1}^m H_i \ \Bigm| \ 
\prod_i \mathfrak{d}_{V_i} (h_i) \,=\,1 \Big\}\, \subset\, \prod_{i=1}^m H_i\, .
$$

The following notation will allow us to write block-diagonal square matrices with many blocks in a convenient way.
For $r$-many square matrices $A_i \,\in\, {\rm M}_{m_i} (\D)$, $1 \,\leq\, i \,\leq\, r$, the block diagonal square
matrix of size $\sum m_i \times \sum m_i$, with $A_i$ as the $i$-th 
block in the diagonal, is denoted by $A_1 \oplus \cdots \oplus A_r$. This is also abbreviated as
$\oplus_{i =1}^r A_i$. Furthermore, if $B \,\in\, {\rm M}_m (\D)$ and $s$ is a positive integer,
then denote
$ B_\blacktriangle^s \,:=\, \underset{s\text{-many}}{\underbrace{B \oplus \cdots \oplus B}}$.

The following lemma is a basic fact which readily follows from the Skolem-Noether theorem.

\begin{lemma}\label{H-conjugate}
Let $\alpha,\, \beta \,\in\, \H^*$ be such that ${\rm Re}(\alpha) \,=\,{\rm Re}(\beta)$ and $| \alpha |\,=\,
| \beta|$. Then there exists an element $\lambda \,\in\, \H^*$ with $|\lambda| \,=\,1$ such that
$ \alpha \,= \,\lambda \beta \lambda^{-1}$.
\end{lemma}

\subsection{Some standard results associated to \texorpdfstring{$\s\l_2(\R)$}{Lg}-triples in semisimple Lie algebras}\label{results-sl2-triples}
For a Lie algebra $\g$ over $\R$, a subset $\{X,\,H,\,Y\} \,\subset\, \g$ is said to be a {\it $\s\l_2(\R)$-triple}
if $X \,\neq\, 0$, $[H,\, X] \,=\, 2X$, $[H,\, Y] \,= \, -2Y$ and $[X,\, Y] \,=\,H$. It is immediate that
$\text{Span}_\R \{X,\,H,\,Y\}$ for a $\s\l_2(\R)$-triple $\{X,\,H,\,Y\} \,\subset\, \g$ is a 
$\R$-subalgebra of $\g$ which is isomorphic to the Lie 
algebra $\s\l_2(\R)$. We now recall a well-known result due to Jacobson and Morozov.

\begin{theorem}[{Jacobson-Morozov, cf.~\cite[Theorem~9.2.1]{CoM}}]\label{Jacobson-Morozov-alg}
Let $X\,\in\, \g$ be a non-zero nilpotent element in a real semisimple Lie algebra $\g$. Then there exist
$H,\,Y\, \in\, \g$ such that $\{X,\,H,\,Y\} \,\subset\, \g$ is a $\s\l_2(\R)$-triple.
\end{theorem}

We now state a result relating two $\s\l_2(\R)$-triples with common elements.

\begin{lemma}[{cf.~\cite[Lemma 3.4.4]{CoM}}]\label{ei}
 Let $X$ be a nilpotent element, and let $H$ be a semisimple element, in a Lie algebra $\g$ such that $\{X,\,H,\,Y_1\}$
and $\{X,\,H,\,Y_2\}$ are two $\s\l_2(\R)$-triples in $\g$. Then $Y_1 = Y_2$.
\end{lemma}

We now record an immediate consequence of the above result.

\begin{lemma}\label{comm-XH} 
Let $\{X,\,H,\,Y\}$ be a $\s\l_2$-triple in the Lie algebra $\g$
of a Lie group $G$. Then $\ZC_{G}(X,H) \,= \,\ZC_{G}(X,H,Y)$.
\end{lemma}

\begin{proof}
To prove the lemma it suffices to show that $\ZC_G(X,H)\,\subset\, \ZC_G(X,H,Y)$. Take any $g \,\in\, \ZC_G (X,H)$.
Then $ \{{\rm Ad}(g) X,\,{\rm Ad}(g) H,\, {\rm Ad}(g) Y\} \,=\, \{X,\,H,\,{\rm Ad}(g)Y\}$ is another $\s\l_2$-triple
in $\g$ containing $X$ and $H$. Using Lemma \ref{ei} we have ${\rm Ad}(g)Y \,=\, Y$, implying that $g \,\in\,\ZC_G (X,H,Y)$. 
\end{proof}

\subsection{Some notions associated to finite dimensional \texorpdfstring{$\s\l_2(\R)$}{Lg}-modules}\label{notions-sl2-modules}
Given an endomorphism $T \,\in\, {\rm End}_\R (W)$, where $W$ is a $\R$-vector space, and any $\lambda \,\in\, \R$, set
$$W_{T, \lambda} \,:=\, \{ w \,\in\, W \,\mid\, T w \,= \,w \lambda \}\, .$$ 
Let $V$ be a right vector space of dimension $n$ over $\D$, where $\D$ is, as before, $\R$ or $ \C $ or $ \H$. 
Let $\{X,\,H,\,Y\} \,\subset\, \s\l (V )$ be a $\s\l_2(\R)$-triple. Note that $V$ is also
a $\R$-vector space using the inclusion $\R \,\hookrightarrow\, \D$. Hence $V$ is a module
over $ \text{ Span}_\R \{ X,\,H,\,Y\} \,\simeq\, \s\l_2(\R)$. 
For any positive integer $d$, let $M(d-1)$ denote the sum of all the $\R$-subspaces $A$ of $V$ such that
\begin{itemize}
	\item $\dim_\R A \,= \,d$, and
	
	\item{} $A$ is an
	irreducible $\text{Span}_\R \{ X,H,Y\}$-submodules of $V$.
\end{itemize} 
Then $M(d-1)$ is the {\it isotypical component} of $V$ containing all the irreducible submodules
of $V$ with highest weight $d-1$. Let
\begin{equation}\label{definition-L-d-1}
L(d-1)\,:= \,V_{Y,0} \cap M(d-1)\, .
\end{equation}
As the endomorphisms $X,\,H,\,Y$ of $V$ are $\D$-linear, 
the $\R$-subspaces $M(d-1)$, $V_{Y,0}$ and $L(d-1)$ of $V$ are also $\D$-subspaces. Let
$$t_{d} \,:=\, \dim_\D L(d-1)\, .$$

\subsection{Partitions and (signed) Young diagrams}\label{sec-partition-Young-diagram}

An {\it ordered set of order $n$} is a $n$-tuple $(v_1,\, \cdots ,\,v_n)$, where $v_1,\, \cdots,\, v_n$
are elements of some set, such that $v_i \,\neq\, v_j$ if $i \,\neq\, j$.
If $w \,\in\, \{ v_1,\, \cdots ,\,v_n\}$, then write
$w \,\in \,(v_1,\, \cdots ,\,v_n)$. For two ordered sets
$(v_1,\, \cdots ,\,v_n)$ and $(w_1,\, \cdots ,\,w_m)$, 
the ordered set $(v_1,\, \cdots ,\,v_n,\, w_1,\, \cdots ,\, w_m)$ will be denoted by
$$(v_1, \,\cdots ,\,v_n) \vee (w_1,\, \cdots ,\,w_m)\, .$$
Furthermore, for $k$-many ordered sets $(v^i_1,  \cdots, v^i_{n_i})$, $1 \leq  i \leq k$,
define the ordered set $(v^1_1, \cdots, v^1_{n_1})$ $ \vee \cdots \vee (v^k_1,\, \cdots,\, v^k_{n_k})$ 
to be the following juxtaposition of ordered sets $(v^i_1,\, \cdots, \,v^i_{n_i})$ with increasing $i$:
$$
(v^1_1,\, \cdots,\, v^1_{n_1}) \vee \cdots \vee (v^k_1,\, \cdots,\, v^k_{n_k})
\,:=\, (v^1_1,\, \cdots,\, v^1_{n_1}, \,\cdots,\, v^k_1,\, \cdots,\, v^k_{n_k})\, .
$$

A {\it partition} of a positive integer $n$ is an object of the form
$[ d_1^{t_1 },\, \cdots ,\, d_s^{t_s}]$, where $t_i,\, d_i \,\in\, \N$, $ 1 \,\leq\, i
\,\leq\, s$, such that
$\sum_{i=1}^{s} t_i d_i \,=\, n$, $t_i \,\geq\, 1$ and $d_{i+1}\,> \,d_i \,>\, 0$ for all $i$; see
\cite[\S~3.1, p.~30]{CoM}.
If $[ d_1^{t_1 },\, \cdots ,\, d_s^{t_s}]$ is a partition of $n$ in the above sense
then $t_i$ is called the {\it multiplicity} of $d_i$. Henceforth,
the multiplicity of $d_i$ will be denoted by $t_{d_i}$; this is to avoid any ambiguity.
Let $\PC(n)$ denote the {\it set of all partitions of $n$}. For
a partition $\d\, =\, [ d_1^{t_{d_1}},\, \cdots ,\, d_s^{t_{d_s}} ]$ of $n$,
define 
\begin{equation}\label{Nd-Ed-Od}
{\N}_{\d} \,:=\, \{ d_i \,\mid\, 1 \,\leq\, i \,\leq\, s \}\, ,\ \
{\E}_{\d} \,:=\, {\N}_{\d}\cap (2\N)\, , \ \ 
{\O}_{\d} \,:=\, {\N}_{\d}\setminus {\E}_{\d} .
\end{equation}
Further define 
\begin{equation}\label{Od1-Od3}
{\O}^1_{\d} \,:=\, \{ d \,\mid\, d \,\in\, \O_\d, \ d \,\equiv\, 1 \,\pmod{4} \}\ \
\text{ and }\ \ {\O}^3_{\d} \,:=\, \{ d \,\mid\, d \,\in\, \O_\d, \ d \,\equiv\, 3\, \pmod{4} \} .
\end{equation}
Following \cite[Theorem 9.3.3]{CoM}, a partition
$\d$ of $n$ will be called {\it even} if ${\N}_{\d} \,=\, {\E}_{\d}$. Let $\PC_{\rm even} (n)$
be the subset of $\PC(n)$ consisting of all even partitions of $n$. We call a partition 
$\d$ of $n$ to be {\it very even} if
\begin{itemize}
\item $\d$ is even, and

\item $t_\eta$ is even for all $\eta \in {\N}_{\d}$.
\end{itemize}
Let $\PC_{\rm v. even} (n)$ be the subset of $\PC(n)$ consisting of all very even partitions of
$n$. Now define $$\PC_1(n) \,:=\, \{ \d \,\in\, \PC(n) \,\mid \,t_\eta ~\text{ is even for all }~ 
\eta \,\in\, \E_\d \}$$
and
$$
\PC_{-1}(n) \,:=\, \{\d \,\in\, \PC(n) \,\mid\,
 t_\theta ~\text{ is even for all }~ \theta \,\in\, \O_\d \}\, .$$
Clearly, we have $\PC_{\rm v. even} (n) \,\subset\, \PC_1(n)$.

Following \cite[p.~140]{CoM} we define a {\it Young diagram} to be a left-justified array of rows
of empty boxes arranged so that no row is shorter than the one below it; the {\it size} of a
Young diagram is the number of empty boxes appearing in it.
There is an obvious correspondence between the set of Young diagrams of size $n$ and the set
$\PC(n)$ of partitions of $n$.
Hence the {\it set of Young diagrams of size $n$} is also denoted by $\PC(n)$.
A {\it signed Young diagram} is a Young diagram in which every box is labeled with $+1$ or $-1$ such
that the sign of $1$ alternate across rows except when the length of the row is of the form
$3\,\pmod{4}$. In the latter case when the length of the row is of the form $3\,\pmod{4}$
we will alternate the sign of $1$ till the last but one and repeat the sign of $1$ in the last
box as in the last but one box; see Remark \ref{CM-correction} why the choices of signs in this case
deviate from that in the previous cases. The sign of $1$ need not alternate down columns. Two
signed Young diagrams are equivalent if and only if each can be obtained from the other by
permuting rows of equal length. The {\it signature of a signed Young diagram} is the
ordered pair of integers $(p,\,q)$ where $p$-many $+1$ and $q$-many $-1$ occur in it. 

We next define certain sets using collections of matrices with entries comprising of signs $\pm 1$, which are
easily seen to be in bijection with sets of equivalence classes of various types of signed Young diagrams.
These sets will be used in parametrizing the nilpotent orbits in the classical Lie algebras.

For a partition $\d \,\in\, \PC (n)$ and $d \,\in\, \N_\d$, we define the subset 
$\Ab_d \,\subset\, {\rm M }_{t_d \times d} (\C)$  of matrices $(m^d_{ij})$ with entries in the set $\{ \pm 1 \}$ as follows : 
\begin{equation}\label{A-d}
 \Ab_d := \{ (m^d_{ij}) \in {\rm M }_{t_d \times d} (\C) \mid (m^d_{ij})  \text{ satisfies } \eqref{yd-def1} \text{ and }  \eqref{yd-def2} \} \,.
\end{equation}

\begin{enumerate}[label = {{\bf Yd}.\arabic*}]
\item~ \label{yd-def1} There is an integer $ 0 \,\leq\, p_d \,\leq\, t_d$ such that
$$
 m^d_{i1} \,:=\, \begin{cases}
+1  & \text{ if } \ 1 \,\leq\, i \,\leq\, p_d\\
-1 & \text{ if } \ p_d \,<\, i \,\leq\, t_d.
\end{cases}
$$
\item~ \label{yd-def2} 
\begin{align*}
m^d_{ij} &:=\, (-1)^{j+1}m^d_{i1} \qquad \text{if } \ 1\,<\,j \,\leq \,d , \ d\,\in\, \E_\d \cup \O^1_\d; \\
m^d_{ij} &:= \begin{cases}
(-1)^{j+1}m^d_{i1} & \text{ if }\ 1<j \,\leq\, d-1 \\
-m^d_{i1} & \text{ if }\ j \,=\, d 
\end{cases},\, \, d \,\in\, \O^3_\d\, .
\end{align*}
\end{enumerate}

For any $(m^d_{ij}) \,\in\, \Ab_d$ set $${\rm sgn}_+ (m^d_{ij})\,:=\, \# \{ (i,\,j) \,\mid\, 1 \,\leq\, i \,\leq\, t_d,\
1 \,\leq\, j \,\leq \, d ,\ m^d_{ij} \,=\, +1 \}$$ and
$${\rm sgn}_-(m^d_{ij})\,:=\, \# \{ (i,\,j) \,\mid\, 1\,\leq\, i \,\leq\, t_d,\ 1 \,\leq \,j \,\leq \,d ,\
m^d_{ij} \,=\, -1 \}\, .$$
Let $\SC_\d(n) \,:= \, \Ab_{d_1} \times \cdots \times \Ab_{d_s}$.
For a pair of non-negative integers $(p,q)$ with $p+q=n$
we now define the subset $\SC_\d(p, q) \subset \SC_\d(n) $ by 
\begin{equation}\label{S-d-pq}
\SC_\d(p, q) := \big\{ (M_{d_1},   \dotsc ,  M_{d_s})  \in \SC_\d(n)  \mid   \sum_{i=1}^s  
{\rm sgn}_+ M_{d_i}  =  p, \,  \sum_{i=1}^s { \rm sgn}_- M_{d_i}  = q \big\}.
\end{equation}
We also define
\begin{equation}\label{yd-Y-pq}
\YC(p,\,q) \,:=\, \{ (\d, \,\sgn ) \,\mid\, \d \,\in\, \PC (n),\ \sgn \,\in \,\SC_\d(p,\,q) \} \, .
\end{equation}

It is easy to see that there is a natural bijection between the set $\YC(p,\,q)$ and
the equivalence classes of signed Young diagrams of size $p+q$ with signature $(p,\,q)$.
Hence, we will call $\YC(p,\,q)$ the {\it set of equivalence classes of signed Young diagrams of size $p+q$
with signature $(p,\,q)$}.

For any $\d \,\in\, \PC(n)$ and $d \,\in\, \N_\d$, define the subset $\Ab_{d, 1} $ of $\Ab_d$ by
$$\Ab_{d, 1} \,:=\, \{(m^d_{ij}) \,\in \,\Ab_d \,\mid\, m^d_{i\, 1} = +1~ \ \forall\ 1 \,\leq\, i \,\leq\, t_d \}\, .$$
Further define $\SC^{\rm even}_{\d}(p,q) \,\subset\, \SC_\d (p,q)$ and $\SC^{\rm odd}_{\d}(n) \,\subset \,\SC_\d (n)$ by
\begin{equation}\label{S-d-pq-even}
\SC^{\rm even}_{\d}(p,\,q) \,:=\, \{ (M_{d_1},\, \dotsc ,\,M_{d_s}) \,\in \,\SC_\d (p,\,q) \,\mid\,
M_\eta \in \Ab_{\eta, 1} ~\ \forall~\  \eta \,\in\, \E_\d \}
\end{equation}
and
\begin{equation}\label{S-d-pq-odd}
\SC^{\rm odd}_{\d}(n) \,:=\, \{ (M_{d_1},\, \dotsc ,\,M_{d_s}) \,\in\, \SC_\d (n) \,\mid\,
M_\theta \,\in\, \Ab_{\theta, 1} ~\ \forall~\  \theta \,\in\, \O_\d \}\, .
\end{equation}

For a pair $(p,q)$ of non-negative integers we define the sets $\YC^{\rm even} (p,q)$ and $\YC^{\rm even}_1 (p,q)$ by
\begin{equation}\label{yd-even-Y-pq}
\YC^{\rm even} (p,q) := \{ (\d, \sgn) \mid~ \d \in \PC (n),~ \sgn \in \SC^{\rm even}_{\d}(p,q) \},
\end{equation}
\begin{equation}\label{yd-1-Y-pq}
\YC^{\rm even}_1(p,q) := \{ (\d, \sgn) \mid~ \d \in \PC_1 (n),~ \sgn \in \SC^{\rm even}_{\d}(p,q) \}.
\end{equation}
Similarly, for a non-negative integer $n$, set 
\begin{equation}\label{yd-odd-Y-pq}
\YC^{\rm odd}(n) \,:= \,\{ (\d, \,\sgn) \,\mid\, ~ \d \,\in\, \PC(n),~\ \sgn \,\in\, \SC^{\rm odd}_{\d}(n) \}, 
\end{equation}
\begin{equation}\label{yd-odd-1-Y-pq}
\YC^{\rm odd}_{-1}(2n)\,:=\,\{ (\d,\,\sgn)\,\mid\, \d \,\in\, \PC_{-1}(2n),~\ \sgn \,\in\,\SC^{\rm odd}_{\d}(2n)\}\, .
\end{equation}

Let $\d \,\in\, \PC(n)$. 
For $\theta \,\in\, \O_\d$ and $M_\theta \,:= \,(m^\theta_{rs})\,\in\,\Ab_\theta $,
define $$l^+_{\theta , i} (M_\theta ) \,:= \,\# \{ j \,\mid\, m^\theta_{ij} \,=\, +1 \}\ \ \text{ and }\ \
l^-_{\theta , i} (M_\theta ) \,:=\, \# \{j \,\mid\, m^\theta_{ij} \,= \,-1 \}$$
for all $1 \,\leq\, i \,\leq\, t_\theta$; set
$$
\SC'_\d(p,\, q) :=
\Bigg\{ (M_{d_1}, \cdots ,M_{d_s}) \in \SC^{\rm even}_\d (p,q) \biggm| \! \!
\begin{array}{cc}
\quad l^+_{\theta , i} (M_\theta )~ \text{ is even for all } ~\theta \,\in\, \O_\d,\ 1 \,\leq\, i \,\leq\, t_\theta \\
 \text{or } l^-_{\theta , i} (M_\theta )~ \text{ is even for all }~ \theta \,\in\, \O_\d,\ 1 \,\leq\, i \,\leq\, t_\theta 
\end{array} \!
\Bigg\}.
$$

\section{The second and first cohomologies of homogeneous spaces}\label{sec-second-first-cohomologies-of-homogeneous-spaces} 
 
In this section we obtain convenient descriptions of the second and the first cohomology groups of homogeneous spaces of 
general connected Lie groups with the purpose of applying them in the computations
done in Sections \ref{sec-Second-Cohomology-of-Nilpotent-Orbits} 
and \ref{sec-First-Cohomology-of-Nilpotent-Orbits}. The main results are Theorems \ref{thm-i1}, \ref{H^1}, and they
hold under the mild restriction that isotropy subgroups have finitely many connected components. Moreover, the results
generalize those of \cite[\S~3]{BC} where it is assumed that the maximal compact subgroups of the ambient Lie groups are semisimple.
We apply the above results to derive a stronger consequence in Theorem \ref{H.21} where the ambient Lie group real simple.

Given a Lie algebra $\mathfrak{a}$ and an integer $n \,\geq\, 0$, let $\Omega^n (\mathfrak{a})$ denote the space of
all $n$-forms on $\mathfrak{a}$.
A $n$-form $\omega \,\in\, \Omega^n (\mathfrak{a})$ is said to
{\it annihilate} a given subalgebra $\mathfrak{b}\,\subset\, \mathfrak{a}$ if $ \omega (X_1,\, \cdots,\, X_n)\,=\,0$
whenever $ X_i \,\in \,\mathfrak{b}$ for some $i$.
Let $\Omega^n (\mathfrak{a}/ \mathfrak{b} )$ denote the space of $n$-forms on $\mathfrak{a}$ which
annihilate $\mathfrak{b}$.

Let $L$ be a compact Lie group with Lie algebra $\l$. Let $J \,\subset\, L $ be a closed subgroup with
Lie algebra ${\mathfrak j}\,\subset\, \l$. 
The space of $J$-invariant $p$-forms on $\l$ will be denoted by $\Omega^p (\l)^J$.
Note that $ \omega \,\in\, \Omega^p (\l)^{J^\circ}$ if and only if 
\begin{equation}\label{inv-cond}
\sum_{i=1}^p \omega (X_1, \,\cdots,\, [Y,\, X_i],\, \cdots,\, X_p)\,=\,0
\end{equation}
for all $Y \,\in\, {\mathfrak j}$ and all $(X_1,\, \cdots,\, X_p ) \,\in\, \l^p$.
For a continuous function $$W\,:\, J\,\longrightarrow\, \Omega^p ( \l)$$ and a Haar measure $\mu_J$ on $J$,
define the integral $\int_J W (g) d \mu_J (g)\,\in\, \Omega^p ( \l)$ as follows:
$$
(\int_J W (g) d \mu_J(g)) (X_1,\, \cdots,\, X_p)\,:=\, \int_J W(g) (X_1,\, \cdots,\, X_p)
d \mu_J (g),\ \ (X_1,\, \cdots,\, X_p) \,\in\, \l^p\, .
$$
The above integral $\int_J W (g) d \mu_J (g)$ is also denoted by $\int_J W d \mu_J$.
The following equations are straightforward.
\begin{equation}\label{formula-1}
d \int_J W d \mu_J \,=\, \int_J d W d \mu_J\, ;
\end{equation}
For any $a \,\in\, L$, 
\begin{equation}\label{formula-2}
{\rm Ad}(a)^* \int_J W d \mu_J \,=\, \int_J {\rm Ad}(a)^* W d \mu_J\, .
\end{equation}

For any $\omega \,\in \,\Omega^p (\l)$, from the left-invariance of the Haar measure $\mu_J$ on
$J$ it follows that
$$\int_J ({\rm Ad}(g)^* \omega) d \mu_J (g) \,\in\, \Omega^p (\l)^J\, .$$

\begin{lemma}\label{inv-forms}
Let $L$ be a compact Lie group with Lie algebra $\l$. Let $p \,\geq\, 1$ be an integer.
\begin{enumerate}
\item If $\omega \,\in\, \Omega^p (\l)$ is invariant then $d \omega \,=\,0$.

\item Every element of $H^p (\l,\, \R)$ contains an unique invariant $\omega \,\in\, \Omega^p (\l)$.

\item If $J \,\subset\, L$ is a closed subgroup, then 
$$
\Omega^p (\l)^J \cap d (\Omega^{p-1} (\l)) \,=\, d (\Omega^{p-1} (\l)^J)\, .
$$

\item If $L$ is connected and $\omega \,\in\, \Omega^2 (\l)$, then
$\omega \,\in \,\Omega^2 (\l)^L $ if and only if $\omega \,\in\, \Omega^2 (\l / [\l, \,\l])$.
\end{enumerate}
\end{lemma}

\begin{proof}
Statement (1) is proved in \cite[p. 102, 12.3]{CE}. Statement (2) is proved in \cite[p. 102, Theorem 12.1]{CE}.

To prove (3), note that it suffices to show that 
$$ 
\Omega^p (\l)^J \cap d (\Omega^{p-1} (\l))
\,\subset\, d (\Omega^{p-1} (\l)^J)\, .
$$
Let $\mu_J$ denote the Haar measure on $J$ such that $\mu_J (J)\,=\,1$.
For any $ \omega \,\in\, \Omega^p (\l)^J \cap d (\Omega^{p-1} (\l))$,
we have $\omega \,= \,d \nu$ for some $\nu \,\in\, \Omega^{p-1} (\l)$. Now as $\omega$ is $J$-invariant, it follows
that
\begin{equation}\label{1}
\omega \,= \,{\rm Ad} (g)^* d \nu \,= \, d {\rm Ad} (g)^* \nu
\end{equation}
for all $g \,\in\, J$. In particular, from \eqref{1} we have
$$
\omega\,= \,\int_J (d {\rm Ad }(g)^* \nu) d \mu_J (g) \,=\, d \int_J ({\rm Ad }(g)^* \nu) d \mu_J (g)\, .
$$
As $\mu_J$ is preserved by the left multiplication by elements of $J$, it now follows that 
$$
\int_J ({\rm Ad }(g)^* \nu) d \mu_J (g) \,\in\, \Omega^{p-1} (\l)^J\, .
$$
This in turn implies that $\omega \,\in \,d (\Omega^{p-1} (\l)^J)$.

The proof of (4) is essentially contained in the proof of \cite[p. 309, Corollary 12.9]{B}; we will give
the details.
Take any $\omega \,\in\, \Omega^2 (\l)^L $. Lemma \ref{inv-forms}(1) says that $d \omega \,=\,0$.
Thus, for all $x,\,y, \,z \,\in\, \l$,
\begin{equation}\label{2}
d \omega (x,\, y,\, z) \,=\, -\omega ([x,\,y],\,z) + \omega ([x,\,z],\,y) - \omega ([y,\,z],\,x) \,=\,0\, .
\end{equation}
As $\omega$ is $L$-invariant, we also have 
\begin{equation}\label{3}
-\omega ([x,\,y],\,z) + \omega ([x,\,z],\,y) \,=\, - ( \omega ([x,\,y],\,z) + \omega (y,\, [x,\,z]) \,=\, 0\, .
\end{equation}
{}From \eqref{2} and \eqref{3} it follows that $\omega ([y,\,z],\,x) \,=\,0$, therefore 
$$
\omega ([\l,\, \l],\, \l) \,=\, 0.
$$
This is equivalent to saying that $\omega \,\in\, \Omega^2 (\l / [\l, \,\l])$.

Conversely, if $\omega ([\l, \,\l], \,\l) \,=\,0$, then it is immediate that $\omega$
satisfies \eqref{inv-cond} for $p\,=\,2$. In particular, as $L$ is connected, we conclude that
$\omega \,\in\, \Omega^2 (\l)^L$. This completes the proof of (4). 
\end{proof}

\begin{theorem}[\cite{Mo}]\label{mostow} 
Let $G$ be a connected Lie group, and let $H \,\subset\, G$ be a closed subgroup with finitely many connected 
components. Let $M$ be a maximal compact subgroup of $G$ such that $M \cap H$ is a maximal compact subgroup of $H$. 
Then the image of the natural embedding $M/ (M \cap H)\, \hookrightarrow\, G/H$ is a deformation retraction of 
$G/H$.
\end{theorem} 

Theorem \ref{mostow} is proved in \cite[p. 260, Theorem 3.1]{Mo} under the assumption that $H$ is connected.
However, as mentioned in \cite{BC}, using \cite[p. 180, Theorem 3.1]{H}, the proof as in \cite{Mo}
goes through when $H$ has finitely many connected components.

Let $G,\, H, \,M$ be as in Theorem \ref{mostow}, and let $K \,:= \,M \cap H$. As
$\ M/K \ \hookrightarrow \ G/H$ is a deformation retraction by Theorem \ref{mostow}, we have 
\begin{equation}\label{homotopy}
{ H}^{i}(G/H ,\, \R) \,\simeq\, { H}^{i}(M/K ,\, \R ) ~\ \ \text{ for all } \ i\, . 
\end{equation}

\begin{theorem}\label{thm-i1}
Let $G$ be a connected Lie group, and let $H \,\subset\, G$ be a closed subgroup with finitely many connected
components. Let $K$ be a maximal compact subgroup of $H$, and let $M$ be a maximal compact subgroup of $G$
containing $K$. Then,
$$ 
 H^2 (G/H, \, \R) \, \simeq \, \Omega^2 \big( \frac{\m}{[ \m, \,\m] + \k}\big) \oplus
[(\z(\k) \cap [\m,\,\m])^*]^{K/K^{\circ}}\, .
$$
\end{theorem}

\begin{proof}
In view of \eqref{homotopy} it is enough to show that 
\begin{equation}\label{gh1}
H^2 (M/K, \, \R) \, \simeq \, \Omega^2 \big( \m /([ \m, \,\m] + \k)\big) \oplus
[(\z(\k) \cap [\m,\,\m])^*]^{K/K^{\circ}}\, .
\end{equation}
As $M$ is compact and connected, from \cite[p. 310, Theorem 30]{Sp} 
and the formula given in \cite[p. 313]{Sp} we conclude that there 
are natural isomorphisms
\begin{equation}\label{e1}
H^i ( M/K,\, \R) \, \simeq \, \frac{ {\rm Ker} \,( d:\Omega^i ( \m/ \k)^K 
\to \Omega^{i+1} ( \m/ \k)^K)}{ d (\Omega^{i-1} ( \m/ \k)^K)} ~\ \ \ \forall\ \ i\, .
\end{equation}
 
Setting $i\,=\, 2$ in \eqref{e1}, 
\begin{equation}\label{e2} 
H^2 (M/K,\, \R) \,\simeq \, \frac{ {\rm Ker} \,( d:\Omega^2 ( \m / \k)^K 
\to \Omega^3 ( \m/ \k)^K)}{ d (\Omega^1 ( \m / \k)^K)}\, . 
\end{equation} 
The numerator and the denominator in \eqref{e2} will be identified. 
 
We claim that 
\begin{equation}\label{e3} 
{\rm Ker} \,( d:\Omega^2 ( \m/ 
\k)^K 
\to \Omega^3 ( \m/ \k)^K)\,= \,\Omega^2 (\m/ \k)^M \oplus d ( \Omega^1 ( \m)^K).
\end{equation} 

To prove the claim, first note that $d (\Omega^2 (\m/ \k)^M ) \,=\, 0$ by
Lemma \ref{inv-forms}(1). Therefore, we have 
$$
\Omega^2 (\m/ \k)^M + d ( \Omega^1 ( \m)^K) \,\subset\, 
{\rm Ker} \,( d:\Omega^2 ( \m/ 
\k)^K \to \Omega^3 ( \m/ \k)^K)\, .
$$

To prove the converse, take any $\omega \,\in\, {\rm Ker} \,( d:\Omega^2 ( \m/ \k)^K 
\to \Omega^3 ( \m/ \k)^K)$. Then by Lemma \ref{inv-forms}(2) there is an element
$ \widetilde{\omega} \,\in\, \Omega^2 (\m)^M$ such that 
\begin{equation}\label{e4} 
\omega -\widetilde{\omega} \,\in\, d (\Omega^1 (\m))\, .
\end{equation} 
As $\omega \,\in\, \Omega^2 (\m)^K$ and $\widetilde{\omega} \,\in\, \Omega^2 (\m)^M$, it follows that 
$ \omega- \widetilde{\omega} \,\in \,\Omega^2 (\m)^K$. So \eqref{e4} and Lemma \ref{inv-forms}(3)
together imply that
$$
\omega -\widetilde{\omega} \,\in\, d (\Omega^1 (\m)^K)\, .
$$
Take any $f \,\in\, \Omega^1 (\m)^K$ such that $\omega - \widetilde{\omega} \,= \,df$. As 
$f \,\in\, \Omega (\m)^K$, it follows that $d f \,\in \,\Omega^2 (\m/ \k)^K$. Thus $ \widetilde{\omega} \,\in 
\, \Omega^2 (\m/ \k)^M $. This in turn implies that $\omega \,\in\,
\Omega^2 (\m/ \k)^M + d ( \Omega^1 ( \m)^K)$. Therefore,
$$
\Omega^2 (\m/ \k)^M + d ( \Omega^1 ( \m)^K) \,\supset\, 
{\rm Ker} \,( d:\Omega^2 ( \m/ 
\k)^K \to \Omega^3 ( \m/ \k)^K)\, .
$$

To complete the proof of the claim, it now remains to show that 
\begin{equation}\label{f1}
\Omega^2 (\m/ \k)^M \cap d ( \Omega^1 ( \m)^K) \,=\, 0\, .
\end{equation}

To prove \eqref{f1}, take any $f_1 \,\in\, \Omega^1 ( \m)^K$ such that $d f_1\,\in\, \Omega^2 (\m/ \k)^M$. From
Lemma \ref{inv-forms}(3) it follows that $d f_1\,=\, d f_2$ for some $f_2 \,\in \, \Omega^1 ( \m)^M$.
But then from Lemma \ref{inv-forms}(1) it follows that $d f_2 \,=\,0$. Thus we have $d f_1 \,=\, d f_2\,=\,0$.
This proves \eqref{f1}, and the proof of the claim is complete. 

Combining \eqref{e2} and \eqref{e3},
\begin{equation}\label{e5} 
H^2 (M/K, \, \R) \,\simeq\, 
\Omega^2 (\m/ \k)^M \oplus\frac{d ( \Omega^1 ( \m)^K)}{d (\Omega^1 (\m/ \k)^K)}\,.
\end{equation}

Moreover, as $M$ is connected, Lemma \ref{inv-forms}(4) implies that
\begin{equation}\label{e6} 
\Omega^2 (\m/ \k)^M \,\simeq\, \Omega^2 (\frac{\m}{[ \m,\, \m] + \k})\, .
\end{equation}

We have
$${\rm Ker}( d : \Omega^1 (\m) \to \Omega^2 (\m)) \,= \,\Omega^1 (\m/ [\m,\, \m]).$$
In view of the above it is straightforward to check that 
\begin{equation}\label{e7}
\frac{d ( \Omega^1 ( \m)^K)}{d (\Omega^1 ( 
\m/ \k)^K)} \,\simeq\, \frac {\Omega^1 ( \m)^K}{ \Omega^1 (\m/\k)^K + \Omega^1 ( \m / [\m,\,\m])^K}\, .
\end{equation}
We will identify the right-hand side of \eqref{e7}.

Consider the adjoint action of $K$ on $\m$. As $K$ is compact, there is a
$K$-invariant inner-product $\<>$ on the $\R$-vector space $\m$. Now decompose $\m$ as follows.
\begin{align}
 \m = & ([\m,\, \m] + \k) + \z (\m)\nonumber\\
 = & ([\m, \,\m] + \k) \oplus\big((([\m,\,\m] + \k)\cap \z (\m))^\perp \cap \z (\m) \big). \label{e8}
\end{align}

We next decompose $[\m, \m] + \k$ as
\begin{equation}\label{e9}
[\m, \,\m] + \k \,=\, 
\big( ([\m,\, \m] \cap \k)^\perp \cap [\m, \,\m] \big) \oplus ([\m,\, \m] \cap \k)
\oplus
\big( ([\m,\, \m] \cap \k)^\perp \cap \k \big)\, .
\end{equation}
Using \eqref{e8} and \eqref{e9} the decomposition of $\m$ is further refined as follows:
\begin{gather}\label{e10}
 \m \,=\, ([\m,\, \m] + \k) \oplus\big((([\m,\,\m] + \k)\cap \z (\m))^\perp \cap \z (\m) \big)\\
 = \big( ([\m, \m] \cap \k)^\perp \cap [\m, \m] \big) \oplus ([\m, \,\m] \cap \k)
\oplus
\big( ([\m,\, \m] \cap \k)^\perp \cap \k \big) \oplus \big((([\m,\,\m] + \k)\cap \z (\m))^\perp \cap \z (\m) \big). \nonumber
\end{gather}
It is clear that all the direct summands in \eqref{e10} are $K$-invariant.
For notational convenience, set $$\a \,:=\, (([\m,\,\m] + \k)\cap \z (\m))^\perp\ \
\text{ and }\ \ \b \,:=\, ([\m, \,\m] \cap \k)^\perp\, .$$
Let 
\begin{equation}\label{e111}
\sigma \, : \, \Omega^1 (\m))\,=\, {\m}^* \, 
\longrightarrow\, ([\m,\,\m] \cap \b)^* \oplus ([\m,\, \m] \cap \k)^* \oplus (\k \cap \b)^* \oplus (\a \cap \z (\m))^*
\end{equation}
be the isomorphism defined by 
$$
f\, \longmapsto\, ( f\vert_{ [\m,\,\m] \cap \b }, f\vert_{[\m,\, \m] \cap \k}, f\vert_{\k \cap \b}, 
f\vert_{\a \cap \z (\m)} )\, . 
$$
As each of the subspaces of $\m$ in \eqref{e10} is 
${\rm Ad}(K)$--invariant, the restriction of the isomorphism $\sigma$ in
\eqref{e111} to $ \Omega^1 ( \m)^K$ induces an isomorphism 
\begin{equation}\label{e12} 
\widetilde{\sigma} \,:\, \Omega^1 ( \m)^K \,\stackrel{\sim}{\longrightarrow}\, 
(([\m,\,\m] \cap \b)^*)^K \oplus (([\m, \,\m] \cap \k)^*)^K \oplus ((\k \cap \b)^*)^K \oplus ((\a \cap \z (\m))^*)^K\, . 
\end{equation} 
 As $\k \,=\,([\m,\,\m] \cap \k) \oplus (\k \cap \b)$ and $[\m,\,\m] \,=\, ( [\m,\,\m] \cap \k)
\oplus ([\m,\, \m] \cap \b)$, it follows that 
\begin{equation}\label{e13}
\widetilde{\sigma} ( \Omega^1 (\m / \k)^K + \Omega^1 ( \m / [\m,\,\m])^K)\,=\, (([\m,\,\m] \cap \b)^*)^K
\oplus ((\k \cap \b)^*)^K \oplus ((\a \cap \z (\m))^*)^K\, .
\end{equation}

Thus from \eqref{e12} and \eqref{e13} it follows that 
\begin{align}
&\frac {\Omega^1 ( \m)^K}{ \Omega^1 (\m/\k)^K + \Omega^1 ( \m / [\m,\,\m])^K} \nonumber\\
&\simeq \frac{(([\m,\,\m] \cap \b)^*)^K \oplus (([\m, \,\m] \cap \k)^*)^K \oplus ((\k \cap \b)^*)^K
\oplus ((\a \cap \z (\m))^*)^K}
 {(([\m,\,\m] \cap \b)^*)^K \oplus ((\k \cap \b)^*)^K \oplus ((\a \cap \z (\m))^*)^K} \nonumber \\
&\simeq (([\m,\, \m] \cap \k)^*)^K\, . \label{e14}
\end{align}

As $[\m, \,\m] \cap \k \,=\, [\k,\,\k] \oplus ( \mathfrak{z}(\k) \cap [\m,\,\m])$, it follows that 
\begin{equation}\label{e15}
(([\m,\, \m] \cap \k)^*)^K \,\simeq\, ([\k,\,\k]^*)^K \oplus (( \mathfrak{z}(\k) \cap [\m,\,\m])^*)^K\, .
\end{equation}
In \cite[\S~3, (3.13)]{BC} it is proved that
\begin{equation}\label{e16} 
([\k, \k]^*)^K \,= \,0\, . 
\end{equation} 

Thus from \eqref{e14}, \eqref{e15} and \eqref{e16} we have
\begin{equation}\label{e17}
\frac {\Omega^1 ( \m)^K}{ \Omega^1 (\m/\k)^K + \Omega^1 ( \m / [\m,\,\m])^K}
\,\simeq \, (( \mathfrak{z}(\k) \cap [\m,\,\m])^*)^K\, .
\end{equation}

Combining \eqref{e7} and \eqref{e17},
\begin{equation}\label{e18}
\frac{d ( \Omega^1 ( \m)^K)}{d (\Omega^1 (\m/ \k)^K)} \,\simeq \, \frac {\Omega^1 ( \m)^K}{ \Omega^1
(\m/\k)^K + \Omega^1 ( \m / [\m,\,\m])^K}\,\simeq\, (( \mathfrak{z}(\k) \cap [\m,\,\m])^*)^K\, .
\end{equation}
Moreover, as $K^\circ$ acts trivially on $(( \mathfrak{z}(\k) \cap [\m,\,\m])^*)$,
\begin{equation}\label{e19}
(( \mathfrak{z}(\k) \cap [\m,\,\m])^*)^K \,\simeq\, (( \mathfrak{z}(\k) \cap [\m,\,\m])^*)^{K/ K^\circ}\, .
\end{equation}
Combining \eqref{e18} and \eqref{e19},
$$
\frac{d ( \Omega^1 ( \m)^K)}{d (\Omega^1 (\m/ \k)^K)} \,\simeq \,
(( \mathfrak{z}(\k) \cap [\m,\,\m])^*)^{K/ K^\circ}\, .
$$
This and \eqref{e6} together imply that the right-hand side of \eqref{e5} coincides with the
right-hand side of \eqref{gh1}. This completes the proof of the theorem.
\end{proof}

\begin{corollary}\label{1dim}
Let $G$, $H$, $K$ and $M$ be as in Theorem \ref{thm-i1}. If $\dim_\R \z(\m) \,=\, 1$, then
$$
{ H}^2(G/H, \,\R) \, \simeq \, [(\z(\k) \cap [\m,\,\m])^*]^{K/K^{\circ}}\, .
$$
\end{corollary}

\begin{proof}
As $M$ is a compact Lie group, we have $\m \,= \,\z(\m) \oplus [\m,\,\m]$. Thus,
$$\dim_\R (\m/([\m,\,\m]+\k))
\,\leq \,1\, .$$ Now the corollary follows from Theorem \ref{thm-i1}.
 \end{proof} 

\begin{corollary}
Let $G$, $H$, $K$ and $M$ be as in Theorem \ref{thm-i1}.
If $K$ is semisimple, then
$$
 { H}^2(G/H, \,\R) \, \simeq \, \Omega^2 (\frac{\m}{[ \m,\, \m] + \k})\, .
$$
\end{corollary}

\begin{proof}
As $K$ is semisimple, we have $\z(\k)\,=\,0$, so it follows from Theorem \ref{thm-i1}. 
 \end{proof}

\begin{theorem}\label{H^1}
Let $G$, $H$, $K$ and $M$ be as in Theorem \ref{thm-i1}. Then
 $$
 H^1(G/H, \,\R) \,\simeq \,\Omega^1(\frac{\m}{[\m,\,\m]+\k})\, .
 $$
\end{theorem}

\begin{proof}
In view of \eqref{homotopy} it is enough to show that $H^1(M/K,\, \R) \,\simeq\, \Omega^1(\m / ( [\m,\,\m]+\k))$.
As $M$ is compact and connected, from \cite[p. 310, Theorem 30]{Sp} and the formula given in \cite[p. 313]{Sp}
it follows that there are natural isomorphisms
\begin{equation}\label{q1}
H^i ( M/K,\, \R) \,\simeq\, \frac{ {\rm Ker} \,( d:\Omega^i ( \m/ \k)^K 
\to \Omega^{i+1} ( \m/ \k)^K)}{ d (\Omega^{i-1} ( \m/ \k)^K)}~\ \ \ \forall\ \ i\, .
\end{equation}
Setting $i\,=\, 1$ in \eqref{q1}, 
\begin{equation}\label{q2} 
H^1 (M/K,\, \R) \, \simeq \, {\rm Ker} \,( d:\Omega^1( \m/ \k)^K \to \Omega^2 ( \m/ \k)^K) \, . 
\end{equation}

We claim that
\begin{equation}\label{q3}
 {\rm Ker} \,( d:\Omega^1( \m/ \k)^K \to \Omega^2 ( \m/ \k)^K) \,=\, \Omega^1(\m/\k)^M\, .
\end{equation}

To prove \eqref{q3}, first note that $\Omega^1(\m/\k)^M \,\subseteq\,\Omega^1( \m/ \k)^K$. From Lemma 
\ref{inv-forms}(1) it follows that $d \alpha \,=\, 0$ for any $\alpha \,\in\, \Omega^1(\m/\k)^M$. Thus
$$
\Omega^1(\m/\k)^M \,\subseteq \,{\rm Ker} \,( d:\Omega^1( \m/ \k)^K \to \Omega^2 ( \m/ \k)^K)\, .
$$

To prove the other way inclusion, take any $\alpha \,\in\,{\rm Ker} \,( d:\Omega^1( \m/ \k)^K \to \Omega^2 ( \m/ \k)^K)$.
Then $d\alpha \,=\,0$ which in turn implies that $\alpha ([X,\,Y])\,=\,0$ for all $X,\,Y \,\in\, \m$. As $M$ is
connected, using \eqref{inv-cond} it follows that $\alpha$ is $M$-invariant. This proves the claim in \eqref{q3}.

As $\Omega^1(\m)^M \,\simeq\, \Omega^1({\m}/{[\m,\,\m]})$, it follows that 
\begin{equation}\label{q4}
 \Omega^1(\m/\k)^M \,\simeq \,\Omega^1(\frac{\m}{[\m,\,\m]+\k})\, .
\end{equation}
Combining \eqref{q2}, \eqref{q3}, \eqref{q4} we have 
$$ H^1(M/K, \, \R) \, \simeq\, \Omega^1(\frac{\m}{[\m,\, \m]+\k})\, .$$
As noted before, the theorem follows from it.
\end{proof}

\begin{corollary}\label{H1}
Let $G$, $H$, $K$ and $M$ be as in Theorem \ref{thm-i1}. Moreover, if $M$ is semisimple, then 
$$
 \dim_\R H^1(G/H , \, \R) \, = \, 0\, .
$$
\end{corollary}
 
\begin{proof}
As $M$ is semisimple, we have $\m \,=\, [\m,\,\m]$, and hence the corollary follows from Theorem \ref{H^1}. 
\end{proof}

Recall that any maximal compact subgroup of a complex semisimple Lie group is semisimple. 
The following corollary now follows form \eqref{homotopy} and Corollary \ref{H1}.

\begin{corollary}\label{H.1}
Let $G$ be a connected complex semisimple Lie group, and let $H \,\subset\, G$ be a closed subgroup with finitely many
connected components. Then
$$ \dim_\R { H}^1(G/H, \,\R) \,= \, 0\, .$$
\end{corollary}
 
In the special case where $G$ is a simple real Lie group, 
the following result is a stronger form of Theorem \ref{thm-i1} and Theorem \ref{H^1}.

\begin{theorem}\label{H.21}
Let $G$ be a connected simple real Lie group, and let $H \,\subset\, G$ be a closed subgroup with
finitely many connected components. Let $K$ be a maximal compact subgroup of $H$ and $M$ a maximal
compact subgroup of $G$ containing $K$. Then
$$ 
{ H}^2(G/H, \R) \  \simeq \ [(\z(\k) \cap [\m,\,\m])^*]^{K/K^{\circ}}
$$
and 
$$ \hspace*{0.6cm} 
 \dim_\R{ H}^1(G/H, \R) = \begin{cases}
1  &  \text{ if } \  \k + [\m,\,\m] \subsetneqq \m \\
0  &  \text{ if } \  \k + [\m,\,\m] = \m .
 \end{cases}
$$
In particular, $\dim_\R H^1(G/H, \,\R) \,\leq \, 1$.
\end{theorem}
 
\begin{proof}
Since $M$ is a maximal compact subgroup of a real simple Lie group, it follows from \cite[Proposition 6.2, p. 382]{He} that $ \dim_\R \z(\m)$ is either $0$ or $1$. In both these cases we have $\Omega^2(\m/([\m,\,\m]+\k)) \,=\,0$. In
view of Theorem \ref{thm-i1} and \eqref{homotopy}, it follows that
$$
{H}^2(G/H, \,\R) \  \simeq \ [(\z(\k) \cap [\m,\,\m])^*]^{K/K^{\circ}}\, . 
$$  
As $G$ is simple, we have  $\dim_\R \z(\m) \,\leq\, 1$. Thus, we have either
$$\, \k + [\m,\,\m] \,\subsetneqq \,\m\ \ \ \text{ or }\ \ \ \k + [\m,\,\m] \,=\, \m \, .$$ 
  Therefore, from Theorem \ref{H^1} and \eqref{homotopy} we conclude that
$$
\dim_\R { H}^1(G/H, \R)  =  \begin{cases}  1  \qquad  \text{ if } \ \k + [\m,\,\m] \subsetneqq \m  \\
                                            0  \qquad  \text{ if } \ \k + [\m,\,\m] = \m\, .
                                 \end{cases}
$$
\end{proof}

\section{Second cohomology groups of nilpotent orbits}\label{sec-Second-Cohomology-of-Nilpotent-Orbits} 

In this section we will compute the second de Rham cohomology groups of the nilpotent orbits in real classical simple
Lie algebras. The case of complex simple Lie algebras was dealt in \cite{BC}, so this case will not be considered here. 

\subsection{Some preparatory results}
Here we prove Theorem \ref{thm-nilpotent-orbit} which is crucial to our computations of 
the second and first cohomology groups of the nilpotent orbits. We begin with the following lemma.

\begin{lemma}\label{reductive-part}
Let ${\bf G}$ be a semisimple algebraic group defined over $\R$, and let $G :=  {\bf G}(\R)$. Let $\{X,\,H,\,Y\}$ be a $\s\l_2(\R)$-triple in
${\rm Lie}\,G$. Then $\ZC_{\bf G} (X,H,Y)$ is a (reductive) Levi subgroup of $\ZC_{\bf G} (X)$ which is defined over $\R$. 
\end{lemma}

\begin{proof}
The nontrivial fact that the group $\ZC_{\bf G} (X,H,Y)$ is a (reductive) Levi subgroup of $\ZC_{\bf G} (X)$ is proved in \cite[p.~50, 
Lemma 3.7.3]{CoM}. Since $X,\,H,\,Y \,\in\, {\rm Lie}\,G$, it is immediate that the group $\ZC_{\bf G} (X,H,Y)$ is defined 
over $\R$.
\end{proof}

As any connected adjoint simple real Lie group is always of the form ${\bf G} (\R)^\circ$ for some ($\R$-simple) algebraic group ${\bf G}$ defined over $\R$, it is enough
to deal with nilpotent orbits in $ {\rm Lie}\, {\bf G}(\R)$ under the adjoint action of ${\bf G} (\R)^\circ$.

\begin{theorem}\label{thm-nilpotent-orbit}
Let ${\bf G}$ be an algebraic group defined over $\R$ such that ${\bf G}$ is $\R$-simple, and let $G :=  {\bf G}(\R)$.
Let $$0\,\not=\, X \,\in\, {\rm Lie}\, G $$
be a nilpotent element and $\OC_X$ be the orbit of $X$ under the adjoint action
of the identity component $G^\circ$ on ${\rm Lie}\, G $. Let $\{X,\,H,\,Y\}$ be a $\s\l_2(\R)$-triple in
$ {\rm Lie}\, G$. Let $K$ be a maximal compact subgroup in $\ZC_{G^\circ}(X,H,Y)$ and $M$ a maximal compact
subgroup of $G^\circ$ containing $K$. Then,
$$ 
{ H}^2(\OC_X,\, \R)   \  \simeq \    [(\z(\k) \cap [\m,\,\m])^*]^{K/K^{\circ}}
$$
and 
$$ 
\dim_\R{ H}^1(\OC_X,\, \R)\,=\,  \begin{cases}
                                        1  \quad  \text{if } \  \k + [\m,\,\m] \,\subsetneqq\, \m \\
                                        0  \quad  \text{if } \  \k + [\m,\,\m] \,=\, \m\,.
                                      \end{cases}$$
\end{theorem}

\begin{proof}
{}From Lemma \ref{reductive-part} it follows that the group $\ZC_{\bf G}(X,H,Y)$ is a (reductive) Levi subgroup of
$\ZC_{\bf G}(X)$. In particular, we have the semidirect product decomposition:
$$
 \ZC_{G^\circ}(X) \,=\, \ZC_{G^\circ}(X,H,Y) (R_u \ZC_{\bf G}(X) (\R))\, ,
$$
where $R_u \ZC_{\bf G}(X)$ is the unipotent radical of $\ZC_{\bf G}(X)$.
As $R_u \ZC_{\bf G}(X) (\R)$ simply connected and nilpotent, this implies that any maximal compact subgroup in
$\ZC_{G^\circ}(X,H,Y)$ is a maximal compact subgroup in $\ZC_{G^\circ}(X)$. 
Since $G^\circ$ is a connected simple real Lie group, the theorem now follows from Theorem \ref{H.21}.
\end{proof}

Let $V$ be a right vector space of dimension $n$ over $\D$, where $\D$ is, as before, $\R$ or $ \C $ or $ \H$. 
Let $\{X,\,H,\,Y\} \,\subset\, \s\l (V )$ be a $\s\l_2(\R)$-triple. We follow the notation established in 	Section \ref{notions-sl2-modules}.
Consider the non-zero irreducible $\text{Span}_\R \{ X, H, Y\}$-submodules of $V$. 
Let $\{d_1,\, \cdots, \, d_s\}$, with $d_1 \,<\, \cdots \,<\, d_s$, be the integers that occur as $\R$-dimension of such
$\text{ Span}_\R \{ X,\,H,\,Y\}$-modules. From Lemma \ref{D-basis}(2) we have
$$ \sum_{i=1}^s t_{d_{i}} d_i \,=\,{\dim}_\D V\,=\, n\, .$$ Thus 
\begin{equation}\label{partition-symbol}
\d \,:=\, \big[d_1^{t_{d_1} },\, \cdots ,\, d_s^{t_{d_s} }\big] \,\in\, \PC(n)\, .
\end{equation}
Consider $\N_\d$, $\E_\d $ and $\O_\d $ defined in Section \ref{sec-partition-Young-diagram}. We have
\begin{align}\label{isotypicalcomp}
V &=\,  \bigoplus_{d \in \N_\d}  M(d-1) \  \    \   \text{ and } \ \ L(d-1) \,=\, V_{Y,0} \cap V_{H, 1-d}  \   \ \text{ for }
\ d \,\geq \,1\, .
\end{align}

Let $(v^d_1,\, \cdots ,\, v^d_{t_d})$
be the ordered $\D$-basis of $L(d-1)$ as in Proposition \ref{unitary-J-basis} for $d \,\in\, \N_\d$.
Then it follows from Proposition \ref{J-basis} and Proposition \ref{unitary-J-basis} that
\begin{equation}\label{old-ordered-basis-part}
\BC^l (d) \,:=\, (X^l v^d_1,\, \cdots ,\,X^l v^d_{t_d})
\end{equation}
is an ordered $\D$-basis of $X^l L(d-1)$ for $0\,\leq\, l \,\leq\, d-1$ with $d\,\in\, \N_\d$. Define
\begin{equation}\label{old-ordered-basis}
\BC(d) \,:=\, \BC^0 (d) \vee \cdots \vee  \BC^{d-1} (d) \ \forall\ d \,\in\, \N_\d\, ,\ \text{ and }\  \BC\,:=\,
\BC(d_1) \vee \cdots \vee  \BC(d_s)\, .
\end{equation}
Let
\begin{equation}\label{algebra-isom}
 \Lambda_\BC \,: \, {\rm End} (V) \,\longrightarrow\, {\rm M}_n(\D)
\end{equation}
be the isomorphism of $\R$-algebras with respect to the ordered basis $\BC$.
Next define the character $$\bigchi_\d \,\colon\, \prod_{d \in \N_\d} {\rm GL} (L(d-1))\,\longrightarrow \,\D^*$$  by
$$
\bigchi_\d \big(A_{t_{d_1}}, \cdots, A_{t_{d_s}}\big) \,:=\, \begin{cases}
                                          \prod_{i=1}^s \big(\det A_{t_{d_i}}\big)^{d_i}  &  \text{ if }\  \D \,=\, \R \text{ or } \C \vspace{.2cm}\\
        \prod_{i=1}^s \big(\text{Nrd}_{{\rm End}_{\H} (L(d_i-1))} A_{t_{d_i}} \big)^{d_i} &  \text{ if }\ \D\,=\, \H\, .
                                                        \end{cases}
$$

Henceforth, $\epsilon \,= \,\pm 1$, $\sigma \,:\, \D \,\longrightarrow\, \D$ will denote 
either the identity map or $\sigma_c$ (defined in Section \ref{sec-epsilon-sigma-forms}) 
when $\D$ is $\C$ or $\H$. Let $\langle \cdot,\, \cdot \rangle \,:\, V \times V \,\longrightarrow\, \D \ $ 
be a $\epsilon$-$\sigma$ Hermitian form. Moreover, assume that  $\{X,H,Y\}$ be a $\s\l_2(\R)$-triple in 
$\s\u (V, \,\langle \cdot,\, \cdot \rangle )$.  Define the form
\begin{equation}\label{new-form}
(\cdot  ,\,\cdot)_{d} \,:\, L(d-1) \times L(d-1 )\,\longrightarrow\, \D\, ,\ \ \
(v,\, u)_d \,:=\,  \langle v \,,\, X^{d-1} u \rangle
\end{equation}
as in \cite[p.~139]{CoM}.
\begin{remark}
	In \cite[p.139]{CoM}, starting with a nilpotent element $X\,\in\,\s\u (V, \,\<>)$, 
	the form in \eqref{new-form} is defined on the highest weight space of $M (d-1)$ involving the element $Y$ of an 
	$\s\l_2 (\R)$-triple $\{ X,\, H,\, Y\}$. However, we work with a basis of $M (d-1)$ constructed using
	$X$ (see Proposition \ref{unitary-J-basis} (2)).
	Hence for our convenience the form in \eqref{new-form} is defined using $X$.
\end{remark}

\begin{lemma}\label{reductive-part-comp}\mbox{} Let ${\rm SL} (V)$ and ${\rm SU} (V, \langle \cdot, \cdot \rangle )$ be the groups as defined in Section \ref{sec-epsilon-sigma-forms}.
\begin{enumerate}
\item The following equality holds:
$$ \ZC_{{\rm SL}(V)} (X,H,Y) 
 = \Bigg\{ g \in {\rm SL}(V)   \Biggm| \! \begin{array}{c}
                            g (X^l L(d-1))\, \subset \,  X^l L(d-1) ,   \vspace{.14cm} \\
                           \! \big[g |_{ X^l L(d-1)}\big]_{{\BC}^l (d)} =   \big[g |_{  L(d-1)} \big]_{{\BC}^0 (d)} \text{ for all } 0 \leq l < d, d\in \N_\d \!
                                \end{array} \!   \Bigg\}. 
$$

\item
In particular, $\ZC_{{\rm SL}(V)}(X,H,Y)\,\simeq\, \big\{ g \in \prod_{d \,\in\, \N_\d} {\rm GL} (L(d-1))
\,\mid\,  \bigchi_\d (g) \,=\,1 \big\}$.

\item If $\{X,\,H,\,Y\}$ is a $\s\l_2(\R)$-triple in $\s\u(V, \<>)$, then 
$$ 
\ZC_{{\rm SU}(V,\langle \cdot,\cdot \rangle)} (X,H,Y) 
 =  \left\{\! g \in {\rm SL}(V)   \middle\vert    \begin{array}{cccc}
                            g (X^l L(d-1))\, \subset \,  X^l L(d-1) ,\vspace{.14cm} \\
                             \!\!   \big[g |_{ X^l L(d-1)}\big]_{{\BC}^l (d)} =   \big[g |_{  L(d-1)} \big]_{{\BC}^0 (d)} , (gx,gy)_d = ( x,y)_d \!\! \vspace{.14cm}\\
                            \text{ for all }  d\in \N_\d,~ 0 \leq l \leq d-1, \text{ and }  x,y \in L(d-1)
                                 \end{array} \! 
 \right\};
$$
here $(\cdot,\, \cdot)_d$ is the form on $L(d-1)$ defined in \eqref{new-form}.

\item
In particular, $ \ZC_{{\rm SU}(V, \<>)} (X,H,Y) \,\simeq\, \big\{ g \in  \prod_{d \in \N_\d}
{\rm U} \big( L(d-1), (\cdot, \cdot)_d \big) \,\mid\,  \bigchi_\d (g) \,=\,1 \big\}$. 
\end{enumerate}
\end{lemma}

\begin{proof}
For notational convenience, denote
$$
\GC \, := \Bigg\{ g \in {\rm SL}(V) \Biggm| \begin{array}{c}
                                           g (X^l L(d-1)) \subset  X^l L(d-1) ; \\
                                          \big[g|_{ X^l L(d-1)}\big]_{{\BC}^l (d)} =  \big[g |_{  L(d-1)}\big]_{{\BC}^0 (d)}\,
                                          \forall\, 0 \leq l \leq d-1, d \in \N_\d
                                       \end{array} 
                                        \Bigg\}. 
$$
Take any $g \,\in \, \ZC_{{\rm SL}(V)}(X,H,Y)$. Then $g(L(d-1))\,\subseteq\, L(d-1)$
by \eqref{isotypicalcomp}.
In particular, it follows that $g(X^l L(d-1))\,\subseteq \,X^l L(d-1)$ because $g$ commutes with $X$.
Let $B_d \,:=\, [g|_{L(d-1)} ]_{\BC^0(d)}$ for all $d\,\in\, \N_\d$.
As $g$ commutes with $X$, it follows that $\big[g |_{ X^l L(d-1)} \big]_{\BC^l(d)} \,= \,B_d$ for
$0 \,\leq\, l \,\leq\, d-1$. This proves that $\ZC_{{\rm SL}(V)} (X,H,Y) \,\subset\, \GC$.

Take any $h \,\in\, \GC$. Then $h ( X^l L(d -1)) \,\subset\, X^l L(d -1)$
for all $ 0\,\leq\, l \,\leq\, d -1$ and $d \,\in\, \N_\d$.
For every $d \,\in\, \N_\d$, let $( a_{ij}^d)$ denote
the matrix $$\big[h\vert_{  L(d-1)} \big]_{\BC^0 (d)}\, \in\, {\rm GL}_{t_{d}} (\D)\, .$$
Then $( a_{ij}^d ) \,=\, \big[h\vert_{ X^l L(d-1)} \big]_{\BC^l(d)}$
for all $0 \,\leq\, l \,\leq\, d -1$.

We will show that $h$ commutes with $X$ and $H$. From \eqref{old-ordered-basis} it follows that ${\BC}$
is a $\D$-basis of $V$. Hence to prove that
$Xh \,=\, h X$ we need to show $Xh (X^{l}{v^d_{j}})\,=\, h X (X^{l}{v^d_{j}})$ for all $1\,\leq\, j \,\leq\,
t_{d}$ and $0 \,\leq\, l \,\leq\, d -1$ with $d \,\in\, \N_\d$. However this follows from the following straightforward
computation: 
\begin{equation}\label{commute-gX}
 h X (X^{l}{v^d_{j}}) \,=\, h X^{l+1}{v^d_{j}} \,=\, \sum_{i=1}^{t_{d}} X^{l+1} v^d_{i} a_{ij}^d
\,=\,X \big( \sum_{i=1}^{t_{d}} X^{l} v^d_{i} a_{ij}^d \big) \,=\, Xh ( X^{l}{v^d_{j}})\, .
\end{equation}
As $H$ acts as multiplication by a scalar in $\R$ (in fact, by a scalar in $\Z$) on the $\D$-basis
$\BC^l(d)$ (of $X^l L(d -1)$) for all $ 0 \,\leq\, l \,\leq\, d-1$ with $d \,\in\, \N_\d$,
it is immediate that $h$ commutes with $H$. 
In view of Lemma \ref{comm-XH}, we conclude that $h$ commutes with $Y$. 
This completes the proof of statement (1).

The third statement follows from statement (1) and Remark \ref{inv-form-old-new}.
\end{proof}

\begin{remark}
When $\D \,=\, \R$ or $\C$, the isomorphisms (2) and (4) in Lemma \ref{reductive-part-comp} were proved in
\cite[p. 251, 1.8]{SS} and \cite[p. 261, 2.25]{SS} using only the Jordan canonical forms. However, as the
non-commutativity of $\H$ creates technical difficulties in extending these results of \cite{SS} to the case of $\D \,=\, \H$,
we take a different approach by appealing to the Jacobson-Morozov theorem and the basic results on the structures of
finite dimensional representations of $\s\l_2 (\R)$.
\end{remark}

\subsection{Second cohomology groups of nilpotent orbits in \texorpdfstring{${\s\l}_n(\R)$}{Lg}}\label{sec-sl-n-R}

In this subsection we compute the second cohomology groups of nilpotent orbits in 
${\s\l}_n(\R)$ under the adjoint action of ${\rm SL}_n (\R)$. We will follow notation as defined in \S \ref{sec-notation}. First recall a standard parametrization of $\NC({\rm SL}_n (\R))$, the set of all nilpotent orbits in $\s\l_n(\R)$.
Let $X \,\in\, {\NC}_ {\s\l_n(\R)}$ be a nilpotent element in $\s\l_n(\R)$ and $\OC_X$ be the corresponding nilpotent
orbit in $\s\l_n(\R)$ under the adjoint action of ${\rm SL}_n (\R)$. First assume that $X\neq 0$. Let $\{X,\, H,\, Y\}\,\subset \,\s\l_n(\R)$ be a $\s\l_2(\R)$-triple.
Let $V:= \R^n$ be the right $\R$-vector space of column vectors.
Let $\{ d_1, \dotsc, d_s\}$ with $ d_1 < \cdots< d_s$ be the finite set of natural numbers that occur as dimension of the non-zero irreducible 
$\text{Span}_\R \{ X,H,Y\}$-submodules of $V$.
Recall that $M(d-1)$ is defined to be the isotypical component of $V$ containing all irreducible submodules of $V$ with highest weight $d-1$ and as in \eqref{definition-L-d-1}, 
we set $L(d-1)\,:= \,V_{Y,0} \cap M(d-1)$.
Let $t_{d_r} \,:=\, \dim_\R L(d_r-1)$ for $ 1 \,\leq\, r \,\leq\, s$. Then $\big[d_1^{t_{d_1} }, \,
\cdots ,\,d_s^{t_{d_s} }\big]
\,\in\, \PC(n)$ because $\sum_{r =1}^s t_{d_r} d_r \,=\,n$. This induces a map
$\Psi_{{\rm SL}_n (\R)} \,\colon\, \NC({\rm SL}_n(\R)) \,\longrightarrow\, \PC(n)$.
It is easy to see that $\Psi_{{\rm SL}_n (\R)} (\OC_X) \,\neq\, [ 1^n ]$ as $X \,\neq\, 0$.
By declaring $\Psi_{{\rm SL}_n (\R)} (\OC_0) \, := \,[ 1^n ]$ we now have a surjective map
\begin{equation}\label{Psi}
\Psi_{{\rm SL}_n (\R)} \,\colon\, {\NC} ({{\rm SL}_n (\R)}) \,\longrightarrow\, {\PC}(n)\, .
\end{equation}
The following known result says that $\Psi_{{\rm SL}_n (\R)} $ is ``almost''
a parametrization of the nilpotent orbits in $\s\l_n(\R)$. 

\begin{theorem}[{\cite[Theorem 9.3.3]{CoM}}]\label{sl-R-parametrization}
For the map $\Psi_{{\rm SL}_n (\R)}$ in \eqref{Psi},
$$\# \Psi_{{\rm SL}_n (\R)}^{-1} (\d)  =
\begin{cases}
 1 & \text{ for all } \  \d \in \PC(n) \setminus  \PC_{\rm even} (n)\\
 2 & \text{ for all } \ \d \in  \PC_{\rm even} (n).
\end{cases}$$ 
\end{theorem}
 
\begin{theorem}\label{sl-n-R}
Let $X \,\in\, \s \l_n(\R)$ be a nilpotent element. Let $\d\,=\,[d_1^{t_{d_1} },\, \cdots ,\,d_s^{t_{d_s} }]\,\in\, \PC(n)$
be the partition associated to the orbit $\OC_X$ (i.e., $\Psi_{{\rm SL}_n (\R)}(\OC_X) \,= \,\d$ in the notation of
Theorem \ref{sl-R-parametrization}). Then the following hold:
\begin{enumerate}
\item If $n\,\geq\, 3$, $\#\O_\d\,=\,1$ and $t_\theta\,=\,2$ for $\theta\,\in\,\O_\d$, then $\dim_\R H^2(\OC_X,\,\R )\,=\,1$.
\item In all the other cases $\dim_\R H^2(\OC_X, \,\R )\,=\,0$.
\end{enumerate}
\end{theorem}

\begin{proof}
This is obvious when $X \,=\,0$, so assume that $X \,\neq\, 0$.

The notation in Lemma \ref{reductive-part-comp} and the paragraph preceding it will be employed.
Let $\{X, \,H, \,Y\}$  $\subset \,\s\l_n(\R)$ be a $\s\l_2(\R)$-triple. 
Let $K$ be a maximal compact subgroup of $\ZC_{{\rm SL}_n(\R)}(X,H,Y)$. Let $M$ be a maximal compact subgroup of
${\rm SL}_n (\R)$ containing $K$.
As $M \,\simeq\, {\rm SO}_n$,  it follows that $\z (\k) \cap [\m,\,\m]\,=\,0$ when $n =2$, and 
$[\m,\,\m]\,=\, \m$ when $n \,\geq\, 3$. Thus using Theorem \ref{thm-nilpotent-orbit},
$$
H^2 ( {\OC}_X,\, \R) \,\simeq\, \begin{cases}
                                   0                         & \text{ if }\ n\,=\,2 \\
                                  \big[\z (\k)^*\big]^{K / K^\circ}  & \text{ if }\ n \,\geq\, 3\, .
                                  \end{cases}
$$
Treating $\R^n$ as a 
${\rm Span}_\R \{ X,H,Y\}$-module through the standard action of $\s\l_n(\R)$, construct  
a $\R$-basis $\BC$ as in \eqref{old-ordered-basis}, and consider the $\R$-algebra isomorphism
$\Lambda_\BC$ in \eqref{algebra-isom}. It now follows from
Lemma \ref{reductive-part-comp}(2) that the restriction of $\Lambda_\BC$ induces an isomorphism of Lie groups: 
\begin{equation}\label{sl-n-R-reductive-part}
\Lambda_\BC \,\colon\, \ZC_{{\rm SL}_n(\R)}(X,H,Y)  \,\stackrel{\sim}{\longrightarrow}\,
S\big( \prod_{d \in \N_\d} {\rm GL}_{t_d}(\R)_{\Delta}^d \big)\, .
\end{equation}
As $\prod_{d \in \N_\d}{\!({\rm O}_{t_d})}_{\Delta}^{d}$ is a maximal compact subgroup of 
$\prod_{d \in \N_\d} {\rm GL}_{t_{d}}(\R)_{\Delta}^{d} $,
and $ S( \prod_{d \in \N_\d} {\rm GL}_{t_{d}}(\R)_{\Delta}^{d})$ is normal in 
$\prod_{d \in \N_\d} {\rm GL}_{t_{d}}(\R)_{\Delta}^{d}$, it follows that
$S \big(\prod _{d \in \N_\d} ({\rm O}_{t_{d}})_{\Delta}^{d} \big)$ is a maximal compact subgroup of
$ S \big( \prod_{d \in \N_\d} {\rm GL}_{t_{d}}(\R)_{\Delta}^{d} \big)$. In view of the above observations 
it is now clear that for $n \,\geq\, 3$, 
\begin{equation}\label{sl-n-R-eq}
H^2 (\OC_X,\, \R)\,\simeq\, [\z (\k)^*]^{K/ K^\circ} \ \ \ \text{ where } \ \  K  \,\simeq \,
\prod_{\eta \in \E_\d} {\rm O}_{t_\eta} \times S \big(\prod _{\theta \in \O_\d} {{\rm O}_{t_\theta}} \big)\, .
\end{equation}

Consider the group $A \,:=\, S ( {\rm O}_{n_1} \times \cdots \times {\rm O}_{n_r})$ for positive integers
$n_1,\, \cdots \,,n_r$. Let $\a$ be the Lie algebra of $A$.
It is then easy to prove (see the proof of Case-2 in \cite[Theorem 5.6]{BC}) that 
\begin{equation}\label{adjoint-action-SO_n}
\dim_\R [\z (\a)]^{A/ A^\circ } \,=\, \begin{cases}
                                     1 & \text{ if\ $r\,=\,1$ ~ and ~ $n_r \,=\,2$} \\
                                     0 & \text{ otherwise}.
                                     \end{cases}
\end{equation}

It is also immediate that if $B_1, B_2$ are Lie groups, $B_3 \,:= \,B_1 \times B_2$, and
$\b_i $, $1\, \leq\, i\, \leq\, 3$, is the Lie algebra of $B_i$, then
\begin{equation}\label{adjoint-action}
 [\z (\b_3)]^{ B_3 / B^\circ_3 }\,\simeq \,[\z (\b_1)]^{ B_1 / B^\circ_1 } \oplus [\z (\b_2)]^{ B_2 / B^\circ_2 }\, . 
\end{equation}
Now the theorem follows from \eqref{adjoint-action-SO_n}, \eqref{adjoint-action} and \eqref{sl-n-R-eq}.
\end{proof}

\subsection{Second cohomology groups of nilpotent orbits in \texorpdfstring{${\s\l}_n(\H)$}{Lg}}\label{sec-sl-n-H}
In this subsection we compute the second cohomology groups of nilpotent orbits in 
${\s\l}_n(\H)$ under the adjoint action of ${\rm SL}_n (\H)$. We will follow notation as defined in \S \ref{sec-notation}. First recall a standard parametrization of $\NC({\rm SL}_n (\H))$, the set of all nilpotent orbits in $\s\l_n(\H)$.
Let $X \,\in\, \s\l_n(\H)$ be a nilpotent element with $\OC_X$ being its orbit in $\s\l_n(\H)$ under the
adjoint action of ${\rm SL}_n (\H)$. First assume that
$X\,\neq\, 0$. Let $\{X,\, H,\, Y\} \,\subset\, \s\l_n(\H)$ be a $\s\l_2(\R)$-triple.
Let $V:= \H^n$ be the right $\H$-vector space of column vectors.
Let $\{d_1, \dotsc, d_s\} \text{ with } d_1 < \cdots < d_s$ be the integers that occur as $\R$-dimensions of non-zero irreducible $\text{Span}_\R \{ X,H,Y\}$-submodules of $V$.
Recall that $M(d-1)$ is defined to be the isotypical component of $V$ containing all irreducible $\text{Span}_\R \{ X,H,Y\}$-submodules of $V$ with highest weight $(d-1)$, and as in \eqref{definition-L-d-1}, we set $L(d-1)\,:= \,V_{Y,0} \cap M(d-1)$. Recall that the space $L(d_r-1)$ is a $\H$-subspace for $r= 1 , \dotsc, s$.
Let $t_{d_r} \,:= \,\dim_\H L(d_r-1)$ for $ 1 \,\leq \,r \,\leq\, s$.  So
$[d_1^{t_{d_1} }, \,\cdots ,\,d_s^{t_{d_s} }]\,\in\, \PC(n)$ as $\sum_{r =1}^s t_{d_r} d_r \,=\,n$.
This gives a map
$\Psi_{{\rm SL}_n(\H)} \,\colon\, \NC({\rm SL} _{n}(\H)) \,\longrightarrow\,\PC(n)$. It follows easily that 
$\Psi_{{\rm SL}_n(\H)}  (\OC_X) \,\neq\, [ 1^n ]$.
By declaring $\Psi_{{\rm SL}_n(\H)}  (\OC_0) \, :=\, [ 1^n ]$ we now have a surjective map  
\begin{equation}\label{psi2}
\Psi_{{\rm SL}_n(\H)}  \,\colon\, \NC({\rm SL} _{n}(\H)) \,\longrightarrow\, \PC(n).
\end{equation}

\begin{theorem}[{\cite[Theorem 9.3.3]{CoM}}]\label{sl-H-parametrization}
The map $\Psi_{{\rm SL}_n(\H)}$ in \eqref{psi2} is a bijection. 
\end{theorem}

As the Lie algebra $\s\l_1(\H)$ is isomorphic to $\s\u(2)$ which is a compact Lie algebra, we will further assume that $n\geq 2$.

\begin{theorem}\label{sl-n-H}
For every nilpotent element $X\,\in\,\s\l_n(\H)$ when $n \geq 2$, 
$$\dim_\R H^2(\OC_X, \,\R) \,=\, 0\, .$$
\end{theorem}

\begin{proof}
We assume that $X \,\neq\, 0$ because the theorem is obvious when $X \,=\,0$.

Suppose that $\Psi_{{\rm SL}_n(\H)}( \OC_X) \,= \,\d$. Using the  
notation in Lemma \ref{reductive-part-comp} and the paragraph preceding it,
let $\{X,\, H,\, Y\} \,\subset\, \s\l_n(\H)$ be a $\s\l_2(\R)$-triple. 
Let $K$ be a maximal compact subgroup in $\ZC_{{\rm SL}_n (\H)}(X,H,Y)$. As ${\rm Sp}(n)$ is a maximal compact
subgroup of ${\rm SL}_n (\H)$, it follows from Theorem \ref{thm-nilpotent-orbit} that
\begin{equation}\label{H2-of-sl-n-H}
H^2 ( {\OC}_X,\, \R) \,\simeq\, [\z (\k)^\ast]^{K / K^\circ}
\end{equation}
for all $ X \,\neq\, 0$.
Treating $\H^n$ as a ${\rm Span}_\R \{ X,H,Y\}$-module via the standard action of $\s\l_n(\H)$, we construct a 
$\H$-basis $\BC$ as in \eqref{old-ordered-basis}, and consider the $\R$-algebra isomorphism $\Lambda_\BC$ in
\eqref{algebra-isom}. It now follows from
Lemma \ref{reductive-part-comp}(2) that the restriction of $\Lambda_\BC$ induces an isomorphism of Lie groups
$$
\Lambda_\BC \,\colon\, \ZC_{{\rm SL}_n(\H)}(X,H,Y)  \,\stackrel{\sim}{\longrightarrow}\,
S \big( \prod_{d \in \N_\d} {\rm GL}_{t_{d}}(\H)_{\Delta}^{d} \big)\,.
$$ 
  
As $ \prod_{d \in \N_\d} {\rm Sp}(t_{d})_\Delta ^{d}$  is a maximal compact subgroup of $\prod_{d \in \N_\d}
{\rm GL}_{t_{d}}(\H)_{\Delta}^{d} $, and $$\prod_{d \in \N_\d} {\rm Sp}(t_{d})_\Delta ^{d} \,\subset\, 
S( \prod_{d \in \N_\d} {\rm GL}_{t_{d}}(\H)_{\Delta}^{d})\, ,$$ it follows that $\prod_{d \in \N_\d}{\rm Sp}(t_{d})_\Delta^{d}$
is a maximal compact subgroup of
 $ S\big(\prod_{d \in \N_\d} {\rm GL}_{t_{d}}(\H)_{\Delta}^{d} \big)$. In particular we have,
$$
K \,\simeq \, \prod_{d \in \N_\d}{\rm Sp}(t_{d})\, .
$$
As $\z(\k) \,=\,0$, it now follows from \eqref{H2-of-sl-n-H} that $\dim_\R H^2(\OC_X ,\, \R)\,=\,0$.
\end{proof}

\subsection{Second cohomology groups of nilpotent orbits in \texorpdfstring{${\s\u}(p,q)$}{Lg}}\label{sec-su-pq}

Let $n$ be a positive integer and $(p,q)$ be a pair of non-negative integers such that $p + q 
\,=\,n$. The aim in this subsection is to compute the second cohomology groups of nilpotent orbits in 
${\s\u}(p,q)$ under the adjoint action of ${\rm SU} (p,q)$. As explained before,
we need to address only non-compact groups; so it will be assumed that $p \,>\,0$ and $q\,>\,0$.
Throughout this subsection $\<>$ 
denotes the Hermitian form on $\C^n$ defined by $\langle x,\, y \rangle \,:=\, \overline{x}^t{\rm 
I}_{p,q} y$, where ${\rm I}_{p,q}$ is as in \eqref{defn-I-pq-J-n}. We will follow notation as defined in \S \ref{sec-notation}.

First recall a standard parametrization of $\NC({\rm SL}_n (\C))$, the set of all nilpotent orbits in $\s\l_n(\C)$. Using
Proposition \ref{J-basis}, this can be done as
in Section \ref{sec-sl-n-R} for $\NC ({\rm SL}_n (\R))$.
Let $X' \,\in\, {\NC}_ {\s\l_n(\C)}$ be a nilpotent element. First assume that $X' \,\neq\, 0$, and
let $\{X',\, H',\, Y'\}
\,\subset\, \s\l_n(\C)$ be a $\s\l_2(\R)$-triple. 
Let $V := \C^n$ be the right $\C$-vector space of column vectors. 
Let $\{c_1,\, \cdots,\, c_l\}$, with $c_1 \,<\, \cdots \,<\, c_l$,
be the finitely many ordered integers that occur as $\R$-dimension of non-zero irreducible
$\text{Span}_\R \{ X',H',Y'\}$-submodules of $V$.
Recall that $M(c-1)$ is defined to be the isotypical component of $V$ containing all irreducible $\text{Span}_\R \{ X' ,H' ,Y' \}$-submodules of $V$ with highest weight $(c-1)$, and as in \eqref{definition-L-d-1}, we set $L(c-1)\,:= \,V_{Y',0} \cap M(c-1)$. Recall that the space $L(c_r-1)$ is a $\C$-subspace for $1\leq r \leq l$.
Let $t_{c_r} \,:=\, \dim_\C L(c_r-1)$ for $ 1\,\leq\, r \,\leq\, l$. Then as $\sum_{r=1}^l
t_{c_r} c_r \,=\,n$, we have
$[c_1^{t_{c_1}}, \,\cdots,\,c_l^{t_{c_l}}]  \,\in\, \PC(n)$. 
Define
$\Psi_{{\rm SL}_n (\C)} ( \OC_{X'})\,:=\, [c_1^{t_{c_1}},\, \cdots ,\,c_l^{t_{c_l}}]$.
It is easy to see that $\Psi_{{\rm SL}_n (\C)} (\OC_X) \,\neq\, [ 1^n ]$ as $X \,\neq\, 0$.
By declaring $\Psi_{{\rm SL}_n (\C)} (\OC_0) \,= \,[ 1^n ]$ we obtain
a bijection $\Psi_{{\rm SL}_n (\C)} \,\colon\, \NC ({\rm SL}_n (\C)) \,\longrightarrow\, \PC (n)$.
As ${\rm SU} (p,q) \,\subset\, {\rm SL}_n (\C)$, consequently $\NC_{{\s\u}(p,q)} \,\subset\, \NC_{\s\l_n(\C)}$,
we have the inclusion map $\Theta_{{\rm SU}(p,q)} \,\colon\, \NC ({\rm SU} (p,q)) \,\longrightarrow\,
\NC ( {\rm SL}_n (\C))$. Let
$$
\Psi'_{{\rm SU} (p,q)}\,:=\, \Psi_{{\rm SL}_n (\C)} \circ \Theta_{{\rm SU}(p,q)}\,
\colon\, \NC ({\rm SU} (p,q))  \,\longrightarrow\, \PC (n)
$$ be the composition.

Let now $X\, \in\, {\s\u}(p,q)$ be a nilpotent element and $\OC_X$ the nilpotent orbit of $X$ in $\s\u(p,q)$. Assume that
$X \,\neq\, 0 $, and let $\{X, H, Y\} \,\subset \,{\s\u}(p,q)$ be a $\s\l_2(\R)$-triple. 
We now apply Proposition \ref{unitary-J-basis}, Remark \ref{unitary-J-basis-rmk}(2), and follow the notation used therein. 
Let $\{d_1,\, \cdots,\, d_s\}$, with $d_1 \,<\, \cdots \,<\, d_s$, be the finite ordered set
of integers that arise as $\R$-dimension of non-zero irreducible 
$\text{Span}_\R \{ X,H,Y\}$-submodules of $V$.
Let $t_{d_r} \,:=\, \dim_\C L(d_r-1)$ for $ 1 \,\leq \,r \,\leq\, s$. Then we have
${\d}\,:=\, [d_1^{t_{d_1} },\, \cdots ,\,d_s^{t_{d_s} }]  \,\in\, \PC(n)$, and  moreover,
$\Psi'_{{\rm SU} (p,q)} (\OC_X) \,=\, {\d}$.

We next assign $\sgn_{\OC_X} \,\in\, \SC_{\d}(p,q)$ to each $\OC_X \,\in\, \NC ({\rm SU}(p,q))$; see
 \eqref{S-d-pq} for the definition of $\SC_{\d}(p,q)$. 
For each $d \,\in\, \N_\d$ (see \eqref{Nd-Ed-Od} for the definition of $\N_\d$) we will define a
$t_d \times d$ matrix $(m^d_{ij})$ in $\Ab_d$ 
which depends only on the orbit $\OC_X$ containing $X$; see \eqref{A-d} for the definition of $\Ab_d$. For this, recall that the form
$(\cdot,\cdot)_{d} \,\colon\, L(d-1) \times L(d-1) \,\longrightarrow\, \C$ defined as in
\eqref{new-form} is Hermitian (respectively, skew-Hermitian) when $d$ is odd (respectively, even). 
Denoting the  signature of $(\cdot,\,\cdot)_{d}$ by $(p_{d},\, q_{d})$ we now define
$$
  m^d_{i1} \,:= \begin{cases}
                   +1  &  \text{ if }  \ 1 \leq i \leq p_d \\
                   -1  &  \text{ if } \  p_d < i \leq t_d  
                    \end{cases};   \  d \in \N_\d,
$$
and 
\begin{align}\label{def-sign-alternate}
 m^d_{ij}  \,  &:= (-1)^{j+1}m^d_{i1} \qquad \text{if   $ 1<j \leq d , \ d\in \E_\d \cup  \O^1_\d$};\\
 m^\theta_{ij}\, &:= \begin{cases}
              (-1)^{j+1}m^\theta_{i1}  & \text{ if } \ 1<j \leq \theta-1 \\
              -m^\theta_{i1}           &    \text{ if } j = \theta
             \end{cases} , \,  \theta\in \O^3_\d.    \label{def-sign-alternate-1}       
\end{align}
The above matrices $(m^d_{ij})$ clearly satisfy the conditions in \eqref{yd-def2}.
We set $\sgn_{\OC_X} \, := \, ( (m^{d_1}_{ij}),\, \cdots , $ \,$ (m^{d_s}_{ij}))$.
It now follows from the last paragraph of Remark \ref{CM-correction} that
$\sgn_{\OC_X}\,\in \,\SC_\d(p,q)$. 
Thus we have the map
$$
\Psi_{{\rm SU} (p,q)} \,\colon\, \NC ({\rm SU} (p,q)) \,\longrightarrow \,\YC(p,q) \, , \quad
\OC_X \,\longmapsto\, \big(\Psi'_{{\rm SU} (p,q)} (\OC_X),\, \sgn_{\OC_X}  \big) \,;
$$
where $\YC(p,q) $ is as in \eqref{yd-Y-pq}.
The following theorem is standard.

\begin{theorem}\label{su-pq-parametrization}
The above map $\Psi_{{\rm SU} (p,q)} \,\colon\, \NC ({\rm SU} (p,q)) \,
\longrightarrow \,\YC(p,q)$ is a bijection.
\end{theorem}

\begin{remark}
Taking into account the error in \cite[Lemma 9.3.1]{CoM} mentioned in Remark \ref{CM-correction}, the 
parametrization in Theorem \ref{su-pq-parametrization} is a modification of the one in 
\cite[Theorem 9.3.3]{CoM}.
\end{remark}

Let $0\,\not=\, X \,\in\, \NC _{\s\u(p,q)}$, and $\{X,\,H,\,Y\}$ be a $\s\l_2(\R)$-triple in $\s\u (p,q)$.
Let $\Psi_{{\rm SU} (p,q)} (\OC_X) \,=\, \big( \d ,\, \sgn_{\OC_X} \big)$.
Then   $\Psi'_{{\rm SU} (p,q)} (\OC_X) \,=\, \d$. Recall that $\sgn_{\OC_X}$ determines the signature
of $(\cdot, \,\cdot)_d$ on $L(d-1)$ for every $d \,\in\, \N_\d$; let $(p_d,\, q_d)$ be the signature
of $(\cdot, \,\cdot)_d$, for $d \,\in\, \N_\d$.
Let $( v^{d}_1,\, \cdots,\, v^{d}_{t_{d}})$ be an ordered $\C$-basis of $L(d-1)$ as in Proposition
\ref{unitary-J-basis}. It now follows from Proposition \ref{unitary-J-basis}(3)(a) that
$( v^{d}_1,\, \cdots,\, v^{d}_{t_{d}} )$ is an orthogonal basis for $(\cdot, \,\cdot)_{d}$. 
We also assume that the vectors in the ordered basis $( v^{d}_1, \,\cdots,\, v^{d}_{t_{d}} )$
satisfies the properties in Remark \ref{unitary-J-basis-rmk}(2).
In view of the signature of $(\cdot, \,\cdot)_d$ we may further assume that 
\begin{align}
\sqrt{-1}( v^{\eta}_j, v^{\eta}_j)_{\eta} & = \begin{cases}
                                               +1  & \text{if }   1 \leq j \leq p_{\eta}  \\
					       -1  &  \text{if }  p_{\eta} < j \leq t_{\eta} 
                                               \end{cases};  \text{  when } \eta  \in \E_\d, \label{orthonormal-basis-even} \\
    ( v^{\theta}_j, v^\theta_j)_{\theta} & = \begin{cases}
                                           +1  &  \text{if } 1 \leq j \leq p_{\theta}  \\
                                           -1  &  \text{if } p_{\theta} < j \leq t_{\theta}
                                          \end{cases};  \text{  when }  \theta  \in \O_\d.\label{orthonormal-basis-odd}
\end{align}

Let $\big\{ \widetilde{w}^d_{jl} \,\mid\, 1 \,\leq\, j \,\leq\, t_{d}, \ 0 \,\leq \,l
\,\leq\, d-1 \big\}$ be the $\C$-basis of $M (d-1)$ constructed using 
$( v^d_1,\, \cdots,\, v^d_{t_{d}} )$ as done in Lemma \ref{orthogonal-basis-C}. 
For each $d \,\in\, \N_\d$, $0 \,\leq\, l \,\leq\, d-1$, set 
$$
V^l (d) \,:=\, \text{ Span}_\C \{ \widetilde{w}^d_{1l}, \,\cdots ,\, \widetilde{w}^d_{t_{d}l} \}\, .
$$
The ordered basis $\big(\widetilde{w}^d_{1l},\, \cdots ,\, \widetilde{w}^d_{t_{d}l} \big)$ of
$ V^l (d) $ will be denoted by $\CC^l (d)$.

\begin{lemma}\label{reductive-part-comp-su-pq}
The following holds: 	
 $$
 \ZC_{ {\rm SU} (p,q) }(X,H,Y) = \Bigg\{ g \in { {\rm SU} (p,q) } \Biggm| 
                    \begin{array}{cc}
                    g ( V^l (d)  ) \, \subset \, V^l (d) \,  \text{ and }    \vspace{.14cm}\\
                  \!  \big[g|_{ V^l(d)}\big]_{{\CC}^l(d)} \! = \big[g |_{ V^0 (d)}\big]_{{\CC}^0 (d)} \text{ for all } d \in \N_\d, \, 0 \leq l < d \!
                   \end{array} \! \Bigg\}.
 $$
\end{lemma}
 
\begin{proof}
As $\ZC_{ {\rm SU} (p,q) } (X,H,Y) \,=\, {\rm SU} (p,q) \cap \ZC_{ {\rm SL}_n (\C)} (X,H,Y)$, using
Lemma \ref{reductive-part-comp}(1) it follows that
 $$
 \ZC_{ {\rm SU} (p,q) }(X,H,Y)= \Bigg\{ g \in { {\rm SU} (p,q) } \Biggm| \negthickspace
                       \begin{array}{c} 
                           g \big(X^l L(d-1) \big) \subset  X^l L(d-1) \ \text{ and }  \vspace{.14cm}\\
                         \!  \big[g |_{ X^l L(d-1)} \big]_{_{{\BC}^l (d)}} \negthickspace = \big[g |_{  L(d-1)}\big]_{_{{\BC}^0 (d)}}  \text{for all } d \in \N_\d,   0 \leq l< d \!
            \end{array} \!
                \Bigg\}.
 $$
For fixed $d\in \N_\d$ we consider the $t_{d} \times 1$-column matrices
$\big[\widetilde{w}^d_{j\,l}\big]_{ 1 \leq j \leq t_{d}}$,
$\big[\widetilde{w}^d_{j\,(d-1-l)}\big]_{  1 \leq j \leq t_{d}}$
and $\big[X^l {v}^d_{j}\big]_{ 1 \leq j \leq t_{d}}$, 
 $\big[X^{d-1-l}v^d_{j}\big]_{  1 \leq j \leq t_{d}}$. Rewriting the definitions in Lemma
\ref{orthogonal-basis-C} when $\eta \in \E_\d$,
 $$
  \big[\widetilde{w}^\eta_{jl}\big] \,=\,  \Big(\big[ X^l{v}^\eta_{j}\big]  +
\big[X^{\eta-1-l}v^\eta_{j}\big]\sqrt{-1}\Big) \frac{1}{\sqrt{2}} \, ; \
 \big[\widetilde{w}^\eta_{j(\eta-1-l)}\big] \,=\,
\Big( \big[X^l{v}^\eta_{j}\big]- \big[X^{\eta-1-l}v^\eta_{j}\big]\sqrt{-1}\Big) \frac{1}{\sqrt{2}}
$$
for $ 0 \,\leq\, l \,< \,\eta/2$.

Furthermore, when $1\, \leq\, \theta \,\in \,\O_\d$, 
$$
\big[\widetilde{w}^\theta_{jl}\big] \,=\,  \Big( \big[ X^l{v}^\theta_{j}\big] +
\big[X^{\theta-1-l}v^\theta_{j}\big]\Big) \frac{1}{\sqrt{2}} \,;   \ 
\big[\widetilde{w}^\theta_{j(\theta-1-l)}\big] \,=\,
\Big( \big[X^l{v}^\theta_{j}\big] - \big[X^{\theta-1-l}v^\theta_{j}\big]\Big) \frac{1}{\sqrt{2}} ,
 $$ 
for all $ 0 \,\leq\, l \,<\, (\theta-1)/2$, while for $l\,=\, ( \theta -1)/2$,
 $$
 \big[\widetilde{w}^\theta_{j \, (\theta-1)/2}\big] \,=\,  \big[X^{(\theta-1)/2}v^\theta_{j}\big]\, .
 $$
 When $\theta \,=\,1$, then $\big[\widetilde{w}^\theta_{j \, }\big] \,=\,  \big[v^\theta_{j}\big]$.

In particular, if $d \,\in\, \N_\d$ is fixed, then for every $0 \,\leq\, l \,\leq\, d-1$
the following holds:
$$g (X^l L(d-1)) \,\subset \, X^l L(d-1)~\ \text{if and only if}~ \
g (  V^l (d)  ) \,\subset \, V^l (d)\, ,$$
and moreover, 
$$[g |_{ X^l L(d-1)}]_{{\BC}^l (d)} \,=\,   [g |_{  L(d-1)}]_{{\BC}^0 (d)}~\ \text{ if and only if }
~\  [g |_{  V^l (d) }]_{{\CC}^l (d)} \,=\,[g |_{  V^0 (d) }]_{{\CC}^0 (d)}\, .$$
In fact, for any $g$ as above,
$\big[g|_{L(d-1)}\big]_{\BC^0(d)} \,=\, \big[g|_{V^0(d) } \big]_{\CC^0(d)}$.
\end{proof}
 
For every $d \,\in\, {\N_\d}$ and $0\,\leq\, l \,\leq\, d-1$,
orderings on the sets $\{ v \,\in\, \CC^l (d) \,\mid\, \langle v, \,v \rangle \,>\, 0 \}$, 
$\{ v \,\in\, \CC^l (d) \,\mid\, \langle v,\, v \rangle \,<\, 0 \}$, will be constructed.
These ordered sets will be denoted by $\CC^l_+ (d)$ and $\CC^l_- (d)$ respectively. 
The construction will be done in three steps according as $d \,\in\, \E_\d$ or
$d \,\in\, \O^1_\d$ or $d \,\in \,\O^3_\d$; see \eqref{Nd-Ed-Od} and \eqref{Od1-Od3} for definitions of $\E_\d,\, \O^1_\d, \, \O^3_\d$. 

For each $\eta \,\in\,  \E_\d$ and $ 0\,\leq\, l \,\leq \,\eta -1$, define 
\begin{align*}
\CC_{+}^l(\eta) &:=  \begin{cases}
                 \big(\, \widetilde{w}^{\eta}_{_{1 \, l}}, \, \cdots , \,  \widetilde{w}^{\eta}_{_{p_{\eta}\,l}} \, \big) &  \text{ if } l \text{ is  even} \vspace{.1cm}\\
                 \big(\,\widetilde{w}^{\eta}_{_{(p_{\eta}+1)\,l}}, \, \cdots ,\,  \widetilde{w}^{\eta}_{_{t_{\eta}\, l}}  \big) &    \text{ if } l \text{ is  odd},
                     \end{cases}\vspace{.1cm} \\
\CC_{-}^l(\eta) &:=  \begin{cases}
                    \big( \, \widetilde{w}^{\eta}_{_{(p_{\eta}+1)\,l}},\  \cdots ,\ \widetilde{w}^{\eta}_{_{t_{\eta}\, l}}  \big) & \text{if } l \text{ is  even} \vspace{.1cm}\\
               \big(\, {\widetilde{w}}^{\eta}_{_{1\, l}}, \  \cdots ,\  {\widetilde {w}}^{\eta}_{_{p_{\eta}\,l}} \, \big) & \text{if } l \text{ is  odd}.
                       \end{cases} 
\end{align*}

For each $\theta \,\in\,  \O^1_\d$, define
\begin{equation}\label{C+theta}
\CC_+^l(\theta) := \begin{cases}
            \big(\, \widetilde{w}^{\theta}_{_{1 \, l}}, \  \cdots ,\  \widetilde{w}^{\theta}_{_{p_{\theta} \, l}} \, \big)   & \text{ if } l \text{ is  even and } 0 \leq l <  (\theta-1)/2 \vspace{.15cm}\\
              \big(\, \widetilde{w}^{\theta}_{_{(p_{\theta} +1)\, l}},\  \cdots ,\  \widetilde{w}^{\theta}_{_{t_{\theta} \, l}} \, \big) &  \text{ if } l \text{ is  odd and } 0 \leq l < (\theta-1)/2 \vspace{.15cm}\\
            \big(\, \widetilde{w}^{\theta}_{_{1 \, l}}, \  \cdots ,\  \widetilde{w}^{\theta}_{_{p_{\theta} \, l}} \, \big)   & \text{ if } l =  (\theta-1)/2 \vspace{.15cm}\\
            \big(\, \widetilde{w}^{\theta}_{_{1 \, l}}, \  \cdots ,\  \widetilde{w}^{\theta}_{_{p_{\theta} \, l}} \, \big)    & \text{ if } l \text{ is  odd and } (\theta+1)/2 \leq l \leq (\theta-1) \vspace{.15cm}\\
             \big(\, \widetilde{w}^{\theta}_{_{(p_{\theta} +1)\, l}},\  \cdots ,\ \widetilde{w}^{\theta}_{_{t_{\theta} \, l}} \, \big)  &  \text{ if } l \text{ is  even and } (\theta+1)/2 \leq l \leq (\theta-1)
               \end{cases}
\end{equation}
and
\begin{equation}\label{C-theta}
\CC_-^l(\theta) := \begin{cases}
              \big(\, \widetilde{w}^{\theta}_{_{(p_{\theta}  +1)\, l}},\  \cdots ,\ \widetilde{w}^{\theta}_{_{t_{\theta} \, l}} \, \big) & \text{ if } l \text{ is  even and } 0 \leq l <  (\theta-1)/2 \vspace{.15cm}\\
                 \big(\, \widetilde{w}^{\theta}_{_{1 \, l}}, \  \cdots ,\  \widetilde{w}^{\theta}_{_{p_{\theta} \, l}} \, \big)   &   \text{ if } l \text{ is  odd and } 0 \leq l < (\theta-1)/2 \vspace{.15cm}\\
         \big(\, \widetilde{w}^{\theta}_{_{(p_{\theta}  +1)\, l}},\  \cdots ,\ \widetilde{w}^{\theta}_{_{t_{\theta} \, l}} \, \big) &  \text{ if }  l =  (\theta-1)/2 \vspace{.15cm}\\         
              \big(\, \widetilde{w}^{\theta}_{_{(p_{\theta} +1)\, l}},\  \cdots ,\ \widetilde{w}^{\theta}_{_{t_{\theta} \, l}} \, \big)  &  \text{ if }   l \text{ is  odd and } (\theta+1)/2 \leq l \leq (\theta-1) \vspace{.15cm}\\
              \big(\, \widetilde{w}^{\theta}_{_{1 \, l}}, \  \cdots ,\  \widetilde{w}^{\theta}_{_{p_{\theta} \, l}} \, \big)  & \text{ if } l \text{ is  even and } (\theta+1)/2 \leq l \leq (\theta-1) .
               \end{cases}
\end{equation}
Similarly, for each $\zeta \,\in\,  \O^3_\d$, define 
\begin{equation}\label{C+zeta}
\CC_+^l(\zeta) := \begin{cases}
            \big(\, \widetilde{w}^{\zeta}_{_{1 \, l}}, \  \cdots ,\  \widetilde{w}^{\zeta}_{_{p_{\zeta} \, l}} \, \big)   & \text{ if } l \text{ is  even and } 0 \leq l <  (\zeta-1)/2 \vspace{.15cm}\\
              \big(\, \widetilde{w}^{\zeta}_{_{(p_{\zeta} +1)\, l}},\  \cdots ,\  \widetilde{w}^{\zeta}_{_{t_{\zeta} \, l}} \, \big) &  \text{ if } l \text{ is  odd and } 0 \leq l < (\zeta-1)/2 \vspace{.15cm}\\
         \big(\, \widetilde{w}^{\zeta}_{_{(p_{\zeta} +1)\, l}},\  \cdots ,\  \widetilde{w}^{\zeta}_{_{t_{\zeta} \, l}} \, \big) & \text{ if } l =  (\zeta-1)/2\vspace{.15cm}\\                 
                 \big(\, \widetilde{w}^{\zeta}_{_{(p_{\zeta} +1)\, l}},\  \cdots ,\ \widetilde{w}^{\zeta}_{_{t_{\zeta} \, l}} \, \big) & \text{ if } l \text{ is  even and } (\zeta+1)/2 \leq l \leq (\zeta-1) \vspace{.15cm}\\
            \big(\, \widetilde{w}^{\zeta}_{_{1 \, l}}, \  \cdots ,\  \widetilde{w}^{\zeta}_{_{p_{\zeta} \, l}} \, \big)    & \text{ if } l \text{ is  odd and } (\zeta+1)/2 \leq l \leq (\zeta-1) 
               \end{cases}
\end{equation}
and
\begin{equation}\label{C-zeta}
 \CC_-^l(\zeta) := \begin{cases}
              \big(\, \widetilde{w}^{\zeta}_{_{(p_{\zeta}  +1)\, l}},\  \cdots ,\ \widetilde{w}^{\zeta}_{_{t_{\zeta} \, l}} \, \big) &  \text{ if } l \text{ is  even and } 0 \leq l <  (\zeta-1)/2 \vspace{.15cm}\\
                 \big(\, \widetilde{w}^{\zeta}_{_{1 \, l}}, \  \cdots ,\  \widetilde{w}^{\zeta}_{_{p_{\zeta} \, l}} \, \big) &  \text{ if } l \text{ is  odd and } 0 \leq l < (\zeta-1)/2 \vspace{.15cm}\\       
      \big(\, \widetilde{w}^{\zeta}_{_{1 \, l}}, \  \cdots ,\  \widetilde{w}^{\zeta}_{_{p_{\zeta} \, l}} \, \big)  & \text{ if } l =  (\zeta-1)/2 \vspace{.15cm}\\        
             \big(\, \widetilde{w}^{\zeta}_{_{1 \, l}}, \  \cdots ,\  \widetilde{w}^{\zeta}_{_{p_{\zeta} \, l}} \, \big) & \text{ if } l \text{ is  even and } (\zeta+1)/2 \leq l \leq (\zeta-1) \vspace{.15cm}\\
             \big(\, \widetilde{w}^{\zeta}_{_{(p_{\zeta} +1)\, l}},\  \cdots ,\ \widetilde{w}^{\zeta}_{_{t_{\zeta} \, l}} \, \big)  &   \text{ if }   l \text{ is  odd and } (\zeta+1)/2 \leq l \leq (\zeta-1) .
               \end{cases}
\end{equation}

For all $d \,\in\, \N_\d$ and $0 \,\leq\, l \,\leq\, d-1$, define
$$
V^l_+ (d) \,:=\, \text{ Span}_\C \{v \,\in\,  \CC^l (d) \,\mid \, \langle v,\, v \rangle\, >\, 0 \},
\ \
V^l_- (d) \,:=\, \text{ Span}_\C \{ v \,\in\,\CC^l (d)\,\mid\, \langle v,\, v \rangle \,<\, 0 \}.
$$
It can be verified using \eqref{orthonormal-basis-even}, \eqref{orthonormal-basis-odd}
together with the orthogonality relations in Lemma \ref{orthogonal-basis-C} that
$\CC^l_+ (d)$ (respectively, $ \CC^l_- (d)$) is indeed an ordered set based on
the (unordered) set
$$\{ v \,\in\, \CC^l (d) \,\mid\, \langle v,\, v \rangle \,>\, 0 \}$$ (respectively,
$\{ v \,\in\, \CC^l (d) \,\mid\, \langle v,\, v \rangle \,<\, 0 \}$) for all
$d \,\in\, {\N_\d}$ and $ 0 \,\leq\, l \,\leq\, d-1$.
In particular, $\CC^l_+ (d)$ and $ \CC^l_- (d)$ are ordered bases of
$V^l_+ (d)$ and $V^l_- (d)$ respectively, for all $d \,\in\, \N_\d$, $0 \,\leq\, l \,\leq\, d-1$.

In the next lemma we specify a maximal compact subgroup of $\ZC_{ {\rm SU} (p,q) } (X,H,Y)$ which will be used in Proposition \ref{max-cpt-su-pq-wrt-onb}.
For notational convenience, we will use $ (-1)^l$ to denote the sign `$+$' or the sign `$-$'
depending on whether $l$ is an even integer or an odd integer.

\begin{lemma}\label{max-cpt-su-pq}
Let $K$ be the subgroup of $\ZC_{ {\rm SU} (p,q) } (X,H,Y)$ consisting of all 
$g\,\in\,
\ZC_{{\rm SU}(p,q)}(X,H,Y)$ satisfying the following conditions:

\begin{enumerate}
 \item $g ( V^l_+ (d)) \subset \, V^l_+ (d)$ and $g ( V^l_- (d)  ) \subset \, V^l_-(d)$,  for all  $ d \in \N_\d \, \text{ and } \, 0 \leq l \leq d -1$.

 \item When $ \eta \in \E_\d $,
 $$\begin{array}{cc}
            \Big[g |_{ V_+^0 (\eta) }\Big]_{{\CC}^0 _+(\eta)} = \Big[g |_{  V^l_{(-1)^{l}} (\eta)}\Big]_{{\CC}^l_{(-1)^{l} } (\eta)}  \qquad \vspace{.12cm}\\
      \Big[g |_{ V_-^0 (\eta) }\Big]_{{\CC}^0_- (\eta)} = \Big[g |_{  V^l_{(-1)^{l+1}} (\eta)}\Big]_{{\CC}^l_{(-1)^{l+1} } (\eta)}   
     \end{array}; \text{ for all }\, 0 \leq l \leq \eta-1.
$$
  \item When $ \theta \in \O^1_\d $,
$$
\Big[g |_{ V_+^0 (\theta) }\Big]_{{\CC}^0_+ (\theta)} = 
             \begin{cases}                                                            
                 \Big[g |_{  V^l_{(-1)^{l}} (\theta)}\Big]_{{\CC}^l_{(-1)^{l}} (\theta)} & \text{ for all } 0 \leq l < (\theta-1)/2 \vspace{.15cm}\\
            \Big[g |_{  V^{(\theta-1)/2}_{+} (\theta)}\Big]_{{\CC}^{(\theta-1)/2}_{+} (\theta)} \vspace{.15cm}  \\
                 \Big[g |_{  V^l_{(-1)^{l+1}} (\theta)}\Big]_{{\CC}^l_{(-1)^{l+1} } (\theta)} &  \text{ for all } (\theta-1)/2 < l \leq \theta-1,
                 \end{cases}                 
$$
$$
  \Big[g |_{ V_-^0 (\theta) }\Big]_{{\CC}^0_- (\theta)} =
     \begin{cases}
      \Big[g |_{  V^l_{(-1)^{l+1}} (\theta)}\Big]_{{\CC}^l_{(-1)^{l+1} } (\theta)} &  \text{ for all } 0 \leq l < (\theta-1)/2 \vspace{.15cm}\\
       \Big[g |_{  V^{(\theta-1)/2}_{-} (\theta)}\Big]_{{\CC}^{(\theta-1)/2}_{-} (\theta)} \vspace{.15cm}  \\
      \Big[g |_{  V^l_{(-1)^{l}} (\theta)}\Big]_{{\CC}^l_{(-1)^{l}} (\theta)} & \text{ for all } (\theta-1)/2 < l \leq \theta-1 .     
     \end{cases}
$$
\item When $ \zeta \in \O^3_\d $,
$$
\Big[g |_{ V_+^0 (\zeta) }\Big]_{{\CC}^0_+ (\zeta)} = 
             \begin{cases}                                                            
                 \Big[g |_{  V^l_{(-1)^{l}} (\zeta)}\Big]_{{\CC}^l_{(-1)^{l}} (\zeta)} &  \text{ for all } 0 \leq l < (\zeta-1)/2 \vspace{.15cm}\\            
      \Big[g |_{  V^{(\zeta-1)/2}_{-} (\zeta)}\Big]_{{\CC}^{(\zeta-1)/2}_{-} (\zeta)} \vspace{.15cm}  \\
                 \Big[g |_{  V^l_{(-1)^{l+1}} (\zeta)}\Big]_{{\CC}^l_{(-1)^{l+1} } (\zeta)} & \text{ for all } (\zeta-1)/2 < l \leq \zeta-1,
                 \end{cases}                 
$$
$$
  \Big[g |_{ V_-^0 (\zeta) }\Big]_{{\CC}^0_- (\zeta)} =
     \begin{cases}
      \Big[g |_{  V^l_{(-1)^{l+1}} (\zeta)}\Big]_{{\CC}^l_{(-1)^{l+1} } (\zeta)} &  \text{ for all } 0 \leq l < (\zeta-1)/2 \vspace{.15cm}\\
      \Big[g |_{  V^{(\zeta-1)/2}_{+} (\zeta)}\Big]_{{\CC}^{(\zeta-1)/2}_{+} (\zeta)} \vspace{.15cm}  \\
            \Big[g |_{  V^l_{(-1)^{l}} (\zeta)}\Big]_{{\CC}^l_{(-1)^{l}} (\zeta)} &  \text{ for all } (\zeta-1)/2 < l \leq \zeta-1.    
     \end{cases}
$$ 
\end{enumerate}
Then $K$ is a maximal compact subgroup of $\ZC_{ {\rm SU}(p,q)}(X,H,Y)$.
\end{lemma}

\begin{proof}
In view of the description of  $\ZC_{ {\rm SU} (p,q) } (X,H,Y)$ in the Lemma
\ref{reductive-part-comp-su-pq} we see that its subgroup
$$
K:= \! \Bigg\{g \in {\rm SU} (p,q) \!   \Biggm| \! \begin{array}{cc}
                               g ( V^l_+ (d)) \subset \, V^l_+ (d) \, , \  \  g ( V^l_- (d)  ) \subset \, V^l_-(d) \,  \text{ and} \vspace{.14cm}\\
                               \!\!\!   \big[g |_{  V^l (d) }\big]_{{\CC}^l (d)} =   \big[g |_{ V^0 (d) }\big]_{{\CC}^0 (d)} \, \text{ for all } d \in \N_\d, \, 0 \leq l < d   \!\!\!
                                        \end{array}
\Bigg\} \! \subset  \ZC_{ {\rm SU} (p,q) } (X,H,Y)
$$
is maximal compact.  
Thus it suffices show that if $g \,\in\, {\rm SU} (p,q)$ and $ g ( V^l_+ (d)) \,\subset \, 
V^l_+ (d)  , \,  g ( V^l_- (d) )  $ \, $\subset \, V^l_-(d)$, then 
$\big[g |_{  V^l (d) }\big]_{{\CC}^l (d)} \,= \,\big[g |_{ V^0 (d) }\big]_{{\CC}^0 (d)} $ for all $0 \leq l \leq d-1$, $ d \in \N_\d$ if and only if $g$ satisfies the conditions (2), (3) and
(4) in the statement of the lemma. To do this, we first record the following relations among the
ordered sets $\CC^l (d), \CC^l_{(-1)^{l+1}} (d)$ and $\CC^l_{(-1)^l} (d)$ for all $d \in \N_\d$:
When $\eta \,\in \E_\d$,
\begin{equation}\label{eqn-max-cpt-su-pq-1}
\CC^l (\eta) = \CC^l_{(-1)^l} (\eta) \vee \CC^l_{(-1)^{l+1}} (\eta) \, \text{~ for } 0 \leq l \leq \eta -1.
\end{equation}
When $\theta \,\in\, \O^1_\d$,
\begin{equation}\label{eqn-max-cpt-su-pq-2}
 \CC^l (\theta) = \begin{cases}
                   \CC^l_{(-1)^{l}} (\theta) \vee \CC^l_{(-1)^{l+1}} (\theta)  &   \text{ for all } 0 \leq l < (\theta -1)/2 \vspace{.2cm} \\
                   \CC^{(\theta-1)/2} _{+1} (\theta) \vee \CC^{(\theta-1)/2} _{-1} (\theta) & \text{ for } l = (\theta -1)/2  \vspace{.2cm}  \\
                   \CC^l_{(-1)^{l+1}} (\theta) \vee \CC^l_{(-1)^{l}} (\theta) & \text{ for all }  (\theta -1)/2 < l \leq \theta-1.
                  \end{cases}
\end{equation}
When $\zeta \,\in\, \O^3_\d$,
\begin{equation}\label{eqn-max-cpt-su-pq-3}
  \CC^l (\zeta) = \begin{cases}
                   \CC^l_{(-1)^{l}} (\zeta) \vee \CC^l_{(-1)^{l+1}} (\zeta) & \text{ for all } 0 \leq l < (\zeta -1)/2 \vspace{.2cm}\\
                   \CC^{(\zeta-1)/2} _{-1} (\zeta) \vee \CC^{(\zeta-1)/2} _{+1} (\zeta) &  \text{ for } l = (\zeta -1)/2 \vspace{.2cm} \\
                   \CC^l_{(-1)^{l+1}} (\zeta) \vee \CC^l_{(-1)^{l}} (\zeta) & \text{ for all }  (\zeta -1)/2 < l \leq \zeta-1.
                  \end{cases}
\end{equation}
Assuming that $g \,\in\, {\rm SU} (p,q), \,g ( V^l_+ (d)) \,\subset \, V^l_+ (d)  , \,
g ( V^l_- (d)  ) \,\subset\,  V^l_-(d)$ and
$$\big[g |_{  V^l (d) }\big]_{{\CC}^l (d)} \,=\,   \big[g |_{ V^0 (d) }\big]_{{\CC}^0 (d)}$$
for all $0 \,\leq\, l \,\leq d-1$, $ d \,\in \,\N_\d$, we next show that $g$ satisfies
the conditions (2), (3) and (4) in the lemma.

In view of \eqref{eqn-max-cpt-su-pq-1}, for all $\eta \,\in\, \E_\d$,
$$
 \Big[g |_{ V^l (\eta) }\Big]_{{\CC}^l_{(-1)^{l}} (\eta)\, \vee \, \CC^l_{(-1)^{l+1}}(\eta)}
= \Big[g |_{ V^l (\eta) }\Big]_{{\CC}^l (\eta)} = \Big[g |_{ V^0 (\eta) }\Big]_{{\CC}^0 (\eta)}
$$
$$ = \begin{pmatrix}
     \big[g |_{ V^0_{+1} (\eta) }\big]_{{\CC}^0_{+1} (\eta)}   &  0  \vspace{.1cm} \\
      0 & \big[g |_{ V^0_{-1} (\eta) }\big]_{{\CC}^0_{-1} (\eta)} 
  \end{pmatrix}.
$$
Thus for all $\eta \,\in\, \E_\d$ and $ 0 \,\leq\, l \,\leq\, \eta-1$,
$$
\Big[g |_{ V^l_{(-1)^{l}} (\eta) }\Big]_{{\CC}^l_{(-1)^{l}} (\eta)} =  \Big[g |_{ V^0_{+1} (\eta) }\Big]_{{\CC}^0_{+1} (\eta)}  \text{ and~ } \Big[g |_{ V^l_{(-1)^{l+1}} (\eta) }\Big]_{{\CC}^l_{(-1)^{l+1}} (\eta)} = \Big[g |_{ V^0_{-1} (\eta) }\Big]_{{\CC}^0_{-1} (\eta)}.
$$
Hence, (2) of the lemma holds.

In view of \eqref{eqn-max-cpt-su-pq-2}, for all $\theta \in \O^1_\d$ and $0 \leq l < (\theta-1)/2$,
$$
 \Big[g |_{ V^l (\theta) }\Big]_{{\CC}^l_{(-1)^{l}} (\theta)\, \vee \, \CC^l_{(-1)^{l+1}}(\theta)}
= \Big[g |_{ V^l (\theta) }\Big]_{{\CC}^l (\theta)} = \Big[g |_{ V^0 (\theta) }\Big]_{{\CC}^0 (\theta)}\\
$$
$$
=
{\left(
 \begin{array}{cc}
     \big[g |_{ V^0_{+1} (\theta) }\big]_{{\CC}^0_{+1} (\theta)} &  0 \vspace{.1cm}  \\
       0 &  \big[g |_{ V^0_{-1} (\theta) }\big]_{{\CC}^0_{-1} (\theta)} 
  \end{array}
  \right)}.
$$
Therefore if $\theta \in \O^1_\d$, then for all  $0 \leq l < (\theta-1)/2$,
$$
\Big[g |_{ V^l_{(-1)^l} (\theta) }\Big]_{{\CC}^l_{(-1)^{l}} (\theta)} = \Big[g |_{ V^0_{+1} (\theta) }\Big]_{{\CC}^0_{+1} (\theta)} \text{ and }  \Big[g |_{ V^l_{(-1)^{l+1}} (\theta) }\Big]_{ \CC^l_{(-1)^{l+1}}(\theta)} = \Big[g |_{ V^0_{-1} (\theta) }\Big]_{{\CC}^0_{-1} (\theta)}. 
$$
  From \eqref{eqn-max-cpt-su-pq-2}, we have 
$$
 \Big[g |_{ V^{(\theta-1)/2} (\theta) }\Big]_{{\CC}^{(\theta-1)/2}_{+} (\theta)\, \vee \, \CC^{(\theta-1)/2}_{-}(\theta)}
= \Big[g |_{ V^{(\theta-1)/2} (\theta) }\Big]_{{\CC}^{(\theta-1)/2} (\theta)} = \Big[g |_{ V^0 (\theta) }\Big]_{{\CC}^0 (\theta)}\\
$$
$$
= \begin{pmatrix}
 \begin{array}{cc}
     \big[g |_{ V^0_{+1} (\theta) }\big]_{{\CC}^0_{+1} (\theta)} &  0 \vspace{.1cm}  \\
       0 &  \big[g |_{ V^0_{-1} (\theta) }\big]_{{\CC}^0_{-1} (\theta)} 
  \end{array}
  \end{pmatrix}.
$$
Thus, 
$$
\Big[g |_{ V^{(\theta-1)/2}_{+} (\theta) }\Big]_{{\CC}^{(\theta-1)/2}_{+} (\theta)} = \Big[g |_{ V^0_{+} (\theta) }\Big]_{{\CC}^0_{+} (\theta)}
, \,   \Big[g |_{ V^{(\theta-1)/2}_{-} (\theta) }\Big]_{{\CC}^{(\theta-1)/2}_{-} (\theta)} = \Big[g |_{ V^0_{-} (\theta) }\Big]_{{\CC}^0_{-} (\theta)}.
$$
When $(\theta-1)/2 < l \leq \theta-1 $, we have
$$
  \Big[g |_{ V^l (\theta) }\Big]_{{\CC}^l_{(-1)^{l+1}} (\theta)\, \vee \, \CC^l_{(-1)^{l}}(\theta)}
= \Big[g |_{ V^l (\theta) }\Big]_{{\CC}^l (\theta)} = \Big[g |_{ V^0 (\theta) }\Big]_{{\CC}^0 (\theta)}\\
$$
$$
= \begin{pmatrix}
     \big[g |_{ V^0_{+1} (\theta) }\big]_{{\CC}^0_{+1} (\theta)} &  0 \vspace{.1cm} \\
       0 &  \big[g |_{ V^0_{-1} (\theta) }\big]_{{\CC}^0_{-1} (\theta)} 
  \end{pmatrix}.
$$
Thus if $\theta \in \O^1_\d$, then for all $ (\theta-1)/2 < l \leq \theta-1$,
$$
\Big[g |_{ V^l _{(-1)^{l+1}}(\theta) }\Big]_{{\CC}^l_{(-1)^{l+1}} (\theta)} = \Big[g |_{ V^0_{+1} (\theta) }\Big]_{{\CC}^0_{+1} (\theta)} \text{ and } \ \Big[g |_{ V^l _{(-1)^{l}}(\theta) }\Big]_{{\CC}^l_{(-1)^{l}} (\theta)} = \Big[g |_{ V^0_{-1} (\theta) }\Big]_{{\CC}^0_{-1} (\theta)} .
$$
Hence, (3) of the lemma holds.

When $ g ( V^l_+ (\zeta)) \subset  V^l_+ (\zeta) , \  g ( V^l_- (\zeta)  ) \subset
V^l_-(\zeta)$ and $\big[g |_{  V^l (\zeta) }\big]_{{\CC}^l (\zeta)} 
=   \big[g |_{ V^0 (\zeta) }\big]_{{\CC}^0 (\zeta)}$ for all $0 \leq l \leq \zeta -1$,
$ \zeta \in \O^3_\d$, using \eqref{eqn-max-cpt-su-pq-3} 
it follows, similarly as above, that (4) of the lemma holds. 

To prove the opposite implication, we assume that $g$ satisfies the conditions
$ g ( V^l_+ (d)) \subset  V^l_+ (d) $, $  g ( V^l_- (d)  ) \subset  V^l_-(d)$ as well as the
conditions
(2), (3), (4) of the lemma. Using the relations \eqref{eqn-max-cpt-su-pq-1}, \eqref{eqn-max-cpt-su-pq-2} and \eqref{eqn-max-cpt-su-pq-3}
it is now straightforward to check that $\big[g |_{  V^l (d) }\big]_{{\CC}^l (d)} =   \big[g |_{ V^0 (d) }\big]_{{\CC}^0 (d)} $  for all $ 0 \leq l \leq d -1, d\in \N_\d$. This completes the proof of the lemma.
\end{proof}

We now introduce some notation which will be required to state Proposition \ref{max-cpt-su-pq-wrt-onb}.
For $d \,\in\, \N_\d$, define
$$
\CC_+ (d) := \CC^0_+ (d) \vee \cdots \vee  \CC^{d-1}_+ (d) \ \text{ and } \ 
\CC_- (d) := \CC^0_- (d) \vee \cdots \vee  \CC^{d-1}_- (d).
$$
Let $\alpha \,:=\, \# \E_\d $,  $\beta \,:=\, \# \O^1_\d$ and $ \gamma \,:=\, \# \O^3_\d $.
We enumerate $$\E_\d =\{ \eta_i \,\mid\, 1 \,\leq\, i \,\leq\, \alpha \}$$ such that
$\eta_i \,<\, \eta_{i+1}$, 
$$\O^1_\d \,=\,\{ \theta_j \,\mid\, 1 \,\leq\, j \,\leq\, \beta \}$$ such that $\theta_j
\,<\, \theta_{j+1}$ and similarly
$$\O^3_\d \,=\,\{ \zeta_j \,\mid\, 1 \,\leq \,j \,\leq\, \gamma \}$$ such that
$\zeta_j \,<\, \zeta_{j+1}$. Now define
$$
\EC_+ := \CC_+ (\eta_1) \vee \cdots \vee \CC_+ (\eta_{\alpha})\, ; \ \  \,
\OC^1_+ := \CC_+ (\theta_1) \vee \cdots \vee \CC_+ (\theta_{\beta})\, ; \ \ \,\OC^3_+ := \CC_+ (\zeta_1) \vee \cdots \vee \CC_+ (\zeta_{\gamma});
$$
$$
\EC_- := \CC_- (\eta_1) \vee \cdots \vee \CC_- (\eta_{\alpha}) ; \ 
\OC^1_- := \CC_- (\theta_1) \vee \cdots \vee  \CC_- (\theta_{\beta})\, \text{ and } \,\OC^3_- := \CC_- (\zeta_1) \vee \cdots \vee  \CC_- (\zeta_{\gamma})  .
$$
Finally we define
\begin{equation}\label{orthogonal-basis-su-pq-final}
 \HC_+ := \EC_+ \vee \OC^1_+ \vee\OC^3_+ , \ \ \HC_- := \EC_- \vee \OC^1_- \vee\OC^3_- \ \text{ and } \
\HC := \HC_+ \vee \HC_-.
\end{equation}
It is clear that $\HC$ is a standard orthogonal basis with
$\HC_+ \,=\, \{ v \,\in\, \HC \,\mid \langle v,\, v \rangle \,=\,1 \}$ and 
$\HC_- \,=\, \{ v \,\in\, \HC \,\mid\, \langle v,\, v \rangle \,=\,-1 \}$. In
particular, $\# \HC_+ \,=\, p$ and $\# \HC_- \,=\,q$.
From the definition of the $\HC_+$ and $\HC_-$ we have the following relations:
$$
\sum_{i=1}^\alpha \frac{\eta_i}{2}t_{\eta_i}  + \sum_{j=1}^\beta \big( \frac{\theta_j+1}{2}p_{\theta_j} +  \frac{\theta_j-1}{2}q_{\theta_j}\big)  + \sum_{k=1}^\gamma\big( \frac{\zeta_k-1}{2}p_{\zeta_k} +  \frac{\zeta_k+1}{2}q_{\zeta_k}\big) =p
$$
and 
$$
\sum_{i=1}^\alpha \frac{\eta_i}{2}t_{\eta_i}  + \sum_{j=1}^\beta \big( \frac{\theta_j-1}{2}p_{\theta_j} +  \frac{\theta_j+1}{2}q_{\theta_j}\big)  + \sum_{k=1}^\gamma\big( \frac{\zeta_k+1}{2}p_{\zeta_k} +  \frac{\zeta_k-1}{2}q_{\zeta_k}\big) =q.
$$

The $\C$-algebra
$$\prod_{i=1}^\alpha \big( {\rm M}_{p_{\eta_i}} (\C) \times
{\rm M}_{q_{\eta_i}} (\C) \big) \times \prod_{j=1}^\beta
\big( {\rm M}_{p_{\theta_j}} (\C)\times {\rm M}_{q_{\theta_j}} (\C) \big)
\times \prod_{k=1}^\gamma \big( {\rm M}_{p_{\zeta_k}} (\C) \times {\rm M}_{q_{\zeta_k}} (\C) \big)$$
is embedded into ${\rm M}_p(\C)$ and ${\rm M}_q(\C)$ in the following two ways:
$$
\Db_p \colon \prod_{i=1}^\alpha  \big( {\rm M}_{p_{\eta_i}} (\C) \times {\rm M}_{q_{\eta_i}} (\C) \big) \times \prod_{j=1}^\beta \big( {\rm M}_{p_{\theta_j}} (\C) \times {\rm M}_{q_{\theta_j}} (\C) \big) \times \prod_{k=1}^\gamma \big( {\rm M}_{p_{\zeta_k}} (\C) \times {\rm M}_{q_{\zeta_k}} (\C) \big) \longrightarrow {\rm M}_p(\C)  
$$
is defined by
$$
\big (A_{\eta_1}, B_{\eta_1},  \cdots , A_{\eta_{\alpha}},  B_{\eta_{\alpha}}; C_{\theta_1}, D_{\theta_1} , \cdots,  C_{\theta_\beta}, D_{\theta_\beta}; E_{\zeta_1}, F_{\zeta_1}, \cdots ,  E_{\zeta_\gamma}, F_{\zeta_\gamma}   \big)$$
$$\longmapsto\, \bigoplus_{i=1}^\alpha  \big( A_{\eta_i} \oplus B_{\eta_i} \big)_\blacktriangle^{\eta_i /2}
$$
$$
\oplus \bigoplus_{j=1}^\beta \Big( \big( C_{\theta_j} \oplus D_{\theta_j}\big)_\blacktriangle ^{\frac{\theta_j-1}{4}} \oplus C_{\theta_j} \oplus \big( C_{\theta_j} \oplus D_{\theta_j}\big)_\blacktriangle ^{\frac{\theta_j-1}{4}} \Big) \\ 
  \oplus \bigoplus_{k=1}^\gamma \Big( \big( E_{\zeta_k}\oplus F_{\zeta_k} \big) _\blacktriangle ^{\frac{\zeta_k + 1}{4}}  \oplus \big( F_{\zeta_k}\oplus E_{\zeta_k} \big) _\blacktriangle ^{\frac{\zeta_k -3}{4}}  \oplus F_{\zeta_k} \Big) ,
$$
and
$$
\Db_q \colon \prod_{i=1}^\alpha \big( {\rm M}_{p_{\eta_i}} (\C) \times {\rm M}_{q_{\eta_i}} (\C) \big) \times \prod_{j=1}^\beta \big( {\rm M}_{p_{\theta_j}} (\C) \times {\rm M}_{q_{\theta_j}} (\C) \big) \times \prod_{k=1}^\gamma \big( {\rm M}_{p_{\zeta_k}} (\C) \times {\rm M}_{q_{\zeta_k}} (\C) \big) \longrightarrow {\rm M}_q(\C)  
$$
is defined by
\begin{align*}
\big(  A_{\eta_1}, B_{\eta_1},  \cdots , & A_{\eta_{\alpha}},  B_{\eta_{\alpha}}; C_{\theta_1}, D_{\theta_1} , \cdots,  C_{\theta_\beta}, D_{\theta_\beta}; E_{\zeta_1}, F_{\zeta_1}, \cdots ,  E_{\zeta_\gamma}, F_{\zeta_\gamma}     \big)\\
\longmapsto\, & \bigoplus_{i=1}^\alpha  \big( B_{\eta_i} \oplus A_{\eta_i} \big)_\blacktriangle^{\eta_i /2}  
 \oplus  \bigoplus_{j=1}^\beta \Big( \big( D_{\theta_j} \oplus C_{\theta_j}\big)_\blacktriangle ^{\frac{\theta_j-1}{4}} \oplus D_{\theta_j} \oplus \big( D_{\theta_j} \oplus C_{\theta_j}\big)_\blacktriangle ^{\frac{\theta_j-1}{4}} \Big) \\ 
 \oplus & \bigoplus_{k=1}^\gamma \Big( \big( F_{\zeta_k}\oplus E_{\zeta_k} \big) _\blacktriangle ^{\frac{\zeta_k + 1}{4}}  \oplus \big( E_{\zeta_k}\oplus F_{\zeta_k} \big) _\blacktriangle ^{\frac{\zeta_k -3}{4}}  \oplus E_{\zeta_k} \Big). 
\end{align*}

Define the characters
$$
\bigchi_p \colon \prod_{i=1}^\alpha \big( {\rm GL}_{p_{\eta_i}} (\C) \times {\rm GL}_{q_{\eta_i}} (\C) \big) \times \prod_{j=1}^\beta \big( {\rm GL}_{p_{\theta_j}} (\C) \times {\rm GL}_{q_{\theta_j}} (\C) \big) \times \prod_{k=1}^\gamma \big( {\rm GL}_{p_{\zeta_k}} (\C) \times {\rm GL}_{q_{\zeta_k}} (\C) \big) \longrightarrow \C^*
$$

$$
\big(   A_{\eta_1}, B_{\eta_1},  \cdots ,  A_{\eta_{\alpha}},  B_{\eta_{\alpha}}; C_{\theta_1}, D_{\theta_1} , \cdots,  C_{\theta_\beta}, D_{\theta_\beta}; E_{\zeta_1}, F_{\zeta_1}, \cdots ,  E_{\zeta_\gamma}, F_{\zeta_\gamma}    \big) 
$$
$$
\longmapsto\, \prod_{i=1}^\alpha( \det  A_{\eta_i}^{\eta_i/2}  \det  B_{\eta_i}^{\eta_i/2} ) \prod_{j=1}^\beta( \det  C_{\theta_j}^{\frac{\theta_j+1}{2} }  \det  D_{\theta_j}^{\frac{\theta_j-1}{2} } ) \prod_{k=1}^\gamma( \det  E_{\zeta_k}^{\frac{\zeta_k-1}{2} } \det  F_{\zeta_k}^{\frac{\zeta_k+1}{2} } )
 $$
and
$$
\bigchi_q \colon \prod_{i=1}^\alpha \big( {\rm GL}_{p_{\eta_i}} (\C) \times {\rm GL}_{q_{\eta_i}} (\C) \big) \times \prod_{j=1}^\beta \big( {\rm GL}_{p_{\theta_j}} (\C) \times {\rm GL}_{q_{\theta_j}} (\C) \big) \times \prod_{k=1}^\gamma \big( {\rm GL}_{p_{\zeta_k}} (\C) \times {\rm GL}_{q_{\zeta_k}} (\C) \big) \longrightarrow \C^*
$$
$$
\big(  A_{\eta_1}, B_{\eta_1}, \cdots ,  A_{\eta_{\alpha}},  B_{\eta_{\alpha}}; C_{\theta_1}, D_{\theta_1} , \cdots ,  C_{\theta_\beta}, D_{\theta_\beta}; E_{\zeta_1}, F_{\zeta_1}, \cdots ,  E_{\zeta_\gamma}, F_{\zeta_\gamma}     \big) 
$$
$$
\longmapsto\, \prod_{i=1}^\alpha( \det  A_{\eta_i}^{\eta_i/2}  \det  B_{\eta_i}^{\eta_i/2} ) \prod_{j=1}^\beta( \det  C_{\theta_j}^{\frac{\theta_j-1}{2} }  \det  D_{\theta_j}^{\frac{\theta_j+1}{2} } ) \prod_{k=1}^\gamma( \det  E_{\zeta_k}^{\frac{\zeta_k+1}{2} } \det  F_{\zeta_k}^{\frac{\zeta_k-1}{2}} ). 
$$

Let $\Lambda_\HC \,\colon\, {\rm End}_\C \C^n \,\longrightarrow\, {\rm M}_n (\C)$ be the
isomorphism of $\C$-algebras induced by the ordered basis $\HC$ defined in \eqref{orthogonal-basis-su-pq-final}. Let $M$ be the maximal compact subgroup of ${\rm SU}(p,q)$ which leaves invariant simultaneously the two subspace spanned by $\HC_+$ and $\HC_-$. Clearly, $\Lambda_\HC (M) = {\rm S} ({\rm U} (p) \times {\rm U} (q))$. 

\begin{proposition}\label{max-cpt-su-pq-wrt-onb} 
Let $X \,\in\, \NC _{\s\u (p,q)}$, $\Psi_{{\rm SU} (p,q)}(\OC_X) \,=\, ( \d ,\, \sgn_{\OC_X} )$.
Let $\alpha \,:=\, \# \E_\d $,  $\beta \,:=\, \# \O^1_\d$ and $ \gamma \,:=\, \# \O^3_\d $.
Let $\{X,H,Y\}$ be a $\s\l_2(\R)$-triple in $\s\u (p,q)$ and 
$(p_d,\, q_d)$ the signature of the form $(\cdot, \cdot)_d$, $d \,\in\, \N_\d$, as defined in \eqref{new-form}.
Let $K$ be the maximal compact subgroup of $\ZC_{{\rm SU} (p,q)} ( X, H, Y)$ as in Lemma
\ref{max-cpt-su-pq}. Then  $\Lambda_\HC (K) \,\subset \,
{\rm S} ({\rm U} (p) \times {\rm U} (q))$ is given by 
\begin{align*} 
\Lambda_\HC(K)\,=\,
 \Bigg\{ \Db_p (g) \oplus \Db_q(g) \Biggm| \begin{array}{c}
                       g \in \prod_{i=1}^\alpha \big( {\rm U}(p_{\eta_i}) \times {\rm U}(q_{\eta_i}) \big) \times \prod_{j=1}^\beta \big( {\rm U}(p_{\theta_j}) \times {\rm U}(q_{\theta_j}) \big) \vspace{.12cm}\\
                                         \times \prod_{k=1}^\gamma \big( {\rm U}(p_{\zeta_k}) \times {\rm U}(q_{\zeta_k})\big) \, , \text{ and } \  ~ \bigchi_p(g) \bigchi_q(g) =1
                                        \end{array} \Bigg\}.
\end{align*}
\end{proposition}

\begin{proof}
This follows by writing the matrices of the elements of the maximal compact subgroup $K$ in Lemma \ref{max-cpt-su-pq} with respect to
the basis $\HC$ as in \eqref{orthogonal-basis-su-pq-final}.
\end{proof}

\begin{theorem}\label{su-pq}
 Let $X \,\in\, \s \u(p,q)$ be a nilpotent element.
 Let $(\d, \,\sgn_{\OC_X}) \,\in\, \YC(p,q)$ be the signed Young diagram 
 of the orbit  $\OC_X$ (that is, $\Psi_{{\rm SU}(p,q)}  (\OC_X) \,=\, (\d,\, \sgn_{\OC_X})$ as in
the notation of Theorem \ref{su-pq-parametrization}). Let $$l\,:=\,
\# \{d\in \N_\d\,\mid\,  p_d \,\neq \,0\} + \# \{d\in \N_\d,\,\mid\, q_d \,\neq\, 0\} \, .$$
Then the following hold:
 \begin{enumerate}
  \item If $\N_\d \,=\, \E_\d$, then\, $\dim_\R H^2(\OC_X, \,\R) \, = \, l-1.$

 \item If $l\,=\,1$ and $\N_\d \,= \,\O_\d$, then \, $\dim_\R H^2(\OC_X, \,\R)  \,= \, 0.$

 \item If $l \,\geq\, 2$ and  $\#\O_\d \,\geq\, 1$, then \, $\dim_\R H^2(\OC_X,\, \R) \, = \, l-2.$
\end{enumerate}
\end{theorem}

\begin{proof}
This is clear when $X \,=\,0$. So assume that $X \,\neq\, 0$. 

Let $\{X,\, H,\, Y\} \,\subset\, \s\u(p,q)$ be a $\s\l_2(\R)$-triple. 
Let $K$ be the maximal compact subgroup of $\ZC_{{\rm SU}(p,q)}(X,H,Y)$ as in 
Lemma \ref{max-cpt-su-pq}, and let $\HC$ be as in \eqref{orthogonal-basis-su-pq-final}.
Let $M$ be the maximal compact subgroup of ${\rm SU}(p,q)$ which leaves invariant
simultaneously the two subspace spanned by $\HC_+$ and $\HC_-$. Then $M$ contains $K$.
It follows either from Proposition \ref{max-cpt-su-pq-wrt-onb} or from Lemma \ref{reductive-part-comp} (4) that
$$
K \,\simeq\, K' \,:=\,
S \big( \prod_{ d \in \N_\d} ( {\rm U} (p_d) \times {\rm U } (q_d))_\Delta^d \big)\, .
$$
This implies that $\dim_\R \z(\k) \,=\, l-1$. We now appeal to Proposition \ref{max-cpt-su-pq-wrt-onb} to make the following observations :
\begin{enumerate}
\item If $\N_\d \,=\, \E_\d$, then $\k \,\subset\, [\m,\,\m] $.
\item If $\# \O_\d \,\geq\, 1$ and $l\,\geq\, 2$, then $\k,\not\subset\, [\m,\,\m]$.
 \end{enumerate} 

 Since $K$ is not necessarily connected, we need to show that the adjoint action of $K$ on
$\z(\k)$ is trivial. For this, first denote
 $$
 \mathbf{L}:= \prod_{ d \in \N_\d} ( {\rm U} (p_d) \times {\rm U } (q_d))_\Delta^d 
 $$
 and identify $K$ with $K'$.
Let $\l$ be the Lie algebra of $\mathbf{L}$. Then $$[{\mathbf L} ,\, {\mathbf L}] \,=\, \prod_{ d \in \N_\d} ( {\rm SU} (p_d) \times {\rm SU } (q_d))_\Delta^d\, .$$
In particular $[{\mathbf L} ,\, {\mathbf L}] \,\subset\, K \subset \mathbf{L}$. Thus $[\l,\, \l]
\,= \,[\k ,\, \k]$, and hence
$ \z(\k) = \k \cap \z(\l)$. Since $\mathbf{L}$ is connected, the adjoint action of
$\mathbf{L}$ is trivial on $\z(\l)$. So the adjoint action of $K$ on $\z(\k)$ is trivial. 
 
{\it Proof of (1)}: 
{}From the above observations it follow that $\k \,\subset\, [\m,\,\m]$ when $\N_\d \,= \,\E_\d$.
As the adjoint action of $K$ on $\z(\k)$ is trivial, we have 
$ \big[(\z(\k)\cap [\m,\,\m])\big]^{K /K^\circ} \,=\,  \z(\k) \cap [\m,\,\m] \,=\,\z(\k)$. In view
of Theorem \ref{thm-nilpotent-orbit} we now have $ \dim_\R H^2(\OC_X,\, \R)\,=\,l-1 $.

{\it Proof of (2)}: Suppose $\d \,= \,[ d^{t_d}] $ where $t_d d\,=\,p+q$. Since $l\,=\,1$,
it follows that either $p_d \,=\, t_d$ or $q_d \,=\, t_d$. In both cases we have
$K \,\simeq\, S\big({\rm U}(t_d)_{\Delta }^{d}\big)$. So $\z(\k)$ is trivial.
Hence, in view of Theorem \ref{thm-nilpotent-orbit} we have $\dim_\R H^2(\OC_X,\, \R)\,=\,0$. 

{\it Proof of (3)}:  
{}From the above observations we have $\z(\k) \not \,\subset\, [\m,\,\m]$ when $\# \O_\d \,\geq 
\,1$ and $l \,\geq\, 2$. Since $\dim_\R \z(\m) \,=\,1$, it follows that $\dim_\R(\z(\k) \cap 
[\m,\,\m])\,=\, \dim_\R \z(\k) -1$.
By Theorem \ref{thm-nilpotent-orbit}, and the fact that the adjoint action of $K$ on $\z(\k)$ is
trivial, we conclude that
$$
  \dim_\R H^2(\OC_X,\, \R) \,=\,  \dim_\R  [(\z(\k)\cap [\m,\,\m])^*]^{K /K^\circ} \,=\, l-2\, .
$$
This completes the proof of the theorem.
\end{proof}

\subsection{Second cohomology groups of nilpotent orbits in \texorpdfstring{${\s\o}(p,q)$}{Lg}}\label{sec-so-pq}

Let $n$ be a positive integer and $(p,\, q)$ be a pair of non-negative integers such that $p + q 
\,=\,n$. The aim in this subsection is to compute the second cohomology groups of nilpotent orbits in the simple Lie algebra ${\s\o}(p,q)$ 
under the adjoint action of ${\rm SO}(p,q)^\circ$. We will further assume that $p,\, q\, >\,0$
as we deal with non-compact groups. Throughout this subsection $\<>$ denotes the symmetric
form on  $\R^n$ defined by $\langle x,\, y \rangle \,:=\, x^t{\rm I}_{p,q} y$, $x,\, y \,\in\, \R^n$, where
${\rm  I}_{p,q}$ is as in \eqref{defn-I-pq-J-n}. In this subsection we will follow 
notation as defined in \S \ref{sec-notation}.

We first need to describe a suitable parametrization of $\NC ({\rm SO} (p,q)^\circ)$, the set of all nilpotent orbits in ${\s\o}(p,q)$ under the adjoint action of ${\rm SO} (p,q)^\circ$.
Let $\Psi_{{\rm SL}_n (\R)} \,:\,  \NC ({\rm SL}_n(\R)) \,\longrightarrow\,\PC (n)$ be the  parametrization of $\NC ({\rm SL}_n (\R))$ as in Theorem \ref{sl-R-parametrization}.
Since ${\rm SO} (p,q) \,\subset\, {\rm SL}_n (\R)$ (consequently as, the set of nilpotent elements
$\NC_{{\s\o}(p,q)} \,\subset\, \NC_{\s\l_n(\R)}$)
we have the inclusion map
$\Theta_{{\rm SO}(p,q)^\circ}\,:\, \NC ({\rm SO} (p,q)^\circ) \,\longrightarrow\,\NC ( {\rm SL}_n (\R))$. Let
$$
\Psi'_{{\rm SO}(p,q)^\circ}\,:= \,\Psi_{{\rm SL}_n (\R)}\circ
\Theta_{{\rm SO}(p,q)^\circ}\,\colon\, \NC ({\rm SO} (p,q)^\circ)
\,\longrightarrow\, \PC (n)
$$ be the composition.
Recall that $\Psi'_{{\rm SO}(p,q)^\circ} ( \NC ({\rm SO} (p,q)^\circ))\,\subset \, \PC_1 (n)$ (this
follows form the first paragraph of Remark \ref{unitary-J-basis-rmk}).
Let $X \,\in\, {\s\o}(p,q)$ be a non-zero nilpotent element and $\OC_X$
the corresponding nilpotent orbit in $\s\o(p,q)$ under the adjoint action of 
${\rm SO} (p,q)^\circ$. Let $\{X,\, H,\, Y\} \,\subset\, {\s\o}(p,q)$ be a $\s\l_2(\R)$-triple. 
We now apply Proposition \ref{unitary-J-basis}, Remark \ref{unitary-J-basis-rmk} (1), and
follow the notation used therein. 
Let $V:= \R^n$ be the right $\R$-vector space of column vectors.
Let $\{d_1,\, \cdots,\, d_s\}$, with $d_1 \,<\, \cdots \,<\, d_s$, be ordered finite set of natural numbers
that occur as dimension of non-zero irreducible 
$\text{Span}_\R \{ X,H,Y\}$-submodules of $V$. Recall that $M(d-1)$ is defined to be the isotypical component of $V$ containing all irreducible Span$_\R \{X,H,Y\}$-submodules of $V$ with highest weight $d-1$ and as in \eqref{definition-L-d-1} we set
$L(d-1)\,:= \,V_{Y,0} \cap M(d-1)$. Let $t_{d_r} \,:=\, \dim_\R L(d_r-1)$,
$1 \,\leq \,r \,\leq\, s$. Then 
$\d\,:=\, [d_1^{t_{d_1}}, \,\cdots ,\, d_s^{t_{d_s}}]\,\in\, \PC_1(n)$, and  moreover,
$\Psi'_{{\rm SO}(p,q)^\circ} (\OC_X) \,=\, {\d}$.

We next assign $\sgn_{\OC_X} \in \SC^{\rm even}_{\d}(p,q)$ to each $\OC_X \,\in\,
 \NC({\rm SO}(p,q)^\circ)$; see \eqref{S-d-pq-even} for the definition of
$\SC^{\rm even}_{\d}(p,q)$.
For each $d\,\in\, \N_\d$ (see \eqref{Nd-Ed-Od} for the definition of $\N_\d$) we will define a $t_d \times  d$ matrix $(m^d_{ij})$ in
$ \Ab_{d}$ that depends only on the orbit $\OC_X$; see \eqref{A-d} for the definition of $\Ab_d$.
For this, recall that the form $(\cdot,\,\cdot)_{d} \,\colon\, L(d-1) \times L(d-1) \,
\longrightarrow\, \R$, defined in \eqref{new-form}, is symmetric or symplectic
according as $d$ is odd or even. 
Denoting the  signature of $(\cdot,\,\cdot)_{d}$ by $(p_{d},\, q_{d})$ when  $d\,\in\, \O_\d$, we now define
\begin{align*}
 m^\eta_{i1} &:= +1 \qquad \text{if } \  1 \leq i \leq t_{\eta}, \quad \eta \in \E_\d\,; \\
 m^\theta_{i1} &:= \begin{cases}
                   +1  & \text{ if } \  1 \leq i \leq p_{\theta} \\
                   -1  &  \text{ if } \ p_\theta < i \leq t_\theta  
                 \end{cases}\,  , \theta \in \O_\d\,;
\end{align*}
and for $j \,>\,1$, define $(m^d_{ij})$ as in \eqref{def-sign-alternate} and \eqref{def-sign-alternate-1}.
 Then the matrices $(m^d_{ij})$ clearly verify \eqref{yd-def2}. Set $\sgn_{\OC_X} \,:=\,
((m^{d_1}_{ij}),\, \cdots,\, (m^{d_s}_{ij}))$. It now follows from the last paragraph of Remark \ref{CM-correction} and the above definition of $m^\eta_{i1}$ for $\eta \,\in\, \E_\d$ that $\sgn_{\OC_X}\,\in\, \SC^{\rm even}_{\d}(p,q)$.
Thus we have the map 
$$
\Psi_{{\rm SO} (p,q)^\circ} \,\colon\, \NC({\rm SO}(p,q)^\circ)\,
\longrightarrow\, \YC^{\rm even}_1(p,q)\, , \ \
\OC_X \,\longmapsto\, \big(\Psi'_{{\rm SO} (p,q)^\circ} (\OC_X),\, \sgn_{\OC_X}  \big) \,;
$$
where  $\YC^{\rm even}_1(p,q)$ is as in \eqref{yd-1-Y-pq}.
The map $\Psi_{{\rm SO} (p,q)^\circ}$ is surjective.  

\begin{theorem}\label{so-pq-parametrization}
For the above map $\Psi_{{\rm SO}(p,q)^\circ}$,
$$ \# \Psi_{{\rm SO}(p,q)^\circ}^{-1} (\d, \sgn)  =
\begin{cases}
  4 &    \text{ for all } \,  \d \in   \PC_{\rm v.even} (n) \\
  2 &    \text{ for all  } \,  \d \in \PC_1(n) \setminus  \PC_{\rm v.even} (n) ,  \, \sgn \in \SC'_{\d}(p,q)   \\  
  1 &    \text{ otherwise}. 
\end{cases}$$ 
\end{theorem}

\begin{remark}
Taking into account the error in \cite[Lemma 9.3.1]{CoM} pointed out in Remark \ref{CM-correction},
the above parametrization
in Theorem \ref{so-pq-parametrization} is a modification of Theorem 9.3.4 in \cite{CoM}.
\end{remark}
 
Let $0\,\not=\, X \in \NC _{\s\o (p,q)}$ and $\{X,H,Y\}\,\subset\, \s\o (p,q)$ be a 
$\s\l_2(\R)$-triple. Let $\Psi_{{\rm SO} (p,q)^\circ} (\OC_X) \,=\, \big( \d, \sgn_{\OC_X} 
\big)$. Then $\Psi'_{{\rm SO} (p,q)^\circ} (\OC_X) \,=\, \d$. Recall that $\sgn_{\OC_X}$ determines 
the signature of $(\cdot, \,\cdot)_\theta$ on $L(\theta -1)$, $\theta \,\in\, \O_\d$; let 
$(p_\theta, q_\theta)$ be the signature of $(\cdot,\, \cdot)_\theta$.

First assume that $\N_\d \,=\, \O_\d$.
Let $( v^{\theta}_1,\, \cdots,\, v^{\theta}_{t_{\theta}})$ be an ordered $\R$-basis of
$L(\theta-1)$ as in Proposition \ref{unitary-J-basis}. 
It now follows from Proposition \ref{unitary-J-basis}(3)(b) that $( v^{\theta}_1,\, \cdots,\,
v^{\theta}_{t_{\theta}} )$ is an orthogonal basis for $(\cdot,\, \cdot)_{\theta}$ when
$\theta \,\in \,\O_\d$. 
We also assume that the vectors in the ordered basis $( v^{\theta}_1, \,\cdots,\,
v^{\theta}_{t_{\theta}} )$ satisfies the properties in Remark \ref{unitary-J-basis-rmk}(1).
In view of the signature of $(\cdot, \,\cdot)_\theta$, $\theta \,\in\, \O_\d$, we may further
assume that
\begin{equation}\label{orthonormal-basis-odd-so-pq}
 ( v^{\theta}_j, v^\theta_j)_{\theta}  =
 \begin{cases}
    +1  &  \text{ if }   1 \leq j \leq p_{\theta}  \\
    -1  &  \text{ if }  p_{\theta} < j \leq t_{\theta}. 
 \end{cases}
\end{equation}

For $\theta \,\in\, \O_\d$,
let $\big\{  {w}^\theta_{jl} \,\mid\, 1 \,\leq\, j \,\leq\, t_{\theta}, \, 0 \,\leq\, l \,
\leq\, \theta-1 \big\}$ be the $\R$-basis of $M (\theta-1)$  as in Lemma \ref{orthogonal-basis-R}.
For each $0 \,\leq\, l \,\leq\, \theta-1$, define
$$
  V^l (\theta) \,:=\, \text{ Span}_\R \{  {w}^\theta_{1l}, \,\cdots ,\,  {w}^\theta_{t_{\theta}l} \}\, .
$$
The ordered basis $\big( {w}^\theta_{1l},\, \cdots ,\,  {w}^\theta_{t_{\theta}l} \big)$ of
$ V^l (\theta) $ is denoted by $\CC^l (\theta)$.

\begin{lemma}\label{reductive-part-comp-so-pq}
For $\N_\d\,= \,\O_\d$,
 $$
 \ZC_{ {\rm SO} (p,q) }(X,H,Y) = \Bigg\{ g \in { {\rm SO} (p,q) } \Biggm| 
                                        \begin{array}{cc}
                             g (  V^l (\theta)  ) \, \subset \, V^l (\theta) \,  \text{ and }    \vspace{.14cm}\\
                             \! \big[g|_{ V^l(\theta)}\big]_{{\CC}^l(\theta)} = \big[g |_{ V^0 (\theta)}\big]_{{\CC}^0 (\theta)} \text{ for all } \theta \in \O_\d, \, 0 \leq l < \theta \!
                                       \end{array}
                                           \Bigg\}.
                                           $$
 \end{lemma}
\begin{proof}
We omit the proof as it is identical to that of Lemma \ref{reductive-part-comp-su-pq}. 
\end{proof}
 
We next impose orderings on the sets $\{v \in \CC^l (\theta) \mid \langle v, v \rangle > 0 \}$,
$\{ v \in \CC^l (\theta) \mid \langle v, v \rangle < 0 \}$.
Define the ordered sets by 
$\CC^l_+ (\theta)$, $\CC^l_- (\theta)$, $\CC^l_+ (\zeta)$ and $\CC^l_- (\zeta)$ as in \eqref{C+theta}, \eqref{C-theta}, \eqref{C+zeta}, \eqref{C-zeta}, respectively according as  $\theta \in \O^1_\d$ or $\zeta \in \O^3_\d$. 
For all $\theta \,\in\, \O_\d$ and $0 \,\leq\, l \,\leq\, \theta-1$, set
$$
V^l_+ (\theta) : = \text{ Span}_\R \{ v \mid  v \in  \CC^l (\theta),\langle v, v \rangle > 0 \}, \quad
V^l_- (\theta) : = \text{ Span}_\R \{ v \mid  v \in  \CC^l (\theta),\langle v, v \rangle < 0 \}.
$$
It is straightforward from \eqref{orthonormal-basis-odd-so-pq}, and the orthogonality relations
in Lemma \ref{orthogonal-basis-R}, that
$\CC^l_+ (\theta)$ and $ \CC^l_- (\theta)$ are indeed ordered bases of $V^l_+ (\theta)$ and
$V^l_- (\theta)$, respectively.

In the next lemma we specify a maximal compact subgroup of $\ZC_{ {\rm SO} (p,q) } (X,H,Y)$ which will be used in Proposition \ref{max-cpt-so-pq-wrt-onb}.
As before, the notation $ (-1)^l $ stands for the sign `$+$' or the sign `$-$' according as $l$
is an even or odd integer.

\begin{lemma}\label{max-cpt-so-pq}
Suppose that $\N_\d \,=\, \O_\d$. Let $K$ be the subgroup of $\ZC_{ {\rm SO} (p,q)}(X,H,Y)$ consisting
of all $g\,\in\, \ZC_{{\rm SO}(p,q)}(X,H,Y)$ such that the following hold:
\begin{enumerate}
 \item $g ( V^l_+ (\theta)) \,\subset \, V^l_+ (\theta)$ and $g ( V^l_- (\theta)  ) \,
\subset \, V^l_-(\theta)$, for all  $ \theta \in \O_\d$ and $ 0 \,\leq\, l \,\leq\, \theta -1$.

\item When $ \theta \in \O^1_\d$,
$$
\Big[g |_{ V_+^0 (\theta) }\Big]_{{\CC}^0_+ (\theta)} = 
             \begin{cases}                                                            
                 \Big[g |_{  V^l_{(-1)^{l}} (\theta)}\Big]_{{\CC}^l_{(-1)^{l}} (\theta)} & \text{ for all } 0 \leq l < (\theta-1)/2 \vspace{.15cm}\\
            \Big[g |_{  V^{(\theta-1)/2}_{+} (\theta)}\Big]_{{\CC}^{(\theta-1)/2}_{+} (\theta)} \vspace{.15cm}  \\
                 \Big[g |_{  V^l_{(-1)^{l+1}} (\theta)}\Big]_{{\CC}^l_{(-1)^{l+1} } (\theta)} &  \text{ for all } (\theta-1)/2 < l \leq \theta-1,
                 \end{cases}                 
$$
$$
  \Big[g |_{ V_-^0 (\theta) }\Big]_{{\CC}^0_- (\theta)} =
     \begin{cases}
      \Big[g |_{  V^l_{(-1)^{l+1}} (\theta)}\Big]_{{\CC}^l_{(-1)^{l+1} } (\theta)} &  \text{ for all } 0 \leq l < (\theta-1)/2 \vspace{.15cm}\\
       \Big[g |_{  V^{(\theta-1)/2}_{-} (\theta)}\Big]_{{\CC}^{(\theta-1)/2}_{-} (\theta)} \vspace{.15cm}  \\
      \Big[g |_{  V^l_{(-1)^{l}} (\theta)}\Big]_{{\CC}^l_{(-1)^{l}} (\theta)} &  \text{ for all } (\theta-1)/2 < l \leq \theta-1 .     
     \end{cases}
$$ 

\item When $ \zeta \in \O^3_\d $,
$$
\Big[g |_{ V_+^0 (\zeta) }\Big]_{{\CC}^0_+ (\zeta)} = 
             \begin{cases}                                                            
                 \Big[g |_{  V^l_{(-1)^{l}} (\zeta)}\Big]_{{\CC}^l_{(-1)^{l}} (\zeta)} &  \text{ for all } 0 \leq l < (\zeta-1)/2 \vspace{.15cm}\\            
      \Big[g |_{  V^{(\zeta-1)/2}_{-} (\zeta)}\Big]_{{\CC}^{(\zeta-1)/2}_{-} (\zeta)} \vspace{.15cm}  \\
                 \Big[g |_{  V^l_{(-1)^{l+1}} (\zeta)}\Big]_{{\CC}^l_{(-1)^{l+1} } (\zeta)} & \text{ for all } (\zeta-1)/2 < l \leq \zeta-1,
                 \end{cases}                 
$$
$$
  \Big[g |_{ V_-^0 (\zeta) }\Big]_{{\CC}^0_- (\zeta)} =
     \begin{cases}
      \Big[g |_{  V^l_{(-1)^{l+1}} (\zeta)}\Big]_{{\CC}^l_{(-1)^{l+1} } (\zeta)} &  \text{ for all } 0 \leq l < (\zeta-1)/2 \vspace{.15cm}\\
      \Big[g |_{  V^{(\zeta-1)/2}_{+} (\zeta)}\Big]_{{\CC}^{(\zeta-1)/2}_{+} (\zeta)} \vspace{.15cm}  \\
            \Big[g |_{  V^l_{(-1)^{l}} (\zeta)}\Big]_{{\CC}^l_{(-1)^{l}} (\zeta)} &  \text{ for all } (\zeta-1)/2 < l \leq \zeta-1.    
     \end{cases}
$$ 
\end{enumerate}
Then $K$ is a maximal compact subgroup of $\ZC_{{\rm SO} (p,q)} (X,H,Y)$.
\end{lemma}

\begin{proof}
 We omit the proof as it is identical to the proof of Lemma \ref{max-cpt-su-pq}. 
\end{proof}

The following lemma is required in the proof of Theorem \ref{so-pq-2} (2)(iv).
Recall that $\BC^0(d)$ is an ordered basis of $L(d-1)$ as in 
\eqref{old-ordered-basis-part} with $l\,=\,0$ and satisfying Remark \ref{unitary-J-basis-rmk} (1).

\begin{lemma}\label{max-cpt-so-p2}
   Suppose that $\Psi_{{\rm SO}(p,2)^\circ}(\OC_X)\,=\, ( [1^{p-2},\, 2^2],\, ((m^1_{ij}), \,
(m^2_{ij})))$, where $(m^1_{ij})$ and $(m^2_{ij})$ are $(p-2) \times 1$ and $2\times 2$
matrices, respectively, satisfying $m^1_{i1} = +1$ with $1 \,\leq\, i \,\leq\, p-2$, $m^2_{i1} 
\,= \,+1$ with $1 \,\leq\, i \,\leq\, 2$, and \ref{yd-def2}.          
Let $K$ be the subgroup of $\ZC_{{\rm SO}(p,2)}(X,H,Y)$ consisting of all $g\,\in\,
\ZC_{{\rm SO}(p,2)}(X,H,Y)$ such that the following hold:
\begin{enumerate}
  \item $g(L(1)) \subset L(1)$, $g(XL(1)) \subset XL(1)$, $\big[g|_{L(1)} \big]_{\BC^0(2)} = \big[g|_{XL(1)} \big]_{\BC^1(2)}$ and \\
  $\big[g|_{L(1)} \big]_{\BC^0(2)}  \begin{pmatrix}
                                         0 & -1 \\
                                         1 & 0
                                          \end{pmatrix}
                                          =\begin{pmatrix}
                                         0 & -1 \\
                                         1 &  0
                                          \end{pmatrix} \big[g|_{L(1)} \big]_{\BC^0(2)}$.
 \item $g(L(0)) \subset L(0)$.
 \end{enumerate}
Then $K$ is a maximal compact subgroup of $\ZC_{{\rm SO}(p,2)}(X,H,Y)$. 
\end{lemma}

  \begin{proof}
  Note that the form $(\cdot,\, \cdot)_1$ defined as in \eqref{new-form} is symmetric
on $L(1-1)\times L(1-1)$ with signature $(p-2,0)$, and 
 the form $(\cdot,\,\cdot)_2$ defined as in \eqref{new-form} is symplectic
on $L(2-1)\times L(2-1)$. Moreover, it follows from Proposition \ref{unitary-J-basis} that
$\BC^0(2) \,=\, (v_1^2 ~;~ v_2^2)$ is a symplectic basis of $L(2-1)$ for $(\cdot,\,\cdot)_2$. 
Now the lemma follows from Lemma \ref{reductive-part-comp}(4) and
Lemma \ref{max-cpt-sp-n-R}(1).
 \end{proof}

We next introduce some notation which will be needed in Proposition \ref{max-cpt-so-pq-wrt-onb} and in Proposition \ref{max-cpt-so-pq-0-wrt-onb}. 
We assume that $\N_\d= \O_\d$. For $\theta \,\in\, \O_\d$, define
$$
\CC_+ (\theta) \,:=\, \CC^0_+ (\theta) \vee \cdots \vee  \CC^{\theta-1}_+ (\theta) \ \ \text{ and } \ \
\CC_- (\theta) \,:=\, \CC^0_- (\theta) \vee \cdots \vee  \CC^{\theta-1}_- (\theta)\, .
$$
Let $\beta \,:=\, \# \O^1_\d$ and $ \gamma \,:=\, \# \O^3_\d $.
We enumerate 
 $\O^1_\d \,=\,\{ \theta_j \,\mid\, 1 \,\leq\, j \,\leq\, \beta \}$ such that $\theta_j 
\,<\, \theta_{j+1}$ and similarly $\O^3_\d \,=\,\{ \zeta_j \,\mid\, 1 \,\leq\, j \,\leq\, \gamma \}$
such that $\zeta_j \,<\, \zeta_{j+1}$. Set
$$
\OC^1_+ \,:=\, \CC_+ (\theta_1) \vee \cdots \vee \CC_+ (\theta_{\beta})\, ; \quad  \OC^3_+ := \CC_+ (\zeta_1) \vee \cdots \vee \CC_+ (\zeta_{\gamma})\,;
$$
$$ 
\OC^1_- := \CC_- (\theta_1) \vee \cdots \vee  \CC_- (\theta_{\beta})\ \ \text{ and } \ \ \OC^3_- := \CC_- (\zeta_1) \vee \cdots \vee  \CC_- (\zeta_{\gamma}).
$$
Now define
\begin{equation}\label{orthogonal-basis-so-pq-final}
 \HC_+ \,:=\,  \OC^1_+ \vee\OC^3_+ , \ \ \HC_- := \OC^1_- \vee\OC^3_- \ \text{ and } \
\HC := \HC_+ \vee \HC_-\, .
\end{equation}
It is clear that $ \HC$ is a standard orthogonal basis of $V$ such that 
$\HC_+\, =\, \{ v \,\in\, \HC \,\mid\, \langle v,\, v \rangle \,=\,1 \}$ and 
$\HC_- \,= \,\{ v \,\in\, \HC \,\mid\, \langle v,\, v \rangle \,=\,-1 \}$. In particular,
$\# \HC_+ \,=\, p$ and $\# \HC_- \,=\,q$. From the definition of $\HC_+$ and $\HC_-$ as given
in \eqref{orthogonal-basis-so-pq-final} we have the following relations:
$$
 \sum_{j=1}^\beta \big( \frac{\theta_j+1}{2}p_{\theta_j} +  \frac{\theta_j-1}{2}q_{\theta_j}\big)  + \sum_{k=1}^\gamma\big( \frac{\zeta_k-1}{2}p_{\zeta_k} +  \frac{\zeta_k+1}{2}q_{\zeta_k}\big) =p
$$
and 
$$
 \sum_{j=1}^\beta \big( \frac{\theta_j-1}{2}p_{\theta_j} +  \frac{\theta_j+1}{2}q_{\theta_j}\big)  + \sum_{k=1}^\gamma\big( \frac{\zeta_k+1}{2}p_{\zeta_k} +  \frac{\zeta_k-1}{2}q_{\zeta_k}\big) =q.
$$

The $\R$-algebra $\prod_{j=1}^\beta \big( {\rm M}_{p_{\theta_j}} (\R) \times
{\rm M}_{q_{\theta_j}} (\R) \big) \times \prod_{k=1}^\gamma
\big( {\rm M}_{p_{\zeta_k}} (\R) \times {\rm M}_{q_{\zeta_k}} (\R) \big)$ is embedded
in ${\rm M}_p(\R)$ and in ${\rm M}_q(\R)$ as follows:
$$
\Db_{p} \colon   \prod_{j=1}^\beta \big( {\rm M}_{p_{\theta_j}} (\R) \times {\rm M}_{q_{\theta_j}} (\R) \big) \times \prod_{k=1}^\gamma \big( {\rm M}_{p_{\zeta_k}} (\R) \times {\rm M}_{q_{\zeta_k}} (\R) \big) \longrightarrow {\rm M}_{p}(\R)  
$$
 \begin{align*}
 \big( C_{\theta_1}, D_{\theta_1} ,& \cdots,  C_{\theta_\beta}, D_{\theta_\beta}; E_{\zeta_1}, F_{\zeta_1}, \cdots ,  E_{\zeta_\gamma}, F_{\zeta_\gamma}   \big) \,\longmapsto
\end{align*}
 $$
      \bigoplus_{j=1}^\beta \Big( \big( C_{\theta_j} \oplus D_{\theta_j}\big)_\blacktriangle ^{\frac{\theta_j-1}{4}} \oplus C_{\theta_j} \oplus \big( C_{\theta_j} \oplus D_{\theta_j}\big)_\blacktriangle ^{\frac{\theta_j-1}{4}} \Big) \\ 
  \oplus  \bigoplus_{k=1}^\gamma \Big( \big( E_{\zeta_k}\oplus F_{\zeta_k} \big) _\blacktriangle ^{\frac{\zeta_k + 1}{4}}  \oplus \big( F_{\zeta_k}\oplus E_{\zeta_k} \big) _\blacktriangle ^{\frac{\zeta_k -3}{4}}  \oplus F_{\zeta_k} \Big)
  $$
and
$$
\Db_{q}  \colon   \prod_{j=1}^\beta \big( {\rm M}_{p_{\theta_j}} (\R) \times {\rm M}_{q_{\theta_j}} (\R) \big) \times \prod_{k=1}^\gamma \big( {\rm M}_{p_{\zeta_k}} (\R) \times {\rm M}_{q_{\zeta_k}} (\R) \big) \longrightarrow {\rm M}_{q}(\R)
$$
$$
\big( C_{\theta_1}, D_{\theta_1} , \cdots,  C_{\theta_\beta}, D_{\theta_\beta}; E_{\zeta_1}, F_{\zeta_1}, \cdots ,  E_{\zeta_\gamma}, F_{\zeta_\gamma} \big) \,\longmapsto
$$
$$
  \bigoplus_{j=1}^\beta \Big( \big( D_{\theta_j} \oplus C_{\theta_j}\big)_\blacktriangle ^{\frac{\theta_j-1}{4}} \oplus D_{\theta_j} \oplus \big( D_{\theta_j} \oplus C_{\theta_j}\big)_\blacktriangle ^{\frac{\theta_j-1}{4}} \Big) \\ 
 \oplus \bigoplus_{k=1}^\gamma \Big( \big( F_{\zeta_k}\oplus E_{\zeta_k} \big) _\blacktriangle ^{\frac{\zeta_k + 1}{4}}  \oplus \big( E_{\zeta_k}\oplus F_{\zeta_k} \big) _\blacktriangle ^{\frac{\zeta_k -3}{4}}  \oplus E_{\zeta_k} \Big). 
 $$
Define two characters
$$
\bigchi_{p} \colon \prod_{j=1}^\beta \big( {\rm O}_{p_{\theta_j}} \times {\rm O}_{q_{\theta_j}}  \big) \times \prod_{k=1}^\gamma \big( {\rm O}_{p_{\zeta_k}}  \times {\rm O}_{q_{\zeta_k}}  \big) \longrightarrow \R\setminus\{0\}
$$
$$
\big( C_{\theta_1}, D_{\theta_1} , \cdots,  C_{\theta_\beta}, D_{\theta_\beta};\, E_{\zeta_1}, F_{\zeta_1}, \cdots ,  E_{\zeta_\gamma}, F_{\zeta_\gamma}   \big) \longmapsto 
$$
$$
   \prod_{j=1}^\beta( \det  C_{\theta_j}^{\frac{\theta_j+1}{2} }  \det  D_{\theta_j}^{\frac{\theta_j-1}{2} } ) \prod_{k=1}^\gamma( \det  E_{\zeta_k}^{\frac{\zeta_k-1}{2} } \det  F_{\zeta_k}^{\frac{\zeta_k+1}{2} } ) =  \prod_{j=1}^\beta \det  C_{\theta_j} \prod_{k=1}^\gamma \det  E_{\zeta_k}
 $$
and
$$
\bigchi_{q}  \colon \prod_{j=1}^\beta \big( {\rm O}_{p_{\theta_j}} \times {\rm O}_{q_{\theta_j}}  \big) \times \prod_{k=1}^\gamma \big( {\rm O}_{p_{\zeta_k}}  \times {\rm O}_{q_{\zeta_k}}  \big) \longrightarrow \R\setminus\{0\}
$$
$$
\big( C_{\theta_1}, D_{\theta_1} , \cdots ,  C_{\theta_\beta}, D_{\theta_\beta};\, E_{\zeta_1}, F_{\zeta_1}, \cdots ,  E_{\zeta_\gamma}, F_{\zeta_\gamma} \big) \,\longmapsto
$$
$$
 \prod_{j=1}^\beta( \det  C_{\theta_i}^{\frac{\theta_i-1}{2} }  \det  D_{\theta_i}^{\frac{\theta_i+1}{2} } ) \prod_{k=1}^\gamma( \det  E_{\zeta_k}^{\frac{\zeta_k+1}{2} } \det  F_{\zeta_k}^{\frac{\zeta_k-1}{2}}) =  \prod_{j=1}^\beta \det D_{\theta_j} \prod_{k=1}^\gamma \det  F_{\zeta_k}.  
$$

Let $\Lambda_\HC \,\colon\, {\rm End}_\R \R^{n} \,\longrightarrow\, {\rm M}_n (\R)$ be the 
isomorphism of $\R$-algebras induced by the ordered basis $\HC$ in 
\eqref{orthogonal-basis-so-pq-final}. Let $M$ be the maximal compact subgroup of ${\rm 
SO}(p,q)$ which leaves invariant simultaneously the two subspaces spanned by $\HC_+$ and 
$\HC_-$. Clearly, $\Lambda_\HC (M) \,=\, {\rm S} ({\rm O} (p) \times {\rm O} (q))$.

\begin{proposition}\label{max-cpt-so-pq-wrt-onb} 
Let $X \,\in \,\NC_{\s\o(p,q)}$, $\Psi_{{\rm SO}(p,q)^\circ}(\OC_X)\, =\,(\d, \,\sgn_{\OC_X})$.
Assume that  $\N_\d\,=\, \O_\d$. Let $\beta \,:=\, \# \O^1_\d$ and $ \gamma \,:=\, \# \O^3_\d $. Let
$\{X,\,H,\,Y\}\,\subset\, \s\o(p,q)$ be a $\s\l_2(\R)$-triple, and let $(p_\theta, \,q_\theta)$ be
the signature of the form $(\cdot,\, \cdot)_\theta$ for all $\theta \,\in \,\O_\d$ as defined
in \eqref{new-form}. Let $K$ be the maximal compact subgroup of $\ZC_{{\rm SO} (p,q)} ( X, H, Y)$
as in Lemma \ref{max-cpt-so-pq}. Then  $\Lambda_\HC (K) \,\subset\,
{\rm S}({\rm O}(p)\times {\rm O}(q))$ is given by
\begin{align*} 
\Lambda_\HC(K) = 
 \Bigg\{ \Db_p (g) \oplus \Db_q(g)  \Biggm| 
                                         \begin{array}{c}
                                         g \in                               
                                       \prod_{j=1}^\beta \big( {\rm O}_{p_{\theta_j}} \times {\rm O}_{q_{\theta_j}} \big) 
                                         \times \prod_{k=1}^\gamma \big( {\rm O}_{p_{\zeta_k}} \times {\rm O}_{q_{\zeta_k}} \big) \vspace{.1cm} \\  \text{ and } \  ~ \bigchi_{p}(g) \bigchi_{q}(g) =1
                                        \end{array} \Bigg\}.
\end{align*}
\end{proposition}

\begin{proof}
This follows by writing the matrices of the elements of the maximal compact subgroup $K$ 
in Lemma \ref{max-cpt-so-pq} with respect to the basis $\HC$ as in 
\eqref{orthogonal-basis-so-pq-final}.
\end{proof}

As the subgroup ${\rm SO} (p,q)^\circ$ is normal in ${\rm SO} (p,q)$, so is $\ZC_{ 
{\rm SO} (p,q)^\circ} (X,H,Y)$ in $\ZC_{ {\rm SO} (p,q)} (X,H,Y)$. In particular, 
as $K$ is a maximal compact subgroup in $\ZC_{ {\rm SO} (p,q)}(X,H,Y)$, it follows 
that $$K_\O \,:=\, K \cap \ZC_{ {\rm SO}(p,q)^\circ}(X,H,Y) \,=\, K\cap {\rm SO} 
(p,q)^\circ$$ is a maximal compact subgroup of $\ZC_{ {\rm SO}(p,q)^\circ}(X,H,Y)$. 
The next proposition gives an explicit description of $\Lambda_\HC (K_\O)$ in $ 
{\rm SO} (p) \times {\rm SO} (q)$.

\begin{proposition}\label{max-cpt-so-pq-0-wrt-onb}
Let $X \in \NC_{\s\o(p,q)}$,  $\Psi_{{\rm SO}(p,q)^\circ}(\OC_X)$ $ =(\d, \sgn_{\OC_X})$.
We assume that $\N_\d= \O_\d$. Let $\beta := \# \O^1_\d$ and $ \gamma := \# \O^3_\d $.  Let $\{X,H,Y\}\subset \s\o(p,q)$ be a $\s\l_2(\R)$-triple, and let $(p_\theta, q_\theta)$ be the signature of the form $(\cdot, \cdot)_\theta$ for all $\theta \in \O_\d$ as defined in \eqref{new-form}. 
Let $K_\O$ be the maximal compact subgroup of $\ZC_{{\rm SO} (p,q)^\circ} ( X, H, Y)$ as in the preceding paragraph. Then  $\Lambda_\HC (K_\O) \subset {\rm SO}(p)\times {\rm SO}(q)$ is given by
\begin{align*}  
 \Bigg\{ \Db_p (g) \oplus \Db_q(g)  \Biggm| 
                                         \begin{array}{c}
                                         g \in                               
                                       \prod_{j=1}^\beta \big( {\rm O}_{p_{\theta_j}} \times {\rm O}_{q_{\theta_j}} \big) 
                                         \times \prod_{k=1}^\gamma \big( {\rm O}_{p_{\zeta_k}} \times {\rm O}_{q_{\zeta_k}} \big) \vspace{.1cm} \\  \text{ and } \  ~ \bigchi_{p}(g) =1 , \ \bigchi_{q}(g) =1
                                        \end{array} \Bigg\}.
\end{align*}
Moreover, the above group is isomorphic to 
$$
S\Big(\prod_{j=1}^\beta {\rm O}_{p_{\theta_j}} \times  \prod_{k=1}^\gamma {\rm O}_{p_{\zeta_k}} \Big) \times S\Big(\prod_{j=1}^\beta {\rm O}_{q_{\theta_j}} \times  \prod_{k=1}^\gamma {\rm O}_{q_{\zeta_k}} \Big).
$$
\end{proposition}

\begin{proof}
Let $V_+$ and $V_-$ be the $\R$-spans of $\HC_+$ and $\HC_-$ respectively.
 Let $M$ be the maximal compact subgroup in  ${\rm SO}(p,q)$ which simultaneously leaves the
subspaces $V_+$ and $V_-$ invariant. 
It is clear that $M^\circ$ is a maximal compact subgroup of ${\rm SO}(p,q)^\circ $. Hence
 $$M^\circ \,=\, {\rm SO}(p,q)^\circ \cap M \,=\,
 \{ g \,\in\, {\rm SO}(p,q) \,\mid\, \det g|_{V_+} \,=\,1,\ \det g|_{V_-} \,=\,1 \}\, .$$
 As $K \,\subset\, M$, we have that $ K \cap {\rm SO} (p,q)^\circ  \,=\, K \cap M^\circ$. The
proposition now follows.
\end{proof}

Set
   $$
   \O_\d^-:= \{\theta \mid \theta \in \O_\d, (\cdot, \cdot)_\theta \text{ is negative definite}\} \text{ and } \O_\d^+:= \{\theta \mid \theta \in \O_\d, (\cdot, \cdot)_\theta \text{ is positive definite}\}.
   $$
 
  \begin{lemma}\label{so-pq-odd}
  Let $X \,\in\, \s \o(p,q)$ be a nilpotent element. 
  Let $(\d, \,\sgn_{\OC_X}) \,\in\, \YC^{\rm even}_1(p,q)$ be the signed Young diagram of the
orbit $\OC_X$ (that is, $\Psi_{{\rm SO}(p,q)^\circ}  (\OC_X) \,= \,(\d, \,\sgn_{\OC_X})$ as in the
 notation of Theorem \ref{so-pq-parametrization}). We moreover assume that $\N_\d \,= \,\O_\d$.
  Let $K_\O$ be the maximal compact subgroup of $\ZC_{{\rm SO} (p,q)^\circ}( X, H, Y)$ as in
Proposition \ref{max-cpt-so-pq-0-wrt-onb}. Let $\k_\O$ be the Lie algebra of $K_\O$. Then the
following hold:
\begin{enumerate}
 \item 
 If $\# (\O_\d \setminus \O_\d^-)=1$,  $\# (\O_\d \setminus \O_\d^+)=1$ and
$p_{\theta_1} = q_{\theta_2} = 2$ for $\theta_1 \in \O_\d \setminus \O_\d^-$, $\theta_2 \in \O_\d \setminus \O_\d^+$, then  $\dim_\R \big[\z(\k_\O)\big]^{K_\O/K^\circ_\O} = 2$.  \vspace{.1cm}

  \item 
 Suppose that either $\# (\O_\d \setminus \O_\d^-)=1$, $p_{\theta_1} = 2$ for
 $\theta_1 \in \O_\d \setminus \O_\d^-$, or 
 $\# (\O_\d \setminus \O_\d^+)=1$, $ q_{\theta_2} = 2$ for $\theta_2 \in \O_\d \setminus \O_\d^+$. Moreover, suppose that both the conditions do not hold simultaneously. Then 
  $\dim_\R \big[\z(\k_\O)\big]^{K_\O/K^\circ_\O} = 1$. \vspace{.1cm}
  
\item 
  In all other cases, $~\dim_\R \big[\z(\k_\O)\big]^{K_\O/K^\circ_\O} = 0$.
\end{enumerate}
 \end{lemma}
  
\begin{proof} 
In view of \eqref{adjoint-action-SO_n},  \eqref{adjoint-action} and Proposition
\ref{max-cpt-so-pq-0-wrt-onb}, the lemma is clear.
\end{proof}

\begin{lemma}\label{centralizer-of-id-comp}
Let $W$ be a finite dimensional vector space over $\R$, and let $\<>'$ be a non-degenerate
symmetric bilinear form on $W$. Let $W_1, W_2 \,\subset\, W$ be subspaces such that $ W_1 \perp W_2$
and $W \,= \,W_1 \oplus W_2$. Let $\<>'_2$ be the restriction of $\<>'$ to $W_2$. 
Then
\begin{align*}
&{\rm SO} (W, \<>')^\circ \cap \{ g \in {\rm SO} (W, \<>') \mid g(W_1) \subset W_1, g (W_2) \subset W_2, g|_{ W_1} = {\rm Id}_{W_1} \, \} \vspace{.17cm}\\
 = \,& \{ g \in {\rm SO} (W, \<>') \mid g(W_1) \subset W_1, g (W_2) \subset W_2, g|_{W_1} = {\rm Id}_{W_1}, g|_{W_2} \in {\rm SO}(W_2, \<>'_2)^\circ \, \}.
\end{align*}
In particular,
$${\rm SO} (W, \<>')^\circ \cap \{ g \,\in\, {\rm SO} (W, \<>') \,\mid\, g(W_1) \,\subset\, W_1,\
g (W_2) \,\subset\, W_2,\ g|_{W_1} \,=\, {\rm Id}_{W_1} \, \}$$
is isomorphic to ${\rm SO}(W_2, \<>'_2)^\circ $.
\end{lemma}

\begin{proof}
Let $(p_2,\, q_2)$ be the signature of $\<>'_2$. If either $p_2 \,=\,0$ or $q_2\,=\,0$, then as
${\rm SO}(W_2, \<>'_2)$ $  =\,{\rm SO}(W_2, \<>'_2)^\circ$ the lemma follows immediately.

Assumption that $p_2 \,>\, 0$ and $q_2 \,>\,0$. In this case, considering an orthogonal basis of
$W_2$ for the form $\<>'_2$ we easily construct a linear map $A \,\colon\, W \,\longrightarrow\, W$
such that $A|_{W_1}\,=\, {\rm Id}_{W_1}$, $A(W_2) \,\subset\, W_2$, $(A|_{W_2})^2\,=\, {\rm Id}_{W_2}$, 
and $A|_{W_2} \,\in\, {\rm SO}(W_2, \<>'_2) \setminus {\rm SO}(W_2, \<>'_2)^\circ$. It is then clear
that $$A\,\in\, {\rm SO}(W, \<>') \setminus {\rm SO}(W, \<>')^\circ\, .$$
Let $\Gamma \,\subset \,{\rm GL}(W)$ be the subgroup generated by $A$ and  $\Gamma' \,\subset\,
{\rm GL}(W_2)$ the subgroup generated by $A|_{W_2}$.
It then follows that ${\rm SO}(W, \<>') \,=\, \Gamma \, {\rm SO}(W, \<>')^\circ$ and 
${\rm SO}(W_2, \<>'_2) \,=\, \Gamma' \, {\rm SO}(W_2, \<>'_2)^\circ$. Now the lemma follows.
\end{proof}

We now describe the second cohomology groups of nilpotent orbits in $\s\o(p,q)$ when $p >0, \, q >0$. 
As we will consider only simple Lie algebras, to ensure simplicity of $\s\o(p,q)$, 
in view of \cite[Theorem 6.105, p. 421]{K} and 
isomorphisms (iv), (v), (vi), (ix), (x) in \cite[Chapter X, \S 6, pp. 519-520]{He},
we need the additional restriction that $(p,q) \not\in \{(1,1), (2,2)\}$.

\begin{theorem}\label{so-pq}
Let $p\,\ne\, 2$, $q \,\ne\, 2$ and $(p,q) \neq (1,1)$. Let $X \,\in\, \s \o(p,q)$ be a
nilpotent element. Let $(\d, \,\sgn_{\OC_X}) \,\in\, \YC^{\rm even}_1(p,q)$ be the signed Young
diagram of the orbit $\OC_X$ (that is, $\Psi_{{\rm SO}(p,q)^\circ}  (\OC_X) \,=\, (\d, \,\sgn_{\OC_X})$
as in the notation of Theorem \ref{so-pq-parametrization}). Then the following hold:
\begin{enumerate}
 \item 
 If $\# (\O_\d \setminus \O_\d^-)\,=\,1$,  $\# (\O_\d \setminus \O_\d^+)\,=\,1$ and $p_{\theta_1}
\,=\, q_{\theta_2} \,=\, 2$ when $\theta_1 \,\in\, \O_\d \setminus \O_\d^-$ and $\theta_2 \,\in\,
\O_\d \setminus \O_\d^+$, then $\dim_\R H^2(\OC_X,\, \R) \,= \, (\#\E_\d + 2)$. \vspace{.1cm}  

\item 
   Suppose that either $\# (\O_\d \setminus \O_\d^-)\,=\,1$ and $p_{\theta_1} \,=\, 2$ for
$\theta_1 \,\in\, \O_\d \setminus \O_\d^-$, or 
 $\# (\O_\d \setminus \O_\d^+)\,=\,1$ and $ q_{\theta_2} \,=\, 2$ for
$\theta_2 \,\in\, \O_\d \setminus \O_\d^+$. Moreover, suppose that
the above two conditions do not hold simultaneously. Then $\dim_\R H^2(\OC_X,\, \R) \,=\, (\#\E_\d + 1)$.
\vspace{.1cm} 

\item
 In all other cases $~\dim_\R H^2(\OC_X, \,\R) \,=\, \#\E_\d$.
\end{enumerate}
\end{theorem}

\begin{proof}
Let $p+q \,=\, n$. As the theorem is evident when $X = 0$, we assume that $X \,\neq\, 0$.

Let $\{X,\, H,\, Y\} \,\subset\, \s\o(p,q)$ be a $\s\l_2(\R)$-triple. 
Let $V\,:=\, \R^{n}$ be the right $\R$-vector space of column vectors. We consider $V$ as a
$\text{Span}_\R \{ X,H,Y\}$-module via its natural $\s\o(p,q)$-module structure. Let 
$$
V_\E \,:=\, \bigoplus_{\eta \in \E_\d} M(\eta -1) ; \quad V_\O \,:=\, \bigoplus_{\theta \in \O_\d} M(\theta -1).
$$
Using Lemma \ref{ortho-isotypical} it follows that $V\,=\,V_\E \oplus V_\O$ is an orthogonal
decomposition of $V$ with respect to $\<>$.  
Let $\<>_{\E}\, :=\, \<> |_{V_\E \times V_\E  }$ and $\<>_{\O} \,:=\, \<> |_{V_\O \times V_\O}$.
Let $ X_{\E} \,:=\, X|_{V_\E}$,  $X _\O \,:=\, X|_{V_\O}$, $H_\E \,:=\, H|_{V_\E}$, $H_\O \,:=\,
 H|_{V_\O}$, $Y_\E\,:=\, Y|_{V_\E}$ and $Y_\O \,:=\, Y|_{V_\O}$.
Then we have the following natural isomorphism
\begin{equation}\label{reductive-part-odd-even}
 \ZC_{{\rm SO}(p,q)}(X,H,Y)~ \simeq~ \ZC_{{\rm SO}(V_\E, \<>_\E)}(X_\E,H_\E,Y_\E) \times \ZC_{{\rm SO}(V_\O, \<>_\O)}(X_\O,H_\O,Y_\O).
\end{equation}
As, the form $(\cdot, \cdot)_\eta$ on $L (\eta-1)$ is non-degenerate and symplectic for all $\eta \in \E_\d$, it follows
 from Lemma \ref{reductive-part-comp} (4) that
 \begin{equation}\label{description-of-ZSO-Ve}
  \ZC_{{\rm SO}(V_\E, \<>_\E)}(X_\E,H_\E,Y_\E) \simeq \prod_{\eta \in \E_\d} {\rm Sp}(t_\eta/2,\, \R).
 \end{equation}
In particular, $\ZC_{{\rm SO}(V_\E, \<>_\E)}(X_\E,H_\E,Y_\E) $ is connected, and hence
using Lemma \ref{centralizer-of-id-comp}, \eqref{reductive-part-odd-even} and \eqref{description-of-ZSO-Ve} it follows that
\begin{equation}\label{reductive-part-odd-even-connected}
 \ZC_{{\rm SO}(p,q)^\circ}(X,H,Y) \simeq  
                              \ZC_{{\rm SO}(V_\E, \<>_\E)}(X_\E,H_\E,Y_\E)  \times \ZC_{{\rm SO}(V_\O, \<>_\O)^\circ}(X_\O,H_\O,Y_\O).
\end{equation}
 Let $K_\E$ be a maximal compact subgroup of $ \ZC_{{\rm SO}(V_\E, \<>_\E)}(X_\E,H_\E,Y_\E) \simeq \prod_{\eta \in \E_\d} {\rm Sp}(t_\eta/2, \R)$.
Setting $\# \O_\d \,:=\, r$, enumerate $\O_\d \,=\,\{ a_1,\, \cdots,\, a_r \}$ such that
$ a_i \,<\, a_{i+1}$ for all $i$. We next set
$\d_\O \,:=\, [ a_1^{t_{a_1}}, \,\cdots, \,a_r^{t_{a_r}}]$. As $\sum_{d \in \O_\d} t_d d
\,=\, \dim_\R V_\O$, we have $\d_\O \,\in\, \PC(\dim_\R V_{\O})$.
We recall that
$ K_\O\,:=\, K \cap \ZC_{{\rm SO}(p,q)^\circ } ( X, H, Y) \,=\, K \cap {\rm SO}(p,q)^\circ$ is
a maximal compact subgroup of $\ZC_{{\rm SO}(p,q)^\circ }(X, H, Y)$, where
$K$ is the maximal compact subgroup of $\ZC_{{\rm SO} (p,q)} ( X, H, Y)$ as in Lemma
\ref{max-cpt-so-pq}.
Let $\widetilde{K}$ be the image of $K_\O \times K_\E$ under the isomorphism
in \eqref{reductive-part-odd-even-connected}.  
It is evident that $\widetilde{K}$ is a maximal compact subgroup of $\ZC_{{\rm SO}(p,q)^\circ}(X,H,Y)$. 
Let $\widetilde{M}$ be a maximal compact subgroup of ${\rm SO}(p,q)^\circ$ containing $\widetilde{K}$. 
Let $\widetilde{\k}$ and $ \widetilde\m$ be the Lie algebras of $\widetilde{K}$ and
$\widetilde{M}$ respectively.
As $p\,\neq\, 2,\ q \,\neq\, 2$, we have $\widetilde\m \,=\, [\widetilde{\m},\, \widetilde{\m}]$.
Then using Theorem \ref{thm-nilpotent-orbit} it follows that, for all $X \,\neq\, 0$,   
$$
H^2 ( {\OC}_X,\, \R) ~ \simeq ~ [\z (\widetilde{\k})^*]^{\widetilde{K} / {\widetilde{K}}^\circ}\, .
$$
Let $\k_\E , \k_\O$ be the Lie algebras of $K_\E, K_\O$ respectively.
As  $K_\E$ is connected, in view of \eqref{adjoint-action} we conclude that $$ [\z
(\widetilde{\k})^*]^{\widetilde{K} / {\widetilde{K}}^\circ} \,\simeq\, \z(\k_\E) \oplus [\z (\k_\O)]^{K_\O / K_\O^\circ}\, .$$
From \eqref{description-of-ZSO-Ve} we have $\k_\E
\,\simeq\, \bigoplus_{\eta \in \E_\d} \u(t_\eta/2)$. In particular, $\dim_\R \z(\k_\E) \,=\, \#\E_\d$.
As $\N_{\d_\O} \,=\, \O_{\d_\O}$, we use Lemma \ref{so-pq-odd} to compute the dimension
of $[\z (\k_\O)]^{K_\O / K_\O^\circ}$. This completes the  proof.
\end{proof}

We will next consider the remaining cases which are not covered in Theorem \ref{so-pq}. These cases are:  
 $(p,q)\,\in\, \{ (2,1),\,(1,2)\}$; $p \,>\, 2,\ q\,=\,2$ and $p\,=\,2, q\, >\, 2$.
Recall the definition of $\YC^{\rm even}_1(p,q)$ as given in \eqref{yd-1-Y-pq}.
If $p \,>\,2$, then we list below the set of signed Young diagrams $\YC^{\rm even}_1(p,2)$ which correspond to non-zero nilpotent orbits in $\s\o(p,2)$.
 \begin{enumerate}[label = {{\bf a}.\arabic*}]
  \item\label{a.1} $\big( [1^{p-1}, 3^1],\, ((m^1_{ij}), (m^3_{ij}))\big)$, where
$(m^1_{ij})$ and $(m^3_{ij})$ are $(p-1) \times 1$ and $ 1 \times 3$ matrices respectively, satisfying $m^1_{i1} = +1, 1 \leq i \leq p-1; \,\, m^3_{i1} = +1, i =1$ and \ref{yd-def2}. \vspace{.2cm}
  
  \item\label{a.2} $\big( [1^{p-1}, 3^1],\, ((m^1_{ij}), (m^3_{ij}))\big)$, where $ (m^1_{ij})$
and $(m^3_{ij})$ are $(p-1) \times 1$ and $ 1 \times 3$ matrices respectively, satisfying $m^1_{i1} = +1, 1 \leq i \leq p-2,~  m^1_{i1} = -1,i=p-1; \,\, m^3_{i1} = -1, i =1$ and \ref{yd-def2}.
  
  \item\label{a.3} $\big( [1^{p-3}, 5^1],\, ((m^1_{ij}), (m^5_{ij}))\big)$, where $(m^1_{ij})$
and $(m^5_{ij})$ are $(p-3) \times 1$ and $ 1 \times 5$ matrices respectively, satisfying $m^1_{i1} = +1, 1 \leq i \leq p-3 ; \,\, m^5_{i1} = +1, i =1$ and \ref{yd-def2}.\vspace{.2cm}
  
   \item\label{a.4} $\big( [1^{p-2}, 2^2],\, ((m^1_{ij}), (m^2_{ij}))\big)$, where
$(m^1_{ij})$ and $(m^2_{ij})$ are $(p-2) \times 1$ and $2\times 2$ matrices respectively, satisfying $m^1_{i1} = +1, 1 \leq i \leq p-2 ; \,\, m^2_{i1} = +1, 1 \leq i \leq2$ and \ref{yd-def2}.
   \end{enumerate}
 
Similarly as above, if $q >2$, then set $\YC^{\rm even}_1(2,q)$ consists of four elements which correspond to non-zero nilpotent orbits in $\s\o(2,q)$.
These are listed below:
 \begin{enumerate}[label = {{\bf b}.\arabic*}]
 \item\label{b.1} $\big( [1^{q-1}, 3^1],\, ((m^1_{ij}), (m^3_{ij}))\big)$, where $ (m^1_{ij})$ and
$(m^3_{ij})$ are $(q-1) \times 1$ and $ 1 \times 3$ matrices respectively, satisfying $m^1_{i1} = -1, 1 \leq i \leq q-1; \,\, m^3_{i1} = -1, i =1$ and \ref{yd-def2}.\vspace{.2cm}
  
  \item\label{b.2} $\big( [1^{q-1}, 3^1],\, ((m^1_{ij}), (m^3_{ij}))\big)$, where $(m^1_{ij})$ and
$(m^3_{ij})$ are $(q-1) \times 1$ and $ 1 \times 3$ matrices respectively, satisfying $m^1_{i1} = +1, i=1,~  m^1_{i1} = -1, 2 \leq i \leq q-1; \,\, m^3_{i1} = +1, i =1$ and \ref{yd-def2}.
  
  \item\label{b.3} $\big( [1^{q-3}, 5^1],\, ((m^1_{ij}), (m^5_{ij}))\big)$, where $(m^1_{ij})$ and
$(m^5_{ij})$ are $(q-3) \times 1$ and $ 1 \times 5$ matrices respectively, satisfying $m^1_{i1} = -1, 1 \leq i \leq q-3 ; \,\, m^5_{i1} = -1, i =1$ and \ref{yd-def2}.\vspace{.2cm}
  
   \item\label{b.4} $\big( [1^{q-2}, 2^2],\, ((m^1_{ij}), (m^2_{ij}))\big)$, where $(m^1_{ij})$ and
$(m^2_{ij})$ are $(q-2) \times 1$ and $2\times 2$ matrices respectively, satisfying $m^1_{i1} = -1, 1 \leq i \leq q-2 ; \,\, m^2_{i1} = +1, 1 \leq i \leq2$ and \ref{yd-def2}.
 \end{enumerate}

\begin{theorem}\label{so-pq-2} 
Let $ \Psi_{{\rm SO}(p,q)^\circ} \,:\, \NC({\rm SO}(p,q)^\circ) \,\longrightarrow\,
\YC^{\rm even}_1(p,q)$ be the parametrization in Theorem \ref{so-pq-parametrization}.
Let $\OC_X \,\in\,\NC({\rm SO}(p,q)^\circ) $. Then the following hold:
\begin{enumerate}
 \item  Suppose $(p,\,q)\,\in\, \{ (2,1),\,(1,2)\}$, then $H^2({\OC}_X,\, \R) \,=\,0$.
 
 \item  Assume that $p> 2,\, q =2$.     \\
 (i)   If $ \Psi_{{\rm SO}(p,2)^\circ} (\OC_X)$ is as in \eqref{a.1}, then $\dim_\R H^2(\OC_X, \,\R)
\,=\, 0$.  \\
 (ii)  If $ \Psi_{{\rm SO}(p,2)^\circ} (\OC_X)$ is as in \eqref{a.2}, then  
 $ \dim_\R H^2(\OC_X, \,\R)= \begin{cases}
                 1 & \text{if } p=4\\
                 0 & \text{otherwise}.
                \end{cases}$\\
 (iii) If $ \Psi_{{\rm SO}(p,2)^\circ} (\OC_X)$ is as in \eqref{a.3}, then $\dim_\R H^2(\OC_X,\, \R)
\,= \,0$.\\
 (iv)   If $ \Psi_{{\rm SO}(p,2)^\circ} (\OC_X)$ is as in \eqref{a.4}, then 
  $ \dim_\R H^2(\OC_X, \,\R)= \begin{cases}
                 1 & \text{if } p=4\\
                 0 & \text{otherwise}.
                \end{cases}$
 
 \item Assume $p= 2 $ and $q>2$. \\
 (i)   If $ \Psi_{{\rm SO}(2,q)^\circ} (\OC_X)$ is as in \eqref{b.1}, then  $\dim_\R H^2(\OC_X,\, \R)
\,=\, 0$.\\
 (ii) If $ \Psi_{{\rm SO}(2,q)^\circ} (\OC_X)$ is as in \eqref{b.2}, then 
  $\dim_\R H^2(\OC_X,\, \R)= \begin{cases}
                 1 & \text{if } q=4\\
                 0 & \text{otherwise}.
                \end{cases} $\\
  (iii) If $\Psi_{{\rm SO}(2,q)^\circ}(\OC_X)$ is as in \eqref{b.3}, then  $ \dim_\R H^2(\OC_X,\, \R)
\,=\, 0$. \\
  (iv)  If $\Psi_{{\rm SO}(2,q)^\circ}(\OC_X)$ is as in \eqref{b.4}, then 
  $\dim_\R H^2(\OC_X,\, \R)= \begin{cases}
                 1 & \text{if } q=4\\
                 0 & \text{otherwise}.
                \end{cases}$
\end{enumerate}
\end{theorem}
  
\begin{proof}
As $X \,\neq\, 0$, we may assume that $X$ lies in a $\s\l_2(\R)$-triple, say $\{ X,H,Y\}$, in
$\s\o(p,q)$. 

{\it Proof of (1):}
Let $\m$ be the Lie algebra of a maximal compact subgroup of ${\rm SO}(p,q)^\circ$. As $(p,q)\in \{ (2,1),(1,2) \}$, we have $[\m,\,\m] = 0$. 
Thus using Theorem \ref{thm-nilpotent-orbit} it follows that $H^2 ( {\OC}_X, \,\R)\, =\,0$.
 
 {\it Proof of (2):}
As $\N_\d \,=\, \O_\d$ in each of the cases (i), (ii) and (iii), 
we will use Proposition \ref{max-cpt-so-pq-0-wrt-onb}.
Let $K$ be the maximal compact subgroup of $\ZC_{{\rm SO}(p,2)}(X,H,Y)$ as given in Lemma
\ref{max-cpt-so-pq}. Let $M$ be the maximal compact subgroup of ${\rm SO}(p,2)$ which
leaves invariant simultaneously the two subspaces spanned by $\HC_+$ and $\HC_-$, where $\HC_+$
and $\HC_-$ are as in \eqref{orthogonal-basis-so-pq-final} with $q=2$. Then 
$M^\circ \,=\, M \cap {\rm SO}(p,2)$ is a maximal compact subgroup of ${\rm SO}(p,2)^\circ$.
Recall that $$K_\O \,:=\, K \cap M^\circ \,=\, K \cap {\rm SO}(p,2)^\circ$$ is a maximal
compact subgroup of  $\ZC_{{\rm SO}(p,2)^\circ}(X,H,Y)$. Then, in the notation of
Proposition \ref{max-cpt-so-pq-0-wrt-onb}, $\Lambda_\HC(K_\O) \,\subset\,
{\rm SO}(p) \times {\rm SO}(2)$. Let $\k_\O $ and $\m$ be the Lie algebras of $K_\O$ and $M^\circ$
respectively. 
 
We now prove (i) of (2).
Suppose $ \Psi_{{\rm SO}(p,2)^\circ} (\OC_X)$ is as in \eqref{a.1}. Using Proposition
\ref{max-cpt-so-pq-0-wrt-onb} it follows that
\begin{equation}\label{max-cpt-embdd-so-B-1}
\begin{split}
 \Lambda_\HC(K_\O) & = \big\{ \Db_p(g) \oplus \Db_q(g)  \bigm| g \in {\rm O}_{p-1} \times {\rm O}_{1}, ~ \bigchi_p(g) =1, \bigchi_q(g)=1 \big\} \\
                   & = \big\{ C \oplus E \bigoplus E \oplus E ~\bigm| ~ C \in {\rm O}_{p-1}, E \in{\rm O}_{1},~ \det C \det E =1   \big\}.
\end{split}
\end{equation}
Therefore, $\z(\k_\O) \cap [\m,\,\m] \,=\, \s\o_2$ when $p\,=\,3$, and $\z(\k_\O)\,=\, 0$ when
$p\,>\,3$. From \eqref{max-cpt-embdd-so-B-1} it follows that $K_\O \,\simeq\, S({\rm O}_2 \times
{\rm O}_1)$ when $p\,=\,3$. Since ${\rm O}_2 /{\rm SO}_2$ acts non-trivially on $\s\o_2$, when
$p\,=\,3$ we have $\big[\z(\k_\O)\cap[\m,\,\m]\big]^{K/K^\circ}\, =\,0$. Thus using Theorem
\ref{thm-nilpotent-orbit}, 
$$H^2(\OC_X,\, \R)\,=\,0$$
for all $p\,>\,2$.

We next give a proof of (ii) of (2).
Assume that $ \Psi_{{\rm SO}(p,2)^\circ} (\OC_X)$ is as in \eqref{a.2}. Using Proposition
\ref{max-cpt-so-pq-0-wrt-onb} and notation therein,
  \begin{align}
  \Lambda_\HC(K_\O) & = \big\{\Db_p(g) \oplus \Db_q(g)  \bigm| g \in  {\rm O}_{p-2} \times {\rm O}_1 \times {\rm O}_1, ~ \bigchi_p(g) =1, \bigchi_q(g) =1 \big\} \nonumber \\
        & = \big\{ C \oplus F \oplus F \bigoplus D \oplus F ~ \bigm| ~ C \in {\rm O}_{p-2}; D ,  F \in{\rm O}_{1};~ \det C =1, \det D \det F =1   \big\}. \label{max-cpt-embdd-so-B-2}
 \end{align} 
It is clear from above that $\z(\k_\O)\cap [\m,\,\m]=\s\o_2$ when $p = 4$ and $\z(\k_\O) = 0$
when $p \neq 4$, $p>2$. When $p=4$, then
$K_\O \simeq {\rm SO}_2 \times S({\rm O}_1 \times {\rm O}_1)$ from \eqref{max-cpt-embdd-so-B-2}. As
${\rm SO}_2$ acts trivially on $\s\o_2$, using Theorem \ref{thm-nilpotent-orbit} we conclude that
$$
 \dim_\R H^2(\OC_X,\, \R)= \begin{cases}
                 1 & \text{ if } p=4\\
                 0 & \text{ otherwise}.
             \end{cases}
$$
 
We now give a proof of (iii) of (2).
Assume that $\Psi_{{\rm SO}(p,2)^\circ} (\OC_X)$ is as in \eqref{a.3}. Using Proposition
\ref{max-cpt-so-pq-0-wrt-onb} and notation therein,
\begin{equation}\label{max-cpt-embdd-so-B-3}
\begin{split}
 \Lambda_\HC(K_\O) & = \big\{ \Db_p(g) \oplus \Db_q(g)  \bigm| g \in {\rm O}_{p-3} \times {\rm O}_{1}, ~ \bigchi_p(g) =1, \bigchi_q(g)=1 \big\} \\
                   & = \big\{ C \oplus E  \oplus E \oplus E \bigoplus E \oplus E  ~\bigm| ~ C \in {\rm O}_{p-3}, E \in{\rm O}_{1},~ \det C \det E =1 \big\}.
\end{split}
\end{equation}
 Therefore, we have $\z(\k_\O) \cap [\m,\,\m] \,= \,\s\o_2$ when $p=5$, and
$\z(\k_\O)= 0$ for $p>2, p\neq 5$. It follows from \eqref{max-cpt-embdd-so-B-3} that
$K_\O \,\simeq\, S({\rm O}_2 \times {\rm O}_1)$ when $p=5$. Since ${\rm O}_2 /{\rm SO}_2$ acts
 non-trivially on $\s\o_2$, in the case when $p=5$ we have
$\big[\z(\k_\O)\cap[\m,\,\m]\big]^{K/K^\circ} =0$. Thus in view of
Theorem \ref{thm-nilpotent-orbit},
$$H^2(\OC_X, \,\R)\,=\,0$$
for all $p\,>\,2$.
  
We now give a proof of (iv) of (2). 
Let $n\,=\,p+2$. Suppose $\Psi_{{\rm SO}(p,2)^\circ} (\OC_X)$ is as in \eqref{a.4}.
We need to construct a standard orthogonal basis as done before. We follow the notation as
in Lemma \ref{max-cpt-so-p2}.
Define, $\AC_+(2)\,:=\, \big( (v_1^2 + Xv^2_2) / \sqrt{2} ,\, (v_2^2 - Xv^2_1)/ \sqrt{2} \big)$
and $$\AC_-(2)\,:=\, \big( (v_1^2 - Xv^2_2) / \sqrt{2},\, (v_2^2 + Xv^2_1)/
\sqrt{2} \big)\, .$$
Finally set $\HC_+:= \BC^0(1) \vee \AC_+(2)$, 
$\HC_-:=\AC_-(2)$ and $\HC := \HC_+ \vee \HC_-$. Then it is clear that $\HC$ is a standard
orthogonal basis of $V$ such that $$\HC_+\,=\, \{v \,\in\, \HC \,\mid\, \langle v,\,v
\rangle\,=\,1\}$$ and
$\HC_-\,=\, \{v\,\in\, \HC \,\mid\, \langle v,\,v \rangle \,=\,-1\}$. In particular, $\# \HC_+\,=\, p$
and $\# \HC_- \,=\, 2$. Let $V_+(2)$, $V_-(2)$ be the spans of $\AC_+(2)$, $\AC_-(2)$ respectively.
Let $K$ be the maximal compact subgroup of $\ZC_{{\rm SO}(p,2)}(X,H,Y)$ as in
Lemma \ref{max-cpt-so-p2}. We observe that if $g \,\in\, K$, then $g(V_+(2))
\,\subset\, V_+(2)$, $g(V_-(2)) \,\subset\, V_-(2)$ and
 $$
 \big[g|_{V_+(2)} \big]_{\AC_+(2)} = \big[g|_{V_-(2)} \big]_{\AC_-(2)} = \big[g|_{L(1)} \big]_{\BC^0(2)}.
 $$ 
Let $\Lambda_\HC\,\colon\,{\rm End}_\R \R^{n}\,\longrightarrow\, {\rm M}_n(\R)$ be the isomorphism
of $\R$-algebras induced by the above ordered basis $\HC$.
Let $M$ be the maximal compact subgroup in ${\rm SO}(p,2)$ which simultaneously leaves the subspaces
spanned by $\HC_+$ and $\HC_-$ invariant. 
Then $M^\circ\,=\, M \cap {\rm SO}(p,2)^\circ$ is a maximal compact subgroup of
${\rm SO}(p,2)^\circ$, and 
$\widetilde{K}\,:=\, K \cap M^\circ$ is a maximal compact subgroup of
$\ZC_{{\rm SO}(p,2)^\circ}(X,H,Y)$. We have the following explicit description of
$\Lambda_\HC(\widetilde{K}) \,\subset\, {\rm SO}(p) \times {\rm SO}(2)$:
\begin{equation}\label{description-of-K}
\Lambda_\HC(\widetilde{K}) = \big\{ A\oplus B \bigoplus B  \bigm|  A \in  {\rm O}_{p-2}, B \in {\rm O}_2; ~ \det A \det B =1 \text{ and } \det B=1  \big\}.
\end{equation}
In particular, $\widetilde{K} \simeq {\rm SO}_{p-2}\times {\rm SO}_2$. Let $\widetilde{\k}$ and
$\m$ be the Lie algebras of $ \widetilde{K}$ and $M^\circ$ respectively. From
\eqref{description-of-K},
 $$
 \z(\widetilde{\k})\cap[\m,\,\m] = \begin{cases}
                                 \s\o_2 & \text{ if } p=4\\
                                  0     & \text{ otherwise}.
                                 \end{cases}
 $$ 
As $\widetilde{K}$ is connected, the conclusion follows from Theorem \ref{thm-nilpotent-orbit}.
This completes the proof of (2).

The proofs of (3)(i), (3)(ii), (3)(iii) and (3)(iv) are similar to those of (2)(i), (2)(ii), 
(2)(iii) and (2)(iv) respectively and hence the details are omitted.
\end{proof}

\subsection{Second cohomology groups of nilpotent orbits in 
\texorpdfstring{${\s\o^*}(2n)$}{Lg}}\label{sec-so*}
Let $n$ be a positive integer. The aim in this subsection is to compute the second cohomology groups of nilpotent 
orbits in the simple real Lie algebra ${\s\o^*}(2n)$ under the adjoint action of ${\rm SO}^*(2n)$.
Throughout this subsection $\<>$ denotes the skew-Hermitian form on $\H^n$ defined by $\langle x, y \rangle \,:=\, 
\overline{x}^t \jb{\rm I}_{n} y$, for $x,y \in \H^n$. We will follow notation as defined in \S \ref{sec-notation}.

We first need to describe a suitable parametrization of $\NC ({\rm SO}^*(2n))$, the set of all nilpotent orbits in ${\s\o^*}(2n)$. 
Let $\Psi_{{\rm SL}_n (\H)} \,:\, \NC ({\rm SL}_n (\H)) \,\longrightarrow\, \PC (n)$ be the 
parametrization as in Theorem \ref{sl-H-parametrization}. As ${\rm SO}^*(2n) \,\subset\, {\rm SL}_n 
(\H)$ (consequently, $\NC_{{\s\o}^*(2n)} \,\subset \,\NC_{\s\l_n(\H)}$) we have the  inclusion 
map $ \Theta_{{\rm SO}^*(2n)} \,:\, \NC ({\rm SO}^*(2n)) \,\longrightarrow\,
\NC ( {\rm SL}_n (\H) )$. Let
$$
\Psi'_{{\rm SO}^*(2n)}\,:=\, \Psi_{{\rm SL}_n (\H)} \circ \Theta_{{\rm 
SO}^*(2n)}\, \colon\, \NC ({\rm SO}^*(2n)) \,\longrightarrow\, \PC (n)
$$ be the composition.
Let $X \in {\s\o^*}(2n)$ be a non-zero nilpotent element and $\OC_X$ be the 
corresponding nilpotent orbit in $\s\o^*(2n)$. Let $\{X,\, H,\, Y\} \,\subset\, {\s\o^*}(2n)$ be a 
$\s\l_2(\R)$-triple. We now apply Proposition \ref{unitary-J-basis}, Remark 
\ref{unitary-J-basis-rmk}(3), and follow the notation used therein. Let $V:= \H^n$ be the 
right $\H$-vector space of column vectors. Let $\{d_1,\, \cdots,\, d_s\}$, with $d_1 \,<\, \cdots \,<\,
d_s$, be the ordered finite set of natural numbers that arise as $\R$-dimension of
non-zero irreducible $\text{Span}_\R \{ X,H,Y\}$-submodules of $V$. 
Recall that $M(d-1)$ is defined to be the isotypical component of $V$ containing all irreducible $\text{Span}_\R \{ X,H,Y\}$-submodules of $V$ with highest weight $(d-1)$, and as in \eqref{definition-L-d-1}, we set $L(d-1)\,:= \,V_{Y,0} \cap M(d-1)$. Recall that the space $L(d_r-1)$ is a $\H$-subspace for $1 \leq r \leq s$.
Let $t_{d_r} \,:=\, \dim_\H L(d_r-1)$, $1 \,\leq\, r \,\leq\, s$.
Then $\d\,:=\, [d_1^{t_{d_1}},\, \cdots 
,\, d_s^{t_{d_s}}] \,\in\, \PC(n)$, and moreover, $\Psi'_{{\rm SO}^*(2n)} (\OC_X) \,=\, \d$.

We next assign $\sgn_{\OC_X} \,\in\, \SC^{\rm odd}_{\d}(n)$ to each $\OC_X \,\in \,\NC({\rm SO}^*(2n))$; see
 \eqref{S-d-pq-odd} for the definition of $\SC^{\rm odd}_{\d}(n)$.
For each $d\,\in\, \N_\d$ (see \eqref{Nd-Ed-Od} for the definition of $\N_\d$) we will define a $t_d \times  d$ matrix $ (m^d_{ij})$ in $\Ab_d$ that
depends only on the orbit $\OC_X$; see \eqref{A-d} for the definition of $\Ab_d$.
For this, recall that the form $(\cdot,\,\cdot)_{d} \,\colon\, L(d-1) \times L(d-1) \,
\longrightarrow\, \H$ defined as in \eqref{new-form} is skew-Hermitian or Hermitian according as $d$
is odd or even.
Denoting the  signature of $(\cdot,\,\cdot)_{\eta}$ by $(p_{\eta},\, q_{\eta})$ when
$\eta\,\in\, \E_\d$, we now define
\begin{align*}
m^\theta_{i1} &:= +1 \qquad   \text{if } \  1 \leq i \leq t_{\theta}, \quad \theta \in \O_\d; \\ 
m^\eta_{i1}  &:= \begin{cases}
                   +1  & \text{ if } \,  1 \leq i \leq p_{\eta} \\
                   -1  & \text{ if } \,  p_\eta < i \leq t_\eta  
                 \end{cases} , \, \eta \in \E_\d\,;
\end{align*}                         
and for $j >1$ we define $(m^d_{ij})$ as in \eqref{def-sign-alternate} and \eqref{def-sign-alternate-1}.
Then the matrices $(m^d_{ij})$ clearly verify \eqref{yd-def2}.
Set $ \sgn_{\OC_X} \,:=\, ((m^{d_1}_{ij}),\, \cdots,\, (m^{d_s}_{ij}))$. 
It now follows from the above definitions of $m^\theta_{i1}, \, \theta \,\in\, \O_\d $ that
$\sgn_{\OC_X} \,\in\, \SC^{\rm odd}_{\d}(n)$.
Thus we have the map 
$$
\Psi_{{\rm SO}^* (2n)} \,\colon \,\NC({\rm SO}^* (2n))
\,\longrightarrow\, \YC^{\rm odd}(n)\, ,\ \
\OC_X \,\longmapsto\, \big(\Psi'_{{\rm SO}^* (2n)} (\OC_X),\, \sgn_{\OC_X}  \big)\, ;
$$
where $ \YC^{\rm odd}(n)$ is as in \eqref{yd-odd-Y-pq}.
The following theorem is standard.

\begin{theorem}[{\rm  \cite[Theorem 9.3.4]{CoM}}]\label{so*-parametrization}
The above map $\Psi_{{\rm SO}^* (2n)}$ is a bijection. 
\end{theorem}

Let $0\,\not=\, X \,\in\, \NC _{\s\o^* (2n)}$ and $\{X,\,H,\,Y\}$ be a $\s\l_2(\R)$-triple in
$\s\o^*(2n)$.
Let $\Psi_{{\rm SO}^* (2n)} (\OC_X) \,=\, \big( \d ,\, \sgn_{\OC_X} \big)$. Then
$\Psi'_{{\rm SO}^* (2n)} (\OC_X) \,=\, \d$. 
Recall that $\sgn_{\OC_X}$ determines the signature of $(\cdot,\, \cdot)_\eta$ on $L(\eta -1)$ for
all $\eta \,\in\, \E_\d$; 
let $(p_\eta, \,q_\eta)$ be the signature of $(\cdot,\, \cdot)_\eta$ on $L(\eta-1)$.
Let $( v^{d}_1,\, \cdots,\, v^d_{t_d} )$ be an ordered $\H$-basis of $L(d-1)$ as in
Proposition \ref{unitary-J-basis}. 
It now follows from Proposition \ref{unitary-J-basis}(3)(a) that $( v^d_1, \,\cdots,\, v^d_{t_d})$
is an orthogonal basis of $L(d-1)$ for the form $(\cdot,\, \cdot)_d$ for all $d \,\in\, \N_\d$.
We also assume that the vectors in the ordered basis $( v^d_1, \,\cdots,\, v^d_{t_d})$ satisfy
the properties in Remark \ref{unitary-J-basis-rmk}(3).
Since $(\cdot,\, \cdot)_\theta$ is skew-Hermitian for all $\theta \,\in \,\O_\d$, using Lemma
\ref{H-conjugate}, we may assume that $( v^\theta_1,\, \cdots,\, v^\theta_{t_\theta} )$ is a
standard orthogonal basis for all $\theta \,\in\, \O_\d$. Thus 
\begin{equation}\label{orthonormal-basis-theta-so*}
( v^{\theta}_j, v^\theta_j)_{\theta}  = \jb \quad \text{ for all } \, 1 \leq j \leq t_\theta, \, \theta \in \O_\d.
\end{equation}
In view of the signature of $(\cdot,\, \cdot)_\eta$, $\eta \,\in\, \E_\d$, we may assume that
\begin{equation}\label{orthonormal-basis-eta-so*}
(v^{\eta}_j, v^\eta_j)_{\eta}  =
 \begin{cases}
  +1  &  \text{ if }   1 \leq j \leq p_{\eta}   \\
  -1  &  \text{ if }  p_{\eta} < j \leq t_{\eta}. 
 \end{cases}
\end{equation}
For $\eta \,\in\, \E_\d,~ 1 \leq r \leq p_\eta$, define 
\begin{equation}\label{orthogonal-basis-p-eta-so*}
w^\eta_{rl} :=
\begin{cases}
   \big(X^l v^{\eta}_r + X^{\eta-1-l} v^{\eta}_r \jb \big)/{\sqrt{2}} & \text{ if } l \text{ is even, } 0 \leq l \leq \eta/2-1 \vspace*{.1cm}\\   
   \big(X^l v^{\eta}_r + X^{\eta-1-l} v^{\eta}_r\jb\big){\ib}/{\sqrt{2}} & \text{ if } l \text{ is odd, } 0 \leq l \leq \eta/2-1 \vspace*{0.1cm}\\  
     \big(X^{\eta-1-l} v^{\eta}_r - X^l v^{\eta}_r \jb \big) {\ib}/{\sqrt{2}}  & \text{ if $l$ is odd, } \eta/2 \leq l \leq \eta -1 \vspace*{.1cm}\\  
  \big( X^{\eta-1-l} v^{\eta}_r - X^l v^{\eta}_r \jb \big)/{\sqrt{2}}  & \text{ if $l$ is even, } \eta/2 \leq l \leq \eta -1.    
\end{cases} 
\end{equation}
Similarly for $\eta \in \E_\d,~ p_\eta < r \leq t_\eta$, define 
\begin{equation}\label{orthogonal-basis-q-eta-so*}
w^\eta_{rl} :=
\begin{cases}
   \big(X^l v^{\eta}_r + X^{\eta-1-l} v^{\eta}_r \jb \big) {\ib}/{\sqrt{2}} & \text{ if } l \text{ is even, } 0 \leq l \leq \eta/2-1 \vspace*{.1cm}\\   
   \big(X^l v^{\eta}_r + X^{\eta-1-l} v^{\eta}_r\jb\big)/{\sqrt{2}} & \text{ if } l \text{ is odd, } 0 \leq l \leq \eta/2-1 \vspace*{0.1cm}\\   
   \big(X^{\eta-1-l} v^{\eta}_r - X^l v^{\eta}_r  \jb\big) /{\sqrt{2}}  & \text{ if $l$ is odd, } \eta/2 \leq l \leq \eta -1 \vspace*{.1cm}\\
   \big(X^{\eta-1-l} v^{\eta}_r - X^l v^{\eta}_r  \jb \big) {\ib}/{\sqrt{2}}  & \text{ if $l$ is even, } \eta/2 \leq l \leq \eta -1.    
\end{cases} 
\end{equation}

Using \eqref{orthonormal-basis-eta-so*} we observe that for all $\eta \,\in \,\E_\d$,
$$\{{w}^{\eta}_{rl} \,\mid\, 0 \,\le\, l \,\le\, \eta-1, \ 1 \,\le\, r \,\le\, t_\eta \}$$ is an
orthogonal basis of $M(\eta -1)$ with respect to $\<>$, where $\langle {w}^{\eta}_{rl},\,
 {w}^{\eta}_{rl} \rangle \,=\, \jb$ for $0 \,\le\, l \,\le \,\eta-1$, $1 \,\le\, r \,\le\, t_\eta$.
For $\eta \,\in\, \E_\d,~ 0 \,\leq\, l \,\leq\, \eta/2 -1$, set 
\begin{equation}\label{defn-W-eta-so*}
W^l(\eta) := \text{Span}_\H \{w^\eta_{r\,l}, w^\eta_{r\,\eta-1-l} \mid 1 \leq r \leq t_\eta\}. 
\end{equation}
Moreover, we define a standard orthogonal basis $\DC^l(\eta)$ of $W^l(\eta)$ with respect to $\<>$ as follows:
\begin{equation}\label{standard-basis-even-so*}
 \DC^l(\eta ):=\begin{cases} 
            \big({w}^{\eta}_{_{1 \, l}}, \cdots,  {w}^{\eta}_{_{p_{\eta}\, l}}\big) \vee 
            \big({w}^{\eta}_{_{1\, (\eta-1-l)}}, \cdots , {w}^{\eta}_{_{p_{\eta} \, (\eta-1-l)}}  \big)\\
            \bigvee  \big(  {w}^{\eta}_{_{(p_{\eta} +1)\, (\eta-1-l)}},  \cdots , {w}^{\eta}_{_{t_{\eta} \, (\eta-1-l)}} \big) \vee  \big({w}^{\eta}_{_{(p_{\eta} +1) \, l}},  \cdots , {w}^{\eta}_{_{t_{\eta} \, l}}  \big)
            & \text{ if } l \text{ is  even} \vspace{.25cm}
            \\
         \big({w}^{\eta}_{_{1\, (\eta-1-l)}}, \cdots , {w}^{\eta}_{_{p_{\eta} \, (\eta-1-l)}}  \big) \vee 
         \big({w}^{\eta}_{_{1 \, l}}, \cdots,  {w}^{\eta}_{_{p_{\eta}\, l}}\big)\\
         \bigvee  \big({w}^{\eta}_{_{(p_{\eta} +1) \, l}},  \cdots , {w}^{\eta}_{_{t_{\eta} \, l}}  \big)  \vee \big({w}^{\eta}_{_{(p_{\eta} +1) \, (\eta-1-l)}},  \cdots , {w}^{\eta}_{_{t_{\eta} \, (\eta-1-l)}}  \big) 
          &  \text{ if } l \text{ is  odd}.
      \end{cases} 
\end{equation}
Now fixing $\theta \in \O^1_\d$, for all $1 \leq r \leq t_\theta$, define 
\begin{equation*}
{w}^{\theta}_{rl}:=
\begin{cases}
   \big(X^l v^{\theta}_r + X^{\theta-1-l} v^{\theta}_r \big)/{\sqrt{2}} & \text{ if } l \text{ is even, } 0 \leq l < (\theta-1)/2 \vspace*{.1cm}\\   
   \big(X^l v^{\theta}_r + X^{\theta-1-l} v^{\theta}_r \big){\ib}/{\sqrt{2}} & \text{ if } l \text{ is odd, } 0 \leq l < (\theta-1)/2 \vspace*{0.05cm}\\   
     ~ X^l v^{\theta}_r    &  \text{ if }  l = (\theta-1)/2 \vspace*{.05cm}\\        
   \big(X^{\theta-1-l} v^{\theta}_r - X^l v^{\theta}_r \big) /{\sqrt{2}}  & \text{ if $l$ is odd, } (\theta + 1)/2 \leq l \leq \theta -1 \vspace*{.1cm}\\
   \big(X^{\theta-1-l} v^{\theta}_r - X^l v^{\theta}_r \big){\ib}/{\sqrt{2}}  & \text{ if $l$ is even, } (\theta+1)/2 \leq l \leq \theta -1.    
\end{cases} 
\end{equation*}
For all $\zeta \in \O^3_\d $ and $ 1 \,\leq\, r \,\leq\, t_\zeta$, define 
\begin{equation*}
{w}^{\zeta}_{rl}:=
\begin{cases}
   \big(X^l v^{\zeta}_r + X^{\zeta-1-l} v^{\zeta}_r \big)/{\sqrt{2}} & \text{ if } l \text{ is even, } 0 \leq l < (\zeta-1)/2 \vspace*{.1cm}\\   
   \big(X^l v^{\zeta}_r + X^{\zeta-1-l} v^{\zeta}_r \big){\ib}/{\sqrt{2}} & \text{ if } l \text{ is odd, } 0 \leq l < (\zeta-1)/2 \vspace*{0.05cm}\\   
       ~ X^l v^{\zeta}_r \ib  &  \text{ if }  l = (\zeta-1)/2 \vspace*{.05cm}\\        
   \big(X^{\zeta-1-l} v^{\zeta}_r - X^l v^{\zeta}_r \big) /{\sqrt{2}}  & \text{ if $l$ is odd, } (\zeta + 1)/2 \leq l \leq \zeta -1 \vspace*{.1cm}\\
   \big(X^{\zeta-1-l} v^{\zeta}_r - X^l v^{\zeta}_r \big) {\ib}/{\sqrt{2}}  & \text{ if $l$ is even, } (\zeta+1)/2 \leq l \leq \zeta -1.    
\end{cases} 
\end{equation*}
Using \eqref{orthonormal-basis-theta-so*} we observe that for all $\theta \in \O_\d$,
$$\{{w}^{\theta}_{rl} \,\mid\, 0 \,\le\, l \,\le\, \theta-1, \ 1 \,\le\, r \,\le\, t_\theta \}$$
is an orthogonal basis of $M(\theta -1)$ with respect to $\<>$, where $\langle {w}^{\theta}_{rl},\,
 {w}^{\theta}_{rl} \rangle \,=\, \jb$ for $0 \,\le\, l \,\le\, \theta-1$, $1\,\le\, r\,\le\, t_\theta$.  
For each $\theta \,\in\, \O_\d$, $0\,\le\, l \,\le\, \theta-1$, set 
\begin{align}\label{defn-V-theta-so*}
V^l(\theta):= \text{Span}_\H\{ w^{\theta}_{rl} \mid 1\le r \le t_\theta \}. 
\end{align}
The standard orthogonal ordered basis $( {w}^\theta_{1l},\, \cdots ,\,  {w}^\theta_{t_{\theta}l} )$
of $ V^l (\theta) $  with respect to $\<>$ is denoted by $\CC^l (\theta)$. 
 
Let $W$ be a right $\H$-vector space and $\<>'$ a non-degenerate skew-Hermitian form on $W$.
Let $\dim_\H W \,=\, m$, and let $\BC'\,:=\, (v_1,\, \cdots, \,v_m)$ be a standard orthogonal basis
of $W$ such that $\langle v_r,\, v_r \rangle' \,=\, \jb$ for all $1\,\leq\, r \,\leq\, m$.
Define $${\rm J}_{\BC'} \,\colon\, W \,\longrightarrow\, W\, ,\ \
\sum_r v_r z_r\,\longmapsto\, \sum_r v_r \jb z_r$$ for all column vectors
$(z_1, \,\cdots,\, z_m)^t \,\in\, \H^m$. Set
\begin{align*}
K_{\BC'} :=& \big\{ g \in {\rm SO}^* (W, \<>') \bigm|   g {\rm J}_{\BC'} = {\rm J}_{\BC'} g  \big\}.
\end{align*}
The following lemma is standard; its proof is omitted.

\begin{lemma} \label{max-cpt-so*-2n}
Let $W$, $\<>'$ and $\BC'$  be as above. Then the following hold:
\begin{enumerate}
 \item $K_{\BC'}$ is a maximal compact subgroup of ${\rm SO}^*(W, \<>')$.
 
 \item $K_{\BC'} = \big\{g \in {\rm SL}(W) \bigm| [g]_{\BC'} = A+\jb B ~ \text{ with }~ A,B \in {\rm M}_m(\R), A+\sqrt{-1}B \in {\rm U}(m) \big\} $.
\end{enumerate}
\end{lemma} 
  
 Recall that $\big\{x \in {\rm End}_\H W \mid x {\rm J}_{\BC'} = {\rm J}_{\BC'} x \big\} = \big\{ x \in {\rm End}_\H W \bigm| [x]_{\BC'} \in {\rm M}_m(\R) + \jb {\rm M}_m(\R)  \big\} $. 
 We now consider the $\R$-algebra isomorphism 
 \begin{equation}\label{R-algebra-isomorphism-so*}
  \Lambda'_{\BC'}  \,\colon\, \big\{x \,\in\, {\rm End}_\H W \,\mid\, x {\rm J}_{\BC'}
\,= \,{\rm J}_{\BC'} x \big\} \,\longrightarrow \,{\rm M}_m(\C)\, , \ \
x\,\longmapsto\, A + \sqrt{-1} B\, ,
\end{equation}
where $A,\,B \,\in\, {\rm M}_m(\R)$ are the unique elements such that $[x]_\BC \,=\, A + \jb B$.
In view of the above lemma it is clear that $\Lambda'_{\BC'}(K_{\BC'})\,=\,{\rm U}(m)$, and
hence $\Lambda'_{\BC'} \,\colon\, K_{\BC'} \,\longrightarrow\, {\rm U}(m)$ is an isomorphism
of Lie groups.

In the next lemma we specify a maximal compact subgroup of $\ZC_{{\rm SO}^* (2n) } (X,H,Y)$ 
which will be used in Proposition \ref{max-cpt-so*-wrt-basis}.
Recall that $\overline{Z}\,:=\, (\sigma_c(z_{rl})) \,\in\, {\rm 
M}_m(\H)$; see Section \ref{sec-epsilon-sigma-forms}.

\begin{lemma}\label{max-cpt-so*-reductive-part}
 Let $K$ be the subgroup of $\ZC_{{\rm SO}^* (2n)}(X,H,Y)$ consisting of elements $g$ in \\
$\ZC_{{\rm SO}^*(2n)}(X,H,Y)$ satisfying the following conditions:
\begin{enumerate}
 \item $g \big( V^l (\theta)\big) \,\subset\, V^l (\theta)$
for all $\theta \,\in\, \O_\d$ and  $0 \leq l \leq \theta -1$.

\item
 For all $ \theta \in \O^1_\d$, there exist  $A_\theta,B_\theta \in {\rm M}_{t_\theta}(\R)$ with $A_\theta + \sqrt{-1}B_\theta \in {\rm U}(t_\theta) $ such that 
  $$
  \big[g|_{V^l(\theta)}\big]_{\CC^l(\theta)} =
  \begin{cases}
   A_\theta + \jb B_\theta  & \text{ if $l$ is even, } 0 \leq l < (\theta -1)/2 \vspace*{0.05cm}\\
   A_\theta - \jb B_\theta  & \text{ if } l \text{ is odd, } 0 \leq l < (\theta-1)/2 \vspace*{0.05cm}\\  
   A_\theta + \jb B_\theta  & \text{ if }  l = (\theta-1)/2 \vspace*{.05cm}\\ 
   A_\theta + \jb B_\theta  &  \text{ if $l$ is odd, } (\theta + 1)/2 \leq l \leq \theta -1 \vspace*{.05cm}\\
   A_\theta - \jb B_\theta  & \text{ if $l$ is even, } (\theta+1)/2 \leq l \leq \theta -1.    
  \end{cases}
 $$ 

 \item For all $ \zeta \in \O^3_\d$, there exist  $A_\zeta,B_\zeta \in {\rm M}_{t_\zeta}(\R)$ with $A_\zeta + \sqrt{-1} B_\zeta \in {\rm U}(t_\zeta) $ such that 
  $$
  \big[g|_{V^l(\zeta)}\big]_{\CC^l(\zeta)} =
  \begin{cases}
   A_\zeta + \jb B_\zeta  & \text{ if $l$ is even, } 0 \leq l < (\zeta -1)/2 \vspace*{0.05cm}\\
   A_\zeta - \jb B_\zeta  & \text{ if } l \text{ is odd, } 0 \leq l < (\zeta-1)/2 \vspace*{0.05cm}\\  
   A_\zeta - \jb B_\zeta  & \text{ if }  l = (\zeta-1)/2 \vspace*{.05cm}\\ 
   A_\zeta + \jb B_\zeta  & \text{ if $l$ is odd, } (\zeta + 1)/2 \leq l \leq \zeta -1 \vspace*{.05cm}\\
   A_\zeta - \jb B_\zeta  & \text{ if $l$ is even, } (\zeta+1)/2 \leq l \leq \zeta -1.    
  \end{cases}
 $$  

 \item 
$g( W^l (\eta)) \,\subset\, W^l (\eta)$
for all $\eta \in \E_\d$ and $0 \leq l \leq \eta/2 -1$.

 \item
 For all $\eta \in \E_\d$, there exist $A_{p_\eta},B_{p_\eta}, C_{p_\eta}, D_{p_\eta} \in {\rm M}_{p_\eta}(\R) $, $A'_{q_\eta},B'_{q_\eta}, C'_{q_\eta}, D'_{q_\eta} \in {\rm M}_{q_\eta}(\R)$  with $A_{p_\eta}+ \jb B_{p_\eta} + \ib (C_{p_\eta}+ \jb D_{p_\eta}) \in {\rm Sp}(p_\eta)$ and  $A'_{q_\eta}+ \jb B'_{q_\eta} + \ib( C'_{q_\eta}+ \jb D'_{q_\eta}) \in {\rm Sp}(q_\eta)$ such that  
$$ 
\big[g|_{W^l(\eta)} \big]_{\DC^l(\eta)}= \begin{pmatrix}
                                         A_{p_\eta}+ \jb B_{p_\eta} & & -C_{p_\eta}+\jb D_{p_\eta} \vspace{.4cm} \\
                                         C_{p_\eta}+ \jb D_{p_\eta} & &  A_{p_\eta}-\jb B_{p_\eta} \\
                                           & &  & A'_{q_\eta} +\jb B'_{q_\eta} && -C'_{q_\eta}+ \jb D'_{q_\eta} \vspace{.4cm} \\
                                           &  & & C'_{q_\eta}+ \jb D'_{q_\eta} &&  A'_{q_\eta}- \jb B'_{q_\eta}
                                           \end{pmatrix}.
 $$
 \end{enumerate}
Then $K$ is a maximal compact subgroup of $\ZC_{ {\rm SO}^* (2n)} (X,H,Y)$.
\end{lemma}

\begin{proof}
Let $\BC^l(d)\,=\, (X^lv^d_1,\, \cdots,\, X^lv^d_{t_d})$ be the ordered basis of $X^lL(d-1)$ for
$0\,\leq \,l \,\leq\, d -1$, $d \in \N_\d$, as in \eqref{old-ordered-basis-part}.
Let $K'$ be the subgroup consisting of all $g \,\in\, \ZC_{{\rm SO}^* (2n)}(X,H,Y)$
satisfying the following properties:
\begin{align}
   & \text{For }\theta \in \O_\d, 0\leq l \leq \theta-1 ,~ g \big(V^l (\theta)\big) \subset \ V^l (\theta),\label{orthogonal-odd-subspace-so*} \\
   & \text{For all } \theta \in \O^1_\d,    
   \big[g|_{ V^l(\theta)}\big]_{{\CC}^l(\theta)} = 
   \begin{cases}
   \big[g |_{ V^0 (\theta)}\big]_{{\CC}^0 (\theta)}   & \text{ if $l$ is even, } 0 \leq l < (\theta -1)/2 \vspace*{0.05cm}\\
   \overline{\big[g |_{ V^0 (\theta)}\big]}_{{\CC}^0 (\theta)}  & \text{ if } l \text{ is odd, } 0 \leq l < (\theta-1)/2 \vspace*{0.1cm}\\  
   \big[g |_{ V^0 (\theta)}\big]_{{\CC}^0 (\theta)}  & \text{ if }  l = (\theta-1)/2 \vspace*{0.1cm}\\ 
   \big[g |_{ V^0 (\theta)}\big]_{{\CC}^0 (\theta)}  &  \text{ if $l$ is odd, } (\theta + 1)/2 \leq l \leq \theta -1 \vspace*{0.1cm}\\
   \overline{\big[g |_{ V^0 (\theta)}\big]}_{{\CC}^0 (\theta)}  & \text{ if $l$ is even, } (\theta+1)/2 \leq l \leq \theta -1,    
   \end{cases}\label{matrix-wrt-orthogonal-odd-1-subspace-so*}\\
  & \text{For all } \zeta \in \O^3_\d,    
  \big[g|_{ V^l(\zeta)}\big]_{{\CC}^l(\zeta)} = 
  \begin{cases}
  \big[g |_{ V^0 (\zeta)}\big]_{{\CC}^0 (\zeta)}   & \text{ if $l$ is even, } 0 \leq l < (\zeta -1)/2 \vspace*{00.1cm}\\
   \overline{ \big[g |_{ V^0 (\zeta)}\big]}_{{\CC}^0 (\zeta)}  & \text{ if } l \text{ is odd, } 0 \leq l < (\zeta-1)/2 \vspace*{00.1cm}\\  
  \overline{\big[g |_{ V^0 (\zeta)}\big]}_{{\CC}^0 (\zeta)}  & \text{ if }  l = (\zeta-1)/2 \vspace*{0.1cm}\\ 
  \big[g |_{ V^0 (\zeta)}\big]_{{\CC}^0 (\zeta)}  &  \text{ if $l$ is odd, } (\zeta + 1)/2 \leq l \leq \zeta -1 \vspace*{.1cm}\\
   \overline{\big[g |_{ V^0 (\zeta)}\big]}_{{\CC}^0 (\zeta)}  & \text{ if $l$ is even, } (\zeta+1)/2 \leq l \leq \zeta -1,    
  \end{cases}  
  \label{matrix-wrt-orthogonal-odd-3-subspace-so*} \vspace{.3cm}\\
   &  g|_{ V^0(\theta)} \text{ commutes with } {\rm J}_{\CC^0(\theta)}, \label{commute-j-odd} \\ 
 & g\big(X^l L(\eta-1)\big) \subset X^l L(\eta-1),\,
 \big[g |_{ X^l L(\eta-1)} \big]_{{\BC}^l (\eta)} = \big[g |_{ L(\eta-1)}\big]_{{\BC}^0 (\eta)} 
 \text{ if } \eta \in \E_\d,   0 \leq l\leq \eta-1;
 \label{orthogonal-even-subspace-so*} \\
 & g( W^l (\eta)) \subset W^l (\eta) \vspace{.3cm}  \text{ for }  \eta \in \E_\d, 0 \leq l \leq \eta/2 -1, \text{ and }   
   g|_{ W^0(\eta)}  \text{ commutes with }  {\rm J}_{\DC^0(\eta)}. \label{commute-j-even} 
\end{align} 
Using  Lemma \ref{max-cpt-so*-2n}(1) it is evident that $K'$ is a maximal compact subgroup of
$\ZC_{{\rm SO}^* (2n)}(X,H,Y)$.  Hence to prove the lemma it suffices to show that $K \,=\, K'$.
Let $g \,\in\, {{\rm SO}^*(2n)}$. From Lemma \ref{max-cpt-so*-2n}(2) it is straightforward that
$g$ satisfies (1), (2), (3) of Lemma \ref{max-cpt-so*-reductive-part} if and only if $g$ satisfies
\eqref{orthogonal-odd-subspace-so*}, \eqref{matrix-wrt-orthogonal-odd-1-subspace-so*}, \eqref{matrix-wrt-orthogonal-odd-3-subspace-so*}
 and \eqref{commute-j-odd}.
 Now suppose that $g \,\in\, {{\rm SO}^*(2n)}$ and $g$ satisfying (4), (5) of
Lemma \ref{max-cpt-so*-reductive-part}. It is clear that \eqref{commute-j-even} holds. We observe
that $$\big[g|_{L(\eta-1)} \big]_{\BC^0(\eta)} = 
 \begin{pmatrix}
   A_{p_\eta}+ \jb B_{p_\eta} + \ib (C_{p_\eta}+ \jb D_{p_\eta}) & 0 \vspace{.4cm}\\
    0 & A'_{q_\eta}+ \jb B'_{q_\eta} + \ib( C'_{q_\eta}+ \jb D'_{q_\eta})                                                                                                                                                                                                                   \end{pmatrix}.
$$
This proves that \eqref{orthogonal-even-subspace-so*} holds.

Now we assume that  $g$ satisfies \eqref{orthogonal-even-subspace-so*} and \eqref{commute-j-even}.
Let $A:= \big[g|_{L(\eta-1)} \big]_{\BC^0(\eta)}$. Then $A=\big[g |_{ X^l L(\eta-1)}\big]_{{\BC}^l (\eta)}$ for $1\leq l \leq \eta -1$.
We observe that 
$$\big[ {\rm J}_{{\DC}^0 (\eta)}\big]_{\BC^0 (\eta) \vee \BC^{\eta -1} (\eta)} = \begin{pmatrix}
                                        & {\rm I}_{p_\eta,q_\eta} \\
                                       -{\rm I}_{p_\eta, q_\eta}  
                                      \end{pmatrix} 
\text{ and }~
\big[ g |_{W^0(\eta)}\big]_{\BC^0 (\eta) \vee \BC^{\eta -1} (\eta)}  = \begin{pmatrix}
                                       A &  \\
                                       & A
                                      \end{pmatrix}.
$$
From \eqref{commute-j-even} it follows that the above two matrices commute, which in turn
implies that $A$ commutes with $ \begin{pmatrix}
         {\rm I}_{p_\eta} &  \\
         & -{\rm I}_{q_\eta}
         \end{pmatrix}$.
Thus $A$ is of the form $A =  \begin{pmatrix}
                                      E_{p_\eta} & 0 \\
                                      0 &  F_{q_\eta}
                                      \end{pmatrix} $
for some matrices $ E_{p_\eta} \in {\rm GL}_{p_\eta}(\H)$ and $F_{q_\eta} \in {\rm GL}_{q_\eta}(\H)$. Write $E_{p_\eta}= A_{p_\eta}+ \jb B_{p_\eta} + \ib (C_{p_\eta}+ \jb D_{p_\eta})  $ and $ F_{q_\eta}= A'_{q_\eta}+ \jb B'_{q_\eta} + \ib( C'_{q_\eta}+ \jb D'_{q_\eta})$  where  $A_{p_\eta},B_{p_\eta}, C_{p_\eta}, D_{p_\eta} \in {\rm M}_{p_\eta}(\R) $, $A'_{q_\eta},B'_{q_\eta}, C'_{q_\eta}, D'_{q_\eta} \in {\rm M}_{q_\eta}(\R)$. We now observe that 
$$
\big[g|_{W^l(\eta)} \big]_{\DC^l(\eta)}=\begin{pmatrix}
					  A_{p_\eta}+ \jb B_{p_\eta} & &-C_{p_\eta}+\jb D_{p_\eta} \vspace{.4cm} \\
                                          C_{p_\eta}+ \jb D_{p_\eta} & & A_{p_\eta}-\jb B_{p_\eta} \vspace{.02cm} \\
                                           & &  & A'_{q_\eta} +\jb B'_{q_\eta} & &-C'_{q_\eta}+ \jb D'_{q_\eta}  \vspace{.4cm}\\
                                           & &  & C'_{q_\eta}+ \jb D'_{q_\eta} & & A'_{q_\eta}- \jb B'_{q_\eta}                                             
                                          \end{pmatrix}
$$
where $\DC^l(\eta)$ is defined as in \eqref{standard-basis-even-so*}.

Recall that $M(\eta -1) = \bigoplus_{l=0}^{\eta/2} W^l(\eta)$ is an orthogonal decomposition of
$M(\eta-1)$ with respect to $\<>$; see \eqref{defn-W-eta-so*} and the paragraph preceding it. As
$\DC^0(\eta)$ is a standard orthogonal basis of $W^0(\eta)$, and 
$g|_{W^0(\eta)}$ commutes with ${\rm J}_{\DC^0(\eta)}$, it follows that $ \Lambda'_{\DC^0(\eta)} (g|_{W^0(\eta)}) \in {\rm U}(2t_\eta)$. In other words,
$$
    \begin{pmatrix}
					  A_{p_\eta}+ \sqrt{-1} B_{p_\eta}  & ~ -C_{p_\eta}+\sqrt{-1} D_{p_\eta} \vspace{.4cm}\\
                                          C_{p_\eta}+ \sqrt{-1} D_{p_\eta}  & ~ A_{p_\eta}-\sqrt{-1} B_{p_\eta} \vspace{.00cm} \\
                                           &   & A'_{q_\eta} +\sqrt{-1} B'_{q_\eta}  &~   -C'_{q_\eta}+ \sqrt{-1} D'_{q_\eta}  \vspace{.4cm}\\
                                           &   & C'_{q_\eta}+ \sqrt{-1} D'_{q_\eta}  &~   A'_{q_\eta}- \sqrt{-1} B'_{q_\eta}                                             
                                          \end{pmatrix} \in {\rm U}(2t_\eta).
 $$
 This implies that $E_{p_\eta} \in {\rm Sp}(p_\eta)$ and $F_{q_\eta} \in {\rm Sp}(q_\eta)$ and (5) of lemma \eqref{max-cpt-so*-reductive-part} holds. This completes the proof.
\end{proof}

We now introduce some notation which will be required to state Proposition \ref{max-cpt-so*-wrt-basis}.
For $\eta \in \E_\d$, set 
$$
\DC(\eta) := \DC^0 (\eta) \vee \cdots \vee  \DC^{\eta/2-1} (\eta)\, ,
$$
and for $\theta \in \O_\d$, set 
$$
\CC(\theta) := \CC^0 (\theta) \vee \cdots \vee  \CC^{\theta-1} (\theta).
$$
Let $\alpha := \# \E_\d $,  $\beta := \# \O^1_\d$ and $ \gamma := \# \O^3_\d $.
We enumerate
$\E_\d =\{ \eta_i \mid 1 \leq i \leq \alpha \}$ such that $\eta_i < \eta_{i+1}$, 
$\O^1_\d =\{ \theta_j \mid 1 \leq j \leq \beta \}$ such that $\theta_j < \theta_{j+1}$ and similarly  $\O^3_\d =\{ \zeta_j \mid 1 \leq j \leq \gamma \}$ such that $\zeta_j < \zeta_{j+1}$. 
Now define
$$
\EC := \DC(\eta_1) \vee \cdots \vee \DC(\eta_{\alpha}) ; \ 
\OC^1 := \CC (\theta_1) \vee \cdots \vee \CC (\theta_{\beta});\text{ and } ~ \OC^3:= \CC(\zeta_1) \vee \cdots \vee \CC (\zeta_{\gamma}).
$$
Also define
\begin{equation}\label{orthogonal-basis-so*-final}
\HC=  \EC \vee \OC^1 \vee \OC^3. 
\end{equation}

For an integer $m$ define the $\R$-algebra embedding
$$\wp_{m, \H} \,\colon\, {\rm M}_m(\H) \,\longrightarrow\, {\rm M}_{2m}(\C)\, ,\ \
R\, \longmapsto\, 
    \begin{pmatrix}
       S & -\overline{T}\\
       T & \overline{S}
     \end{pmatrix}
$$
where $S,T \in {\rm M}_m(\C)$ are the unique elements such that $R= S+ \jb T$.
The following map is an $\R$-algebra embedding of
$ \prod_{i=1}^{\alpha} \big( {\rm M}_{p_{\eta_i}}(\H) \times {\rm M}_{q_{\eta_i}}(\H) \big)
      \times \prod_{j=1}^\beta {\rm M}_{t_{\theta_j}} (\C) \times \prod_{k=1}^\gamma 
{\rm M}_{t_{\zeta_k}} (\C)$ into ${\rm M}_n(\C)$. Define
      $$
\Db \colon \prod_{i=1}^{\alpha} \big( {\rm M}_{p_{\eta_i}}(\H) \times {\rm M}_{q_{\eta_i}}(\H) \big)
  \times \prod_{j=1}^\beta {\rm M}_{t_{\theta_j}} (\C) \times \prod_{k=1}^\gamma {\rm M}_{t_{\zeta_k}} (\C)  \longrightarrow {\rm M}_n(\C)  
$$
by
\begin{align*}
  \big(  C_{\eta_1}, D_{\eta_1}, \cdots , C_{\eta_\alpha}, D_{\eta_\alpha};     
  ~& A_{\theta_1}, \cdots,  A_{\theta_\beta} ;~ B_{\zeta_1},\cdots ,  B_{\zeta_\gamma} \big) \\
\longmapsto 
  & \bigoplus_{i=1}^\alpha \Big( \wp_{p_{\eta_i},\H}(C_{\eta_i}) \oplus \wp_{{q_{\eta_i}},\H}(D_{\eta_i}) \Big)_\blacktriangle ^{\frac{\eta}{2}}\\ 
        \oplus &  \bigoplus_{j=1}^\beta \Big( \big( A_{\theta_j} \oplus \overline{A}_{\theta_j}\big)_\blacktriangle ^{\frac{\theta_j-1}{4}} \oplus A_{\theta_j} \oplus \big( A_{\theta_j} \oplus \overline{A}_{\theta_j}\big)_\blacktriangle ^{\frac{\theta_j-1}{4}} \Big) \\ 
  \oplus & \bigoplus_{k=1}^\gamma \Big( \big( B_{\zeta_k}\oplus \overline{B}_{\zeta_k} \big) _\blacktriangle ^{\frac{\zeta_k + 1}{4}}  \oplus \big( B_{\zeta_k}\oplus \overline{B}_{\zeta_k} \big)_\blacktriangle^{\frac{\zeta_k -3}{4}}  \oplus \overline{B}_{\zeta_k} \Big) .
\end{align*}

It is clear that $\HC$ in \eqref{orthogonal-basis-so*-final} is a standard orthogonal basis of
$V$ with respect to $\<>$.
Let $$\Lambda'_\HC \,\colon\, \{ x \,\in\, {\rm End}_\H \H^{n} \,\mid\,  x{\rm J}_{\HC}
\,=\, {\rm J}_{\HC} x \} \,\longrightarrow\, {\rm M}_{n} (\C)$$ be the isomorphism of $\R$-algebras
induced by the above ordered basis $\HC$. Recall that $\Lambda'_\HC \,\colon\, K_\HC $ \, $
\longrightarrow\, {\rm U}(n)$ is an 
isomorphism of Lie groups.

\begin{proposition}\label{max-cpt-so*-wrt-basis} 
Let $X \in \NC_{\s\o^*(2n)}$, $\Psi_{{\rm SO}^* (2n)} (\OC_X)$ $ = \big(\d, \sgn_{\OC_X} \big)$.
Let $\alpha := \# \E_\d $,  $\beta := \# \O^1_\d$ and $ \gamma := \# \O^3_\d $. Let $\{X,H,Y\} \subset \s\o^* (2n)$ be a $\s\l_2(\R)$-triple; let $(p_\eta, q_\eta)$ be the signature of the form $(\cdot, \cdot)_\eta$, for  $\eta \in \E_\d$, as defined in \eqref{new-form}. 
Let $K$ be the maximal compact subgroup of $\ZC_{{\rm SO}^* (2n)} (X, H, Y)$ as in
Lemma \ref{max-cpt-so*-reductive-part}.
Then $\Lambda'_\HC(K) \subset {\rm U}(n)$ is given by
$$
\Lambda'_\HC(K) = \bigg\{ \Db(g) \biggm| 
                        g \in \prod_{i=1}^\alpha \big(  {\rm Sp}(p_{\eta_i}) \times  {\rm Sp}(q_{\eta_i}) \big) \times \prod_{j=1}^\beta  {\rm U}(t_{\theta_j})  \times \prod_{k=1}^\gamma {\rm U}(t_{\zeta_k}) \bigg\}.
$$
\end{proposition}
 
\begin{proof}
This follows by writing the matrices of the elements of the maximal compact subgroup $K$
with respect to the basis $\HC$ in \eqref{orthogonal-basis-so*-final}.
\end{proof}

As we only consider simple Lie algebras, to ensure simplicity of $\s \o^*(2n)$, 
in view of \cite[Theorem 6.105, p. 421]{K} and the isomorphisms (vii), (xi) in \cite[Chapter X, \S6, pp.519-520]{He}, we will further need to assume that $n \geq 3$.

\begin{theorem}\label{so*}
Let $X \in \s \o^*(2n)$ be a nilpotent element when $n\geq 3$. Let $(\d, \sgn_{\OC_X}) \in
\YC^{\rm odd}(n)$ be the signed Young diagram of the orbit
$\OC_X$ (that is, $\Psi_{{\rm SO}^*(2n)} (\OC_X) = (\d, \sgn_{\OC_X})$ in the notation of
Theorem \ref{so*-parametrization}). Then
$$\dim_\R H^2(\OC_X,\R) =
\begin{cases}
 0   &  \text{ if }\, \# \O_\d = 0 \\
 \# \O_\d-1 &  \text{ if }\, \# \O_\d \geq 1 \, .
\end{cases}
$$
\end{theorem}

\begin{proof}
As the theorem is evident when $X \,=\, 0$, we assume that $X \,\neq\, 0$.

In the proof we will use the notation established above. Let $\{X, H, Y\} \subset \s \o^*(2n)$ 
be a $\s\l_2(\R)$-triple.
Let $K$ be the maximal compact subgroup of $\ZC_{{\rm SO}^* (2n)} (X, H, Y)$ as in
Lemma \ref{max-cpt-so*-reductive-part}.
Let $\HC$ be as in \eqref{orthogonal-basis-so*-final}, and let 
$K_{\HC}$ be the maximal compact subgroup of ${\rm SO}^*(2n)$ as in the Lemma 
\ref{max-cpt-so*-2n}(1). Then $K\subset K_\HC$. It follows either from Proposition 
\ref{max-cpt-so*-wrt-basis} or from Lemma \ref{reductive-part-comp}(4) that
$$
K \simeq \prod_{\eta \in \E_\d} \big({\rm Sp}(p_\eta)\times{\rm Sp}(q_\eta) \big) \times \prod_{\theta \in \O_\d} {\rm U}(t_\theta).
$$
In particular, $K$ is connected and $\dim_\R \z(\k)= \# \O_\d$. Let $\k_\HC$ be the Lie algebra of
$K_\HC$. We now appeal to Proposition \ref{max-cpt-so*-wrt-basis} to conclude that
$\z(\k)\,\subset\, [\k_\HC, \,\k_\HC]$ when $\# \O_\d =0$, and $\z(\k)\,\not\subset\, [\k_\HC,\, \k_\HC]$ when $\# \O_\d>0$.
As $\dim_\R \z(\k_\HC)=1$, in view of Theorem \ref{thm-nilpotent-orbit} we have that for all
$X \neq 0$,
$$\dim_\R H^2(\OC_X,\R) = \dim_\R \z(\k) \cap [\k_\HC, \k_\HC] =
\begin{cases}
 0   &  \text{ if }\, \# \O_\d = 0 \\
 \# \O_\d-1 &  \text{ if }\, \# \O_\d \geq 1 \, .
\end{cases}
$$
This completes the proof of the theorem.
\end{proof}

\subsection{Second cohomology groups of nilpotent orbits in 
\texorpdfstring{${\s\p}(n,\R)$}{Lg}}\label{sec-sp-n-R}

Let $n$ be a positive integer. The aim in this subsection is to compute the second cohomology groups of nilpotent 
orbits in the simple real Lie algebra ${\s\p}(n,\R)$ under the adjoint action of ${\rm Sp}(n,\R)$. Throughout this 
subsection $\<>$ denotes the symplectic form on $\R^{2n}$ defined by $\langle x, y \rangle := 
x^t{\rm J}_{n} y$, $x,\,y \,\in\, \R^{2n}$, where ${\rm J}_{n}$ is as in \eqref{defn-I-pq-J-n}.
We will follow notation as defined in \S \ref{sec-notation}.

Let $\Psi_{{\rm SL}_{2n} (\R)} \,: \,\NC ({\rm SL}_{2n} (\R)) \,\longrightarrow\, \PC(2n)$ be
the parametrization of nilpotent orbits in $\s\l_{2n}(\R)$; see Theorem \ref{sl-R-parametrization}.
As ${\rm Sp} (n,\R) \,\subset\, {\rm SL}_{2n} (\R)$ (consequently, $\NC_{{\s\p}(n,\R)} \,\subset
\,\NC_{\s\l_{2n}(\R)}$) we have the inclusion map
 $\Theta_{{\rm Sp} (n,\R)} \,:\, \NC ({\rm Sp} (n,\R)) \,\longrightarrow\,  \NC ( {\rm SL}_{2n} (\R) )$.
Let 
$$
\Psi'_{{\rm Sp} (n,\R)}\,:= \Psi_{{\rm SL}_{2n} (\R)} \circ \Theta_{{\rm Sp} (n,\R)}
\, \colon\, \NC ({\rm Sp} (n,\R))  \,\longrightarrow\, \PC (2n)
$$
be the composition. From
Remark \ref{unitary-J-basis-rmk}(1) it follows that
$\Psi'_{{\rm Sp} (n,\R)} ( \NC ({\rm Sp} (n,\R))) \,\subset\,  \PC_{-1} (2n)$.
Let $0\,\not=\, X \in {\s\p}(n,\R)$ be a nilpotent element and $\OC_X$ be the corresponding
nilpotent orbit in $\s\p(n,\R)$. Let $\{X,\, H, \,Y\} \,\subset\, {\s\p}(n,\R)$ be a
$\s\l_2(\R)$-triple. We now apply Proposition \ref{unitary-J-basis}, Remark \ref{unitary-J-basis-rmk} (1), and
follow the notation used therein. 
Let $V:= \R^{2n}$ be the right $\R$-vector space of column vectors. Let $\{d_1,\, \cdots, \,d_s \}$, with
$d_1 \,<\, \cdots \,<\, d_s$, be the ordered finite set of natural numbers that arise as
dimension of non-zero irreducible 
$\text{Span}_\R \{ X,H,Y\}$-submodules of $V$.
Recall that $M(d-1)$ is defined to be the isotypical component of $V$ containing all irreducible submodules of $V$ with highest weight $d-1$ and as in \eqref{definition-L-d-1}, we set $L(d-1)\,:= \,V_{Y,0} \cap M(d-1)$.
Let $t_{d_r} \,:=\, \dim_\R L(d_r-1)$ for $1\,\leq\, r \,\leq\, s$. Then 
$\d\,:=\, [d_1^{t_{d_1}},\, \cdots ,\,d_s^{t_{d_s}}]  \,\in\, \PC_{-1}(2n)$, and  moreover,
$\Psi'_{{\rm Sp} (n,\R)} (\OC_X) \,=\, {\d}$.

We next assign $\sgn_{\OC_X} \,\in\, \SC^{\rm odd}_{\d}(2n)$ to each $\OC_X \,\in\,
\NC({\rm Sp} (n,\R))$; see \eqref{S-d-pq-odd} for the definition of
$\SC^{\rm odd}_{\d}(2n)$.
For each $d\,\in\, \N_\d$ (see \eqref{Nd-Ed-Od} for the definition of $\N_\d$) we will define a $t_d \times  d$ matrix $(m^d_{ij})$ in
$ \Ab_{d}$ that depends only on the orbit $\OC_X$; see \eqref{A-d} for the definition of $\Ab_d$.
For this, recall that the form $(\cdot,\,\cdot)_{d} \colon L(d-1) \times L(d-1) \,\longrightarrow\, \R$
defined as in \eqref{new-form} is symmetric or symplectic
according as $d$ is even or odd. Denoting the  signature of
$(\cdot,\,\cdot)_{\eta}$ by $(p_{\eta},\, q_{\eta})$ when $\eta\,\in\, \E_\d$, we now define
\begin{align*}
m^\theta_{i1} &:= +1 \qquad   \text{if } \  1 \leq i \leq t_{\theta}, \quad \theta \in \O_\d;\\ 
m^\eta_{i1} &:= \begin{cases}
                   +1  & \text{ if } \,  1 \leq i \leq p_{\eta} \\
                   -1  & \text{ if } \, p_\eta < i \leq t_\eta  
                 \end{cases},~ \eta \in \E_\d\,;
              \end{align*}
and for $j >1$, define $(m^d_{ij})$ as in \eqref{def-sign-alternate} and \eqref{def-sign-alternate-1}.
Then the matrices $(m^d_{ij})$ clearly verify \eqref{yd-def2}. Set $\sgn_{\OC_X} \,:=\,
((m^{d_1}_{ij}),\, \cdots, \,(m^{d_s}_{ij}))$. 
It now follows from the above definition of $m^\theta_{ij}, \theta \in \O_\d$  that
$\sgn_{\OC_X} \,\in\, \SC^{\rm odd}_{\d}(2n)$.  
Thus we have a map $$\Psi_{{\rm Sp} (n, \R)} \,\colon\, \NC({\rm Sp}(n, \R)) \,\longrightarrow \,
\YC_{-1}^{\rm odd}(2n)\, ,\ \
\OC_X \,\longmapsto\, \big(\Psi'_{{\rm Sp} (n,\R)} (\OC_X),\, \sgn_{\OC_X}  \big)\, ;
$$
where $\YC_{-1}^{\rm odd}(2n)$ is as in \eqref{yd-odd-1-Y-pq}.

\begin{theorem}[{\rm  \cite[Theorem 9.3.5]{CoM}}]\label{sp-n-R-parametrization}
The above map $\Psi_{{\rm Sp} (n,\R)}$ is a bijection. 
\end{theorem}

Let $0\,\not=\, X \,\in\, \NC _{\s\p (n,\R)}$ and $\{X,H,Y\}$ be a $\s\l_2(\R)$-triple in $\s\p (n,\R)$. 
Let $\Psi_{{\rm Sp} (n,\R)} (\OC_X) \,= \,( \d ,\, \sgn_{\OC_X} )$.
Recall that $\sgn_{\OC_X}$ determines the signature of $(\cdot,\, \cdot)_\eta$ on $L(\eta -1)$ for
all $\eta \,\in\, \E_\d$; 
let $(p_\eta, \,q_\eta)$ be the signature of $(\cdot,\, \cdot)_\eta$ on $L(\eta-1)$.
Let $( v^{d}_1,\, \cdots,\, v^d_{t_d} )$ be a $\R$-basis of $L(d-1)$ as in Proposition
\ref{unitary-J-basis}.  
It now follows from Proposition \ref{unitary-J-basis}(3)(c) that $( v^\eta_1,\, \cdots,\,
v^\eta_{t_\eta})$ is an orthogonal basis of $L(\eta-1)$ for the form $(\cdot,\, \cdot)_\eta$.
We also assume that the vectors in the basis $( v^d_1,\, \cdots,\, v^d_{t_d})$ satisfy properties
in Remark \ref{unitary-J-basis-rmk}(1). 
In view of the signature of $(\cdot,\, \cdot)_\eta$, we may further assume that
\begin{equation}\label{orthonormal-basis-eta-sp-n-R}
 ( v^{\eta}_j, v^\eta_j)_{\eta}  =
 \begin{cases}
  +1  &  \text{ if }   1 \leq j \leq p_{\eta}  \\
  -1  &  \text{ if }  p_{\eta} < j \leq t_{\eta} 
 \end{cases} ; \  \eta  \in \E_\d.
\end{equation}
For all $\theta \,\in \,\O_\d$, as $(\cdot\, ,\, \cdot)_\theta$ is a symplectic form, we may
assume that $( v^{\theta}_1,\, \cdots,\, v^{\theta}_{t_{\theta}/2}\, ; \, v^{\theta}_{t_{\theta}/2+1},\,
\cdots, $ $  v^{\theta}_{t_{\theta}} )$ is a symplectic basis of $L(\theta - 1)$; see
Section \ref{sec-epsilon-sigma-forms} for the definition of a symplectic basis.
This is equivalent to saying that, for all $\theta \,\in\, \O_\d$, 
\begin{equation}\label{symplectic-basis-theta-sp-n-R}
 (v_j^\theta, v^\theta_{ t_\theta/2 + j})_\theta = 1 \text{ for } 1 \le j \le t_\theta/2 \text{ and }  (v_j^\theta, v^\theta_{i})_\theta = 0 \text{ for all } i \neq j + t_\theta/2.
\end{equation}
Now fixing $\theta \,\in\, \O_\d $, for all $ 1\,\leq\, j \,\leq\, t_\theta$, define 
\begin{equation}\label{orthogonal-basis-V-theta-sp-n-R}
{ w}^{\theta}_{jl}:=
\begin{cases}
   \big(X^l v^{\theta}_j + X^{\theta-1-l} v^{\theta}_j \big)\frac{1}{\sqrt{2}} & \text{ if } 0 \leq l < (\theta-1)/2 \vspace*{.1cm}\\
        X^l v^{\theta}_j    &  \text{ if }  l = (\theta-1)/2 \vspace*{.1cm}\\
   \big(X^{\theta-1-l} v^{\theta}_j - X^l v^{\theta}_j \big) \frac{1}{\sqrt{2}}  & \text{ if } (\theta-1)/2 < l \leq \theta -1.
\end{cases} 
\end{equation}

For $\theta \,\in \,\O_\d$, $0 \,\le\, l \,\le\, \theta-1$, set 
\begin{equation}\label{defn-V-theta-sp-n-R}
 V^l(\theta)\,:=\, \text{Span}_\R\{ w^{\theta}_{jl} \,\mid\, 1\,\le\, j \,\le\, t_\theta \}\, .
\end{equation}
The ordered basis $\big( {w}^\theta_{1l},\, \cdots ,\,  {w}^\theta_{t_{\theta}l} \big)$ of
$ V^l (\theta)$ is denoted by $\AC^l (\theta)$.  Let $\BC^l(d)\,=\, (X^lv^d_1,\, \cdots,\,
 X^lv^d_{t_d})$ be the ordered basis of $X^l L(d-1)$ for $0\,\leq\, l \,\leq\, d -1,\, d
\,\in\, \N_\d$ as in \eqref{old-ordered-basis-part}.

\begin{lemma}\label{reductive-part-comp-sp-n-R}
The following holds:
$$
 \ZC_{ {\rm Sp} (n,\R) }(X,H,Y)\! =\! \left\{ g \in { {\rm Sp} (n,\R)} \middle\vert \! 
                                        \begin{array}{ccc}
                             g (V^l (\theta)) \subset \ V^l (\theta) \  \text{and }    \vspace{.14cm}\\
                             \big[g|_{ V^l(\theta)}\big]_{{\AC}^l(\theta)} = \big[g |_{ V^0 (\theta)}\big]_{{\AC}^0 (\theta)} \text{for all } \theta \in \O_\d, 0 \leq l < \theta ~; \vspace{.14cm}\\
                               g(X^l L(\eta-1)) \subset  X^l L(\eta-1)  \text{ and }    \vspace{.14cm}\\
                          \!   \big[g |_{ X^l L(\eta-1)} \big]_{{\BC}^l (\eta)}\! \!=\! \big[g |_{  L(\eta-1)}\big]_{{\BC}^0 (\eta)}  \text{for all } \eta \in \E_\d,   0 \leq l< \eta \!
                                       \end{array}
                                           \right\}\!.
                                           $$                                           
\end{lemma}

\begin{proof}
The proof is similar to that of the Lemma \ref{reductive-part-comp-su-pq}; the details are
omitted.
\end{proof}
 
Using \eqref{symplectic-basis-theta-sp-n-R} and \eqref{orthogonal-basis-V-theta-sp-n-R} we 
observe that for each $\theta \,\in \,\O_\d$ the space $M (\theta-1)$ is a direct sum of the 
subspaces $V^l (\theta)$, $ 0 \leq l \leq \theta-1$, which are mutually orthogonal with respect 
to $\<>$. We now re-arrange the ordered basis $\AC^l(\theta)$ of 
$V^l (\theta)$ to obtain a symplectic basis $\CC^l(\theta)$ of $ V^l (\theta)$ with respect to 
$\<>$ as follows. For $\theta \,\in\, \O^1_\d$, define
$$
\CC^l(\theta) := \begin{cases}
            \big( {w}^{\theta}_{_{1 \, l}}, \  \cdots ,\   {w}^{\theta}_{_{t_{\theta}/2 \, l}}  \big) \vee  \big(  {w}^{\theta}_{_{(t_{\theta}/2 +1)\, l}},\  \cdots ,\   {w}^{\theta}_{_{t_{\theta} \, l}}  \big) &  \text{ if } l \text{ is  even and } 0 \leq l <  (\theta-1)/2 \vspace{0.1cm}\\
              \big( {w}^{\theta}_{_{(t_{\theta}/2 +1)\, l}},\  \cdots ,\  {w}^{\theta}_{_{t_{\theta} \, l}} \big)\vee  \big(  {w}^{\theta}_{_{1 \, l}}, \  \cdots ,\   {w}^{\theta}_{_{t_{\theta}/2 \, l}}  \big) &  \text{ if } l \text{ is  odd and } 0 \leq l < (\theta-1)/2 \vspace{0.1cm}\\
            \big(  {w}^{\theta}_{_{1 \, l}}, \  \cdots ,\   {w}^{\theta}_{_{t_{\theta}/2 \, l}}  \big) \vee \big( {w}^{\theta}_{_{(t_{\theta}/2 +1)\, l}},\  \cdots ,\   {w}^{\theta}_{_{t_{\theta} \, l}} \big) & \text{ if } l =  (\theta-1)/2 \vspace{0.1cm}\\
                  \big( {w}^{\theta}_{_{(t_{\theta}/2 +1)\, l}},\  \cdots ,\  {w}^{\theta}_{_{t_{\theta} \, l}} \big)\vee  \big(  {w}^{\theta}_{_{1 \, l}}, \  \cdots ,\   {w}^{\theta}_{_{t_{\theta}/2 \, l}}  \big) & \text{ if } l \text{ is  even and } (\theta+1)/2 \leq l \leq (\theta-1) \vspace{0.1cm}\\
            \big(  {w}^{\theta}_{_{1 \, l}}, \  \cdots ,\   {w}^{\theta}_{_{t_{\theta}/2 \, l}}  \big) \vee \big(  {w}^{\theta}_{_{(t_{\theta}/2 +1)\, l}},\  \cdots ,\   {w}^{\theta}_{_{t_{\theta} \, l}} \big) & \text{ if } l \text{ is  odd and } (\theta+1)/2 \leq l \leq (\theta-1) . 
               \end{cases}
$$
Similarly, for each $\zeta \,\in\,\O^3_\d$, define
$$
\CC^l(\zeta) := \begin{cases}
             \big( {w}^{\zeta}_{_{1 \, l}},  \cdots ,   {w}^{\zeta}_{_{t_{\zeta}/2 \, l}}  \big) \vee \big(  {w}^{\zeta}_{_{(t_{\zeta}/2 +1)\, l}},  \cdots , {w}^{\zeta}_{_{t_{\zeta} \, l}} \big)   &  \text{ if } l \text{ is  even and } 0 \leq l <  (\zeta-1)/2 \vspace{0.1cm}\\
             
              \big( {w}^{\zeta}_{_{(t_{\zeta}/2 +1)\, l}},  \cdots ,   {w}^{\zeta}_{_{t_{\zeta} \, l}} \big) \vee \big( {w}^{\zeta}_{_{1 \, l}},  \cdots ,   {w}^{\zeta}_{_{t_{\zeta}/2 \, l}}  \big)   &  \text{ if } l \text{ is  odd and } 0 \leq l < (\zeta-1)/2 \vspace{0.1cm}\\
              
                   \big( {w}^{\zeta}_{_{(t_{\zeta}/2 +1)\, l}},  \cdots ,   {w}^{\zeta}_{_{t_{\zeta} \, l}} \big) \vee \big( {w}^{\zeta}_{_{1 \, l}},  \cdots ,   {w}^{\zeta}_{_{t_{\zeta}/2 \, l}}  \big) & \text{ if } l =  (\zeta-1)/2 \vspace{0.1cm}\\
                   
                  \big( {w}^{\zeta}_{_{(t_{\zeta}/2 +1)\, l}},  \cdots ,   {w}^{\zeta}_{_{t_{\zeta} \, l}} \big) \vee \big( {w}^{\zeta}_{_{1 \, l}},  \cdots ,   {w}^{\zeta}_{_{t_{\zeta}/2 \, l}}  \big) & \text{ if } l \text{ is  even and } (\zeta+1)/2 \leq l \leq (\zeta-1) \vspace{.1cm}\\
                  
             \big( {w}^{\zeta}_{_{1 \, l}},  \cdots ,   {w}^{\zeta}_{_{t_{\zeta}/2 \, l}}  \big) \vee \big(  {w}^{\zeta}_{_{(t_{\zeta}/2 +1)\, l}},  \cdots , {w}^{\zeta}_{_{t_{\zeta} \, l}} \big)  & \text{ if } l \text{ is  odd and } (\zeta+1)/2 \leq l \leq (\zeta-1) . 
             \end{cases}
$$

For $\eta \,\in\, \E_\d$, $0\,\le\, l \,\le\, \eta/2 -1$, set
\begin{equation}\label{defn-W-eta-sp-n-R}
  W^l(\eta)\,:=\,  X^l L (\eta-1) + X^{\eta-1-l} L(\eta-1)\, .
\end{equation} 
We moreover re-arrange the ordered basis $\BC^l (\eta) \vee \BC^{\eta-1-l} (\eta)$ of
$W^l(\eta)$ and obtain new basis $\DC^l(\eta)$ as follows:
\begin{equation}\label{symplectic-basis-D-eta}
\DC^l(\eta):= \begin{cases}
               \big(X^l v_1, \cdots, X^l v_{p_\eta}\big) \vee \big(X^{\eta-1-l} v_{p_\eta +1}, \cdots, X^{\eta-1-l} v_{t_\eta}\big)\\
               \bigvee  \big(X^{\eta-1-l} v_{1}, \cdots, X^{\eta-1-l} v_{p_\eta}\big) \vee  \big(X^l v_{p_\eta +1}, \cdots, X^l v_{t_\eta}\big)   & \text{ if }l \text{  is even} \vspace{.25cm}\\
             
             \big(X^{\eta-1-l} v_{1}, \cdots, X^{\eta-1-l} v_{p_\eta}\big)  \vee \big(X^l v_{p_\eta +1}, \cdots, X^l v_{t_\eta}\big)\\
               \bigvee \big(X^l v_1, \cdots, X^l v_{p_\eta}\big) \vee \big(X^{\eta-1-l} v_{p_\eta +1}, \cdots, X^{\eta-1-l} v_{t_\eta}\big) & \text{ if } l \text{ is odd}.               
 \end{cases}
 \end{equation}
Using \eqref{orthonormal-basis-eta-sp-n-R} it can be easily verified that $\DC^l(\eta)$ is a symplectic basis with respect to $\<>$.
 
Let $J_{\CC^l(\theta)}$ be the complex structure on $V^l(\theta)$ associated to the basis 
$\CC^l(\theta)$ for $\theta \in \O_\d$, $0 \leq l \leq \theta-1$, and let $J_{\DC^l(\eta)}$ be 
the complex structure on $W^l(\eta)$ associated to the basis $\DC^l(\eta)$ for $\eta \in \E_\d, 
0 \leq l \leq \eta-1$; see Section \ref{sec-epsilon-sigma-forms} for the definition of such 
complex structures.

The next lemma is a standard fact where we recall, without a proof, an explicit description of a maximal compact subgroup in a symplectic group.
Let $V'$ be a $\R$-vector space, $\<>'$ be a non-degenerate symplectic form on $V'$ and ${\BC'}$ be a symplectic basis of $V'$. 
Let $J_{\BC'}$ be the complex structure on $V'$ associated to $\BC'$.
Let $2m := \dim_\R V'$. We set
$$
K_{\BC'} := \{ g \in {\rm Sp} (V', \<>') \mid   g J _{\BC'} = J_{\BC'} g  \}.
$$

\begin{lemma} \label{max-cpt-sp-n-R}
Let $V', \<>', \BC'$ and $J_{\BC'}$ be as above. Then 
 \begin{enumerate}
 \item 
 $K_{\BC'} $ is a maximal compact subgroup in ${\rm Sp} (V', \<>')$.
\item $K_{\BC'} = \Big\{ g \in {\rm SL} (V') \Bigm| \big[ g \big]_{\BC'} = \begin{pmatrix}
                                         A & -B\\
                                         B &  A
                                        \end{pmatrix}
                                        \text{ where } A + \sqrt{-1} B \in {\rm U}(m) \Big\}.$ 
\end{enumerate}                              
\end{lemma}            
                                        
Define the $\R$-algebra isomorphism 
\begin{equation}\label{R-algebra-isomorphism-sp}
 \widetilde{\Lambda}_{\BC'} \,\colon\, \{ x \,\in\, {\rm End}_\R V' \,\mid\, x J_{\BC'} \,=\,
 J_{\BC'} x \} \longrightarrow {\rm M}_m (\C), \quad
 x \longmapsto A + \sqrt{-1} B 
\end{equation}
where $ \big[ x \big]_{\BC'} = \begin{pmatrix}
                                         A & -B\\
                                         B &  A
                                        \end{pmatrix}$.
In view of Lemma \ref{max-cpt-sp-n-R} it is clear that $\widetilde{\Lambda}_{\BC'} ( K_{\BC'})
\,= \, {\rm U}(m)$, and thus $\widetilde{\Lambda}_{\BC'} \colon  K_{\BC'} \to {\rm U}(m)$ 
is an isomorphism of Lie groups.

In the next lemma we describe a suitable maximal compact subgroup of $\ZC_{ {\rm Sp} (n, \R) } (X,H,Y)$ which will be used in Proposition \ref{max-cpt-sp-n-R-wrt-basis}.
Define the $\R$-algebra embedding 
$$
\wp_{m, \C} \,: \,{\rm M}_m(\C) \,\longrightarrow\, {\rm M}_{2m}(\R)\, ,\ \
R\,\longmapsto\, \begin{pmatrix}
         S & -T\\
         T & S
      \end{pmatrix}
      $$
where $S,T \in {\rm M}_m(\R)$ are the unique elements such that $R= S+ \sqrt{-1} T$.

\begin{lemma}\label{max-cpt-reductive-part-sp-n-R}
Let $K$ be the subgroup of $\ZC_{ {\rm Sp} (n,\R) } (X,H,Y)$ consisting of elements $g$ in \\
$\ZC_{ {\rm Sp} (n,\R) } (X,H,Y)$ satisfying the following conditions:
 \begin{enumerate}
\item For all  $ \theta \in \O_\d $ and $ 0 \leq l \leq \theta -1$, \,  $g ( V^l (\theta)) \subset  V^l (\theta)$.
\item For all $ \theta \in \O^1_\d$, there exist  $A_\theta,B_\theta \in {\rm M}_{t_\theta/2}(\R)$ with $A_\theta + \sqrt{-1}B_\theta \in {\rm U}(t_\theta/2) $ such that 
  $$
  \big[g|_{V^l(\theta)}\big]_{\CC^l(\theta)} =
  \begin{cases}
  \wp_{t_\theta/2,\C}( A_\theta + \sqrt{-1} B_\theta)  & \text{ if $l$ is even, } 0 \leq l < (\theta -1)/2 \vspace*{0.05cm}\\
  \wp_{t_\theta/2,\C}( A_\theta -  \sqrt{-1} B_\theta)  & \text{ if } l \text{ is odd, } 0 \leq l < (\theta-1)/2 \vspace*{0.05cm}\\  
  \wp_{t_\theta/2,\C}( A_\theta +  \sqrt{-1} B_\theta ) & \text{ if }  l = (\theta-1)/2 \vspace*{.05cm}\\ 
  \wp_{t_\theta/2,\C}( A_\theta + \sqrt{-1} B_\theta  )&  \text{ if $l$ is odd, } (\theta + 1)/2 \leq l \leq \theta -1 \vspace*{.05cm}\\
  \wp_{t_\theta/2,\C}( A_\theta -  \sqrt{-1} B_\theta ) & \text{ if $l$ is even, } (\theta+1)/2 \leq l \leq \theta -1.    
  \end{cases}
 $$ 
\item For all $ \zeta \in \O^3_\d$, there exist  $A_\zeta,B_\zeta \in {\rm M}_{t_\zeta/2}(\R)$ with $A_\zeta + \sqrt{-1} B_\zeta \in {\rm U}(t_\zeta/2) $ such that 
  $$
  \big[g|_{V^l(\zeta)}\big]_{\CC^l(\zeta)} =
  \begin{cases}
   \wp_{t_\zeta/2,\C}( A_\zeta + \sqrt{-1} B_\zeta ) & \text{ if $l$ is even, } 0 \leq l < (\zeta -1)/2 \vspace*{0.05cm}\\
   \wp_{t_\zeta/2,\C}(A_\zeta - \sqrt{-1} B_\zeta )& \text{ if } l \text{ is odd, } 0 \leq l < (\zeta-1)/2 \vspace*{0.05cm}\\  
   \wp_{t_\zeta/2,\C}( A_\zeta - \sqrt{-1} B_\zeta ) & \text{ if }  l = (\zeta-1)/2 \vspace*{.05cm}\\ 
   \wp_{t_\zeta/2,\C}(A_\zeta + \sqrt{-1} B_\zeta )& \text{ if $l$ is odd, } (\zeta + 1)/2 \leq l \leq \zeta -1 \vspace*{.05cm}\\
   \wp_{t_\zeta/2,\C}( A_\zeta - \sqrt{-1} B_\zeta )& \text{ if $l$ is even, } (\zeta+1)/2 \leq l \leq \zeta -1.    
  \end{cases}
 $$  
\item For all $\eta \in \E_\d$ and $ 0 \leq l \leq \eta -1$, \, $ g ( X^l L(\eta-1)) \subset X^l L(\eta-1)$.
\item  For all $\eta \in \E_\d$, there exist $C_\eta \in {\rm O}_{p_\eta}$ and $D_\eta \in {\rm O}_{q_\eta}$ such that  \\ 
  $ \big[ g|_{X^l L(\eta-1)} \big]_{\BC^l(\eta)} = \begin{pmatrix}
                                         C_\eta & 0\\
                                         0 & D_\eta
                                        \end{pmatrix}$.
\end{enumerate}
Then $K$ is a maximal compact subgroup of $\ZC_{ {\rm Sp} (n,\R)}(X,H,Y)$.
\end{lemma}

\begin{proof}
For our convenience we begin by introducing a new notation. Let $m$ be an integer. For a matrix $Z$ in ${\rm M}_{2m} (\R)$, define 
$$                     
                           Z^\dagger :={\begin{pmatrix}
                                         0 & {\rm I}_m\\
                                         {\rm I}_m & 0
                                     \end{pmatrix}}
                                 Z {\begin{pmatrix}
                                         0 & {\rm I}_m\\
                                         {\rm I}_m & 0
                                         \end{pmatrix}}^{-1}.   
                                         $$
Note that 
$ \begin{pmatrix}
      P & -R\\
      R &  P
  \end{pmatrix}^{\! \! \dagger}  =  {\begin{pmatrix}
                                    P & R\\
                                   -R & P
                              \end{pmatrix}}$ for matrices $P,R \in {\rm M}_m (\R)$.

 Let $K'\, \subset\, \ZC_{ {\rm Sp} (n,\R) } (X,H,Y)$ be the subgroup consisting of
all elements $g$ satisfying the following conditions:
\begin{align}
   & \text{ For all } \theta \in \O_\d,  0 \leq l \leq \theta-1 ,~  g \big(V^l (\theta)\big) \subset \ V^l (\theta).
   \label{orthogonal-odd-subspace} \vspace{.3cm} \\
   & \text{For all } \theta \in \O^1_\d,    
  \big[g|_{ V^l(\theta)}\big]_{{\CC}^l(\theta)} = 
  \begin{cases}
  \big[g |_{ V^0 (\theta)}\big]_{{\CC}^0 (\theta)}   & \text{ if $l$ is even, } 0 \leq l < (\theta -1)/2 \vspace*{0.1cm}\\
    \big[g |_{ V^0 (\theta)}\big]^{ \dagger}_{{\CC}^0 (\theta)}  & \text{ if } l \text{ is odd, } 0 \leq l < (\theta-1)/2 \vspace*{0.1cm}\\  
  \big[g |_{ V^0 (\theta)}\big]_{{\CC}^0 (\theta)}  & \text{ if }  l = (\theta-1)/2 \vspace*{.05cm}\\ 
  \big[g |_{ V^0 (\theta)}\big]_{{\CC}^0 (\theta)}  &  \text{ if $l$ is odd, } (\theta + 1)/2 \leq l \leq \theta -1 \vspace*{.05cm}\\
   \big[g |_{ V^0 (\theta)}\big]^{\dagger}_{{\CC}^0 (\theta)}  & \text{ if $l$ is even, } (\theta+1)/2 \leq l \leq \theta -1,    
  \end{cases}\label{matrix-wrt-orthogonal-odd-1-subspace-sp} \\
   & \text{For all } \zeta \in \O^3_\d,    
  \big[g|_{ V^l(\zeta)}\big]_{{\CC}^l(\zeta)} = 
  \begin{cases}
  \big[g |_{ V^0 (\zeta)}\big]_{{\CC}^0 (\zeta)}   & \text{ if $l$ is even, } 0 \leq l < (\zeta -1)/2 \vspace*{0.1cm}\\
    \big[g |_{ V^0 (\zeta)}\big]^{\dagger}_{{\CC}^0 (\zeta)}  & \text{ if } l \text{ is odd, } 0 \leq l < (\zeta-1)/2 \vspace*{0.1cm}\\  
  \big[g |_{ V^0 (\zeta)}\big]^{\dagger}_{{\CC}^0 (\zeta)}  & \text{ if }  l = (\zeta-1)/2 \vspace*{.05cm}\\ 
  \big[g |_{ V^0 (\zeta)}\big]_{{\CC}^0 (\zeta)}  &  \text{ if $l$ is odd, } (\zeta + 1)/2 \leq l \leq \zeta -1 \vspace*{.05cm}\\
   \big[g |_{ V^0 (\zeta)}\big]^{\dagger}_{{\CC}^0 (\zeta)}  & \text{ if $l$ is even, } (\zeta+1)/2 \leq l \leq \zeta -1,    
  \end{cases} \label{matrix-wrt-orthogonal-odd-3-subspace-sp} \\ 
   & g|_{ V^0(\theta)} \text{ commutes with } J_{{\CC}^0 (\theta)}, \label{commute-J-odd} \\
   & g\big(X^l L(\eta-1)\big) \subset  X^l L(\eta-1), ~~~ \big[g |_{ X^l L(\eta-1)} \big]_{{\BC}^l (\eta)} =  \big[g |_{  L(\eta-1)}\big]_{{\BC}^0 (\eta)}  \text{ for all } \eta \in \E_\d, \label{orthogonal-even-subspace} \\  
   &  0 \leq l\leq \eta-1,   \vspace{.3cm} \nonumber\\  
   &  g|_{ W^0(\eta)}  \text{ commutes with }   J_{{\DC}^0 (\eta)}. \label{commute-J-even} 
 \end{align}
Using Lemma \ref{reductive-part-comp-sp-n-R} and Lemma \ref{max-cpt-sp-n-R}(1) it is clear
that $K'$ is a maximal compact subgroup of $\ZC_{ {\rm Sp} (n,\R) } (X,H,Y)$. 
Hence to prove the lemma it suffices to show that $K \,= \,K'$. Let $g \,\in\,  {\rm Sp} (n,\R) $.
Using Lemma \ref{max-cpt-sp-n-R} (2) it is straightforward to check that $g$ satisfies
(1), (2), (3) in the statement of the lemma if and only if $g$ satisfies
\eqref{orthogonal-odd-subspace}, \eqref{matrix-wrt-orthogonal-odd-1-subspace-sp},
\eqref{matrix-wrt-orthogonal-odd-3-subspace-sp} and \eqref{commute-J-odd}.
Now suppose that $g \,\in\, {\rm Sp} (n,\R)$ and $g$ satisfies (4) and (5) in the statement
of the lemma. It is clear that \eqref{orthogonal-even-subspace} holds. 
We observe that  
$$
\big[ J_{{\DC}^0 (\eta)}\big]_{{\DC}^0(\eta)} = \begin{pmatrix}
                                        0  & -{\rm I}_{t_\eta}\\
                                       {\rm I}_{t_\eta} &0 \\
                                      \end{pmatrix} ~
\text{~ and }~
\big[ g|_{W^0(\eta)} \big]_{{\DC}^0(\eta)} = \begin{pmatrix}
                                        C_\eta &   \\
                                        & D_\eta  \\
                                        & & C_\eta\\
                                        & & & D_\eta
                                      \end{pmatrix}
                                      $$ 
where $\DC^0(\eta)$ is defined by setting $l=0$ in \eqref{symplectic-basis-D-eta}.
  From the matrix representations as above, it is clear that $J_{{\DC}^0 (\eta)}$ and  $g|_{W^0(\eta)}$ 
  commute. This proves that \eqref{commute-J-even} holds.

Now we assume that $g$ satisfies \eqref{orthogonal-even-subspace} and \eqref{commute-J-even}.
It is clear that (4) in the statement of the lemma holds.
Note that $A\,:=\,\big[g |_{  L(\eta-1)}\big]_{_{{\BC}^0 (\eta)}}\,=\,
\big[g |_{ X^l L(\eta-1)}\big]_{{\BC}^l (\eta)}$ for $1\leq l \leq \eta -1$.
We observe that 
$$\big[ J_{{\DC}^0 (\eta)}\big]_{\BC^0 (\eta) \vee \BC^{\eta -1} (\eta)} = \begin{pmatrix}
                                         0 & -{\rm I}_{p_\eta,q_\eta} \\
                                       {\rm I}_{p_\eta,q_\eta} &0 
                                      \end{pmatrix} ~
\text{~ and }~
\big[ g |_{W^0(\eta)}\big]_{\BC^0 (\eta) \vee \BC^{\eta -1} (\eta)}  = \begin{pmatrix}
                                       A &  \\
                                       & A
                                      \end{pmatrix}.
                                      $$                                     
{}From  \eqref{commute-J-even} it follows that the above two matrices commute, which in turn implies that $A$ commutes with $ \begin{pmatrix}
         {\rm I}_{p_\eta} &  \\
         & -{\rm I}_{q_\eta}
\end{pmatrix}$.
Thus $A$ is of the form $A =  \begin{pmatrix}
                                      C & 0 \\
                                      0 & D
                                      \end{pmatrix} $
for some matrices $C \in {\rm GL}_{p_\eta} ( \R)$ and $D \in {\rm GL}_{q_\eta} ( \R)$.
Now observe that 
$$
\big[ g |_{W^0(\eta)}\big]_{\DC^0 (\eta)} 
= \begin{pmatrix}
   C  \\
   & D \\
   & & C\\
   & & & D
  \end{pmatrix}.
$$
As $g|_{W^0(\eta)}$ commutes with $J_{{\DC}^0 (\eta)}$, it follows that 
$$
\begin{pmatrix}
 C & 0  \\
 0 & D  
 \end{pmatrix} +\sqrt{-1} \begin{pmatrix}
 0 & 0  \\
 0 & 0  
 \end{pmatrix} \in {\rm U} (t_\eta).
$$
Thus, $C \in {\rm O}_{p_\eta} $ and $ D \in {\rm O}_{q_\eta} $ and (5) in the statement of
the lemma holds.
This completes the proof.
\end{proof}

We next introduce some notation which will be needed in Proposition \ref{max-cpt-sp-n-R-wrt-basis}.
Recall that the positive parts of the symplectic basis $\DC(\eta)$, $\CC(\theta)$ are denoted 
by $\DC_+(\eta)$, $\CC_+(\theta)$ respectively; see Section \ref{sec-epsilon-sigma-forms}. 
Similarly, the negative parts of $\DC(\eta)$, $\CC(\theta)$ are denoted by $\DC_-(\eta)$, 
$\CC_-(\theta)$ respectively. For $\eta \,\in\, \E_\d$, set
$$
\DC_+ (\eta) \,:=\, \DC^0_+ (\eta) \vee \cdots \vee  \DC^{\eta/2-1}_+ (\eta) \ \text{ and } \ 
\DC_- (\eta) \,:=\, \DC^0_- (\eta) \vee \cdots \vee  \DC^{\eta/2 -1}_- (\eta).
$$
For $\theta \,\in\, \O_\d$, set
$$
\CC_+ (\theta) \,:=\, \CC^0_+ (\theta) \vee \cdots \vee  \CC^{\theta-1}_+ (\theta) \ \text{ and } \ 
\CC_- (\theta) \,:=\, \CC^0_- (\theta) \vee \cdots \vee  \CC^{\theta-1}_- (\theta).
$$
Let $\alpha\, :=\, \# \E_\d $, $\beta \,:=\, \# \O^1_\d$ and $ \gamma \,:=\, \# \O^3_\d $.
We enumerate $\E_\d \,=\,\{ \eta_i \,\mid\, 1 \,\leq\, i \,\leq\, \alpha \}$ such that
$\eta_i \,< \,\eta_{i+1}$, and 
$\O^1_\d \,=\,\{ \theta_j \,\mid\, 1 \leq j \leq \beta \}$ such that
$\theta_j \,< \,\theta_{j+1}$; similarly enumerate $\O^3_\d \,=\,\{ \zeta_j \,\mid\, 1 \,\leq\, j
\,\leq\, \gamma \}$ such that $\zeta_j \,<\, \zeta_{j+1}$. Now define
$$
\EC_+ \,:=\, \DC_+ (\eta_1) \vee \cdots \vee \DC_+ (\eta_{\alpha})\, ; \ \  \,
\OC^1_+ \,:=\, \CC_+ (\theta_1) \vee \cdots \vee \CC_+ (\theta_{\beta})\, ; \ \ \,\OC^3_+ := \CC_+ (\zeta_1) \vee \cdots \vee \CC_+ (\zeta_{\gamma});
$$
$$
\EC_- := \DC_- (\eta_1) \vee \cdots \vee \DC_- (\eta_{\alpha}) ; \ 
\OC^1_- := \CC_-(\theta_1) \vee \cdots \vee \CC_-(\theta_{\beta})\, \text{ and } \,\OC^3_- := \CC_- (\zeta_1) \vee \cdots \vee \CC_-(\zeta_{\gamma}).
$$
Also we define
\begin{equation}\label{symplectic-basis-final}
\HC_+ := \EC_+ \vee \OC^1_+ \vee\OC^3_+ , \ \ \HC_- := \EC_- \vee \OC^1_- \vee\OC^3_- \ \text{ and } \
\HC := \HC_+ \vee \HC_-. 
\end{equation}

As before, for a matrix $A\,=\, (a_{ij}) \,\in\, {\rm M}_r(\C)$, define $\overline{A}
\,:=\, (\overline{a}_{ij}) \,\in\, {\rm M}_r(\C)$.
Let
$$
\Db \colon \prod_{i=1}^\alpha  \big( {\rm M}_{p_{\eta_i}} (\R) \times {\rm M}_{q_{\eta_i}} (\R) \big) \times \prod_{j=1}^\beta {\rm M}_{t_{\theta_j}/2} (\C) \times \prod_{k=1}^\gamma {\rm M}_{t_{\zeta_k}/2} (\C)  \longrightarrow {\rm M}_n(\C)  
$$
be the $\R$-algebra embedding defined by
\begin{align*}
\big(C_{\eta_1}, D_{\eta_1},  \cdots ,& C_{\eta_{\alpha}}, D_{\eta_{\alpha}}; A_{\theta_1}, \cdots, 
A_{\theta_\beta} ; B_{\zeta_1}, \cdots ,  B_{\zeta_\gamma} \big) \\
     \longmapsto\, & \bigoplus_{i=1}^\alpha  \big( C_{\eta_i} \oplus D_{\eta_i} \big)_\blacktriangle^{\eta_i /2} \oplus \bigoplus_{j=1}^\beta \Big( \big( A_{\theta_j} \oplus \overline{A}_{\theta_j}\big)_\blacktriangle ^{\frac{\theta_j-1}{4}} \oplus A_{\theta_j} \oplus \big( A_{\theta_j} \oplus \overline{A}_{\theta_j}\big)_\blacktriangle ^{\frac{\theta_j-1}{4}} \Big) \\ 
  \oplus & \bigoplus_{k=1}^\gamma \Big( \big( B_{\zeta_k}\oplus \overline{B}_{\zeta_k} \big) _\blacktriangle ^{\frac{\zeta_k + 1}{4}}  \oplus \big( B_{\zeta_k}\oplus \overline{B}_{\zeta_k} \big)_\blacktriangle^{\frac{\zeta_k -3}{4}}  \oplus \overline{B}_{\zeta_k} \Big) .
\end{align*}

It is clear that the basis $\HC$ in \eqref{symplectic-basis-final} is a symplectic basis of $V$ with
respect to $\<>$. Let $\widetilde{\Lambda}_\HC \,\colon\, \{ x \,\in\, {\rm End}_\R \R^{2n}
\,\mid\, x J_\HC \,=\,J_\HC x  \} \,\longrightarrow\, {\rm M}_{n} (\C)$ be the isomorphism of
$\R$-algebras induced by the above symplectic basis $\HC$. Recall that $\widetilde{\Lambda}_\HC \,:\,
K_\HC \,\longrightarrow\, {\rm U} (n)$ is an isomorphism of Lie groups. Using \eqref{commute-J-odd}
and \eqref{commute-J-even} we observe that the group $K$ defined in Lemma
\ref{max-cpt-reductive-part-sp-n-R} satisfies the condition $K \,\subset\, K_\HC$.
 
\begin{proposition}\label{max-cpt-sp-n-R-wrt-basis} 
Let $X \,\in\, \NC _{\s\p(n,\R)}$ and $\Psi_{{\rm Sp}(n,\R)} (\OC_X)\, =\, (\d,\, \sgn_{\OC_X})$.
Let $\alpha \,:=\, \# \E_\d $, $\beta \,:=\, \# \O^1_\d$ and $ \gamma \,:=\, \# \O^3_\d $. Let
 $\{X,H,Y\}$ be a $\s\l_2(\R)$-triple in $\s\p (n,\R)$, and let $(p_\eta, \,q_\eta)$ be the signature
of $(\cdot,\, \cdot)_\eta$, $\eta \,\in\, \E_\d$, as defined in \eqref{new-form}.
Let $K$ be the maximal compact subgroup of $\ZC_{{\rm Sp} (n,\R)} (X, H, Y)$ as in Lemma
\ref{max-cpt-reductive-part-sp-n-R}. Then $\widetilde{\Lambda}_\HC(K) \subset {\rm U}(n)$ is given by
$$
\widetilde{\Lambda}_\HC(K)= \bigg\{ \Db(g) \biggm| 
                        g \in   \prod_{i=1}^\alpha \big( {\rm O}_{p_{\eta_i}} \times {\rm O}_{q_{\eta_i}} \big) \times \prod_{j=1}^\beta  {\rm U}(t_{\theta_j}/2)  \times \prod_{k=1}^\gamma {\rm U}(t_{\zeta_k}/2) \bigg\}.
$$
\end{proposition}

 \begin{proof}
  This follows by writing the matrices of the elements of the maximal compact subgroup $K$ in
Lemma \ref{max-cpt-reductive-part-sp-n-R} with respect to the symplectic basis $\HC$
in \eqref{symplectic-basis-final}.
 \end{proof}

\begin{theorem}\label{sp-n-R}
Let $X \,\in\, \s \p(n,\R)$ be a nilpotent element.
Let $(\d,\, \sgn_{\OC_X}) \,\in\, \YC^{\rm odd}_{-1}(2n)$ be the signed Young diagram of the orbit
$\OC_X$ (that is, $\Psi_{{\rm Sp}(n,\R)}  (\OC_X) \,= \,(\d,\, \sgn_{\OC_X})$ as in the notation of
Theorem \ref{sp-n-R-parametrization}). Then
$$
\dim_\R H^2(\OC_X,\, \R)=\begin{cases}
                              0    & \text{ if } \,  \# \O_\d =0 \\
                       \# \O_\d-1  & \text{ if } \,  \# \O_\d \geq 1.
                             \end{cases}
$$
\end{theorem}

\begin{proof} As the theorem is evident when $X \,=\,0$ we assume that $X \,\neq\, 0$.

Let $\{X,\, H,\, Y\}\,\subset\, \s\p(n,\R)$ be a $\s\l_2(\R)$-triple.
Let $K$ be the maximal compact subgroup of $\ZC_{{\rm Sp}(n,\R)}(X,H,Y)$ as in Lemma
\ref{max-cpt-reductive-part-sp-n-R}. 
Let $\HC$ be as in \eqref{symplectic-basis-final} and $K_{\HC}$ the maximal compact subgroup
of ${\rm Sp}(n,\R)$ as in Lemma \ref{max-cpt-sp-n-R}(1). Then $K \,\subset\, K_{\HC}$.
Let $\k_{\HC}$ be the Lie algebra of $K_\HC$.
Using Proposition \ref{max-cpt-sp-n-R-wrt-basis} it follows that 
$\z(\k) \,\subset\, [\k_{\HC},\, \k_{\HC}]$ when $ \# \O_\d \,=\,0$, and  $\z(\k) \,\not\subset\,
 [\k_{\HC}, \,\k_{\HC}]$ when $ \# \O_\d \,\geq\, 1$.
As $\dim_\R \z(\k_{\HC})\,=\,1$, it follows that $$\dim_\R \z(\k)\cap [\k_{\HC},\, \k_{\HC}]
\,=\, \dim_\R \z(\k) -1$$ when $ \# \O_\d \,\geq\, 1$. 
The group ${\rm O}_2 /{\rm SO}_2\,=\,\Z/2\Z$ acts non-trivially on $\s\o_2$ and the group
${\rm U}(m)$ acts trivially on $\z(\u(m))$. We next use the observation in
\eqref{adjoint-action} to conclude that
$$
\dim_\R \big[ \z(\k)\cap [\k_{\HC}, \k_{\HC}] \big]^{K/{K^\circ}} = \begin{cases}
                              0    & \text{ if } \, \# \O_\d =0 \\
                       \# \O_\d-1  & \text{ if } \, \# \O_\d \geq 1 \,.
                             \end{cases}
$$
Now the theorem follows from Theorem \ref{thm-nilpotent-orbit}.
\end{proof}

\subsection{Second cohomology groups of nilpotent orbits in 
\texorpdfstring{${\s\p}(p,q)$}{Lg}}\label{sec-sp-pq}

Let $n$ be a positive integer and $(p,q)$ be a pair of non-negative integers with $p + q \,=\,n$. 
In this subsection we compute the second cohomology groups of nilpotent orbits in ${\s\p}(p,q)$ under 
the adjoint action of ${\rm Sp}(p,q)$. As we do not need to deal with compact groups, we will further 
assume that $p \,>\,0$ and $q\,>\,0$. Throughout this subsection $\<>$ denotes the Hermitian form
on $\H^n$  defined by $\langle x,\, y \rangle \,:=\, \overline{x}^t{\rm I}_{p,q} y$, $x,\,y\, \in\,
\H^n$, where ${\rm I}_{p,q}$ is as in \eqref{defn-I-pq-J-n}. We will follow notation as defined in \S \ref{sec-notation}.

Let  $\Psi_{{\rm SL}_n (\H)} \,:\, \NC ({\rm SL}_n (\H)) \,\longrightarrow\, \PC (n)$ be the
parametrization as in Theorem \ref{sl-H-parametrization}.
As ${\rm Sp} (p,q) \,\subset\, {\rm SL}_n (\H)$ (consequently, $\NC_{{\s\p}(p,q)} \,\subset\,
\NC_{\s\l_n(\H)}$)
we have the inclusion map $\Theta_{{\rm Sp}(p,q)} \,: \,\NC ({\rm Sp} (p,q)) \,
\longrightarrow\,\NC ( {\rm SL}_n (\H) )$. Let
$$\Psi'_{{\rm Sp}(p,q)}\,:=\, \Psi_{{\rm SL}_n (\H)} \circ \Theta_{{\rm Sp}(p,q)} \,
\colon\, \NC ({\rm Sp} (p,q))  \,\longrightarrow\, \PC (n)$$ be the composition.
Let $0\,\not=\, X \,\in\, {\s\p}(p,q)$ be a nilpotent element and $\OC_X$ be the corresponding nilpotent orbit in
${\s\p}(p,q)$. Let $\{X,\, H,\, Y\} \,\subset\, {\s\p}(p,q)$ be a $\s\l_2(\R)$-triple. 
We now apply Proposition \ref{unitary-J-basis}, Remark \ref{unitary-J-basis-rmk}(3), and follow
the notation used therein. 
Let $V := \H^n$ be the right $\H$-vector space of column vectors.
Let $\{d_1,\, \cdots,\, d_s\}$, with $d_1 \,<\, \cdots \,<\, d_s$, be a ordered finite subset of
natural numbers that arise as $\R$-dimension of non-zero irreducible 
$\text{Span}_\R \{ X,H,Y\}$-submodules of $V$.
Recall that $M(d-1)$ is defined to be the isotypical component of $V$ containing all irreducible $\text{Span}_\R \{ X,H,Y\}$-submodules of $V$ with highest weight $(d-1)$, and as in \eqref{definition-L-d-1}, we set $L(d-1)\,:= \,V_{Y,0} \cap M(d-1)$. Recall that the space $L(d_r-1)$ is a $\H$-subspace for $1\leq r \leq s$.
Let $t_{d_r} \,:=\, \dim_\H L(d_r-1)$ for 
$1\,\leq\, r \,\leq\, s$. Then 
$\d\,:=\, [d_1^{t_{d_1}},\, \cdots ,\,d_s^{t_{d_s}}]\,\in\, \PC(n)$, and  moreover,
$\Psi'_{{\rm Sp}(p,q)} (\OC_X) \,=\, \d$.

We next assign $\sgn_{\OC_X} \,\in\, \SC^{\rm even}_{\d}(p,q)$ to each $\OC_X \,\in\, \NC({\rm Sp}(p,q))$; see 
\eqref{S-d-pq-even} for the definition of $\SC^{\rm even}_{\d}(p,q)$.
For each $d \,\in\, \N_\d$ (see \eqref{Nd-Ed-Od} for the definition of $\N_\d$)  we will define a $t_d \times d$ matrix $(m^d_{ij})$ in $\Ab_d$ 
which depends only on the orbit $\OC_X$ containing $X$; see \eqref{A-d} for the definition of $\Ab_d$.
For this, recall that the form
$(\cdot,\,\cdot)_{d} \,\colon\, L(d-1) \times L(d-1) \,\longrightarrow\,
\H$ defined as in \eqref{new-form} is Hermitian or skew-Hermitian
according as $d$ is odd or even.
Denoting the signature of  $(\cdot,\, \cdot)_{\theta}$ by $(p_{\theta}, \,q_{\theta})$ 
when $\theta\,\in\, \O_\d$, we now define 
\begin{align*}
 m^\eta_{i1} & := +1 \qquad   \text{if } \  1 \leq i \leq t_{\eta}, \quad \eta \in \E_\d\,; \\
 m^\theta_{i1} &:= \begin{cases}
                   +1  & \text{ if } \,  1 \leq i \leq p_{\theta} \\
                   -1  & \text{ if } \, p_\theta < i \leq t_\theta  
                 \end{cases} ,\, \theta \in \O_\d\,;
               \end{align*}
and for $j \,>\,1$, define $(m^d_{ij})$ as in \eqref{def-sign-alternate} and
\eqref{def-sign-alternate-1}. 
Then the matrices $(m^d_{ij})$ clearly verify \eqref{yd-def2}. Set $\sgn_{\OC_X} \,:=\,
((m^{d_1}_{ij}),\, \cdots ,\, (m^{d_s}_{ij}))$.
It now follows from the last paragraph of Remark \ref{CM-correction} 
that $\sgn_{\OC_X}\,\in\, \SC^{\rm even}_\d(p,q)$. Thus we have the map
$$
\Psi_{{\rm Sp} (p,q)}\,:\,\NC({\rm Sp}(p,q)) \,\longrightarrow\, \YC^{\rm even}(p,q)\, ,\ \
\OC_X \,\longmapsto\, \big(\Psi'_{{\rm Sp} (p,q)} (\OC_X),\, \sgn_{\OC_X}  \big)\, ;
$$
where $\YC^{\rm even}(p,q)$ is as in \eqref{yd-even-Y-pq}.

\begin{theorem}\label{sp-pq-parametrization}
The above map $\Psi_{{\rm Sp}(p,q)}$ is a bijection.
\end{theorem}
 
\begin{remark}
On account of the error in \cite[Lemma 9.3.1]{CoM} (mentioned in Remark \ref{CM-correction}), the
above parametrization in Theorem \ref{sp-pq-parametrization} is a modification of the one
given in \cite[Theorem 9.3.5]{CoM}.
\end{remark}

\begin{theorem}\label{sp-pq}
 Let $X \,\in\, \s \p(p,q)$ be a nilpotent element. 
 Let $(\d,\, \sgn_{\OC_X}) \,\in\, \YC^{\rm even}(p,q)$ be the signed Young diagram of the orbit
$\OC_X$ (that is, $\Psi_{{\rm Sp}(p,q)}  (\OC_X) \,=\, (\d,\, \sgn_{\OC_X})$ in the notation of
Theorem \ref{sp-pq-parametrization}). Then $\dim_\R H^2(\OC_X,\, \R) \, =\, \# \E_\d$.
\end{theorem}

\begin{proof}
As the theorem is evident for $X \,=\, 0$, we assume that $X \,\neq\, 0$.

Let $\{X,\, H,\, Y\} \,\subset\, \s\p(p,q)$ be a $\s\l_2(\R)$-triple. 
 Let $V\,:= \,\H^n$ be the right $\H$-vector space of column vectors.
 We consider $V$ as a $\text{Span}_\R \{ X,H,Y\}$-module via its natural $\s\p(p,q)$-module
structure. Let 
$$
V_\E : = \bigoplus_{\eta \in \E_\d} M(\eta -1)\,; \qquad V_\O : = \bigoplus_{\theta \in \O_\d} M(\theta -1).
$$
Using Lemma \ref{ortho-isotypical} we see that $V=V_\E \oplus V_\O$ is an orthogonal decomposition
of $V$ with respect to $\<>$.  
Let $\<>_ \E := \<> |_{V_\E \times V_\E  }$ and $\<>_ \O := \<> |_{V_\O \times V_\O}$. 
Let $ X _\E := X|_{V_\E}$,  $X _\O := X|_{V_\O}$, $H_\E := H|_{V_\E}$, $H_\O := H|_{V_\O}$, $Y_\E:= Y|_{V_\E}$ and  $Y_\O := Y|_{V_\O}$.
Then we have the following natural isomorphism:
\begin{equation}\label{reductive-part-odd-even-sp}
 \ZC_{{\rm Sp}(p,q)}(X,H,Y)~ \simeq~ \ZC_{{\rm SU}(V_\E, \<>_\E)}(X_\E,H_\E,Y_\E) \times \ZC_{{\rm SU}(V_\O, \<>_\O)}(X_\O,H_\O,Y_\O).
\end{equation}
Recall that the non-degenerate form $(\cdot, \,\cdot)_d$ on $L (d-1)$ is skew-Hermitian for all
$d \,\in \,\E_\d$ and Hermitian for all $d \,\in\, \O_\d$; see Remark \ref{unitary-J-basis-rmk}.
Moreover, for any $\theta \,\in\, \O_\d$, the signature of $(\cdot,\, \cdot)_\theta$ is
$(p_\theta,\, q_\theta)$. It follows from Lemma \ref{reductive-part-comp}(4) that
$$
  \ZC_{{\rm SU}(V_\E, \<>_\E)}(X_\E,H_\E,Y_\E) \simeq \prod_{\eta \in \E_\d} {\rm SO^*}(2t_\eta)
 \,\, {\rm and} \, \, 
   \ZC_{{\rm SU}(V_\O, \<>_\O)}(X_\O,H_\O,Y_\O) \simeq \prod_{\theta \in \O_\d} {\rm Sp}(p_\theta, q_\theta). 
 $$
 In particular,  $\ZC_{{\rm SU}(V_\E, \<>_\E)}(X_\E,H_\E,Y_\E) $ and $\ZC_{{\rm SU}(V_\O, \<>_\O)}(X_\O,H_\O,Y_\O)$ are both connected groups.
 Let $K_\E$ be a maximal compact subgroup of $ \ZC_{{\rm SU}(V_\E, \<>_\E)}(X_\E,H_\E,Y_\E)
\simeq \prod_{\eta \in \E_\d}{\rm SO^*}(2t_\eta)$, and let $K_\O$ be a maximal compact subgroup of $\ZC_{{\rm SU}(V_\O, \<>_\O)}(X_\O,H_\O,Y_\O) \simeq \prod_{\theta \in \O_\d} {\rm Sp}(p_\theta, q_\theta)$. 
 Let $K$ be the image of $ K_\E \times K_\O$ under the isomorphism as in \eqref{reductive-part-odd-even-sp}. It is clear that $K$  
 is a maximal compact subgroup of $\ZC_{{\rm Sp}(p,q)}(X,H,Y)$. 
 Let $M$ be a maximal compact subgroup of ${\rm Sp}(p,q)$ containing $K$. As $M \simeq {\rm Sp}(p)\times {\rm Sp}(q)$ is semisimple,
and $K$ is connected, using Theorem \ref{thm-nilpotent-orbit} we have    
$$
H^2 ({\OC}_X,\, \R) ~ \simeq ~ \z (\k),  \quad \text{for all $X \neq 0$}.
$$
Let $\k_\O$ and $ \k_\E$ be the Lie algebras of $K_\O$ and $ K_\E$ respectively. 
As $K_\O$ is semisimple, we have $\z(\k_\O) \,=\,0$. Hence, $\z(\k) \,\simeq\, \z(\k_\E) \oplus \z (\k_\O) \,=\,\z(\k_\E)$. Since $\k_\E \,\simeq\, \bigoplus_{\eta \in \E_\d} \u(t_\eta)$, we
have $\dim_\R \z(\k_\E) \,=\,  \# \E_\d$. This completes the proof.
\end{proof}
 
\section{First cohomology groups of nilpotent orbits}\label{sec-First-Cohomology-of-Nilpotent-Orbits}

In this section we apply the results of the previous section to
compute the first cohomology groups of the nilpotent orbits in the classical simple Lie algebras. 
We begin by showing that in the case of complex simple Lie algebras the first cohomology groups of all the nilpotent orbits vanish.  

\begin{theorem}\label{complex-simple-H1}
Let $\g$ be a complex simple Lie algebra. Then $H^1( \OC_X,\, \R) \,=\,0$ for every nilpotent
element $X \,\in\, \g$.
\end{theorem}

\begin{proof}
 Any maximal compact subgroup of a simple complex Lie Group is simple. The theorem follows from Corollary \ref{H.1}.
\end{proof}

\begin{theorem}\label{sl-n-H-H1}
Let $\g$ be either $ \s\l_n(\H)$ or $\s\p(p,q)$. Then $H^1( \OC_X, \,\R)\, =\,0$ for every
nilpotent element $X \,\in\, \g$.
\end{theorem} 

\begin{proof}
  Let $G$ be ${\rm SL}_n (\H)$ or ${\rm Sp} (p,q)$ according as $\g$ is $ \s\l_n(\H)$ or $\s\p(p,q)$.
  Then any maximal compact subgroup of $G$ is simple.
  The theorem now follows from Theorem \ref{thm-nilpotent-orbit}. 
\end{proof}

\begin{theorem}\label{sl-n-R-H1}
 Let $X \,\in\, \s \l_n(\R)$ be a non-zero nilpotent element. Then 
 $$ \dim_\R H^1(\OC_X, \,\R)  \,=\,
 \begin{cases}
  1 & \text{ if } n=2\\
  0 & \text{ if } n \geq 3.
 \end{cases}
$$
\end{theorem} 
              
\begin{proof}
We follow the notation of the proof of Theorem \ref{sl-n-R}. When $n \,\geq\, 3$, it is clear
that $\m \,=\, [\m,\,\m]$. 
When $n\,=\,2$, we have $\m \,\simeq\, \s\o_2$ and $\Psi_{{\rm SL}_n(\R)}(\OC_X)\,=\, [2^1]$. 
Thus, using \eqref{sl-n-R-reductive-part} we see that $\k\,=\,0$. Now the
theorem follows from Theorem \ref{thm-nilpotent-orbit}.
\end{proof}

\begin{theorem}\label{su-pq-H1}
 Let $X \,\in\, \s \u(p,q)$ be a nilpotent element. 
 Let $(\d,\, \sgn_{\OC_X}) \,\in\, \YC(p,q)$ be the signed Young diagram of the orbit $\OC_X$ (that
is, $\Psi_{{\rm SU}(p,q)}(\OC_X) \,=\, (\d,\, \sgn_{\OC_X})$ in the notation of Theorem
\ref{su-pq-parametrization}). Let $l\,:\,= \# \{d\,\in\, \N_\d\,\mid\, p_d \neq 0\} + \#
\{d\in \N_\d\,\mid\, q_d \,\neq\, 0\}$. Then the following hold:
 \begin{enumerate}
 \item If $\N_\d \,=\, \E_\d$, then  $\dim_\R H^1(\OC_X,\, \R) \, = \, 1.$
 \item If $l\,=\,1$ and $\N_\d \,=\, \O_\d$, then  $\dim_\R H^1(\OC_X,\, \R) \, = \, 1.$
 \item If $l \,\geq\, 2$ and $\#\O_\d \,\geq \,1$, then  $\dim_\R H^1(\OC_X,\, \R)  \,= \, 0.$
\end{enumerate}
\end{theorem}

\begin{proof}
 We follow the notation of the proof of Theorem \ref{su-pq}. We now appeal to Proposition
\ref{max-cpt-su-pq-wrt-onb} to make the following observations:
 \begin{enumerate}
  \item If $\N_\d \,=\, \E_\d$, then $\k\, \subset\, [\m,\,\m] $. Hence, $\k+ [\m,\,\m]\, \subsetneqq \,\m$.  
  \item If $\d\,=\, [d^{t_d}]$, then $\z(\k)\,=\, 0$.  Hence, $\k+ [\m,\,\m] \,\subsetneqq\, \m$.  
  \item If $\# \O_\d\,\geq\, 1$ and $l\,\geq\, 2$, then $\k + [\m,\,\m] \,=\, \m$.
 \end{enumerate} 
As $\dim_\R \z(\m)=1$, the theorem follows from Theorem \ref{thm-nilpotent-orbit}.
\end{proof}

We next describe the first cohomology groups of nilpotent orbits in the simple Lie algebra $\s\o(p,q)$ when $p >0, \, q >0$.
Recall that in view of \cite[Theorem 6.105, p. 421]{K} and isomorphisms
(iv), (v), (vi), (ix), (x) in \cite[Chapter X, \S 6, pp. 519-520]{He},
to ensure simplicity of $\s\o(p,q)$, we further assume that $(p,q) \not\in \{(1,1), (2,2)\}$; see
\S \ref{sec-so-pq} also.

\begin{theorem}\label{so-pq-H1}
Consider $\s \o(p,q)$, and assume that $p\,\neq\, 2,\, q \,\neq\, 2$ and $(p,q) \neq (1,1)$. Then $H^1(\OC_X,\,\R) \,=\,0$ for all nilpotent elements $X$ in $\s \o(p,q)$.
\end{theorem}

\begin{proof}
 Let $\m,\, \k$ be as in the proof of Theorem \ref{so-pq}. Since $p\,\neq\, 2,\, q \,\neq\, 2$,
we have $\m\,=\, [\m,\,\m]$. Using Theorem 
 \ref{thm-nilpotent-orbit} we conclude that $H^1(\OC_X,\,\R)\,=\,0$.
\end{proof}

We will now consider the remaining cases of $\s \o(p,q)$ which are not covered in Theorem
\ref{so-pq-H1}; they are: $p > 2, q=2$; $p=2, q > 2$ and $(p,q)\in \{ (2,1),(1,2)\}$.
In Section \ref{sec-so-pq} it was observed that when $p>2, q=2$, the non-zero nilpotent orbits
correspond to only four possible signed Young diagrams as given in \eqref{a.1}, \eqref{a.2},
\eqref{a.3}, \eqref{a.4}, and similarly, when $p=2, q>2$, the non-zero nilpotent orbits
correspond to only four possible
 signed Young diagrams as given in \eqref{b.1}, \eqref{b.2}, \eqref{b.3}, \eqref{b.4}.

\begin{theorem}\label{so-pq-H1-2}
Let $\Psi_{{\rm SO} (p,q)^\circ}$ be the parametrization in Theorem \ref{so-pq-parametrization}. Let
$\OC_X \,\in\,\NC({\rm SO}(p,q)^\circ)$. Then the following hold:
  \begin{enumerate}
   \item  Suppose  $(p,q)\,\in\, \{(2,1),\,(1,2)\}$, then $ H^1 ( {\OC}_X,\, \R) \,=\,1$.
   
 \item  Assume that $p\,>\, 2 $ and $q\,=\,2$. \\
   (i)   If $\Psi_{{\rm SO}(p,2)^\circ}(\OC_X)$ is as in either \eqref{a.1} or
\eqref{a.2} or \eqref{a.3}, then $\dim_\R H^1(\OC_X,\, \R) \,=\, 1$.\\
   (ii)  If $\Psi_{{\rm SO}(p,2)^\circ}(\OC_X)$ is as in \eqref{a.4}, then $ H^1(\OC_X,\, \R)\,=\,0$. 
 \item  Assume that $p\,=\, 2 $ and $q\,>\,2$.\\
   (i)   If $\Psi_{{\rm SO}(2,q)^\circ}(\OC_X)$ is as in \eqref{b.1} or \eqref{b.2} or \eqref{b.3},
then  $\dim_\R H^1(\OC_X, \,\R) \,=\, 1$. \\
   (ii)  If $\Psi_{{\rm SO}(2,q)^\circ}(\OC_X)$ is as in \eqref{b.4}, then $H^1(\OC_X, \,\R)\,=\,0$.
  \end{enumerate}
   \end{theorem}

\begin{proof}
As $X \,\neq\, 0$, it lies in a $\s\l_2(\R)$-triple, say $\{ X,\,H,\,Y\}$, in $\s\o(p,q)$. 
 
 {\it Proof of (1): } Let $K'$ be a maximal compact subgroup of $\ZC_{{\rm SO}(p,q)^\circ}(X,H,Y)$. 
 Let $\k'$ be the Lie algebra of $K'$ and $\m$ the Lie algebra of a maximal compact subgroup of ${\rm SO}(p,q)^\circ$ which contains $K'$. 
 When $(p,q) \in \{ (2,1),(1,2)\}$, we have $\dim_\R \m =1$ and $\Psi'_{{\rm SO}(p,q)^\circ}(\OC_X) = [3^1]$. In particular, $\dim_\R L(3-1) =1$. Using Lemma \ref{reductive-part-comp} (4) we have $\k'=0$. Hence, using Theorem \ref{thm-nilpotent-orbit}, we have $\dim_\R H^1(\OC_X, \,\R)=1$.
  
 {\it Proof of (2):} 
 We first prove (2)(i).
Let $\Psi_{{\rm SO}(p,2)^\circ}(\OC_X)$ be as in \eqref{a.1}, \eqref{a.2} or \eqref{a.3}. Let $K$ and $M$ be the maximal compact subgroups of 
$\ZC_{{\rm SO}(p,2)}(X,H,Y)$ and ${\rm SO}(p,2)$ respectively, as defined in the first
paragraph of the proof of Theorem \ref{so-pq-2}(2).
Recall that $K_\O \,:=\, K \cap M^\circ \,=\, K \cap {\rm SO}(p,2)^\circ$ is a maximal compact subgroup of  $\ZC_{{\rm SO}(p,2)^\circ}(X,H,Y)$. 
Let $\k_\O $ and $\m$ be the Lie algebras of $K_\O$ and $M^\circ$ respectively.
Using \eqref{max-cpt-embdd-so-B-1}, \eqref{max-cpt-embdd-so-B-2}, \eqref{max-cpt-embdd-so-B-3} for
the signed Young diagrams \eqref{a.1}, \eqref{a.2}, \eqref{a.3} respectively, 
we observe that in all the cases
$\k_\O \,\subset\, [\m,\,\m]$. Now (2)(i) follows from Theorem \ref{thm-nilpotent-orbit}.
 
 We next prove (2)(ii).
 Let $\widetilde\k $ and $ \m$ be as in the proof of (2)(iv) of Theorem \ref{so-pq-2}. Then using
\eqref{description-of-K}, we have $\widetilde\k+[\m,\,\m]\,=\,\m$. 
The statement (2)(ii) now follows using Theorem \ref{thm-nilpotent-orbit}.
 
 The proofs of (3)(i) and (3)(ii) are similar to those of (2)(i) and (2)(ii) respectively.
\end{proof}

As we deal with nilpotent orbits in simple Lie algebras, to ensure simplicity of $\s\o^*(2n)$, in our next result we
further assume that $n\geq 3$; see \S \ref{sec-so*} also.

\begin{theorem}\label{so*-H1}
 Let $X \,\in\, \s \o^*(2n)$ be a nilpotent element when $n\geq 3$.
 Let $(\d, \,\sgn_{\OC_X}) \,\in \,\YC^{\rm odd}(n)$ be the signed Young diagram of the orbit
$\OC_X$ (that is, $\Psi_{{\rm SO}^*(2n)} (\OC_X) \,= \,(\d,\, \sgn_{\OC_X})$ in the notation of
Theorem \ref{so*-parametrization}). Then
 $$
\dim_\R H^1(\OC_X,\, \R)=\begin{cases}
                               1  & \text{ if } \, \# \O_\d =0 \\
                               0  & \text{ if } \, \# \O_\d \geq 1\,.
                             \end{cases}
$$
\end{theorem}

\begin{proof}
 We  follow the notation of the proof of Theorem \ref{so*}. Using Proposition
\ref{max-cpt-so*-wrt-basis} we have $\k \,\subset\, [\k_\HC,\k_\HC]$ when $ \# \O_\d \,=\, 0$, and
$\k + [\k_\HC,\k_\HC] \,=\, \k_\HC$ when $ \# \O_\d \,\geq\, 1$. As $\dim_\R \z(\k_\HC)\,=\,1$,
the proof is completed by Theorem \ref{thm-nilpotent-orbit}.
\end{proof}

\begin{theorem}\label{sp-n-R-H1}
Let $X \,\in\, \s \p(n,\R)$ be a nilpotent element.
Let $(\d, \,\sgn_{\OC_X}) \,\in\, \YC^{\rm odd}_{-1}(2n)$ be the signed Young diagram of the orbit
$\OC_X$ (that is, $\Psi_{{\rm Sp}(n,\R)}  (\OC_X) \,=\, (\d,\, \sgn_{\OC_X})$
in the notation of Theorem \ref{sp-n-R-parametrization}). Then
$$
\dim_\R H^1(\OC_X,\, \R)=\begin{cases}
                               1  & \text{ if } \, \# \O_\d =0 \\
                               0  & \text{ if } \, \# \O_\d \geq 1\,.
                               \end{cases}
$$
\end{theorem}

\begin{proof}
We follow the notation of the proof of Theorem \ref{sp-n-R}. Using Proposition
\ref{max-cpt-sp-n-R-wrt-basis}, we conclude that $\k \,\subset\, [\k_\HC,\,\k_\HC]$ when $ \# \O_\d\, =\, 0$
and $\k + [\k_\HC,\,\k_\HC] \,=\, \k_\HC$ when $ \# \O_\d \geq\, 1$. As $\dim_\R \z(\k_\HC)\,=\,1$, 
the proof is completed by Theorem \ref{thm-nilpotent-orbit}.
\end{proof}

\section*{Appendix. Basic results on nilpotent orbits}

This appendix is devoted to working out certain details on the structures of the nilpotent orbits in classical Lie algebras so that 
they can be applied in our computations in Section \ref{sec-Second-Cohomology-of-Nilpotent-Orbits}.

We recall that the structures of the nilpotent orbits in classical Lie algebras over $\R$ or $\C$
are due to Springer and Steinberg and they are obtained in \cite[p. 249, 1.6; p. 259, 2.19]{SS} using elementary linear algebra.
However, when classical Lie algebras over $\H$ are considered, due to non-commutativity of $\H$
the above results do not seem to allow direct extensions.
Here, by applying the Jacobson-Morozov Theorem and basics of  $\s\l_2 (\R)$-representation theory,
we overcome this and moreover, work on Lie algebras involving $\R$, $\C$ and $\H$ simultaneously.
This method is suggested in \cite[\S~3.1--3.3, pp. 174-180]{M} and in \cite[\S~9.3, p. 139]{CoM}. 

Following the above approach we also detect an error in \cite[Lemma 9.3.1, p. 139]{CoM} which we point out in
Remark \ref{CM-correction}. This led us to modify the definition of signed Young diagrams as given in \cite[p. 140]{CoM} and
choose different signs in the last columns of the associated matrices, as done in \ref{yd-def2}. 

We follow the notation established in Section \ref{notions-sl2-modules}.
The next lemma is an elementary application of the standard structure theory of irreducible $\s\l_2(\R)$-modules.
Let $\D$ denote either $\R$, $\C$ or $\H$, as before.

\begin{lemma}\label{D-basis}
	Let $V$ be a right $\D$-vector space, and let $\{X,\,H,\,Y\} \,\subset\, \s\l (V )$ be a $\s\l_2(\R)$-triple. 
	Let $d$ be a positive integer such that $M(d-1) \,\neq\, 0$. Let
	$\{ w_1,\, w_2,\, \cdots, \,w_{t_d} \}$ be any $\D$-basis of $L(d-1)$. Then
	\begin{enumerate}
		\item  $X^d w_j\,=\,0$ and $H(X^lw_j)\,=\, X^l w_j (2l+1-d)$ for  all $1 \,\leq\, j \,\leq\, t_d$;
		
		\item the set $\{X^l w_j  \, \mid  \  1\,\leq \,j \,\leq \,t_d, ~ \  0 \,\leq\, l \,\leq\, d-1  \}$ is a $\D$-basis of $M(d-1)$; 
		
		\item  the $\R$-Span of $\{ w_j,\, Xw_j,\, \cdots,\, X^{d-1}w_j\}$ is an irreducible ${\rm Span}_\R \{ X,\,H,\,
		Y\} $-submodule of $M(d-1)$, and moreover, if $W_j$ is the $\D$-Span of $\{ w_j,\,Xw_j,\, \cdots,\, X^{d-1}w_j\}$, then 
		\begin{equation}\label{LXl}
		M(d-1) \, = \, W_1 \oplus W_2 \oplus \cdots \oplus W_{t_d}\,
		=\, L(d-1) \oplus X L(d-1)\oplus \cdots \oplus X^{d-1}L(d-1)\, . 
		\end{equation}
	\end{enumerate}
\end{lemma}

\begin{proof}
	As $M(d-1)$, $V_{Y,0}$ and $L(d-1)$ of are $\D$-subspaces of $V$, it suffices to prove the lemma for
	$\D \,=\, \R$. We have the following relations: for $1\leq j \leq t_d$ and $0\leq l \leq d-1$,
	\begin{equation}\label{eigsp}
	Hw_j\,=\, w_j(1-d)\, , \ \ \  H(X^lw_j)\,=\, w_j(2l+1-d)\, .
	\end{equation}
	Using induction on $l$, it follows from the relations $[H,\, X] \,=\, 2X,\, [H,\,Y] \,=\, -2Y,\, [X,\, Y]\,=\,H$, that
	$Y X^l v\,=\,  (X^{l-1}v )l (d-l)$ for all $v \,\in\, L (d-1)$ and $l\, > \, 0$. This in turn implies that 
	\begin{equation}\label{YXrel}
	Y^lX^lw_j\,=\, w_j(l!)(d-1)(d-2)\cdots (d-l)\, .
	\end{equation}
	Note that $X^{d}w_j\,=\,0$ because $d-1$ is the highest weight. From 
	\eqref{eigsp} it follows that $X^l w_j$ and $X^k w_i$ are linearly independent if 
	$l\,\neq \,k \,;\ 0\,\leq\, l,\, k \,\leq\, d-1 \,;\ 1\,\leq\, j,i \,\leq\, t_d$. Furthermore, \eqref{YXrel} implies that
	for each $l$ with $0\,\leq\, l \,<\, d$, the vectors
	$\{ X^lw_j \,\mid \, 1\,\leq\, j \,\leq\, t_d \} $ are $\R$-linearly independent. It is a basic fact that
	$\dim_\R M(d-1)\,=\, d \dim_\R L(d-1)$.
	Consequently, $$\{X^l w_j  \, \mid \  1\,\leq\, j \,\leq\, t_d, \  0 \,\leq\, l \,\leq\, d-1  \}$$ is a $\R$-basis of
	$M(d-1)$. This proves (2). Part (3) follows immediately from (2).
\end{proof}

When $\D= \R$ or $\C$ Proposition \ref{J-basis} follows from \cite[p. 249, 1.6]{SS}. 

\begin{proposition}\label{J-basis}
	Let $\{X,H,Y\} \subset \s\l (V) $ be a $\s\l_2(\R)$-triple, where $V$ is a right $\D$-vector space. 
	For all $d \,\in \,\N_\d$ and for any $\D$-basis of 
	$L(d -1)$, say, $\{ v^d_j  \,\mid\,  1 \,\leq\, j \,\leq\, t_d\,:=\, \dim_\D L(d-1) \}$ the following two hold:
	\begin{enumerate}
		\item $ X^d v^d_j \,=\, 0$ and  $H(X^l v^d_j)\,= \, X^lv^d_j(2l + 1 - d) $ for $ 1\,\leq\, j \,\leq\, t_d$,
		$0 \,\leq\, l \,\leq\, d-1$, $d \,\in\, \N_\d$.
		
		\item  For all $d \,\in \,\N_\d$, the set $\{ X^lv^d_j \,\mid\,  1\,\leq\, j \,\leq\, t_d,\  0 \,\leq\, l
		\,\leq\, d-1 \}$ is a $\D$-basis of $M(d-1)$. In particular, 
		$\{ X^l v^d_j \,\mid  \, 1\,\leq\, j \,\leq\, t_d, \  0 \,\leq\, l \,\leq\, d-1,\ d \,\in\, \N_\d \}$
		is a $\D$-basis of $V$. 
	\end{enumerate}
\end{proposition}

\begin{proof}
	This follows from Lemma \ref{D-basis} and \eqref{isotypicalcomp}.
\end{proof} 

Henceforth, $\sigma \,:\, \D \,\longrightarrow\, \D$ will denote either the identity map or $\sigma_c$ (defined
in Section \ref{sec-epsilon-sigma-forms}) when $\D$ is $\C$ or $\H$.  Let $V$ be a right $\D$-vector space,
 and let $\langle \cdot,\, \cdot \rangle \,:\, V \times V \,\longrightarrow\, \D \ $ be a $\epsilon$-$\sigma$ 
 Hermitian form. Let  $X$ be a non-zero nilpotent element
in $\s\u (V, \,\langle \cdot,\, \cdot \rangle )$. Using Theorem  \ref{Jacobson-Morozov-alg}, there exists
$ H, \, Y\,\in\, \s\u (V,\, \langle \cdot,\, \cdot \rangle )$ such that $\text{Span}_\R \{ X,\,H,\,Y \}$
is isomorphic to $ \s\l_2 (\R)$. Thus, $V$ becomes a $\text{Span}_\R \{ X,\,H,\,Y\}$-module.

We record the following straightforward but useful fact.

\begin{lemma}[{cf.\,\cite[\S~2.4, p. 171]{M}}]\label{extreme} 
	Let $\sigma \,: \,\D \,\longrightarrow\, \D$ be either the identity map or $\sigma_c$ when $\D$ is $\C$ or $\H$.
	Let $ \langle \cdot,\, \cdot \rangle \,:\, V \times V \,\longrightarrow\, \D \ $ be a $\epsilon$-$\sigma$ Hermitian form.
	Suppose  $A \in {\rm End}_\D(V)$ such that $\langle Ax,\, y \rangle +\langle x,\,Ay \rangle \,=\,0 $ for all $x,\,y\,\in\,
	V$. Let $v$ and $ w$ be two nonzero elements in $V$ such that $Av\,=\,v \lambda$ and $Aw \,= \, w\mu $ for some 
	$\lambda , \,\mu \,\in\, \R$. If $\lambda + \mu \,\neq\, 0$, then $~\langle v,\,w\rangle \,= \,0$.
\end{lemma}

\begin{proof} As $\langle Av,\,w\rangle + \langle v, \,Aw\rangle \,=\,0$, it follows immediately that
	$ \langle v,\, w\rangle (\lambda + \mu )\, =\, 0 $. Now the lemma follows because $ \lambda + \mu \,\neq\, 0 $.
\end{proof}

\begin{lemma}[{cf.\,\cite[\S~3.2, p. 178]{M}}]\label{ortho-isotypical}
	Let $V$ be a right $\D$-vector space, and $\epsilon\,=\, \pm 1$. Let $\sigma$ be as in Lemma \ref{extreme}, and
	let $$\langle \cdot,\, \cdot \rangle \,: \,V \times V \,\longrightarrow\, \D \ $$ be a $\epsilon$-$\sigma$ Hermitian form.  
	Let $\{X,\,H,\,Y\} \,\subset\, \s\u (V,\, \langle \cdot,\,
	\cdot \rangle ) $ be a $\s\l_2(\R)$-triple.  Let $\d$ be as in \eqref{partition-symbol}, and
	let $d,\, d' \,\in\, \N_\d$ be such that $d \,\neq\, d'$.
	Then $M(d-1) $ and $M(d' -1)$ are orthogonal with respect to $\langle \cdot,\,\cdot \rangle$.
	In particular, the Hermitian form $\<>$ on $M(d-1)$ is non-degenerate for all $d$.
\end{lemma}

\begin{proof}
	We may assume that $d \,>\,d'$. Let $v \,\in\, L(d-1)$ and $u \,\in\, L(d'-1)$. By Lemma \ref{extreme} 
	we have that $\langle v, X^lu \rangle =0$ when $0 \leq l \leq d'-1$. Moreover, $X^l u = 0 $ if $l \geq d'$. Thus 
	$\langle v,\, X^lu \rangle \,=0\,$ for all $l \,\geq\, 0$.  
	Hence, $\langle X^h v,\, X^lu \rangle \,=\,(-1)^h \langle  v,\, X^{l+h}u \rangle \,=\,0$. Now the lemma follows from
	\eqref{LXl} of Lemma \ref{D-basis}.
\end{proof}

The following lemma, which further decomposes each isotypical component $M(d-1)\, \subset\, V$ into orthogonal subspaces, 
seems basic. However, as we are unable to locate it in the literature, we include a proof here. 

\begin{lemma}\label{unitary-J-basis-lemma} 
	Let $V$ be a right $\D$-vector space, and $\epsilon \,= \,\pm 1$. Let $\sigma \,:\, \D \,\longrightarrow\, \D$ be either
	the identity map or $\sigma_c$ when $\D$ is $\C$ or $\H$.
	Let $\langle \cdot,\, \cdot \rangle \,:\, V \times V \,\longrightarrow\,\D$ be a (non-degenerate) $\epsilon$-$\sigma$
	Hermitian form. Let $\{X,\,H,\,Y\} \,\subset\, \s\u (V, \,\langle \cdot,\,\cdot \rangle ) $ be a $\s\l_2(\R)$-triple. Let
	$\d$ be as in \eqref{partition-symbol}, $d \,\in\, \N_\d$ and $t_d \,:= \,\dim_\D L(d-1)$.
	Then there exists a $\D$-basis $\{ w_1,\, \cdots,\, w_{t_d} \}$ of $L(d-1)$ such that
	the set $$\{X^l w_j  \, \mid  \  1\,\leq\, j \,\leq\, t_d, \  0 \,\leq \,l \,\leq\, d-1  \}$$
	is a $\D$-basis of $M(d-1)$, and moreover,
	the value of $\langle \cdot,\, \cdot \rangle$ on a pair of these basis vector is $0$, except in the following cases:
	\begin{itemize}
		\item[{\rm (1)}] If $\sigma \,= \,\sigma_c $, then $ \langle X^lw_{j},\, X^{d -1-l}w_{j} \,\rangle\, \in \D^*$.
		
		\item[{\rm (2)}] If $\sigma \,= \,{\rm Id}$ and $\epsilon \,=\,1$, then $\langle X^lw_{j},\, X^{d -1-l}w_{j} \rangle
		\,\in \,\D^*$ for $ d$ odd, and $$\langle X^lw_{j},\, X^{d -1-l}w_{j+1} \rangle \,\in\, \D^*$$ for $ d$ even and $j$ odd.
		
		\item[{\rm (3)}] If $\sigma \,=\, {\rm Id}$ and $\epsilon \,=\, -1$, then
		$\langle X^lw_{j},\, X^{d -1-l}w_{j} \rangle \,\in\, \D^* $ for $d$ even, and 
		$$ \langle X^lw_{j},\, X^{d-1 -l}w_{j+1} \rangle \,\in\, \D^* $$ for $d$ odd and $j$ odd.
	\end{itemize}
\end{lemma}

\begin{proof}
	We use induction on $\dim_\D V$. The proof is divided into two parts.
	
	{\it Part 1.}\, In this part assume that one of the following two holds:
	\begin{itemize}
		\item{} $\D \,= \,\R$, $\sigma \,= \,\text{Id}$, and $(-1)^{d-1} \epsilon \,=\,1$;
		
		\item{}  $\D\,=\,\H$ or $ \C $ and $ \sigma \,=\, \sigma_c $.
	\end{itemize}
	
	We claim that there is an element $x_1 \,\in\, L(d-1)$ such that $\langle x_1,\, X^{d-1}x_1 \rangle \,\neq\, 0$.
	
	To prove the claim by contradiction, assume that $\langle x,\, X^{d-1}x \rangle \,=\, 0$ for all $x \,\in \,L(d-1)$. Lemma
	\ref{extreme} implies that $\langle z_{1},\, X^{l}z_{2}\rangle\,=\, 0$ 
	for $l \,\neq\, d-1$ and $z_1,\,z_2 \,\in\, L(d-1)$. Fix a nonzero element $x \,\in\,  L(d-1)$. 
	Since $\langle \cdot ,\, \cdot \rangle $ is non-degenerate on $M(d-1) \times M(d-1)$, there exists
	an element $y\,\in\, L(d-1)$ such that $\langle x, \,X^{d-1}y\rangle \,\neq \,0$. 
	As $x+y \,\in\, L(d-1)$, we also know that $\langle x+y, \ X^{d-1}(x+y) \rangle  \,=\, 0$.
	
	Now we will arrive at a contradiction considering the two cases separately.
	
	First assume that $\D \,=\, \R$, $\sigma \,=\, \text{Id}$ and $(-1)^{d-1} \epsilon \,=\,1$.
	Then $$0 \,=\,\langle x+y, \,\ X^{d-1}(x+y) \rangle \,=\, 2 \langle x,\, X^{d-1}y\rangle\, .$$
	This is evidently a contradiction.
	
	Next assume that $\D\,=\,\C$. Writing $\langle x,\, X^{d-1}y\rangle \,= \,a_1 + \sqrt{-1} b_1$ where $a_1,b_1\in \R$ and multiplying $y$ by an
	appropriate scalar from $\C$ if required, we may assume that $a_1 \,\neq\, 0$ as well as $b_1\,\neq\, 0$.
	Now the condition $\langle x+y, \, X^{d-1}(x+y) \rangle  \,=\, 0$ implies that $(a_1 + \sqrt{-1} b_1)+(-1)^{d-1} \epsilon(a_1 - \sqrt{-1} b_1)
	\,=\,0$. This contradicts the fact that both $a_1$ and $b_1$ are non-zero.
	
	Finally, assume that $\D\,=\, \H$. Writing $\langle x,\, X^{d-1}y\rangle \,=\, a_1 + \ib b_1 + \jb c_1 + \kb d_1$ where $a_1,b_1,c_1,d_1\in \R$ and  
	multiplying $y$ by an appropriate scalar from $\H$ if needed, we may assume that $a_1 \,\neq\, 0$ and $b_1\,\neq\, 0$. Then,
	\begin{align*}
	\langle x&\,, \ X^{d-1}y\rangle + (-1)^{d-1} \epsilon \sigma (\langle x, \ X^{d-1}y\rangle)  \,=\,0\, .
	\end{align*} 
	From this it follows that  $(a_1 + \ib b_1 + \jb c_1 + \kb d_1) + (-1)^{d-1} \epsilon  (a_1 - \ib b_1 - \jb c_1- \kb d_1)\,=\,0$.
	This gives a contradiction as both $a_1$ and $b_1$ are nonzero. This completes the proof of the claim.
	
	Let $W$ be the $\D$-Span of $\{ X^l x_1  \,\mid  \, 0 \,\leq\, l \,\leq \,d-1  \}$, where $x_1$ is the element
	of $L(d-1)$ in the above claim. As the vectors
	$\{ X^l x_1 \, \mid  \, 0 \,\leq\, l \,\leq\, d-1  \}$ are $\D$-linearly independent, and $\langle x_1, \,X^{d-1}x_1 \rangle
	\,\neq\, 0$, it follows that 
	$\langle \cdot, \,\cdot \rangle $  is non-degenerate on $W$. Hence, $$V \,= \,W \oplus W^\bot \, ,$$ where  
	$W^{\bot} \,:=\, \{ v \,\in\, V  \,\mid\,  \langle v,\, W \rangle \,=\,0 \} $.
	As $\{X,\,H,\,Y\} \,\subset\, \s\u (V, \<>)$, it follows immediately that $X,\,H,\,Y$ leave $W^{\bot}$ invariant.
	Let $$X_1 \,:=\, X\vert_{W^\bot}\,,\ \ H_1 \,:=\, H\vert_{W^\bot}\, ,\ \ Y_1 \,:=\, Y\vert_{W^\bot}\, .$$
	Let $\<>'$ be the restriction of $\<>$ to ${W^\bot}$. 
	Then $$\{ X_1,\,H_1,\, Y_1\}\,\subset\, \s\u(W^\bot , \<>')$$ is a $\s\l_2(\R)$-triple. 
	Let $M_{W^\bot} (d-1)$ be the isotypical component of $W^\bot$ consisting of sum of all $\R$-subspaces $B$ of $W^\bot$
	with $\dim_\R B \,=\, d$ which are also irreducible $\text{Span}_\R \{ X_1 ,H_1 ,Y_1 \}$-submodules of $W^\bot$. Then 
	we have $M_{W^\bot} (d-1) \,=\, W^\bot \cap M (d-1)$ and $M( d-1) \,= \,W \oplus M_{W^\bot} (d-1)$. 
	Since $\dim_\D W^\bot \,<\, \dim_\D V$, from the induction hypothesis, $M_{W^\bot} (d-1)$
	has a $\D$-basis satisfying (1), (2), (3) of the lemma. This $\D$-basis of $M_{W^\bot} (d-1)$
	together with the $\D$-basis $$\{X^l x_1 \,\mid  \, 0 \,\leq\, l\,\leq\, d-1\}$$ of
	$W$ will give the required $\D$-basis of $M(d-1)$. This completes the proof using induction on $\dim_\D V$.
	
	{\it Part 2 :}\,
	Here we deal with the remaining case where $\D\,=\,\R$, $\sigma\,=\, \text{Id}$ and $(-1)^{d-1} \epsilon\,=\, -1$.
	
	For all $x \,\in\, L(d-1)$, as $$\langle x , \,  X^{d-1}x \rangle \,=\, (-1)^{d-1} \epsilon \langle x , \,  X^{d-1}x \rangle
	\,=\, -\langle x , \,  X^{d-1}x \rangle\, ,$$
	it is clear that $\langle x , \,  X^{d-1}x \rangle \,=\, 0$. Lemma \ref{extreme} gives
	that $$\langle z_{1},\,
	X^{l}z_{2}\rangle\,=\, 0$$ for $l \,\neq\, d-1, \ z_1,\,z_2 \,\in\, L(d-1)$. 
	Fix any nonzero $x_1 \,\in\,L(d-1)$.
	Since $\langle \cdot ,\, \cdot \rangle$ is non-degenerate on $M(d-1) \times M(d-1)$, there exists $y_1\,\in\, L(d-1)
	\setminus x_1\D$ such that $\langle x_1,\, X^{d-1}y_1\rangle \,\neq\, 0$.
	Let $W'$ be the $\D$-Span of $\{ X^l x_1, \,X^l y_1\, \mid  \ 0 \,\leq\, l \,\leq\, d-1\}$.  As the vectors $\{ X^l x_1, \,
	X^l y_1 \,\mid  \ 0 \,\leq\, l \,\leq\, d-1\}$
	are $\D$-linearly independent, and $\langle x_1,\, X^{d-1}y_1\rangle\,\neq\, 0$, it follows that 
	$\langle \cdot, \cdot \rangle $ is non-degenerate on $ W' $. As before, define 
	$ W'^{\bot} \,:=\, \{ v \,\in\, V \, \mid  \, \langle v,\, W' \rangle \,=\,0 \}$.
	As $V \,=\, W' \oplus W'^\bot $, and $\dim_\D W'^\bot \,<\, \dim_\D V$, repeating the argument in part 1 the
	proof is completed. 
\end{proof}

The next result is an analogue of Proposition \ref{J-basis} in the presence of a $\epsilon$-$\sigma$ Hermitian form.
When $\D= \R$ or $\C$ Proposition \ref{unitary-J-basis} follows from \cite[p. 259, 2.19]{SS}. 

\begin{proposition}\label{unitary-J-basis}
	Let $V$ be a right $\D$-vector space, $\epsilon \,= \,\pm 1$, $\sigma \,\colon\, 
	\D \,\longrightarrow\, \D$ is either the identity map or it is $\sigma_c$ when $\D$ is $\C$ or $\H$.
	Let $\langle \cdot,\, \cdot \rangle \,:\, V \times V \,\longrightarrow\, \D$ be a $\epsilon$-$\sigma$ Hermitian form.
	Let $\{X,\,H,\,Y\}\,\subset\, \s\u (V,\, \langle \cdot,\, \cdot \rangle ) $ be a $\s\l_2(\R)$-triple. 
	Let $d \,\in\, \N_\d$ and $t_{d}\,:= \,\dim_\D L(d-1)$. 
	Then for all $d \,\in\, \N_\d$, there exists a $\D$-basis $\{ {v^d_j}\,\mid \, 1\,\leq\, j \,\leq \,t_d \}$ of 
	$L(d -1)$ such  that the following three hold:
	\begin{enumerate}
		\item $X^d v^d_j \,=\, 0$ and $H(X^l v^d_j)\,=\,  X^l v^d_j(2l + 1 - d)$ for all
		$1\,\leq\, j \,\leq\, t_d$, $0 \,\leq\, l \,\leq\, d-1$ and $d \,\in\, \N_\d$.
		
		\item  For all $d \,\in\, \N_\d$, the set $\{ X^lv^d_j\,\mid\,  1\,\leq\, j \,\leq\, t_d,\ 0 \,\leq\, l
		\,\leq\, d-1 \}$ is a $\D$-basis of $M(d-1)$. In particular, 
		$$\{ X^l v^d_j \,\mid\, 1\,\leq\, j \,\leq\, t_d,\  0 \,\leq\, l \,\leq\, d-1,\ d \,\in\, \N_\d\}$$
		is a $\D$-basis of $V$. 
		
		\item  The value of $\langle \cdot  , \,\cdot \rangle$ on any pair of the above basis vectors is $0$, except in the
		following cases: 
		\begin{itemize} 
			\item  If $\sigma \,=\, \sigma_c $, then $ \langle X^l v^d_j,\, X^{d -1 -l} v^d_j \rangle \,\in\, \D^*$.
			
			\item If $\sigma \,=\, {\rm Id}$ and $\epsilon \,=\,1$, then $\langle X^l v^d_j ,\, X^{d -1-l} v^d_j \rangle \,\in\, \D^*$ 
			when $d \,\in\, \O_\d$, and 
			$$\langle X^l v^d_j ,\, X^{d-1 -l} v^d_{j+1} \rangle \,\in\, \D^* $$ when $ d \,\in\,\E_\d $ and $j$ is odd. 
			
			\item If $\sigma \,= \,{\rm Id}$ and $\epsilon \,= \,-1$, then
			$\langle X^l v^d_j ,\, X^{d -1-l} v^d_j \rangle \,\in\, \D^*$ when $d \in \E_\d$, and
			$$\langle X^l v^d_j ,\, X^{d-1 -l} v^d_{j+1} \rangle \,\in\, \D^* $$ when $d\,\in\,\O_\d $ and $j$ is odd.
		\end{itemize}
	\end{enumerate}
\end{proposition}

\begin{proof}
	Lemma \ref{ortho-isotypical} gives the
	orthogonal decomposition $ V  \,=  \,\bigoplus_{d \in \N_\d} M(d-1)$ with respect to the
	non-degenerate form $\<>$ on $V$. The proposition now follows from Lemma \ref{unitary-J-basis-lemma}.
\end{proof} 

Let $ (\cdot, \cdot)_{d}$ be the form on $ L(d-1)$ as defined in \eqref{new-form}.

\begin{remark}\label{inv-form-old-new}
	Observe that if $\{ X,\, H, \,Y\}$ is a $\s\l_2 (\R)$-triple in the Lie algebra $\s\u (V, \<>)$,
	and $g \,\in\, \ZC_{{\rm GL}(V)}(X, H, Y)$, then 
	$(g x,\, gy)_{d}\,= \,(x,\, y)_{d}$ for all $x,\,y \,\in\, L(d-1)$ if and only if $ \langle g v ,\, gw \rangle
	\,=\, \langle v  ,\, w \rangle$ for all $v,\,w \,\in\, M(d-1)$.
\end{remark}

\begin{remark}\label{unitary-J-basis-rmk}
	It is easy to see that when $\<>$ is Hermitian, then the form $(\cdot ,\,\cdot)_{d}$ is Hermitian (respectively, skew-Hermitian)
	if $d$ is odd (respectively, even). 
	When $\<>$ is symmetric, then $(\cdot ,\,\cdot)_{d}$ is symmetric (respectively, symplectic) if $d$ is odd
	(respectively, even).
	When $\<>$ is symplectic, then  $(\cdot ,\,\cdot)_{d}$ is  symplectic (respectively, symmetric) if $d$ is odd
	(respectively, even).
	Lastly, when $\<>$ is skew-Hermitian, then  $(\cdot ,\, \cdot)_{d}$ is skew-Hermitian (respectively, Hermitian)
	if $d$ is odd (respectively, even).
	
	From Lemma \ref{unitary-J-basis-lemma} it follows that $(\cdot ,\, \cdot)_{d}$ is
	non-degenerate.
	The $\D$-basis elements $$\{ v^d_j \,\mid\,  1\,\leq\, j\, \leq\, t_d \}$$ of $L(d-1)$ in Proposition \ref{unitary-J-basis}
	are modified as follows:
	\begin{enumerate} 
		\item 
		If $\D\,=\,\R$ and  $\epsilon \,=\,1$, by suitable rescaling each element of
		$\{ v^d_j \,\mid\,  1\,\leq\, j\, \leq\, t_d \}$ we may assume that 
		\begin{itemize}
			\item  $\langle v^d_j,\, X^{d-1} v^d_j \rangle\,=\, \pm 1$ when
			$d \,\in\, \O_\d$, and
			
			\item $\langle v^d_j ,\, X^{d-1} v^d_{j+1} \rangle \,=\, 1$ when $d \,\in\, \E_\d$ and $j$ is odd.
		\end{itemize}
		In particular, $(v_1^d,\, \cdots,\, v_{t_d}^d)$ is an standard orthogonal basis of $L(d-1)$ with respect to
		$(\cdot, \,  \cdot)_{d}$ for $d \,\in\, \O_\d$. 
		If $\D\,=\,\R$ and  $\epsilon \,=\,-1$, analogously we may assume that the elements of
		the $\R$-basis  $$\{ v^d_j \,\mid\,  1\,\leq\, j \,\leq \,t_d \}$$ of $L(d-1)$ in Proposition \ref{unitary-J-basis}
		satisfy the condition that 
		\begin{itemize}
			\item $\langle v^d_j,\, X^{d-1} v^d_j \rangle \,=\, \pm 1$ when
			$d \,\in \,\E_\d$, and
			
			\item $\langle v^d_j,\, X^{d-1} v^d_{t_d/2 +j} \rangle \,=\, 1 $ when
			$d \,\in\, \O_\d$ and $1\,\leq\, j\,\leq\, t_d/2$.
		\end{itemize}
		In particular, $(v_1^d,\,\cdots,\, v_{t_d}^d)$ is an orthogonal basis for $d \,\in \,\E_\d$, and
		$$(v_1^d,\, \cdots,\, v_{t_d/2}^d;\, v^d_{t_d/2 +1},\, \cdots ,\,v_{t_d}^d )$$ is a symplectic basis for
		$d \,\in\, \O_\d$ of $L(d-1)$ with respect to $(\cdot, \, \cdot)_{d}$.
		
		\item 
		If $\D\,=\,\C$, $\epsilon \,=\,1$ and $ \sigma \,=\,\sigma_c$, rescaling the elements of the $\C$-basis
		$\{ v^d_j \,\mid \, 1\,\leq\, j \,\leq\, t_d \}$ we may assume that
		\begin{itemize}
			\item $\langle v^d_j ,\, X^{d-1} v^d_j \rangle\,=\, \pm 1$ when 
			$d\,\in\, \O_\d$, and
			
			\item $\langle v^d_j ,\, X^{d-1} v^d_j \rangle \,=\, \pm \sqrt{-1}$ when $d \,\in\,  \E_\d$.
		\end{itemize}
		In particular, $(v_1^d,\, \cdots,\, v_{t_d}^d)$ is an orthogonal basis of $L(d-1)$ with respect to $(\cdot, \, \cdot)_{d}$
		for $d \,\in\, \N_\d$.
		
		\item 
		If $\D\,=\, \H$, $\epsilon\, =\,1$ and $\sigma \,=\, \sigma_c$, after rescaling and conjugating 
		the elements of the $\H$-basis $\{ v^d_j \,\mid\,  1\,\leq\, j \,\leq\, t_d \}$ of $L(d-1)$
		by suitable scalars (see Lemma \ref{H-conjugate}) the elements of the $\H$-basis, we may assume that 
		\begin{itemize}
			\item $\langle v^d_j,\, X^{d-1} v^d_j \rangle\,=\, \pm 1$ when $d \,\in\, \O_\d$, and
			
			\item $\langle v^d_j,\, X^{d-1} v^d_j \rangle\,=\, \jb$ when $d \,\in\, \E_\d$.
		\end{itemize} 
		If $\D\,=\,\H$, $\epsilon \,=\, -1$ and $\sigma  \,=\, \sigma_c$, analogously we may assume that the elements of
		the $\H$-basis $\{ v^d_j \,\mid\,  1\,\leq\, j \,\leq \,t_d \}$ of $L(d-1)$
		satisfy
		\begin{itemize}
			\item $\langle v^d_j, \,X^{d-1} v^d_j \rangle\,=\, \pm 1$ when $d \,\in\, \E_\d$, and
			
			\item $\langle v^d_j,\, X^{d-1} v^d_j \rangle\,= \,\jb$ when $d \,\in\, \O_\d$. 
		\end{itemize}
	\end{enumerate}
\end{remark}

Let $( v^{d}_1,\, \cdots,\, v^{d}_{t_{d}} )$ be an ordered $\D$-basis of $L(d-1)$ as in 
Proposition \ref{unitary-J-basis} satisfying the properties as Remark \ref{unitary-J-basis-rmk}.
The proofs of the following lemmas are straightforward and they are omitted.

\begin{lemma}\label{orthogonal-basis-R}
	Let  $\D\,=\, \R$, $\sigma \,=\,{\rm Id}$ and $\epsilon \,=\,1$. Fix $d \,\in\, \N_\d $ and $1 \,\leq\, j \,\leq\, t_d$.
	\begin{enumerate}
		\item If $\eta \,\in\, \E_\d$ and $j$ is odd, define
		$$
		{w}^{\eta}_{jl}\,:=\,  
		\begin{cases}
		\big(X^l v^{\eta}_j  + X^{\eta-1-l} v^{\eta}_{j+1} \big)\frac{1}{\sqrt{2}}  &  \text{ if }\ 0 \,\leq\, l \,\leq\,
		\eta -1 \vspace*{.2cm} \\ 
		\big(  X^{2\eta-1-l} v^{\eta}_j  - X^{l-\eta} v^{\eta}_{j+1} \big) \frac{1}{\sqrt{2}}   & \text{ if }\ \eta\,\leq\, l
		\,\leq\, 2\eta -1.
		\end{cases}
		$$
		Then 
		$ \langle  {w}^{\eta}_{jl},\, {w}^{\eta}_{j(2\eta-1-l)} \rangle \,=\,0$, and $ \langle  {w}^{\eta}_{jl}, \,
		{w}^{\eta}_{jl} \rangle  \,=\, (-1)^{l}\,  \langle  v^{\eta}_j,\, X^{\eta-1}v^{\eta}_{j+1} \rangle$.

		\item For $\theta \,\in\, \O_\d$, define
		$$ \ \
		{ w}^{\theta}_{jl}\,:=\,
		\begin{cases}
		\big( X^l v^{\theta}_j  +  X^{\theta-1-l} v^{\theta}_j \big)\frac{1}{\sqrt{2}}    & \text{ if }\ 0\,\leq\,l\,<\, (\theta-1)/2 \vspace*{.2cm}\\
		X^l v^{\theta}_j     & \text{ if }\  l \,= \,(\theta-1)/2 \vspace*{.2cm}\\
		\big(   X^{\theta-1-l} v^{\theta}_j - X^l v^{\theta}_j \big) \frac{1}{\sqrt{2}}  & \text{ if }\ (\theta-1)/2 \,<\,
		l \,\leq \,\theta -1.
		\end{cases} 
		$$
		Then 
		$ \langle{w}^{\theta}_{jl}, \,{w}^{\theta}_{j(\theta-1-l)} \rangle \,=\,0$
		and
		$$
		\langle  {w}^{\theta}_{jl}, \,{w}^{\theta}_{jl} \rangle  \,=\, 
		\begin{cases}
		(-1)^{l}   \langle  v^{\theta}_j, X^{\theta-1}v^{\theta}_j \rangle  & \text{ if }\ 0 \,\leq\, l \,<\, (\theta-1)/2 \vspace*{.14cm}\\
		
		(-1)^{l}  \langle  v^{\theta}_j, X^{\theta-1}v^{\theta}_j \rangle   & \text{ if } \  l \,=\, (\theta-1)/2 \vspace*{.14cm}\\
		
		(-1)^{l+1} \langle  v^{\theta}_j, X^{\theta-1}v^{\theta}_j \rangle  &    \text{ if }\ (\theta-1)/2 \,<\, l \,\leq\, \theta -1.
		\end{cases}
		$$
	\end{enumerate}
	Therefore, for any $\theta \,\in\, \O_\d$, $$\langle  {w}^{\theta}_{jl},\, {w}^{\theta}_{jl')} \rangle \,=\, 0$$
	when $l \,\neq\, l'$ and $0 \,\leq\, l,\, l' \,\theta-1$.
\end{lemma}

\begin{lemma}\label{orthogonal-basis-C}
	Let  $\D\,= \,\C$, $\sigma \,= \,\sigma_c $ and $\epsilon \,=\,1$. Fix $d \,\in\, \N_\d$ and $1\,\leq\, j\,\leq\, t_d$.
	\begin{enumerate}
		\item For $\eta \,\in\, \E_\d$, define
		$$
		{\widetilde w}^{\eta}_{jl}\,:=\,
		\begin{cases}
		\big({X^l v^{\eta}_j } + X^{\eta-1-l} v^{\eta}_j\sqrt{-1} \big) \frac{1}{\sqrt{2}}  & \text{ if }\ 0 \,\leq\, l \,<\, \eta/2  \vspace*{.2cm} \\
		\big( X^{\eta-1-l} v^{\eta}_j - X^l v^{\eta}_j\sqrt{-1} \big) \frac{1}{\sqrt{2}}    &  \text{ if }\ \eta/2\,\leq\, l\,\leq\, \eta -1.
		\end{cases}
		$$
		Then $ \langle  \widetilde{w}^{\eta}_{jl}, \,\widetilde{w}^{\eta}_{j(\eta-1-l)} \rangle\,=\,0$  and $ \langle
		\widetilde{w}^{\eta}_{jl},\, \widetilde{w}^{\eta}_{jl} \rangle  \,=\, (-1)^l\sqrt{-1}  \langle  v^{\eta}_j, X^{\eta-1}v^{\eta}_j \rangle.$
		
		\item For $\theta \,\in\, \O_\d$, define
		$$ \ 
		{\widetilde w}^{\theta}_{jl}\,:=\,
		\begin{cases}
		\big( X^l v^{\theta}_j  +  X^{\theta-1-l} v^{\theta}_j  \big) \frac{1}{\sqrt{2}} & \text{ if }\ 0 \,\leq\, l \,<\, (\theta-1)/2 \vspace*{.2cm}\\
		
		X^l v^{\theta}_j     & \text{ if }\  l \,=\, (\theta-1)/2 \vspace*{.2cm} \\
		
		\big( X^{\theta-1-l}v^{\theta}_j - X^l v^{\theta}_j\big) \frac{1}{\sqrt{2}}  &  \text{ if }\ (\theta-1)/2 \,<\, l \,\leq\, \theta -1.
		\end{cases}
		$$
		Then 
		$ \langle  \widetilde{w}^{\theta}_{jl}, \,\widetilde{w}^{\theta}_{j(\theta-1-l)} \rangle \,=\,0$, 
		and
		$$
		\langle  \widetilde{w}^{\theta}_{jl}, \widetilde{w}^{\theta}_{jl} \rangle  \,=\, 
		\begin{cases}
		(-1)^{l}  \langle  v^{\theta}_j, X^{\theta-1}v^{\theta}_j \rangle  & \text{ if }\ 0 \,\leq \,l \,<\, (\theta-1)/2 \vspace*{.2cm} \\
		
		(-1)^{l}  \langle  v^{\theta}_j, X^{\theta-1}v^{\theta}_j \rangle  & \text{ if }\  l \,=\, (\theta-1)/2 \vspace*{.2cm}\\
		
		(-1)^{l+1}  \langle  v^{\theta}_j, X^{\theta-1}v^{\theta}_j \rangle & \text{ if }\ (\theta-1)/2 \,<\, l \,\leq\, \theta -1.
		\end{cases}
		$$
	\end{enumerate}
	Therefore, for any $\theta \,\in\, \O_\d$, $$\langle  \widetilde{w}^{\theta}_{jl},\, \widetilde{w}^{\theta}_{jl')} \rangle
	\,=\, 0$$ when $l \,\neq\, l'$ and $0 \,\leq\, l ,\,l' \,\leq \,\theta-1$.
\end{lemma}

\begin{lemma}\label{orthogonal-basis-H}
	Let  $\D= \H$, $\sigma = \sigma_c$ and $\epsilon \,=\,1$. Fix ${d}$ and $1 \,\leq\, j \,\leq\, t_{d}$. 
	\begin{enumerate}
		\item For $\eta\,\in\, \E_\d$, define
		$$
		\widehat{w}_{jl}^{\eta}\,:=\,
		\begin{cases}
		\big( {X^l v^{\eta}_j } + X^{\eta-1-l} v^{\eta}_j \alpha_j \big)\frac{1}{\sqrt{2}}   & \text{ if } \ 
		0 \,\leq\, l \,<\, \eta/2   \vspace*{.2cm} \\
		\big( X^{\eta-1-l} v^{\eta}_j - X^l v^{\eta}_j \alpha_j \big)\frac{1}{\sqrt{2}} 
		& \text{ if }  \ \eta/2 \,\leq\, l\,\leq\, \eta -1
		\end{cases}
		$$ 
		where $\alpha_j \,= \,\langle v^{\eta}_j , \,X^{\eta-1} v^{\eta}_j  \rangle$. Then
		$$ \langle  \widehat{w}^{\eta}_{jl},\, \widehat{w}^{\eta}_{j(\eta-1-l)} \rangle \,=\,0 \ \ \text{ and }\ \
		\langle \widehat{w}^{\eta}_{jl},\, \widehat{w}^{\eta}_{jl} \rangle\,=\, (-1)^{l+1}{\rm Nrd}(\langle v^{\eta}_j ,\,
		X^{\eta-1} v^{\eta}_j \rangle)\, .$$
		
		\item When $\theta \,\in\, \O_\d$, define
		$$
		{\widehat w}^{\theta}_{jl}\,:=\,
		\begin{cases}
		\big( X^l v^{\theta}_j  +  X^{\theta-1-l} v^{\theta}_j \big)\frac{1}{\sqrt{2}} & \text{ if }\ 0 \,\leq\, l \,<\, (\theta-1)/2 \vspace*{.1cm}\\
		X^l v^{\theta}_j    & \text{ if }\  l \,= \,(\theta-1)/2 \vspace*{.1cm}\\
		\big( X^{\theta-1-l} v^{\theta}_j - X^l v^{\theta}_j \big)\frac{1}{\sqrt{2}} &  \text{ if }\ (\theta-1)/2 \,<\, l \,\leq\, \theta -1.
		\end{cases} 
		$$
		Then 
		$ \langle \widehat{w}^{\theta}_{jl},\, \widehat{w}^{\theta}_{j(\theta-1-l)} \rangle \,=\,0 $, 
		and
		$$
		\langle \widehat{w}^{\theta}_{jl},\, \widehat{w}^{\theta}_{jl} \rangle\, =\, 
		\begin{cases}
		(-1)^{l}\,   \langle  v^{\theta}_j, X^{\theta-1}v^{\theta}_j \rangle & \text{ if }\ 0 \,\leq\, l \,<\, (\theta-1)/2 \vspace*{.14cm}\\
		(-1)^{l}  \langle  v^{\theta}_j, X^{\theta-1}v^{\theta}_j \rangle & \text{ if }\  l \,=\, (\theta-1)/2 \vspace*{.14cm}\\
		(-1)^{l+1}\,  \langle  v^{\theta}_j, X^{\theta-1}v^{\theta}_j \rangle  &    \text{ if }\ (\theta-1)/2 \,<\, l \,\leq\, \theta -1.
		\end{cases}
		$$
	\end{enumerate}
	Therefore, for any $d\,\in \,\N_\d$, $$\langle  \widehat{w}^{d}_{jl},\, \widehat{w}^{d}_{jl'} \rangle \, =\, 0$$
	when $l \,\neq \,l'$ and $0 \,\leq \,l,\,l' \,\leq\, d-1$.
\end{lemma}

The next corollary, which closely follows \cite[Lemma 9.3.1]{CoM}, gives a direct correspondence between the signature
of $(\cdot, \, \cdot)_{d}$ on $L(d-1)$ and the signature of $\langle \cdot ,\, \cdot \rangle$ on $M(d-1)$
when both $\<>$ and $(\cdot ,\, \cdot)_{d}$ have signatures. 
In part (3) of the corollary we record a correct version of a result in \cite[Lemma 9.3.1]{CoM}.
Recall that $\O^1_\d$ and $\O^3_\d$ are as in \eqref{Od1-Od3}. 

\begin{corollary}\label{c-m9.3.1}
	Let  $\<>$ be a $\epsilon$-$\sigma$ Hermitian form on $V$. Assume that $\epsilon\,=\,1$, that is,
	the form $\<>$ is symmetric or Hermitian.
	\begin{enumerate}
		\item  If $d \,\in\, \E_\d$, then the signature of $\langle \cdot ,\, \cdot \rangle$ on $M(d-1)$ is 
		$\big(\dim_\D M(d-1)/2  , \,\dim_\D M(d-1)/ 2\big)$.
		
		\item  If $d \,\in\, \O^1_\d$, and $(p_d,\, q_d)$ is the signature
		of $(\cdot  ,\,\cdot)_{d}$, then the signature of $\langle \cdot \, , \cdot \rangle$ on $M(d-1)$ is  
		$$ ((\dim_\D M(d-1) + p_d-q_d )/ 2 , \, (\dim_\D M(d-1) + q_d -p_d) / 2).$$
		
		\item If $d \,\in\, \O^3_\d$, and $(p_d, \,q_d)$ is the
		signature of $(\cdot  ,\,\cdot)_{d}$, then the signature of $\langle \cdot\, , \cdot \rangle$ on $M(d-1)$ is  
		$$( (\dim_\D M(d-1)  + q_d -p_d )/ 2, \, (\dim_\D M(d-1) + p_d -q_d) / 2).$$ 
	\end{enumerate}
\end{corollary}

\begin{proof}
	This follows directly from Lemmas \ref{orthogonal-basis-R}, \ref{orthogonal-basis-C} and \ref{orthogonal-basis-H}.
\end{proof}

\begin{remark}\label{CM-correction}
	We will now point out an error in \cite[p. 139, Lemma 9.3.1]{CoM}, and also explain why
	the definition of $m^d_{id}$ in the case of $ d \,\in \,\O^3_\d$ as in \ref{yd-def2} (in
	Section \ref{sec-partition-Young-diagram}) is different from that in the case of $d \,\in\,
	\E_\d \cup \O^1_\d$. Let $V$ be a right $\D$-vector space, and
	let $\<>$ be a Hermitian (respectively symmetric) form on $V$ if $\D \,=\, \H, \, \C$
	(respectively, $\D \,=\, \R$). Take a $\s\l_2 (\R)$-triple
	$\{X, \,H,\, Y\} \,\subset\, \s\u( V, \<>)$. Note that if $d \,\in\, \O^3_\d$, then
	the form $( \cdot,\, \cdot )_d$ in \eqref{new-form} is Hermitian
	(respectively, symmetric) when $\D \,=\, \H,\, \C$ (respectively, $\D\,=\, \R$).
	Let $(p_d, \,q_d)$ be the signature of $(\cdot,\, \cdot )_d$ when $d\,\in\, \O_\d^3$. Corollary \ref{c-m9.3.1}(3) says that
	the signature of $\<>$ restricted to $M(d-1)$ is
	$$((\dim_\D M(d-1)  + q_d -p_d )/2 ,\, (\dim_\D M(d-1) + p_d -q_d )/2 )$$ when $d \,\in\, \O^3_\d$. 
	Set the signs in first column of the matrix $(m^d_{ij})$ as in \ref{yd-def1}, and thus define
	$m^d_{i1} \,= \,+1$ when $1 \,\leq \,i \,\leq\, p_d $, and define $m^d_{i1} \,=\, -1$
	when $p_d \,<\, i \,\leq\, t_d$.
	
	However, in the case of $d \,\in\, \O^3_\d$, if we, following \cite[p.~139, Lemma 9.3.1]{CoM}, define $$m^d_{ij} \,=\,
	(-1)^{j+1}m^d_{i1}$$ for $1\,<\,j \,\leq \,d$,
	then it can be easily verified that $$({\rm sgn}_+(m^d_{ij}),\, {\rm sgn}_-(m^d_{ij})) \,= \,
	((\dim_\D M(d-1)  + p_d - q_d )/2 ,\, (\dim_\D M(d-1) + q_d -p_d )/2 )\, .$$
	Thus, if $d  \,\in\, \O^3_\d$ and $p_d \,\neq\, q_d$, then appealing to 
	Corollary \ref{c-m9.3.1}(3) we see that the signature of $\<>$ restricted to $M(d-1)$ does not
	coincide with $({\rm sgn}_+(m^d_{ij}),\, {\rm sgn}_-(m^d_{ij}))$.
	This shows that the second statement of \cite[p.~139, Lemma 9.3.1]{CoM} is not true
	when $d\,\in\, \O^3_\d$ and $p_d \,\neq\, q_d$ (this means that $r \,\equiv\, 2\, \pmod{4}$ in the notation of
	\cite[p. 139, Lemma 9.3.1]{CoM}).
	Recall that in \ref{yd-def2} (see Section \ref{sec-partition-Young-diagram}), when $d \,\in\, \O^3_\d$
	we have defined $m^d_{ij} \,= \,(-1)^{j+1}m^d_{i1}$ when $1\,<\,j \,\leq\, d-1$ while $m^d_{id}\,:=\, -m^d_{i1}$. Using the
	definitions of $m^d_{i1}$ as above
	we have that $$({\rm sgn}_+(m^d_{ij}),\, {\rm sgn}_-(m^d_{ij}))\,=\, 
	((\dim_\D M(d-1)  + q_d -p_d )/2 ,\, (\dim_\D M(d-1) + p_d -q_d )/2)\, .$$ 
	Thus, if we define $m^d_{ij}$ as in  \ref{yd-def1} and \ref{yd-def2}, then
	the signature of $\<>$ on $M(d-1)$ does coincide with $({\rm sgn}_+(m^d_{ij}),\, 
	{\rm sgn}_-(m^d_{ij}))$ for $d\,\in\, \N_\d$.
\end{remark}

\end{document}